\documentclass[10pt]{amsart}
\usepackage{color}
\usepackage{verbatim}
\usepackage{mathtools}
\usepackage{amstext}
\usepackage{amsthm}
\usepackage{amssymb}
\usepackage{kotex}
\usepackage[unicode=true,
 bookmarks=false,
 breaklinks=false,pdfborder={0 0 1},colorlinks=false]
 {hyperref}
\hypersetup{
 colorlinks,linkcolor={red!60!black},citecolor={green!60!black},urlcolor={blue!60!black}}

\makeatletter
\numberwithin{equation}{section}
\numberwithin{figure}{section}
\theoremstyle{plain}

\theoremstyle{remark}

\theoremstyle{plain}

\theoremstyle{plain}

\usepackage{lineno}
\usepackage{footnote}
\usepackage{enumitem}
\usepackage{setspace}
\usepackage{tikz}
\usepackage{cite}
\makesavenoteenv{tblr}

\newtheorem{theorem}{Theorem}[section]
\newtheorem{proposition}[theorem]{Proposition}
\newtheorem{lemma}[theorem]{Lemma}

\theoremstyle{remark}
\newtheorem{remark}[theorem]{\protect\remarkname}

\providecommand{\propositionname}{Proposition}
\providecommand{\remarkname}{Remark}
\providecommand{\theoremname}{Theorem}

\providecommand{\lemmaname}{Lemma}

\newcommand{\nut}[1]{\overset{\circ}{#1}}

\global\long\def\R{\mathbf{\mathbb{R}}}%
\global\long\def\C{\mathbf{\mathbb{C}}}%
\global\long\def\Im{\mathrm{Im}}%
\global\long\def\Re{\mathrm{Re}}%
\global\long\def\norm#1{\|#1\|}%
\global\long\def\ol#1{\overline{#1}}%
\global\long\def\tilde#1{\widetilde{#1}}%
\global\long\def\calN{\mathcal{N}}%
\global\long\def\bbR{\mathbf{\mathbb{R}}}%
\global\long\def\bbC{\mathbf{\mathbb{C}}}%
\global\long\def\bbN{\mathbf{\mathbb{N}}}%
\global\long\def\calR{\mathcal{R}}%
	\global\long\def\calP{\mathcal{P}}%
\global\long\def\gmm{\gamma}%
\global\long\def\td#1{\widetilde{#1}}%

\makeatother

\providecommand{\lemmaname}{Lemma}
\providecommand{\propositionname}{Proposition}
\providecommand{\remarkname}{Remark}
\providecommand{\theoremname}{Theorem}

\title[Blow-up for half-wave equation]{Blow-up construction and instability for mass-critical half-wave equation with slightly superthreshold mass}

\author{Taegyu Kim}
\email{k1216300@kias.re.kr}
\address{(current) Korea Institute for Advanced Study, 80 Hoegi-ro, Dongdaemun-gu, Seoul 02455, Korea (previous) Department of Mathematical Sciences, Korea Advanced Institute of Science and Technology, 291 Daehak-ro, Yuseong-gu, Daejeon 34141, Korea}

\author{Soonsik Kwon}
\email{soonsikk@kaist.edu}
\address{Department of Mathematical Sciences, Korea Advanced Institute of Science
and Technology, 291 Daehak-ro, Yuseong-gu, Daejeon 34141, Korea}

\author{Jeongheon Park}
\email{jse05002@kaist.ac.kr}
\address{Department of Mathematical Sciences, Korea Advanced Institute of Science
and Technology, 291 Daehak-ro, Yuseong-gu, Daejeon 34141, Korea}

\begin{document}

\begin{abstract}
We study the blow-up dynamics for the $L^2$-critical focusing half-wave equation on the real line, a nonlocal dispersive PDE arising in various physical models. As in other mass-critical models, the ground state solution becomes a threshold between the global well-posedness and the existence of a blow-up. The first blow-up construction is due to Krieger, Lenzmann and Rapha\"el, in which they constructed the minimal mass blow-up solution at the threshold mass. In this paper, we construct finite-time blow-up solutions with mass slightly exceeding the threshold. This is inspired by similar results in the mass-critical NLS by Bourgain and Wang, and their instability by Merle, Rapha\"el and Szeftel. We exhibit a blow-up profile driven by the rescaled ground state, with a decoupled dispersive radiation component. We rigorously describe the asymptotic behavior of such solutions near the blow-up time, including sharp modulation dynamics. Furthermore, we demonstrate the instability of these solutions by constructing non-blow-up solutions that are arbitrarily close to the blow-up solutions. The main contribution of this work is to overcome the nonlocal setting of half-wave and to extend insights from the mass-critical NLS to a setting lacking pseudo-conformal symmetry.

\end{abstract}

\maketitle

\tableofcontents{}

\section{Introduction}

\subsection{Setting of the problem} We consider the $L^2$-critical focusing half-wave equation on $\bbR$:
\begin{equation}\label{half-wave}\tag{HW}
    \begin{cases}
        &i\partial_t u = Du - |u|^2u, \quad u:I\times \mathbb{R} \rightarrow \mathbb{C},
        \\
        &u(t_0)=u_0.
    \end{cases}
\end{equation}
Here, $I \subset \mathbb{R}$ is an interval that contains $t_0 \in I$, and $D$ denotes the pseudodifferential operator defined by the Fourier multiplier with symbol $|\xi|$. Nonlocal evolution equations such as \eqref{half-wave} arise in a range of physical contexts, including the continuum limits of lattice models \cite{KLS2013CMP}, models for wave turbulence \cite{Physics1997JNS,Physics2001PhysD}, and gravitational collapse \cite{Physics2007CPAM,FrohlichLenzmann2007CPAM}.

Apart from its physical origins, the half-wave equation \eqref{half-wave} has attracted mathematical interest due to its nonlocal structure and critical scaling. At the same time, the absence of classical symmetries such as Lorentz, Galilean, or pseudo-conformal invariance presents substantial analytical challenges. A central problem in its analysis is the formation of singularities in finite time. In this direction, Krieger, Lenzmann, and Rapha\"el \cite{KLR2013ARMAhalfwave} constructed minimal-mass blow-up solutions and identified that the ground state is a sharp threshold dividing line between global existence and finite-time blow-up. Besides the minimal-mass blow-up solution, the dynamics above this threshold remains poorly understood. In this work, we construct other finite-time blow-up solutions having slightly supercritical mass. These are referred to as \textit{Bourgain--Wang type solutions}. Our study is inspired by similar works in the mass-critical NLS. In there, Bourgain and Wang \cite{BourgainWang1997} originally constructed solutions that blow up at the pseudo-conformal rate. Their instability was shown in \cite{MRS2013AJM}. Inspired by the results in NLS, our goal is to construct blow-up solutions with the same blow-up scenario as the minimal blow-up solutions in \cite{KLR2013ARMAhalfwave} and to show their instability.

We begin with elementary facts, symmetries, and conservation laws. Equation \eqref{half-wave} is invariant under the $L^2$-critical scaling
\begin{align*}
    u(t,x)\mapsto\lambda^{-\frac{1}{2}}u(\lambda^{-1}t,\lambda^{-1}x),\quad(\lambda>0),
\end{align*}
which preserves the $L^2$-norm. It also enjoys space and time translation symmetry
\begin{align*}
    u(t,x)\mapsto u(t+t_{0},x+x_{0}),\quad((t_{0},x_{0})\in\mathbb{R}\times \bbR),
\end{align*}
and the phase rotation symmetry;
\begin{align*}
    u(t,x)\mapsto e^{i\gamma}u(t,x),\quad(\gamma\in\mathbb{R}).
\end{align*}
Moreover, associated conservation laws include the conservation of mass, momentum, and energy: 
\begin{equation}
    \begin{aligned}
        M(u) &= \int_{\mathbb{R}} |u(t,x)|^2 dx,
    \\
    P(u) &= -\int_{\mathbb{R}} i\nabla u(t,x)\overline{u}(t,x)dx,  
    \\
    E(u) &= \frac{1}{2}\int_{\mathbb{R}} |D^{1/2}u(t,x)|^2 dx - \frac{1}{4}\int_{\mathbb{R}}|u(t,x)|^4 dx. 
    \end{aligned} \label{eq:conservation laws}
\end{equation}
Since the scaling symmetry of \eqref{half-wave} preserves the $L^2$-norm, the equation is referred to as $L^2$-critical, or mass-critical. This aligns it with other well-studied mass-critical models, such as the nonlinear Schrödinger equation \eqref{NLS}, with which it shares several structural features, including critical scaling and the conservation of mass, energy and momentum. However, the half-wave equation \eqref{half-wave} differs in important ways: It lacks Lorentz, Galilean, or pseudo-conformal invariance.

The \emph{ground state} solution plays a pivotal role in the dynamics of \eqref{half-wave} and serves as a threshold for global behavior. It is characterized as the unique positive solution $Q\in H^{1/2}(\bbR)$ to the nonlinear  elliptic equation
\begin{equation*}
    DQ+Q-Q^3 =0.
\end{equation*}
Although the existence of such a solution is obtained by standard variational arguments, the uniqueness and qualitative properties of $Q$ are nontrivial due to the nonlocal nature of the equation. The uniqueness was established by Frank and Lenzmann \cite{FrankLenzmann2013Acta}. The ground state $Q$ also provides the best constant in the associated Gagliardo--Nirenberg inequality
\begin{align*}
    \|u\|_{L^4}^4\leq C \||D|^{1/2}u\|_{L^2}^2\|u\|_{L^2}^2.
\end{align*}
As a consequence, the energy functional admits the coercive lower bound for a subthreshold mass
\begin{align*}
    E(u) \geq \frac{1}{2}\norm{u}_{\dot{H}^{1/2}}^2\left[1-\frac{\norm{u}_{L^2}^2}{\norm{Q}_{L^2}^2}\right].
\end{align*}
More precisely, by combining this with the conservation of mass and energy, one deduces that for any initial data $u_0\in H^{1/2}(\bbR)$ with $\|u_0\|_{L^2} < \|Q\|_{L^2}$, the solution $u(t)$ exists globally in time and satisfies the uniform bound
\begin{align*}
    \|u(t)\|_{H^{1/2}(\bbR)}\le C(M(u_0),E(u_0)).
\end{align*}
On the threshold $M(u)=M(Q)$, the minimal blow-up solution is constructed in \cite{KLR2013ARMAhalfwave}. However, the uniqueness of the minimal blow-up solution is not known, and so the threshold dynamics is barely understood.   

\subsection{Main result}
In this paper, we look at slightly superthreshold mass solutions. More precisely, we construct Bourgain--Wang type finite-time blow-up solutions to \eqref{half-wave} and prove their instability.
Our main result continues the line of investigation initiated in the mass-critical NLS in \cite{BourgainWang1997, MRS2013AJM}. We are ready to state our main theorem.
\begin{theorem}[Bourgain--Wang type solutions]\label{main result 1}
There exists a small universal constant $\alpha^*>0$ such that the following holds. Fix $(E_0,P_0)\in\mathbb{R}_+\times \mathbb{R}$. Let $z^*\in H^{1/2+}$ be an asymptotic profile constructed in Proposition ~\ref{construction of asymptotic profile} that satisfies
$\|z^*\|_{H^{1/2}}<\min\{\alpha^*,E_0\}$. 
Then, there exists a corresponding blow-up solution
\begin{equation*}
    u \in \mathcal{C}\big([t_1,0);H^{1/2+}(\mathbb{R})\big),\qquad t_1<0,
\end{equation*}
and parameters      $(\lambda^*,x^*,\gamma^*)\in\mathbb{R}_+\times\mathbb{R}\times\mathbb{R}$ satisfying 
\begin{equation}
    u(t,x)-\frac{1}{\sqrt{\lambda^*}\,|t|}\,
    Q\!\left(\frac{x-x^* t^3}{\lambda^* t^2}\right)
    e^{i\gamma^*/t} \rightarrow z^*,\quad \text{in } L^2,
\end{equation}
as $t \to 0^-$, and
\[
    \|u\|_{L^2}^2 = \|Q\|_{L^2}^2 + \|z^*\|_{L^2}^2,
    \quad 
    E(u(t))=E_0,\quad P(u(t))=P_0.
\]
\end{theorem}

Although the statement above only describes the leading order asymptotics, the blow-up solution actually admits a refined $H^{1/2}$–decomposition with explicit modulation parameters.

\begin{remark}[Sharp description for blow-up solutions]
    Given $(E_0,P_0) \in \mathbb{R}_+ \times \mathbb{R}$ and the corresponding asymptotic profile $z^*$, there exists a universal constant $C>0$ such that $u$ satisfies, for $t \in [t_1,0)$,
    \begin{equation}\label{H1/2 decomposition}
        \begin{aligned}
            &\norm{u(t) - S_{\mathrm{HW}}(t) - z(t)}_{L^2} \le C\,\alpha^* |t|^{5/4},\\
            &\norm{u(t) - S_{\mathrm{HW}}(t) - z(t)}_{H^{1/2}} \le C\,\alpha^* |t|^{1/4}.
        \end{aligned}
    \end{equation}
    Here $z(t)$ denotes the flow of \eqref{half-wave} with initial data $z^*$, and
    \begin{equation}\label{S HW def}
        S_{\mathrm{HW}}(t,x)
        \coloneqq \frac{1}{\lambda_0^{1/2}(t)}
        Q_{\mathcal{P}(b_0(t),\nu_0(t),0)}
        \left(\frac{x-\nut{x}_0(t)}{\lambda_0(t)}\right)
        e^{i\gamma_0(t)},
    \end{equation}
    where $Q_{\mathcal{P}(b,\nu,0)}$ is the modified profile constructed in Proposition~\ref{Singular profile}.
    The modulation parameters $(\lambda_0,\gamma_0,\nut{x}_0,b_0,\nu_0)(t)$ satisfy, for $t \in [t_1,0)$,
    \begin{equation}\label{modulation asymptotics}
        \begin{aligned}
            \lambda_0(t) &= \lambda^* t^2 + \mathcal{O}(t^4), \qquad
            b_0(t) = -2\lambda^* t + \mathcal{O}(t^3),\\
            \nut{x}_0(t) &= x^* t^3 + \mathcal{O}(t^5), \qquad
            \nu_0(t) = 3x^* t^2 + \mathcal{O}(t^4),\\
            \gamma_0(t) &= -\frac{1}{\lambda^* t} + \mathcal{O}(t),
        \end{aligned}
    \end{equation}
    where the constants $\lambda^*$ and $x^*$ depend on $(E_0,P_0)$ and are given by
    \[
        \lambda^* = e_1 \bigl(E_0 - E(z^*)\bigr), \qquad
        x^* = p_1 \lambda^* \bigl(P_0 - P(z^*)\bigr),
    \]
    with $e_1$ and $p_1$ being universal constants.

    We also emphasize that the constructed asymptotic profile $z^*$ is highly degenerate at $x=0$: all its low-order derivatives vanish up to a sufficiently high order. This spatial flatness at the origin is the hallmark of Bourgain--Wang type solutions, reflecting that the blow-up and radiation parts are essentially decoupled as $t\to0$ in $H^{1/2}$.
\end{remark}

Together with its construction, we obtain the following instability property.

\begin{theorem}[Instability of Bourgain--Wang type solutions]\label{main result 2}
    Given $(E_0,P_0)\in\mathbb{R}_+\times\mathbb{R}$, 
    let $u \in \mathcal{C}([t_1,0);H^{1/2+})$ 
    be the blow-up solution from Theorem~\ref{main result 1}
    corresponding to the asymptotic profile $z^*$. Then there exists a sequence of solutions $\{u_n\}_{n\in\mathbb{N}}$ to \eqref{half-wave} with the following properties:
    \[
       \text{each } u_n \in \mathcal{C}([t_1,-t_1];H^{1/2+}(\mathbb{R})) 
    \quad\text{and does not blow up at } t=0,\quad \text{for all }n\in\mathbb{N},
    \]
    and for each $t\in[t_1,0)$,
    \[
        u_n(t)\rightarrow u(t)\quad\text{in }H^{1/2}\ \text{as }n\rightarrow\infty.
    \]
\end{theorem}

We provide a more explicit description of the family of non-blow-up solutions that generate the sequence $\{u_n\}$. They can be regarded as perturbations of the Bourgain--Wang solution along a direction associated with the phase rotation of $e^{it} Q$.

\begin{remark}[On the construction of the approximating sequence in Theorem~\ref{main result 2}]
    The sequence of solutions $\{u_n\}$ in Theorem~\ref{main result 2} is extracted from a one-parameter family 
    of non-blow-up solutions $\{u_{\eta}\}_{\eta \in (0,\eta^*)}$ ($0<\eta^*\ll1$) to \eqref{half-wave}, 
    by choosing a sequence $\eta_n \to 0$ as $n \to \infty$.

    More precisely, $u_\eta$ admits the approximate description
    \[
        u_{\eta}(t,x) \approx \frac{1}{\lambda_{\eta}(t)^{1/2}}\,
        Q_{\mathcal{P}(b_\eta(t),\nu_{\eta}(t),\eta)}
        \left(\frac{x-\nut{x}_{\eta}(t)}{\lambda_{\eta}(t)}\right)
        e^{i\gmm_{\eta}(t)} \;+\; z(t,x),
    \]
    where the modulation parameters 
    $(\lambda_\eta,\nut{x}_{\eta},\gmm_\eta,b_\eta,\nu_\eta)(t)$ 
    satisfy, for all $\eta \in (0,\eta^*)$ and $t \in [t_1,-t_1]$,
    \begin{align*}
        \lambda_\eta(t) \sim t^2 + \eta, 
        \quad \nut{x}_{\eta}(t) &\sim t^3+\eta t, \quad
        \gmm_\eta(t) \sim \frac{1}{\sqrt{\eta}}
        \left(\arctan\!\left(\frac{t}{\sqrt{\eta}}\right)+\frac{\pi}{2}\right), \\
         b_{\eta}(t) &\sim -t,\quad 
        \nu_{\eta}(t) \sim t^2+\eta. 
    \end{align*}
    In particular, at $t=0$ we have
    \[
        \lambda_\eta(0) \sim \eta, \qquad 
        \gmm_\eta(0) \sim \frac{1}{\sqrt{\eta}},
    \]
    so that $u_\eta$ does not blow up at $t=0$.

    The introduction of the parameter $\eta$ prevents blow-up by slightly decreasing 
    the mass of the core profile: indeed, 
    \[
        \norm{Q_{\mathcal{P}(b_\eta,\nu_\eta,\eta)}}_{L^2} 
        < \norm{Q}_{L^2} \qquad (\eta > 0),
    \]
    while the core part and the radiation part $z^*$ remain almost decoupled.
    This family $\{u_\eta\}$ thus provides the natural approximation of the 
    Bourgain--Wang type blow-up solution $u$ obtained in Theorem~\ref{main result 1}. 
\end{remark}

\subsection{Previous results}\label{previous results}
The half-wave equation \eqref{half-wave} is a nonlocal, $L^2$-critical non-dispersive model lacking pseudo-conformal invariance. A key structural foundation is the ground state \(Q\), whose existence and uniqueness (up to symmetries) for the elliptic profile \(D Q + Q = Q^3\) in one dimension were rigorously established by Frank and Lenzmann~\cite{FrankLenzmann2013Acta} via variational and spectral methods.

Krieger, Lenzmann, and Rapha\"{e}l~\cite{KLR2013ARMAhalfwave} constructed finite-time blow-up solutions at the threshold $\norm{u_0}_{L^2} = \norm{Q}_{L^2}$ in one dimension, showing that the blow-up rate satisfies $\norm{D^{1/2}u(t)}_{L^2} \sim 1 / |t|$ as $t \to 0$. For higher dimensions, see \cite{Georgiev2Dhalfwaveblowup, Georgiev3Dhalfwaveblowup}.

We also mention the work of G\'{e}rard, Lenzmann, Pocovnicu, and Rapha\"{e}l~\cite{GLPR2018annPDE}, who constructed a two-soliton solution exhibiting a transient turbulent regime with dramatic $H^1$-growth from arbitrarily small $L^2$-mass, followed by a long-time saturation phase. While their setting differs from the minimal blow-up regime, their analysis reveals rich dynamical behavior for solutions below the threshold.

For other results in this model, see: weak stability of multi‑soliton sums~\cite{Li2024WeakStability}, multi‑bubble blow‑up~\cite{CaoSuZhang2022}, inhomogeneous mass‑critical blow‑up~\cite{Li2022Inhomogeneous}, nondispersive/traveling waves and failure of small‑data scattering~\cite{BellazziniGeorgievVisciglia2018}, and ill‑posedness of the cubic half-wave or related fractional NLS~\cite{ChoffrutPocovnicu2018}.

In contrast, the present work focuses on the super-threshold regime: we construct a family of finite-time blow-up solutions with mass strictly above $\norm{Q}_{L^2}$, exhibiting a non-trivial asymptotic part decoupled from the blow-up core. This constitutes the first Bourgain--Wang type construction for \eqref{half-wave}, and we further prove the instability of these solutions, as a continuation of the framework developed by Krieger–Lenzmann–Rapha\"{e}l.

For a comparison, we recall results for the $L^2$-critical nonlinear Schr\"{o}dinger equation:
\begin{equation}\label{NLS}\tag{NLS}
i \partial_t u + \Delta u + |u|^\frac{4}{d} u = 0,\quad  (t,x) \in I \times \bbR^d.
\end{equation}
In two dimensions, Bourgain and Wang~\cite{BourgainWang1997} originally constructed blow-up solutions at the superthreshold with prescribed asymptotic profile $z^*$, using the pseudo-conformal symmetry. The instability of these Bourgain--Wang solutions was later rigorously established by Merle, Rapha\"{e}l, and Szeftel~\cite{MRS2013AJM}, where it was shown that the instability occurs transversally to the pseudo-conformal manifold. In particular, the set of initial data splits into two open sets: one leading to global scattering and the other to log-log type stable blow-up. As a similar result, for the self-dual Chern-Simons-Schr\"{o}dinger equation, the construction of pseudo-conformal blow-up solutions and their instability were established in \cite{KimKwon2019}. 

The pseudo-conformal symmetry plays an essential role in the analysis of~\cite{MRS2013AJM, KimKwon2019}, allowing a precise control of the pseudo-conformal phase and enabling a refined modulation approach near the blow-up regime. Such a structure is absent in the half-wave setting, presenting analytical challenges for constructing and analyzing analogous solutions.

Our main contribution is to demonstrate that analogous Bourgain–Wang type dynamics—and their instability—can be realized for the nonlocal half-wave equation, which lacks pseudo-conformal symmetry.

\subsection{Discussions on main results}
\ \vspace{5bp}

\emph{1. Method and novelties.}
We use the modulation analysis for the backward construction. The overall strategy of our finite-time blow-up construction is based on \cite{MRS2013AJM}. We construct a one-parameter family of near-blow-up solutions that exhibit the instability mechanism of the blow-up solutions. More precisely, we approximate the singular part of the solution and decouple it from the radiation part by imposing a degeneracy condition on the asymptotic profile. The singular part follows a soliton profile, which is modified by corrector terms resulting from tail computations \cite{RaphaelRodnianski2012}. However, there are major difficulties in our setting: \textit{the nonlocal nature of \eqref{half-wave}} and \textit{the lack of the pseudo-conformal symmetry}.

One of our novelties lies in identifying the correct instability mechanism and constructing a modified singular profile $Q_{\mathcal{P}}$ in the absence of pseudo-conformal symmetry. Unlike the \eqref{NLS} case, where the pseudo-conformal phase explicitly provides the unstable direction in terms of ODE-level, we rely on the spectral structure of the linearized operator $L_Q$. In particular, we detect a mass-deficient direction in its generalized kernel, along which the $L^2$ mass of the singular part becomes slightly smaller than that of $Q$. This recovers, at a spectral level, the same type of ejection mechanism as in the \eqref{NLS} case. Based on this observation, we identify the direction in which the blow-up profile constructed in \cite{KLR2013ARMAhalfwave} becomes unstable and construct a corrected profile $Q_{\mathcal{P}(b,\nu,\eta)}$. Compared to the minimal blow-up solution \cite{KLR2013ARMAhalfwave}, in our setting, there is a lack of smallness coming from the threshold condition. Because of this, we have to find a refined ODE dynamics of the modulation parameters and a modified profile.  

Another novelty lies in the construction of the asymptotic profile $z^*$. The nonlocal nature of the differential operator $D$ creates a major difficulty: spatial degeneracy near the origin of the initial data
does not directly propagate into temporal degeneracy due to interaction with far from the origin. Hence, only such an assumption on $z^*$ gives a strongly harmful interaction with the main singular part. To enforce a sufficiently high space-time degeneracy at $(0,0)$, and decouple from the singular part, we impose additional codimension conditions to $z^*$ on the Fourier side and use a Borsuk--Ulam type topological argument to guarantee the existence of suitable initial data.

Finally, we emphasize that our construction is carried out in the
non-radial setting with a prescribed momentum, which is technically
delicate due to the absence of Galilean symmetry. 

\emph{2. Comparison with \eqref{NLS}.}
As mentioned in Section~\ref{previous results}, our work is motivated from similar results in mass-critical NLS \cite{BourgainWang1997,MRS2013AJM}. Here, we briefly comment on some difference to the \eqref{half-wave}. 
In \eqref{NLS}, subthreshold and threshold dynamics are well-understood. If $M(u)<M(Q_{\text{NLS}})$, then the solution is global and scatters \cite{KenigMerle2006Invent,Dodson2012JAMS}. On the threshold $M(u)=M(Q_{\text{NLS}})$, the dynamics is completely classified, i.e., soliton, pseudo-conformal blow-up solution, or scattering \cite{Merle1993Duke,Dodson2024AnalPDEd1,Dodson2024AnnPDEd2}. In the half-wave \eqref{half-wave}, there is no dispersion even for the small initial data. In fact, there are traveling wave solutions whose mass is strictly smaller than $Q$. (See \cite[Theorem 1.1]{KLR2013ARMAhalfwave}.) See also \cite{GLPR2018annPDE,BGLV2019CMP}. The threshold dynamics for \eqref{half-wave} is much less known. Even the uniqueness of the minimal blow-up solution is still open \cite{KLR2013ARMAhalfwave}. Above the threshold, there are crucial progress in \eqref{NLS}. Merle and Rapha\"{e}l have shown that solutions with negative energy and slightly superthreshold mass should blow up with \emph{log-log corrected} self-similar rate \cite{MerleRaphael2005AnnMath,MerleRaphael2003GAFA,Raphael2005MathAnnalen,MerleRaphael2004Invent,MerleRaphael2006JAMS,MerleRaphael2005CMP}. There are blow-up solutions with pseudo-conformal rate \cite{BourgainWang1997,MRS2013AJM}. However, in \eqref{half-wave} there are no blow-up results prior to the current result. In this case, because the virial identity is absent, Glassey's convexity argument fails to demonstrate a blow-up.
On the other hand, in \eqref{NLS} non-blow-up solutions in the constructed family of \cite{MRS2013AJM} scatter in both time directions using the stability of scattering and the pseudo-conformal symmetry. However, in \eqref{half-wave} we have no further information outside a time interval $[t_1,-t_1]$.
Finally, we note that Theorem~\ref{main result 2} is slightly weaker than that for \eqref{NLS} in \cite{MRS2013AJM}, though our result suffices to claim the instability. This is connected to a uniqueness issue. See more detail in Remark~\ref{remark:instability}.

\subsection{Strategy of the proof}

Our proof is inspired by the approach in \cite{BourgainWang1997, MRS2013AJM}, where Bourgain–Wang type blow-up solutions are constructed and their instability is established via modulation analysis. We adapt this strategy to the half-wave equation, building on the blow-up profile construction in \cite{KLR2013ARMAhalfwave}. The adaptation to the half-wave equation presents several key challenges, including the absence of pseudo-conformal and Galilean symmetries as well as the nonlocal nature of the equation.

\textbf{1. Approximate solution.}
We aim to construct a blow-up solution of the form
\[
    u(t,x) \approx \frac{1}{\lambda^{1/2}(t)}Q\left(\frac{x-\nut{x}(t)}{\lambda(t)}\right)e^{i\gmm(t)} + z(t,x),
\]
where
\[
    \lambda(t) \sim t^2,\quad \nut{x}(t) \sim t^3,\quad \gamma(t)\sim -\frac{1}{t}.
\]
This corresponds to concentration of the singular-core mass at the
origin:
\[
    \frac{1}{\lambda(t)}
    \left|Q\left(\frac{x-\nut{x}(t)}{\lambda(t)}\right)\right|^2 dx
    \rightharpoonup
    \|Q\|_{L^2}^2\delta_0
\]
in the sense of measures. At the level of the solution itself,
\[
    u(t)\rightharpoonup z^*
    \qquad\text{weakly in }L^2(\mathbb R).
\]

To achieve decoupling of the nonlinear interaction between the blow-up profile and the radiation part $z$ in the regime $|t|\ll 1$, we construct an asymptotic profile $z^*$ so that $z(t,x)$ exhibits sufficient degeneracy at $(0,0)$ in space-time. However, due to the nonlocal nature of the operator $D = |\nabla|$, a simple pointwise degeneracy of $z^*$ at $x=0$ is insufficient to ensure decoupling, due to long-range interactions in the convolution kernel.

To overcome this difficulty, we construct $z^*$ on the Fourier side. For a fixed Schwartz function $f \in \mathcal{S}(\mathbb{R})\setminus\{0\}$, we define
\begin{equation}\label{proof strategy : asymptotic profile ansatz}
    \widehat{z_f^*}(\xi) \coloneqq \widehat{f}(\xi)(d_1\xi+\cdots + d_N\xi^N),
\end{equation}
where $\|f\|_{H^{A_0}} < \min\{\alpha^*,E_0\}$ for large universal constant $A_0 \gg 1$.

Our goal is to choose the coefficients $\{d_j\}$ so that for sufficiently large $m \in \mathbb{N}$, the solution $z(t,x)$ to the half-wave equation with initial data $z(0) = z_f^*$ satisfies the high-order degeneracy condition
\begin{equation}\label{proof strategy : degeneracy}
    \left((\partial_t^{k_1}\nabla^{k_2}z)(0,0),\partial_t^{k_3}\nabla^{k_4}(Dz)(0,0)\right) = (0,0)
\end{equation}
for all $k_1+k_2 \leq m$ and $k_3+k_4 \leq m-1$.

By the local well-posedness theory in $H^{\frac{1}{2}+}$ (Theorem~\ref{Cauchy theory}) and the explicit Fourier-side ansatz \eqref{proof strategy : asymptotic profile ansatz}, the degeneracy condition \eqref{proof strategy : degeneracy} translates into a system of algebraic equations for the coefficients $\{d_j\}_{j=1}^N$. These equations are odd polynomials in $\{d_j\}$, arising from the cubic nonlinearity in the half-wave equation.

Taking $N = N(m) > (m+1)^4$, the system has a nontrivial solution on the unit sphere $\mathbb{S}^{2N-1} \subset \mathbb{C}^N$, which is guaranteed by the Borsuk–Ulam theorem. Hence, we obtain coefficients $\{d_j\}$ such that \eqref{proof strategy : degeneracy} holds. Furthermore, the $H^{\frac{1}{2}+}$ norm of the radiation remains uniformly controlled:
\[
    \|z(t)\|_{H^{\frac{1}{2}+}} < \alpha^*, \quad \text{for all } |t| \ll 1.
\]
The construction of the asymptotic profile $z_f^*$ via a \emph{Fourier-side degeneracy mechanism}, together with the use of topological arguments (Borsuk–Ulam) to guarantee the existence of coefficients satisfying high-order cancellation, is a key novelty of our approach. In particular, this method circumvents the lack of pseudo-conformal symmetry and provides a robust framework for decoupling radiation in the nonlocal setting of the half-wave equation.

\textbf{2. Blow-up profile with unstable direction}
With the prescribed asymptotic profile $z_f^*$, we simultaneously prove Theorem~\ref{main result 2} and Theorem~\ref{main result 1} by constructing a family of solutions $\{u_\eta(t)\}_{\eta \in (0,\eta^*)}$ to \eqref{half-wave} on $[t_1, -t_1]$, uniformly in $\eta$. This approach is referred to as \emph{backward construction}, and the desired blow-up solution $u(t)$ is then obtained by taking the limit $\eta \to 0^+$.

We consider initial data of the form
\begin{equation}\label{proof strategy : initial data of u_eta}
    u_{\eta}(0,x) = \frac{1}{\lambda_{\eta}^{1/2}(0)}Q_{\mathcal{P}(b_\eta(0),\nu_{\eta}(0),\eta)}\left(\frac{x - \nut{x}_\eta(0)}{\lambda_\eta(0)}\right)e^{i\gamma_\eta(0)} + z_f^*(x),
\end{equation}
where $|(\lambda_\eta, b_\eta, \nu_\eta, \nut{x}_\eta, \gamma_\eta)(0)| \sim |(\eta, 0, \eta, 0, \eta^{-1/2})|$ are chosen in Lemma~\ref{approximated dynamics}, and $Q_{\mathcal{P}(b,\nu,\eta)}$ is a modified blow-up profile close to $Q$, with slowly modulated parameters $(b,\nu)$ and a small perturbation $\eta$.

Unlike the minimal blow-up solution in \cite{KLR2013ARMAhalfwave}, our setting corresponds to an above-threshold regime:
\[
    \|u_\eta(t)\|_{L^2} > \|Q\|_{L^2}, \quad \text{for small } \eta > 0.
\]
However, due to the degeneracy of $z_f^*$ at the origin and the resulting weak interaction between the singular core and radiation, the leading-order behavior is still governed by the blow-up profile $Q_{(b,\nu)}$ constructed in \cite[Proposition 4.1]{KLR2013ARMAhalfwave}. Furthermore, from the structure of the generalized kernel, we find that the generalized kernel corresponding to the phase rotation $\rho_1$ yields a mass decreasing direction:
\begin{equation}\label{proof strategy : instability direction}
    \left.\frac{d}{d\eta}\right|_{\eta=0} \|Q_{(b,\nu)} + \eta \rho_1\|_{L^2}^2 = 2(Q_{(b,\nu)}, \rho_1)_r \approx 2(Q, \rho_1)_r.
\end{equation}
From the kernel relations \eqref{generalized kernel relations}, we have $(Q, \rho_1)_r < 0$, hence the direction of $\rho_1$ corresponds to a \emph{mass-deficient} direction for $\eta > 0$.

Combining with the sharp lower bound on the energy functional \eqref{sharp energy lower bound inequality}, we anticipate that the $\eta \rho_1$ corrector prevents blow-up at $t=0$, consistent with
\[
    \|Q_{(b,\nu)} + \eta \rho_1\|_{L^2} < \|Q\|_{L^2},
\]
which further confirms that $\rho_1$ is an \emph{unstable direction}. For $\eta=0$, the singular profile reduces to the minimal-mass blow-up profile of \cite{KLR2013ARMAhalfwave}. For $\eta>0$, the singular core is displaced in a mass-deficient direction, whereas the full solution remains slightly super threshold because of the radiation component. The construction below shows that the corresponding approximating solution remains regular across $t=0$ on a uniform time interval.

Guided by this insight, we construct the singular core profile as a high-order expansion.
\[
    Q_{\mathcal{P}} \coloneqq Q + P = Q + \sum_{(p,q,r)} (ib)^p(i\nu)^q \eta^r R_{p,q,r},
\]
with the first few coefficients given by
\[
    R_{1,0,0} \coloneqq S_1, \quad R_{0,1,0} \coloneqq G_1, \quad R_{0,0,1} \coloneqq \rho_1.
\]
After renormalization into the $(s,y)$-variables, the profile $Q_{\mathcal{P}}$ satisfies
\begin{equation}\label{proof strategy : singular part ansatz}
\begin{aligned}
   (\partial_s & - \frac{\lambda_s}{\lambda}\Lambda - \frac{\nut{x}_s}{\lambda}\nabla + \tilde{\gamma}_s i) Q_{\mathcal{P}} + i\left(D Q_{\mathcal{P}} + Q_{\mathcal{P}} - |Q_{\mathcal{P}}|^2 Q_{\mathcal{P}}\right) \\
   &= \overrightarrow{\text{Mod}}(t) \cdot (-i\Lambda Q_{\mathcal{P}}, -i\nabla Q_{\mathcal{P}}, -Q_{\mathcal{P}}, i\partial_b Q_{\mathcal{P}}, i\partial_\nu Q_{\mathcal{P}}) + \Psi_{\mathcal{P}} \approx 0,
\end{aligned}
\end{equation}
where $\Psi_{\mathcal{P}}$ is the profile error and $\overrightarrow{\text{Mod}}(t)$ is the modulation vector defined in \eqref{proof strategy : Mod(t) = 0 choice}.

To make $\Psi_{\mathcal{P}}$ small, we construct higher-order terms $R_{p,q,r}$ by inverting the linearized operator $L_Q$. This step is technically delicate: choosing appropriate constants $c_1$, $c_2$, $c_3 \in \mathbb{R}$, and $c_4 < 0$ in the modulation equations to minimize $\Psi_{\mathcal{P}}$ to a desired order involves careful analysis of cancellation. Although this is not a new phenomenon, it is a technically demanding part of our construction, particularly due to the presence of the additional $\eta$-direction. See Proposition~\ref{Singular profile} with its proof and remark~\ref{where and why we choose c_j} for more details.

\textbf{3. Formal law of modulation parameters and modulation analysis.}
We decompose the solution $u_\eta(t)$ with the initial data \eqref{proof strategy : initial data of u_eta} into the form
\begin{equation}\label{proof strategy : decompose of the solution u_eta}
    u_\eta(t,x) = \frac{1}{\lambda^{1/2}(t)}\left[Q_{\mathcal{P}(b,\nu,\eta)} + \epsilon\right]\left(t, \frac{x - \nut{x}(t)}{\lambda(t)}\right)e^{i\gamma(t)} + z(t,x),
\end{equation}
 and investigate the dynamics of the perturbation $\epsilon$. After renormalization into the $(s,y)$-variables, the equation for $\epsilon$ becomes
\begin{equation}\label{proof strategy : epsilon flow}
    \partial_s \epsilon + iL_Q[\epsilon] \approx \overrightarrow{\text{Mod}}(t) \cdot \overrightarrow{V} + i\mathcal{E},
\end{equation}
where $i\mathcal{E}$ is a small error term that contains interactions between the singular core and the radiation part, as well as higher order terms that arise from cubic non-linearity $|u|^2 u$ in \eqref{half-wave}.

Our goal is to control $\epsilon$ in $H^{1/2+}$, which requires careful selection of modulation parameters $(\lambda, \nut{x}, \gamma, b, \nu)(t)$. We define
\begin{equation}\label{proof strategy : Mod(t) = 0 choice}
    \overrightarrow{\text{Mod}}(t) \coloneqq 
        \begin{pmatrix}
            \frac{\lambda_s}{\lambda} + b \\
            \frac{\nut{x}_s}{\lambda} - \nu - c_2 b^2 \nu \\
            \gamma_s - 1 \\
            b_s + \left(\frac{1}{2} + c_3 \eta \right)b^2 + c_1 b^4 + c_4 \nu^2 + \eta \\
            \nu_s + b\nu
        \end{pmatrix}
    \approx 0.
\end{equation}

From Lemma~\ref{approximated dynamics}, the exact solution $(\lambda_\eta, \nut{x}_\eta, \gamma_\eta, b_\eta, \nu_\eta)(t)$ to solve $\overrightarrow{\text{Mod}}(t) = 0$ satisfies the approximate dynamics
\[
    b_\eta^2 + \eta \lesssim \lambda_\eta, \quad b_\eta \sim -t, \quad \nut{x}_\eta \sim t^3 + \eta t, \quad \lambda_\eta \sim \nu_\eta.
\]

To enforce the modulation law \eqref{proof strategy : Mod(t) = 0 choice}, we impose five orthogonality conditions on $\epsilon$:
\[
    (\epsilon, i\Lambda Q_{\mathcal{P}})_r = (\epsilon, i\partial_b Q_{\mathcal{P}})_r = (\epsilon, i\partial_\eta Q_{\mathcal{P}})_r = (\epsilon, i\nabla Q_{\mathcal{P}})_r = (\epsilon, i\partial_\nu Q_{\mathcal{P}})_r = 0.
\]
Differentiating these orthogonality conditions in time and applying \eqref{proof strategy : epsilon flow}, we obtain
\[
    |\overrightarrow{\text{Mod}}(t)| \lesssim |(\epsilon, Q_{\mathcal{P}})_r| + \alpha^* \lambda^{2+},
\]
where the smallness of the error term $\mathcal{E}$ plays a key role. More precisely, $\mathcal{E}$ is bounded by $\alpha^* \lambda^{2+}$, as shown in Lemma~\ref{interaction error order}.

To control the key quantity $(\epsilon, Q_{\mathcal{P}})_r$, we differentiate in $s$:
\begin{equation}
    \partial_s (\epsilon, Q_{\mathcal{P}})_r \approx (\epsilon, L_Q[iQ])_r + |(\partial_b Q_{\mathcal{P}}, Q_{\mathcal{P}})_r| \cdot |\overrightarrow{\text{Mod}}(t)| + (i\mathcal{E}, Q)_r.
\end{equation}
From the identity $\|R_{1,0,0}\|_{L^2}^2 = 2(Q, R_{2,0,0})_r$ and the generalized kernel relations \eqref{generalized kernel relations}, we obtain the bound
\begin{equation*}
    (\partial_b Q_{\mathcal{P}}, Q_{\mathcal{P}})_r = \mathcal{O}(b^2 + \nu + \eta),
\end{equation*}
which yields the estimate
\[
    |(\epsilon, Q)_r| \lesssim \lambda^{1/2-} \|\epsilon\|_{L^2}.
\]
This additional smallness of $(\epsilon, Q)_r$ is one of the main motivations to impose the above orthogonality conditions. It allows us to close the bootstrap argument and control $\epsilon$ uniformly in $H^{1/2+}$ norm over the backward time interval $[t_1, 0]$.

\textbf{4. Bootstrap argument and energy bounds.}
To close the bootstrap argument, we seek to control the $H^{1/2+}$-norm of $\epsilon$ uniformly in $\eta$, which is essential for applying a soft compactness argument at the end. 

For the $H^{1/2}$-control, we employ a linearized energy functional augmented by a virial-type correction. For sufficiently large $A \gg 1$, we define
\begin{equation}\label{proof strategy : energy functional J_A}
    \begin{aligned}
        \mathcal{J}_A(u) \coloneqq& \frac{1}{2} \int |D^{1/2}\epsilon^{\sharp}|^2 + \frac{1}{2\lambda} \int |\epsilon^{\sharp}|^2 
        - \int \left[F(W+\epsilon^{\sharp}) - F(W) - F'(W)\cdot \epsilon^{\sharp}\right] \\
        &\quad + \frac{b}{2} \Im\left( \int A\phi'\left(\frac{x - \nut{x}}{A\lambda}\right) \nabla \epsilon^{\sharp} \cdot \overline{\epsilon^{\sharp}} \right),
    \end{aligned}
\end{equation}
where $\phi:\mathbb{R} \rightarrow \mathbb{R}$ is a smooth weight function satisfying $\phi''(y) \geq 0$ and
\[
    \phi'(y) \coloneqq 
    \begin{cases}
        y, & \text{for } 0 \leq y \leq 1, \\
        3 - e^{-y}, & \text{for } y \geq 2.
    \end{cases}
\]
Here, we define the nonlinear energy components as
\[
    F(u) = \frac{1}{4}|u|^4, \quad f(u) = u|u|^2, \quad \text{and} \quad F'(u) \cdot h = \Re\left(f(u)\cdot \overline{h}\right),
\]
so that $F'(u) \cdot h$ captures the linear term in the Taylor expansion of $F(u + h) - F(u)$.

We then show that there exists $t_1 = t_1(\alpha^*) < 0$ such that for all $t \in [t_1, 0]$, uniformly in $\eta$, we have
\[
    \mathcal{J}_A(u) \gtrsim \frac{1}{\lambda} \|\epsilon\|_{H^{1/2}}^2, \quad \frac{d}{dt} \mathcal{J}_A(u) \geq  -o\left( \frac{1}{\lambda} \|\epsilon\|_{H^{1/2}}^2 \right).
\]
This energy functional originates from the conservation law and virial identity applied to the leading-order dynamics of $\epsilon$, particularly the scaling contribution $ib \Lambda \epsilon$ arising in the renormalized equation (see \eqref{energy bound: flow of epsilon}). The idea of this correction term was first introduced by Rapha\"{e}l–Szeftel in \cite{RS2011JAMSinhomo} to compensate for the noncoercivity of the linearized energy near the singular profile.

By applying the fundamental theorem of calculus in time, and using the fact that $\epsilon(0) = 0$, we obtain an upper bound for $\|\epsilon(t)\|_{H^{1/2}}$ over $[t_1, 0]$.

However, the half-wave equation lacks a local smoothing effect, and small data do not scatter. Therefore, traditional Strichartz estimates are not available to control the $H^{1/2+}$-norm. Instead, we work on the Fourier side using fractional calculus techniques as in \cite[Appendix E]{KLR2013ARMAhalfwave}.

Specifically, in the Fourier domain, we consider the time derivative of the high-regularity norm:
\[
    \frac{1}{2} \frac{d}{dt} \|\epsilon^{\sharp}\|_{\dot{H}^{1/2+\delta}}^2 
    = -\Im\left( D^{1/2+\delta}\left(-D\epsilon^{\sharp} + |\epsilon^{\sharp}|^2 \epsilon^{\sharp} + \text{(higher-order terms)}\right),\, D^{1/2+\delta}\epsilon^{\sharp} \right).
\]
The self-adjointness of the fractional Laplacian $D = |\nabla|$ implies that
\[
    \Im\left(D^{1/2+\delta}(D\epsilon^{\sharp}), D^{1/2+\delta} \epsilon^{\sharp} \right) = 0,
\]
and the remaining terms can be controlled using fractional Leibniz rules and Lemma~\ref{log loss of the L-infty estiamte}, provided that we have the bootstrap assumption $\|\epsilon^{\sharp}\|_{H^{1/2+\delta}} \lesssim 1$.

Combining the energy-virial method for $H^{1/2}$ and the Fourier-based estimate for $H^{1/2+}$, we close the bootstrap argument and control $\epsilon$ uniformly in both norms in $\eta$. Finally, using a soft compactness argument, we pass to the limit $\eta \to 0^+$ and obtain the existence of a blow-up solution with a mass slightly above the threshold, together with its instability with respect to the perturbation direction $\rho_1$.

\vspace{5bp}
\noindent\textbf{Acknowledgements.} We appreciate Junseok Kim for brining a topological lemma, Borsuk--Ulam Theorem to our attention. T.~Kim was partially supported by the National Research Foundation of Korea, RS-2019-NR040050, RS-2022-NR069873, and RS-2024-00333393, and by a KIAS Individual Grant (MG105201) at Korea Institute for Advanced Study. J.~Park was partially supported by the National Research Foundation of Korea, RS-2019-NR040050, RS-2022-NR069873, and RS-2024-00333393. S.~Kwon was partially supported by the National Research Foundation of Korea, RS-2019-NR040050 and RS-2022-NR069873.

\section{Preliminaries}
\noindent\textbf{Notations.}
Throughout the paper, we use the standard notation
$\mathbb{R}$, $\mathbb{C}$, and $\mathbb{N}$ for the sets of real numbers, 
complex numbers, and natural numbers, respectively. 
For a complex number $A = A_1 + iA_2 \in \mathbb{C}$ with $A_1, A_2 \in \mathbb{R}$, we denote
\[
    \Re A \coloneqq A_1, \qquad \Im A \coloneqq A_2.
\]
We also denote $\R_+=(0,\infty)$. For quantities $A\in\bbC$ and $B\geq0$, we denote $A \lesssim B$ or $A=\mathcal{O}(B)$ if $|A|\leq CB$ holds for some implicit constant $C$. For $A,B\geq0$, we say $A \sim B$ when $A \lesssim B$ and $B \lesssim A$. Similarly, for $A\geq 0$ and $B\in \bbC$, we write $A\gtrsim B$ if $B\lesssim A$. If $C$ depends on some parameters $m$, then we write $\lesssim_m,\sim_m$, and $\gtrsim_m$ to indicate this dependence. We also use the symbol $\sim$ schematically to denote algebraic expansions up to harmless constants; for example,
\[
    |A+B|^2(A+B)\sim A^3+A^2B+AB^2+B^3.
\]
We also write $\langle \xi \rangle = (1+|\xi|^2)^{\frac{1}{2}}$.

We denote $\nabla$ by an ordinary spatial derivative. If $f : \mathbb{R} \rightarrow \mathbb{C}$, we define
\begin{equation*}
    \nabla^m f \coloneqq \frac{d^mf}{dx^m},
\end{equation*}
where $m \in \mathbb{N}$. We also denote $\widehat{f}$ and $\check{f}$ by the Fourier transform and the inverse Fourier transform of $f$, respectively, and define the nonlocal spatial derivative by
\begin{equation*}
    \widehat{D^sf}(\xi) \coloneqq |\xi|^s\widehat{f}(\xi) \quad \text{for} \quad s \geq 0,
\end{equation*}
We use the following notations for inner products:
\[
    (f,g)_r := \Re\!\int_{\R} f\,\overline{g}, \qquad
    (f,g) := \int_{\R} f\,\overline{g}, \qquad
    A\cdot B := \sum_{j=1}^n a_j b_j,
\]
for $f,g\in L^2(\R)$ and $A=(a_1,\dots,a_n),\,B=(b_1,\dots,b_n)\in\C^n$.
The scaling operator is defined by
\[
    \Lambda f := \tfrac12 f + x\cdot\nabla f.
\]
Unless otherwise stated, we abbreviate
\[
    \int f := \int_{\R} f, \qquad
    \int u(t) := \int_{\R} u(t,x)\,dx \quad (u:I\times\R\to\C).
\]
Since in most places, the singular part corresponding to the blow-up component and the asymptotic part live on different scales, we introduce the following notation to focus on the singular part. Let $s$ and $y$ be variables defined by
\begin{equation}
    \frac{ds}{dt} \coloneqq \frac{1}{\lambda}, 
    \qquad 
    y \coloneqq \frac{x - \nut{x}}{\lambda}.
\end{equation}
For functions $f$ and $g$, we define $f^{\sharp}$ in the $(t,x)$-scale and $g^{\flat}$ in the $(s,y)$-scale by
\begin{align}
    f^{\sharp}(t,x) &\coloneqq \frac{1}{\lambda^{1/2}} 
    f\!\left(s, \frac{x - \nut{x}}{\lambda}\right) 
    e^{i\gamma}\bigg|_{s = s(t)}, \\
    g^{\flat}(s,y) &\coloneqq \lambda^{1/2} 
    g\!\left(t, \lambda y + \nut{x}\right) 
    e^{-i\gamma}\bigg|_{t = t(s)}.
\end{align}
Here, $(\lambda, \nut{x}, \gamma)(t)$ are the dynamical parameters. 
Note that
\begin{equation}
    f = [f^{\sharp}]^\flat,
    \qquad
    g = [g^\flat]^\sharp.
\end{equation}

\vspace{5bp}
\noindent\textbf{Function spaces.}
For $1\le p\le \infty$, we define the usual $L^p$-norms by
\[
    \norm{f}_{L^p} \coloneqq \left(\int_{\mathbb{R}} |f(x)|^p\,dx\right)^{1/p}.
\]
For $s\ge0$, the homogeneous and inhomogeneous fractional Sobolev norms are defined by
\begin{align*}
    \norm{f}_{\dot{H}^s} &\coloneqq \norm{D^s f}_{L^2},\\
    \norm{f}_{H^s} &\coloneqq \norm{\langle\nabla\rangle^s f}_{L^2},
\end{align*}
where $D^s$ and $\langle\nabla\rangle^s$ are the Fourier multiplier operators 
with symbols $|\xi|^s$ and $\langle\xi\rangle^s$, respectively.  
We write $H^s \equiv H^s(\mathbb{R})$ and $\dot{H}^s \equiv \dot{H}^s(\mathbb{R})$ for simplicity. In particular,
\[
    \norm{f}_{H^s} \sim_s \norm{f}_{L^2} + \norm{f}_{\dot{H}^s}.
\]
For an interval $I\subset\mathbb{R}$ and $0\le s <\frac 12$, we set
\[
    \norm{f}_{H^s(I)} \coloneqq \norm{f\,1_I}_{H^s(\mathbb{R})},
\]
where $1_I$ denotes the characteristic function of $I$.  
We will use a refined Sobolev embedding estimate in Section~\ref{closing the bootstrap arguments}.
\begin{lemma}[Refined Sobolev embedding]\label{log loss of the L-infty estiamte}
Let $s>1/2$ and $u\in H^s(\mathbb{R})$. Then
\[
    \norm{u}_{L^{\infty}} \le C_s\norm{u}_{H^{1/2}}
    \Bigl[\log\!\left(2+\frac{\norm{u}_{H^s}}{\norm{u}_{H^{1/2}}}\right)\Bigr]^{1/2},
\]
where $C_s>0$ depends only on $s$.
\end{lemma}
See \cite{KLR2013ARMAhalfwave}, Lemma D.1 for a proof. 

The standard Sobolev embedding $H^{1/2}(\mathbb{R})\hookrightarrow L^p(\mathbb{R})$ for $p \in [2,\infty)$ 
will also be used implicitly, so we do not state it separately.

In addition, we will use a commutator estimate in the proof of Theorem~\ref{main result 1}.
\begin{lemma}\label{Lemma:commutator estimate for localizing mass}
    Let $R>0$ be a constant and $\chi : \bbR \rightarrow \bbC$ be a smooth function defined by
    \[
        \chi(x) \coloneqq   
        \begin{cases}
            0, \quad \text{for } |x|\leq 1&\\
            1, \quad \text{for } |x|\geq 2.
        \end{cases}
    \]
    Let $\chi_R(x) \coloneqq \chi(x/R)$. Then, we have the following commutator estimate:
    \begin{equation}\label{commutator estimate for localizing mass}
    \norm{[\chi_R,iD]}_{L^2\rightarrow L^2} \lesssim \norm{\nabla \chi_R}_{L^{\infty}}.
    \end{equation}
\end{lemma}
See \cite{Steinharmonic-book} for a proof.

\vspace{5bp}
\noindent\textbf{Cauchy theory.}
We recall the local well-posedness of the half-wave equation \eqref{half-wave}. See \cite[Theorem D.1]{KLR2013ARMAhalfwave} for the proof.

\begin{theorem}[\cite{KLR2013ARMAhalfwave}, Theorem D.1]\label{Cauchy theory}
Let $s \ge 1/2$. For every initial datum $u_0 \in H^s(\mathbb{R})$, 
there exists a unique solution $u \in \mathcal{C}([t_0,T);H^s(\mathbb{R}))$ 
of \eqref{half-wave}, where $t_0 < T(u_0) \le \infty$ is its maximal time of existence. Moreover:
\begin{enumerate}
    \item (\textbf{Blow-up alternative in $H^{1/2}$}) Either $T(u_0)=+\infty$, or if $T(u_0)<\infty$, 
    then $\norm{u(t)}_{H^{1/2}} \to \infty$ as $t\to T-$.
    \item (\textbf{Continuous dependence}) If $s>1/2$, the flow map $u_0\mapsto u(t)$ is Lipschitz on the bounded subsets of $H^s(\mathbb{R})$.
    \item (\textbf{Global existence for subthreshold solutions}) If $\norm{u_0}_{L^2}<\norm{Q}_{L^2}$, 
    then $T(u_0)=+\infty$.
\end{enumerate}
\end{theorem}
The global existence for subthreshold solutions follows from the sharp energy lower bound
\begin{equation}\label{sharp lower bound for energy}
    E(u) \ge \frac12\norm{u}_{\dot{H}^{1/2}}^2
    \Bigl[1-\frac{\norm{u}_{L^2}^2}{\norm{Q}_{L^2}^2}\Bigr],
\end{equation}
which is from the sharp Gagliardo--Nirenberg inequality.

\vspace{5bp}
\noindent\textbf{Ground state.}
The half-wave equation \eqref{half-wave} admits stationary solutions of the form $u(t,x)=e^{it}Q(x)$. Here, $Q$ is the ground state solution 
\begin{equation}\label{Q equation}
    DQ+Q-Q^3=0.
\end{equation}
This is a nonlocal elliptic equation. Frank and Lenzmann \cite{FrankLenzmann2013Acta} have shown that there exists a unique positive even solution $Q$ to \eqref{Q equation}. Note that $E(e^{it}Q)=0$.

\vspace{5bp}
\noindent\textbf{Linearized operator and its coercivity.}
We now linearize \eqref{Q equation} and recall its basic spectral properties.  
For a complex-valued function $v$, define a linearized operator $L_v$ at $v$
\[
    L_v f \coloneqq Df + f - |v|^2 f - 2\,\Re\{\overline{v} f\}\,v.
\]

In particular, for $v=Q$, the soliton of \eqref{half-wave}, 
the linearized operator $L_Q$ satisfies the following generalized kernel relations
(see \cite{KLR2013ARMAhalfwave}):
\begin{equation}\label{generalized kernel relations}
    \begin{aligned}
        L_Q[iQ]&=0, \qquad L_Q[\nabla Q]=0,\\
        L_Q[\Lambda Q]&=-Q, \qquad L_Q[iG_1]=-i\nabla Q,\\
        L_Q[iS_1]&=i\Lambda Q, \qquad L_Q[\rho_1]=S_1.
    \end{aligned}
\end{equation}
Here, $S_1$, $G_1$, and $\rho_1$ are defined in (4.15), (4.16), and (5.13) of \cite{KLR2013ARMAhalfwave} and 
\begin{equation}\label{positivity of L_Q 0}
    (\Lambda Q,S_1)_r >0,\quad (-\nabla Q, G_1)_r>0.
\end{equation}
We also recall the following coercivity property of $L_Q$.

\begin{lemma}[Coercivity of $L_Q$]\label{Coercivity 1}
There exists a universal constant $\kappa_0>0$ such that for all 
$\epsilon \in H^{1/2}(\mathbb{R})$,
\begin{equation}\label{Coercive state with kernel}
    (L_Q[\epsilon],\epsilon)_r 
    \ge \kappa_0 \norm{\epsilon}_{H^{1/2}}^2
    - \frac{1}{\kappa_0}
      \bigl\{(\epsilon,\phi)_r^2
      +(\epsilon,\nabla Q)_r^2
      +(\epsilon,iQ)_r^2\bigr\},
\end{equation}
where $\phi=\phi(x)>0$ is the unique (up to sign) real-valued ground-state eigenfunction 
of $L_Q$, normalized by $\norm{\phi}_{L^2}=1$ and satisfying $L_Q[\phi]=\lambda\phi$ 
for some $\lambda<0$.
\end{lemma}

Lemma~\ref{Coercivity 1} gives a standard coercivity property of $L_Q$ 
away from the three distinguished directions 
$\{\phi,\nabla Q,iQ\}$. 
In practice, it is more convenient to impose orthogonality 
with respect to the generalized kernel 
$\{Q, S_1, G_1, i\rho_1\}$.
For the modulation analysis, it is more convenient to use the following
alternative coercivity estimate involving the generalized-kernel directions
(see \cite[Lemma~B.4]{KLR2013ARMAhalfwave}):
\begin{lemma}\label{Coercivity 2}
    There exists a universal constant $\kappa_1>0$ such that for $\epsilon \in H^{1/2}(\mathbb{R})$, we have 
        \begin{equation}\label{Coercive state with generalized kernel}
            (L_Q[\epsilon],\epsilon)_r \geq \kappa_1\norm{\epsilon}_{H^{1/2}}^2 -\frac{1}{\kappa_1}\left\{ (\epsilon, Q)_r^2 + (\epsilon, S_1)_r^2+(\epsilon, G_1)_r^2+(\epsilon, i\rho_1)_r^2\right\}.
        \end{equation}
\end{lemma}
Combined with the generalized kernel relations in \eqref{generalized kernel relations},
this coercivity property plays a key role in revealing the instability mechanism
through the modulation equations and will be crucial to close the bootstrap argument
in Section~\ref{closing the bootstrap arguments}.

\section{Approximate solutions}\label{sec:approx sol}

We construct a one-parameter family of approximate solutions \( u_\eta(t) \) to \eqref{half-wave}, depending on a small parameter \( \eta > 0 \). This construction is motivated by the instability mechanism established for \eqref{NLS} in \cite{MRS2013AJM}, concerning the Bourgain--Wang blow-up solutions initially constructed in \cite{BourgainWang1997}.

Our approach is inspired by the instability mechanism developed in \cite{MRS2013AJM}, which employs a decomposition into a blow-up profile and a radiation component. Due to the absence of symmetries like the pseudo-conformal in the half-wave equation \eqref{half-wave}, directly applying the techniques from \cite{MRS2013AJM} is not feasible. We turn to the blow-up profile specifically developed for the half-wave equation by \cite{KLR2013ARMAhalfwave}, sharing similarities with the tail computation method developed in \cite{RaphaelRodnianski2012,MerleRaphaelRodnianski2013Invention,MerleRaphaelRodnianski2015CambJMath}. Using this structure, we introduce an instability parameter analogous to the pseudo-conformal direction's role in the NLS scenario.
To implement this strategy, we consider an approximate decomposition near \( t = 0 \) of the form
\[
u_\eta(t) \approx v_\eta(t) + z(t),
\]
where \( v_\eta \) is a singular profile depending on \( \eta \), and \( z \) is a radiation part independent of \( \eta \) that converges to an asymptotic profile \( z^* \in H^{1/2}(\mathbb{R}) \) as \( t \to 0^- \).

The profile $v_\eta$ is constructed by perturbing the minimal blow-up profile as described in \cite{KLR2013ARMAhalfwave} in an unstable manner. Meanwhile, the radiation component $z$ is crafted to be sufficiently small and flat near $(t,x) = (0,0)$. The construction of a suitable radiation part is nontrivial due to the nonlocal nature of the half-wave operator, prompting the use of a topological method to achieve the necessary flatness of the approximate solution. As a result, the interaction between $v_\eta$ and $z$ remains negligible as $t \to 0$, ensuring that the sum $v_\eta + z$ provides a good approximate solution to \eqref{half-wave} near the time of blow-up.
To finalize the description of the approximate solution, we solve the formal modulation law that governs the blow-up dynamics of the modulation parameters in $v_\eta$, as dictated by the profile's structure.

\subsection{Construction of the approximate radiation}\label{construction of the approximate radiation}

In this subsection, we construct an asymptotic profile $z^*$ together with its associated approximate radiation $z(t,x)$ solving the nonlinear equation \eqref{half-wave} with the initial data $z^*$ at $t=0$. The profile $z^*$ is carefully designed so that the resulting solution $z(t,x)$ exhibits sufficient flatness close to $(t,x)=(0,0)$. In \eqref{NLS}, the high degeneracy assumption on initial data $z^*$ naturally propagates in time, inducing sufficient flatness at $(t,x)=(0,0)$ by standard Cauchy theory combined with the smoothness of the nonlinearity; see \cite{BourgainWang1997,MRS2013AJM}. However, in our case, due to the nonlocal nature of the operator $D = |\nabla|$,
the evolution near the origin is influenced by the entire distribution of the initial data. This global dependence prevents spatial degeneracy near the origin from naturally converting into time degeneracy at $t=0$. Hence, the spatial degeneracy at the initial time is not sufficient on its own. To handle this difficulty, we need to use an additional argument to construct the asymptotic profile $z^*$ so that the corresponding solution $z(t,x)$ achieves the necessary flatness property near $(t,x)=(0,0)$. More precisely, we apply a topological argument based on the Borsuk–Ulam theorem to overcome nonlocal obstruction, which cannot be handled by standard local Cauchy theory arguments.

In the construction of $z^*$, we hope that $z^*$ is as generic as possible. For this purpose, for a given function $f\in S(\bbR)$ such that $\norm{f}_{H^{A_0}} \ll 1$ for large $A_0\gg 1$, we construct $z^*_f \in S(\mathbb{R})$, parameterized by $f$ and finitely many coefficients so that the corresponding solution $z(t,x)$ is expected to satisfy the required flatness conditions at $(t,x)=(0,0)$. Specifically, we design $z^*_f$ in its Fourier side
\begin{equation}\label{profile of z^*_f}
    \widehat{z^*_f} = \widehat{f}(\xi)(d_1 \xi + \cdots + d_N \xi^N).
\end{equation}
The construction relies on carefully adjusting the coefficients $\{d_j\}_{1\leq j \leq N}$ in $\widehat{z^*_f}$, with $N$ chosen sufficiently large relative to the required flatness order.

By virtue of Cauchy theory for \eqref{half-wave} in $H^{1/2+}(\mathbb{R})$ (as stated in Lemma~\ref{Cauchy theory}), the solution $z(t,x)$ corresponding to \eqref{half-wave} exhibits continuous dependence on the initial condition $z^*_f$. Notably, the flatness condition at the point $(t,x) = (0,0)$ can be expressed as a finite-dimensional algebraic system involving the unknowns $\{d_j\}$. Subsequently, the Borsuk--Ulam theorem ensures that there is a set of coefficients meeting these criteria, thus finalizing the construction. We now encapsulate this construction in the following proposition.

\begin{proposition}\label{construction of asymptotic profile}(Construction of the asymptotic profile)
Let $f \in \mathcal{S}(\mathbb{R}) \setminus \{0\}$ and $m \in \mathbb{N}$ be given, and define
\[
    N = N(m) \coloneqq (m+2)^4 + 1.
\]
There are constants $C \gg 1$ and $0<\alpha^*<1$ so that the following hold: Assume that 
\begin{equation}\label{smallness of f}
    \norm{f}_{H^{C\cdot(N+m+1)}} < \alpha^*\min\{1,E_0\}.
\end{equation}
Then, there exists a nonzero sequence $\{d_j\}_{1 \leq j \leq N}\in\C$ (depending on $f$ and $m$) such that the function $z_f^*$ defined by \eqref{profile of z^*_f} yields a solution $z(t,x)$ to \eqref{half-wave} with $z(0,\cdot)=z^*_f$ satisfying the following properties:

\begin{itemize}
    \item The solution $z(t,x)$ exists globally in time and satisfies 
    \begin{equation}\label{control of norm z(t) for geq 1/2}
        \|z(t)\|_{H^{m+1}} \lesssim_m \alpha^*, \quad \text{for all } t \in [-1,1].
    \end{equation}

    \item Moreover, $z(t,x)$ exhibits degeneracy at $(t,x) = 0$ in the following sense: for all $t,x \in [-1,1]$,
    \begin{equation}\label{degeneracy for z}
    \begin{aligned}
        |z(t,x)| &\lesssim_{m} \alpha^*(|t| + |x|)^{m+1}, \\
        |D z(t,x)| &\lesssim_{m} \alpha^*(|t| + |x|)^m, \\
        |\nabla z(t,x)| &\lesssim_{m} \alpha^*(|t| + |x|)^m.
    \end{aligned}
    \end{equation}
\end{itemize}
\end{proposition}

\begin{proof}
    Given $m \in \mathbb{N}$ and $f \in \mathcal{S}(\mathbb{R}) \setminus \{0\}$, we begin by choosing coefficients $\{d_j\}_{1 \leq j \leq N} \subset \mathbb{C}^N$, depending on $f$ and $m$, such that the resulting solution $z(t,x)$ to \eqref{half-wave} satisfies the following degeneracy conditions at the origin:
    \begin{equation}
        \left( (\nabla^{k_1} \partial_t^{k_2} z)(0,0),\, (\nabla^{k_3} \partial_t^{k_4} D z)(0,0) \right) = (0,0), \label{eq: z degeneracy goal}
    \end{equation}
    for all multi-indices $k_1, k_2, k_3, k_4 \in \mathbb{N}_{\geq 0}$ with $k_1 + k_2 \leq m$ and $k_3 + k_4 \leq m-1$. This will be achieved using the local well-posedness theory for \eqref{half-wave} established in Lemma~\ref{Cauchy theory} and topological arguments.

   The key step in this argument is to observe that the left-hand side of \eqref{eq: z degeneracy goal} depends on $\{d_j\}$ via an odd polynomial. This structural property will allow us later to apply a topological argument.

    To verify this, we begin by noting that since $z$ is a solution to \eqref{half-wave}, we have
    \begin{equation}\label{radiation profile eq:1}
        \partial_t z = -i(Dz - |z|^2 z).
    \end{equation}
    By repeatedly applying \eqref{radiation profile eq:1} inductively, we express each time derivative in terms of spatial derivatives and nonlinearities. In particular, we obtain
    \begin{equation}\label{radiation profile eq:2}
        \begin{aligned}
            (\nabla^{k_1} \partial_t^{k_2} z)(t,x) =\; 
            &(-i)^{k_2} \nabla^{k_1}(D^{k_2} z)(t,x) \\
            &- \sum_{j=1}^{k_2} (-i)^j \nabla^{k_1} D^{j-1} \left( \partial_t^{k_2 - j} (|z|^2 z) \right)(t,x).
        \end{aligned}
    \end{equation}
    Using \eqref{radiation profile eq:2} inductively and applying the Leibniz rule at each step, we can express the term $\partial_t^{k_2 - j}(|z|^2 z)$ as a linear combination of expressions of the form
    \begin{equation}\label{radiation profile eq:3}
        D^{j_1}\left(P_1(z,\overline{z}) D^{j_2} \left(P_2(z,\overline{z}) D^{j_3} \left( \cdots D^{j_k} P_k(z,\overline{z}) \right) \right) \right),
    \end{equation}
    for some nonnegative integers $\{j_\ell\}_{1 \leq \ell \leq k}$ satisfying $0 \leq j_1 + \cdots + j_k \leq k_2 - j$, where each $P_j(z,\overline{z})$ is a monic polynomial in $z$ and $\overline{z}$.
    Substituting this expression into \eqref{radiation profile eq:2}, we see that all time derivatives in $(\nabla^{k_1} \partial_t^{k_2} z)(t,x)$ are rewritten in terms of spatial derivatives and nonlinear monomials in $z$ and $\overline{z}$. Then, by the Fourier inversion formula, we have
    \[
        D^{j_\ell}(z^a \overline{z}^b)(t,0) 
        = \int_{\mathbb{R}} |\xi|^{j_\ell} 
        \underbrace{\widehat{z} * \cdots * \widehat{z}}_{\text{$a$ times}} 
        * \underbrace{\widehat{\overline{z}} * \cdots * \widehat{\overline{z}}}_{\text{$b$ times}} (t,\xi) \, d\xi.
    \]
    Since we are evaluating at $(t,x) = (0,0)$, and the local well-posedness theory from Lemma~\ref{Cauchy theory} ensures that all derivatives of the solution at the origin coincide with those of the radiation profile $z_f^*$ at time $t=0$, we have
    \[
    (\nabla^p D^q P_j(z, \overline{z}))(0,0) 
    = (\nabla^p D^q P_j(z_f^*, \overline{z_f^*}))(0).
    \]
    Combining this with the expression in \eqref{radiation profile eq:2}, we conclude that each term in $(\nabla^{k_1} \partial_t^{k_2} z)(0,0)$ is a polynomial in $\{d_j\}$, since $z_f^*$ is explicitly constructed from $\{d_j\}$ via \eqref{profile of z^*_f}. The same conclusion holds for $(\nabla^{k_3} \partial_t^{k_4} D z)(0,0)$ by applying the same argument to one additional derivative. Moreover, since $-z$ is also a solution to \eqref{half-wave} with initial data $-z_f^*$, we have
    \begin{equation}\label{eq: z degeneracy k1 k2}
        (\nabla^{k_1} \partial_t^{k_2} [-z])(0,0) = - (\nabla^{k_1} \partial_t^{k_2} z)(0,0).
    \end{equation}
    This identity implies that $(\nabla^{k_1} \partial_t^{k_2} z)(0,0)$ must be an odd polynomial in $\{d_j\}$. Indeed, replacing $\{d_j\}$ by $\{-d_j\}$ corresponds to replacing $z_f^*$ with $-z_f^*$, and hence $z$ with $-z$, which flips the sign of the expression. Applying the same argument to $(\nabla^{k_3} \partial_t^{k_4} D z)(0,0)$, we conclude that it is also an odd polynomial in $\{d_j\}$. 

    We now turn to the existence of a nontrivial choice of coefficients $\{d_j\}$ satisfying \eqref{eq: z degeneracy goal}. This will be established through a topological argument. Let $F$ denote the left-hand side of \eqref{eq: z degeneracy goal}:
    \[
        F(d_1, \dots, d_N) 
        \coloneqq \left( (\nabla^{k_1} \partial_t^{k_2} z)(0,0),\, (\nabla^{k_3} \partial_t^{k_4} D z)(0,0) \right).
    \]
    From conditions $k_1 + k_2 \leq m$ and $k_3 + k_4 \leq m - 1$, the total number of equations is $\binom{m+2}{2} + \binom{m+1}{2}=(m+1)^2$. Hence, we may view $F$ as a map
    \[
        F: \mathbb{C}^N \to \mathbb{C}^{(m+1)^2} \simeq \mathbb{R}^{2((m+1)^2)}.
    \]
    As shown above, $F$ is a polynomial function in $\{d_j\}$ and is odd, that is, $F(-d_1, \dots, -d_N) = -F(d_1, \dots, d_N)$. We now restrict the domain of $F$ to the unit sphere $\mathbb{S}^{2N-1} \subset \mathbb{C}^N \simeq \mathbb{R}^{2N}$ and regard $F$ as a continuous odd map
    \[
        F: \mathbb{S}^{2N-1} \to \mathbb{R}^{2(m+1)^2}.
    \]
    Since we defined $N(m) = (m+2)^4 + 1$, we have
    \[
    2N - 1 > 2 (m+1)^2,
    \]
    so, the dimension of the domain strictly exceeds that of the target space. By the Borsuk--Ulam theorem, any continuous odd map from $\mathbb{S}^n$ to $\mathbb{R}^n$ must vanish somewhere. Therefore, there exists a nontrivial choice of coefficients $\{d_j\} \subset \mathbb{S}^{2N - 1} \subset \mathbb{C}^N$ such that
    \[
        F(d_1, \dots, d_N) = 0.
    \]
    We now claim that if there exist constants $C\gg 1$ and $0<\alpha^*\ll 1$ such that 
    \[
        \norm{f}_{H^{C\cdot(N+m+1)}} < \alpha^*\min\{1,E_0\},
    \]
    then both \eqref{control of norm z(t) for geq 1/2} and \eqref{degeneracy for z} hold. Under this smallness assumption on $f$, we first establish \eqref{control of norm z(t) for geq 1/2}.
    Since $\{d_j\} \subset \mathbb{S}^{2N-1}$ and by the Cauchy--Schwarz inequality, for each $C>1$ we obtain
    \begin{equation}\label{smallness of z_f^*}
        \norm{z_f^*}_{H^{C\cdot(m+1)}}^2 \leq \int \langle \xi \rangle^{2C\cdot(m+1)}|\widehat{f}(\xi)|^2 (\sum_{j=1}^{N} |d_j|^2)\cdot(\sum_{j=1}^N |\xi|^{2j}) d\xi \lesssim_m \norm{f}_{H^{C\cdot{(N+m+1)}}}^2.
    \end{equation}
    Choose $\alpha^*>0$ small so that
    \begin{equation}\label{L^2 smallness of z(t)}
        \|z_f^*\|_{L^2} <\frac{1}{\sqrt{2}} \|Q\|_{L^2}.
    \end{equation}
    As a consequence, the global well-posedness of $z(t)$ follows from Lemma~\ref{Cauchy theory}.
    
    We now derive upper and lower bounds on the $H^{1/2}$ norm of $z(t)$ using energy and mass conservation. By combining the sharp lower bound of the energy \eqref{sharp lower bound for energy} with conservation laws \eqref{eq:conservation laws}, we obtain
    \begin{equation}\label{sharp energy lower bound inequality}
    \begin{aligned}
        E(z_f^*) &= E(z(t)) 
        \geq \frac{1}{2} \|D^{1/2} z(t)\|_{L^2}^2 \left( 1 - \frac{\|z(t)\|_{L^2}^2}{\|Q\|_{L^2}^2} \right) 
        \geq \frac{1}{4} \|D^{1/2} z(t)\|_{L^2}^2, \\
        E(z(t)) &\geq \frac{1}{2} \|D^{1/2} z_f^*\|_{L^2}^2 \left( 1 - \frac{\|z_f^*\|_{L^2}^2}{\|Q\|_{L^2}^2} \right) 
        \geq \frac{1}{4} \|D^{1/2} z_f^*\|_{L^2}^2.
    \end{aligned}
    \end{equation}
    Therefore \eqref{sharp energy lower bound inequality} with mass conservation gives
    \begin{equation}\label{s=1/2 estimate}
        \frac{1}{2}\norm{z_f^*}_{H^{1/2}} \leq \norm{z(t)}_{H^{1/2}} \leq 2\norm{z_f^*}_{H^{1/2}}, \text{ for } t\in [-1,1].
    \end{equation}

    Following \cite[Theorem D.2]{KLR2013ARMAhalfwave}, we use a Gronwall-type argument. More precisely, define
    \[
        g(t) \coloneqq \frac{\|z(t)\|_{H^\sigma}}{\sup_{t \in [0,1]} \|z(t)\|_{H^{1/2}}},
    \]
    so that by \eqref{s=1/2 estimate}, we have for all $t \in [0,1]$:
    \begin{equation}\label{Gronwall inequality}
        2 + g(t) \leq (2 + g(0))^{e^{Ct}},
    \end{equation}
    for some universal constant $C > 0$. From the lower bound of $\norm{z(t)}_{H^{1/2}}$ in \eqref{s=1/2 estimate}, we have 
    \[
        g(0) \leq 2\norm{z_f^*}_{H^{\sigma}}\norm{z_f^*}_{H^{1/2}}^{-1},
    \]
    and hence combining \eqref{Gronwall inequality} we obtain
    \[
        \frac{\norm{z(t)}_{H^{\sigma}}}{2\norm{z_f^*}_{H^{1/2}}} \leq g(t) \leq (2 + g(0))^{e^{C}}-2 \leq (2+2\norm{z_f^*}_{H^{\sigma}}\norm{z_f^*}_{H^{1/2}}^{-1})^{e^C}, \text{ for all } t \in [0,1].
    \]
    Therefore, the Gagliardo--Nirenberg inequality yields
    \begin{align*}
        \|z(t)\|_{H^\sigma} &\leq 2^{1+e^C}(\norm{z_f^*}_{H^{1/2}}^{e^{-C}} + \norm{z_f^*}_{H^{1/2}}^{\left(e^{-C}-1\right)}\norm{z_f^*}_{H^{\sigma}})^{e^C}\\
        &\leq C_\sigma \norm{z_f^*}_{H^{C\cdot \sigma}},
    \end{align*}
    for some constant $C \gg 1$.
    By the time-reversibility of the half-wave equation (i.e., replacing $z(t,x) \mapsto \overline{z}(-t,x)$), the same estimate holds for all $t \in [-1,0]$. In particular, applying the same estimate as in \eqref{smallness of z_f^*} with $\sigma=m+2$, we have
    \[
         \norm{z(t)}_{H^{m+2}}\lesssim_m \alpha^*, \qquad t\in[-1,1].
    \]
    This completes the proof of \eqref{control of norm z(t) for geq 1/2}.

    It remains to prove \eqref{degeneracy for z}. By the Taylor expansion of $z(t,x)$ at $(t,x) = (0,0)$, for each $t,x \in [-1,1]$ we have
    \begin{equation}\label{Taylor expansion of z}
    \begin{aligned}
        z(t,x) - &\sum_{k_1 + k_2 \leq m} \frac{1}{k_1! k_2!} (\partial_t^{k_2} \nabla^{k_1} z)(0,0) t^{k_2} x^{k_1}\\
        &= \sum_{k_1 + k_2 = m+1} \frac{1}{k_1! k_2!} (\partial_t^{k_2} \nabla^{k_1} z)(\beta^*_{(t,x)} t, \beta^*_{(t,x)} x) t^{k_2} x^{k_1},
    \end{aligned}
    \end{equation}
    where $\beta^*_{(t,x)} \in (0,1)$ is a constant depending on $(t,x)$. Note that $(\beta^*_{(t,x)} t, \beta^*_{(t,x)} x) \in (-1,1)^2$.
    By the construction of the coefficients $\{d_j\}_{1 \leq j \leq N}$, the first $m$-th order derivatives of $z$ at $(0,0)$ vanish, so the second term on the left-hand side of \eqref{Taylor expansion of z} vanishes. To complete the proof of \eqref{degeneracy for z}, it remains to show that for all $k_1 + k_2 = m+1$,
    \[
        \| \partial_t^{k_2} \nabla^{k_1} z(\beta^*_{(t,x)} t) \|_{L^\infty} \lesssim_{m} \alpha^*.
    \]
    This estimate follows from \eqref{radiation profile eq:2} and the bound \eqref{control of norm z(t) for geq 1/2}, once we show that for each $1 \leq j \leq k_2$,
    \begin{equation}\label{radiation construction eq:4}
        \| \nabla^{k_1} D^{j-1} \partial_t^{k_2 - j} (|z|^2 z)(\beta^*_{(t,x)} t) \|_{L^\infty} \lesssim_{m} \alpha^*.
    \end{equation}
    Indeed, \eqref{radiation construction eq:4} follows from the fact that $\nabla^{k_1}D^{j-1}\partial_t^{k_2-j}(|z|^2 z)(t,x)$ can be expressed as a finite linear combination $\sum a_\mu \nabla^{k_1} D^{j-1} \mathcal{T}_\mu$, where each $\mathcal{T}_\mu$ is a term of the form \eqref{radiation profile eq:3} and $a_\mu \in \mathbb{C}$. To estimate these in $L^\infty$ in the $x$-variable, we apply Sobolev embedding on each factor and then the product rule iteratively:
    \begin{align*}
        & \| \nabla^{k_1} D^{j-1} D^{j_1} \left( P_1(z,\overline{z}) D^{j_2} \left( P_2(z,\overline{z}) D^{j_3} \left( \cdots D^{j_k} P_k(z,\overline{z}) \right) \right) \right) (\beta^*_{(t,x)} t) \|_{L^\infty} \\
        &\lesssim_m \| P_1(z,\overline{z})(\beta^*_{(t,x)} t) \|_{H^{1 + k_1 + j - 1 + j_1}} 
        \times \cdots \times 
        \| P_k(z,\overline{z})(\beta^*_{(t,x)} t) \|_{H^{1 + k_1 + j - 1 + j_1 + \cdots + j_k}} \\
        &\lesssim_{m} \| z(\beta^*_{(t,x)} t) \|_{H^{k_1 + k_2}}^M,
    \end{align*}
    for some $M \in \mathbb{N}$, where each $P_\ell(z,\overline{z})$ is a monic polynomial in $z$ and $\overline{z}$, and $\sum_{\ell=1}^k j_\ell \leq k_2 - j$. Since $|\beta^*_{(t,x)} t| < 1$ and hence $\norm{z(\beta^*_{(t,x)} t)}_{H^{m+1}} \lesssim_{m} \alpha^*$ from \eqref{control of norm z(t) for geq 1/2}. Therefore, we conclude that the desired degeneracy estimate for $z(t,x)$ at $(0,0)$ is satisfied. In the same manner, we obtain the corresponding degeneracy estimates for $Dz(t,x)$ and $\nabla z(t,x)$. This completes the proof of \eqref{degeneracy for z}.
\end{proof}

\begin{remark}
From now on, we fix the Schwartz function $f$. To guarantee that the radiation part $z$ is sufficiently decoupled from the singular profile $Q_{\mathcal{P}}$ at the energy level, we choose the flatness parameter $m$ to be large; specifically, we fix $m=32$ (see Lemma~\ref{interaction error order}). This choice ensures that the interaction error between $z$ and $Q_{\mathcal{P}}$ is sufficiently small, making it negligible in the modulation analysis which will be carried out in Section~\ref{Main thm to main bootstrap} and Section~\ref{proof of main bootstrap lemma}. Since these choices will be made uniformly throughout the remainder of the analysis, we treat $m$ and $N(m)$ as universal constants from this point onward. Hence, we fix $A_0 \coloneqq C\cdot(N(32)+32+1)$.
\end{remark}

\subsection{Construction of the blow-up profile}
In this subsection, we discuss the construction of the blow-up profile. Unlike the \eqref{NLS}, the equation \eqref{half-wave} does not possess symmetries such as pseudo-conformal or Galilean invariance. Consequently, identifying a natural blow-up profile via the pseudo-conformal phase, as done in \eqref{NLS}, is unachievable. Therefore, the authors in \cite{KLR2013ARMAhalfwave} developed a blow-up profile for the minimal mass blow-up solution by exploiting the generalized kernel of the linearized operator around the ground state $Q$ and performing a formal derivation of the related modulation law. Nevertheless, the constructed minimal mass blow-up solution in \cite{KLR2013ARMAhalfwave} serves as an analogue of the pseudo-conformal blow-up solution for \eqref{NLS} at the threshold, particularly in terms of its \textit{instability mechanism}. This is evident through the spectral property of the linearized operator $L_Q$. Specifically, $L_Q$ has a negative eigenvalue, paired with an eigenfunction $\phi$ that is not derived from any symmetry of equation \eqref{half-wave}. To control this instability, the authors deduce \eqref{Coercive state with generalized kernel} 
and maintain the smallness of $(\epsilon,Q)_r$ by enforcing the mass threshold condition. This inner product captures the core of the instability. 
To quantify this instability, we incorporate a (fixed) unstable parameter $\eta$ into the profile. For $\eta>0$, the corresponding generalized-kernel direction decreases the mass of the singular core, while the full solution may still have mass above the threshold because of the radiation component. This is a similar scenario as demonstrated in \eqref{NLS}  \cite{MRS2013AJM}, in which the authors demonstrate the instability of pseudo-conformal blow-up.
We use the \textit{tail computation} to construct the modified profile, developed in \cite{RaphaelRodnianski2012,MerleRaphaelRodnianski2013Invention,MerleRaphaelRodnianski2015CambJMath}.  We note that in \cite[Proposition 4.1]{KLR2013ARMAhalfwave} the authors use tail computation to construct the profile $Q_{(b,\nu)}$. We are in a similar line but add an instability parameter $\eta$  for Bourgain--Wang type blow-up solutions. On the way of the construction of the profile, we also identify the associated modulation laws and blow-up laws at the formal level.

We will introduce an additional fixed parameter $\eta >0$, where $\eta \ll 1$, to develop a set of solutions $v_{\eta} $. Considering the anticipated instability, we demonstrate that for small $\eta$, where $\eta \ne 0$, $v_{\eta}$ does not experience blow-up, and instead, a solution of the Bourgain-Wang type blow-up solution is obtained as the limit as $\eta \to 0$. We initiate this process by applying a renormalization to the solution $v_\eta$ of \eqref{half-wave}, with the renormalized solution denoted by  $v_\eta^{\flat}$ ;
\begin{equation*}
	\frac{ds}{dt} := \frac{1}{\lambda},\quad y:=\frac{x-\nut{x}}{\lambda},\quad v_\eta^{\flat}(s,y) := \left.\lambda(t)^{\frac{1}{2}}v_\eta(t,\lambda(t)y+\nut{x}(t))e^{-i\gamma(t)}\right|_{t=t(s)},
\end{equation*}
where $\lambda(t)$ is a scaling parameter, $\nut{x}(t)$ is a translation parameter, and $\gamma(t)$ is a phase rotation parameter. Then, $v_\eta^{\flat}$ solves  
\begin{equation}\label{eq:renormalised v}
	(\partial_s-\frac{\lambda_s}{\lambda}\Lambda-\frac{\nut{x}_s}{\lambda}\nabla + \tilde{\gamma}_s i)v_\eta^{\flat} +i( Dv_\eta^{\flat} + v_\eta^{\flat} - |v_\eta^{\flat}|^2v_\eta^{\flat} )=0,
\end{equation}
where $\tilde{\gamma}_s = \gamma_s-1$. 

Consider an approximate solution of the form 
$v_\eta^{\flat} \approx Q_{(b,\nu)} + \eta R$, and explain how to choose the corrector profile $R : \mathbb{R} \rightarrow \mathbb{C}$. Our goal is to construct a blow-up solution that blows up with the same rate as the minimal blow-up solution, which has a positive energy. To guarantee this, we need to choose a direction such that $v_{\eta}^{\flat}$ has subthreshold mass, 
and so the energy is positive. Since $b$ and $\nu$ are slowly modulated parameters 
such that $Q_{(b,\nu)}$ is close to $Q$, a direct calculation gives
\begin{equation}\label{mass escape ansatz}
    \frac{\partial}{\partial \eta}
    \left\{\norm{Q_{(b,\nu)}+\eta R}_{L^2}^2\right\}\bigg|_{\eta=0} 
    = 2(Q_{(b,\nu)},R)_r \approx 2(Q,R)_r.
\end{equation}

Recalling the generalized kernel relations in \eqref{generalized kernel relations}, 
we adopt the notation
\[
    R_{1,0,0}\coloneqq S_1, \qquad R_{0,1,0}\coloneqq G_1, \qquad R_{0,0,1} \coloneqq \rho_1,
\]  
We choose 
\[
    R = R_{0,0,1} = \rho_1, \qquad L_Q[R]=R_{1,0,0}=S_1,
\]
so that, using $L_Q[\Lambda Q]=-Q$, the identity \eqref{mass escape ansatz} yields
\[
    2(Q,R)_r = -2(\Lambda Q, R_{1,0,0})_r < 0.
\]
This shows that, for small $\eta>0$, the direction $R=R_{0,0,1}$ decreases the mass of the singular core. The non-blow-up property of the corresponding exact solutions will be established later in the construction of the family $u_\eta$.
We introduce the corrector profiles by $P$ with modulation parameters $(b,\nu,\eta)$, and denote the modified profile $Q_{\calP}$ with a general expansion of the form
\begin{equation}
		\begin{aligned}
			Q_{\mathcal{P}} &=Q+P= Q + \sum_{p,q,r\in \mathbb{N}_0}(ib)^p(i\nu)^q\eta^r R_{p,q,r}.
		\end{aligned} \label{eq:modi profile general expan}
\end{equation}
In the context mentioned above, we define $R_{0,0,1} \coloneqq R$. Specifically, $R_{1,0,0}$, $R_{0,1,0}$, and $R_{0,0,1}$ are associated with the directions $ib$, $i\nu$, and $\eta$, respectively. These belong to the generalized kernel of the linearized operator $L_Q$. Using the tail computation procedure, we will construct a Taylor series expansion up to a designated order, ensuring the profile error reaches an acceptable level. We begin by inserting $v_\eta^{\flat}=Q_{\calP}$ into \eqref{eq:renormalised v}: 
\begin{align*}
    (\partial_s&-\frac{\lambda_s}{\lambda}\Lambda-\frac{\nut{x}_s}{\lambda}\nabla + \tilde{\gamma}_s i)Q_{\mathcal{P}} +i( DQ_{\mathcal{P}} + Q_{\mathcal{P}} - |Q_{\mathcal{P}}|^2Q_{\mathcal{P}} )
    \\
    =&-i\overrightarrow{\td{\text{Mod}}} \cdot \overrightarrow{\td V}+iL_QP-i\calN_{\calP}
\end{align*}
where $\overrightarrow{\td{\text{Mod}}}\cdot \overrightarrow{\td V}$ is modulation terms and is given by
\begin{align*}
    \overrightarrow{\td{\text{Mod}}} &= \left(\frac{\lambda_s}{\lambda}, \frac{\nut{x}_s}{\lambda}, \tilde{\gamma}_s,b_s , \nu_s ,\eta_s \right), \quad
    \\
    \overrightarrow{\td V} &= (- i\Lambda Q_{\mathcal{P}}, - i\nabla Q_{\mathcal{P}}, - Q_{\mathcal{P}}, i\partial_b Q_{\mathcal{P}}, i\partial_\nu Q_{\mathcal{P}}, i\partial_\eta Q_{\mathcal{P}} ).
\end{align*}
The nonlinear terms $\mathcal{N}_{\mathcal{P}}$ are the collection of cubic or higher terms with respect to $R_{p,q,r}$,
\begin{align}\label{definition of calN_P}
		\calN_{\calP}\coloneqq|Q_{\mathcal{P}}|^2Q_{\mathcal{P}}-Q^3-Q^2P-2\Re(Q^2P).
\end{align}

We will construct the profile \( Q_{\mathcal{P}} \) in such a way that the term \( iL_Q P - i\mathcal{N}_{\mathcal{P}} \) decomposes into one part that is absorbed into \( \overrightarrow{\td V} \), allowing us to determine the modulation laws, while ensuring that the remaining terms are small.


According to this strategy, we state the construction of the modified profile.
\begin{proposition}[Modified profile]\label{Singular profile}
  Let $\mathcal{P}=(b,\nu,\eta)$  be a modulation parameter collection with
  \begin{align*}
		\bbN_{\calP}\coloneqq\{(p,q,r)\in \bbN_0^3, 1\leq p+2q+2r < 6 \}.
	\end{align*} There is a modified profile $Q_{\mathcal{P}}\in H^1(\bbR)$ of the form 
    \begin{equation}\label{modified profile of Q}
		\begin{aligned}
			Q_{\mathcal{P}} &=Q+P= Q + \sum_{(p,q,r)\in \bbN_{\calP}}(ib)^p(i\nu)^q\eta^r R_{p,q,r},
		\end{aligned}
	\end{equation}
    where   $R_{p,q,r}\in H^1(\bbR)$ real profiles, and constants  $c_1,c_2,c_3\in \bbR$, $c_4<0$ so that the profile error $\Psi_{\mathcal{P}}$ satisfies the approximate identity
	\begin{equation}\label{modified profile eror}
		\begin{aligned}
		    \Psi_{\mathcal{P}}	
            =& iL_{Q}P- i\mathcal{N}_{\mathcal{P}} + b\Lambda Q_{\mathcal{P}} - (\nu+c_2b^2\nu)\nabla Q_{\mathcal{P}}\\
            &-\left(\frac{b^2}{2} +\eta+c_1b^4+ c_3b^2\eta+c_4\nu^2\right)\partial_bQ_{\mathcal{P}} - b\nu\partial_{\nu}Q_{\mathcal{P}},
		\end{aligned}
	\end{equation}
	and estimates, for $m \in \{0,1\}$,
	\begin{equation}\label{profile error estimate}
		\norm{\Psi_{\mathcal{P}}}_{H^m} \lesssim_m \sum_{p+2q+2r\geq6} |b|^p|\nu|^q\eta^r,\quad 
        |\nabla ^m \Psi_{\mathcal{P}}(y)| \lesssim_m \sum_{p+2q+2r\geq6} |b|^p|\nu|^q\eta^r\langle y \rangle ^{-2}.
	\end{equation}
	Moreover, for each \( (p,q,r) \) the profiles $R_{p,q,r}$ satisfy the following regularity and decay bounds:
	\begin{equation}\label{decay of R(pqr)}
		\begin{aligned}
			&\norm{R_{p,q,r}}_{H^m}+\norm{\Lambda R_{p,q,r}}_{H^m}+\norm{\Lambda^{2}R_{p,q,r}}_{H^m} \lesssim_m 1 \\
			&|R_{p,q,r}(y)|+|\Lambda R_{p,q,r}(y)|+|\Lambda^2 R_{p,q,r}(y)| \lesssim \langle y \rangle ^{-2}.
		\end{aligned}
	\end{equation}
\end{proposition}
\begin{remark}
   The constants \( c_1, c_2, c_3, c_4 \in \mathbb{R} \) that appear in modulation laws are fixed and do not change with time. Unlike previous research ~\cite{KLR2013ARMAhalfwave}, where a less precise control of the profile error sufficed, our analysis requires a more sophisticated profile construction and higher precision in modulation laws to minimize the profile error. Further details will be discussed in Section~\ref{Set up of the modulation analysis}.
\end{remark}
\begin{remark}[Formal modulation law]\label{rem:formal law}
  Based on Proposition~\ref{Singular profile}, we arrive at a formal ODE system of modulation parameters, given by:
\begin{align}
    \begin{array}{ccc}
        \frac{\lambda_s}{\lambda} + b = 0, & \frac{\nut{x}_s}{\lambda} - \nu - c_2 b^2 \nu = 0, & \tilde{\gamma}_s = 0, \\
        b_s + \frac{b^2}{2} + \eta + c_1 b^4 + c_3 b^2 \eta + c_4 \nu^2 = 0, & \nu_s + b \nu = 0, & \eta_s = 0.
    \end{array} \label{eq:formal modul}
\end{align}
It is important to highlight that this system is expressed in the renormalized time \( s \), where \( s = \infty \) indicates the blow-up time \( t = 0 \), when $\eta=0$. Although the complete system \eqref{eq:formal modul} features nonlinear terms that complicate direct analysis, these higher-order terms can be considered as perturbations on the assumption \( b(t=0) = 0 \), as the modulation parameter \( b \) is small close to time $t=0$. This leads us to study a simplified ODE system, which encapsulates the leading-order dynamics:
\begin{align}
    \begin{array}{ccc}
        \frac{\lambda_s}{\lambda} + b = 0, & \frac{\nut{x}_s}{\lambda} - \nu = 0, & \tilde{\gamma}_s = 0, \\
        b_s + \frac{b^2}{2} + \eta + c_4 \nu^2 = 0, & \nu_s + b \nu = 0, & \eta_s = 0.
    \end{array} \label{eq:formal modul perturb}
\end{align}
Note that \( c_4 < 0 \) (See Appendix~\ref{Proof of c_4<0}, for its proof). \eqref{eq:formal modul perturb} admits three conserved quantities:
\begin{align}
    \frac{b^2+2\eta-2c_4\nu^2}{\lambda} \equiv \ell,\quad \frac{\nu}{\lambda}=\nu_0, \quad \eta\equiv \eta_0. \label{eq:formal conserv}
\end{align}
From this, we compute $\lambda_{tt}$,
\begin{align*}
    \lambda_{tt}=-b_t=\frac{\ell}{2}+2c_4\nu_0^2\lambda.
\end{align*}
First, if \( \nu_0 = 0 \), the equation reduces to
\begin{align}
    b(t)=-\frac{\ell}{2}t,\quad \lambda(t)=\frac{\ell}{4}t^2+\frac{2\eta_0}{\ell} ,\quad t<0. \label{eq:lambda without nu0}
\end{align}
On the other hand, if \( \nu_0 \neq 0 \), then \( \lambda \) satisfies a second-order linear ODE:
\begin{align*}
    \lambda_{tt}=\frac{\ell}{2}+2c_4\nu_0^2\lambda,
\end{align*}
with initial conditions
\begin{align*}
    \lambda_t(0) = -b(0) = 0, \quad \lambda(0) = \frac{\ell - \sqrt{\ell^2 - 16|c_4| \nu_0^2\cdot \eta_0}}{4|c_4| \nu_0^2}.
\end{align*}
Here, the initial condition is determined by the conserved quantities in \eqref{eq:formal conserv} and the simplified system \eqref{eq:formal modul perturb}. The corresponding solution is given by:
\begin{equation}
    \begin{aligned}
    \lambda(t;\eta_0,\nu_0)&=  \frac{\ell}{4|c_4|\nu_0^2} 
    - \frac{\sqrt{\ell^2 - 16|c_4|\nu_0^2\cdot \eta_0}}{4|c_4|\nu_0^2}\cos\left(\sqrt{2|c_4|\nu_0^2} t\right),
    \\
    b(t;\eta_0,\nu_0)&=-\frac{\sqrt{\ell^2 - 16|c_4|\nu_0^2 \cdot \eta_0}}{2\sqrt{2|c_4|\nu_0^2}} \sin\left(\sqrt{2|c_4|\nu_0^2} t\right).
\end{aligned} \label{eq:lambda with nu0}
\end{equation}
One can verify that taking the limit $\nu_0\to 0$,  \eqref{eq:lambda with nu0} converges to \eqref{eq:lambda without nu0}. Moreover, when $\eta_0=0$, the solution satisfies $\lambda(t;0,\nu_0)\to 0$ as $t\to 0$. This sheds light on the fact that the Galilean parameter does not introduce any instability into the blow-up dynamics.
\end{remark}

Before we start the proof of Proposition~\ref{Singular profile}, we recall a technical lemma concerning $(L_Q)^{-1}$:
\begin{lemma}[\cite{KLR2013ARMAhalfwave}, Appendix A]
\label{lem:L inverse bdd}
    Let $f,g\in H^k(\bbR)$ be real-valued functions for some $k\geq0$,
    and suppose that
    \[
        (f,Q)_r=0,
        \qquad
        (g,\nabla Q)_r=0.
    \]
    Then we have the regularity bounds
    \begin{align*}
        \|L_Q^{-1}[if]\|_{H^{k+1}}
        \lesssim_k \|f\|_{H^k},\qquad
        \|L_Q^{-1}[g]\|_{H^{k+1}}
        \lesssim_k \|g\|_{H^k},
    \end{align*}
    and the decay estimates
    \begin{align*}
        \|\langle x\rangle^2L_Q^{-1}[if]\|_{L^\infty}
        \lesssim
        \|\langle x\rangle^2f\|_{L^\infty},\qquad
        \|\langle x\rangle^2L_Q^{-1}[g]\|_{L^\infty}
        \lesssim
        \|\langle x\rangle^2g\|_{L^\infty}.
    \end{align*}
\end{lemma}
The proof of Lemma \ref{lem:L inverse bdd} can be found in \cite[Appendix A]{KLR2013ARMAhalfwave}.

\begin{proof}[Proof of Proposition~\ref{Singular profile}] 
    We define the total order of parameters as  $k \coloneqq p + 2q + 2r$, where $(p,q,r) \in \mathbb{N}_{\mathcal{P}}$ . The order of parameter is motivated by the relative sizes of the parameters in the formal laws  \eqref{eq:formal conserv}, $b^2+\eta\sim \lambda$ and $|\nu| \sim \lambda$. 
    
    We begin by briefly outlining the strategy. The goal is to minimize the profile error  $\Psi_{\mathcal{P}}$ of modified profile $Q_\mathcal{P} \coloneqq Q+P= Q + \sum_{(p,q,r)\in \bbN_{\calP}}(ib)^p(i\nu)^q\eta^r R_{p,q,r}$, where
    \begin{equation}
        \begin{aligned}
            \Psi_{\mathcal{P}} &\coloneqq iL_QP - \mathcal{P}_1(b,\nu,\eta)\partial_b Q_\mathcal{P} - \mathcal{P}_2(b,\nu,\eta)\partial_\nu Q_{\mathcal{P}}\\
            &\quad + (b + \mathcal{P}_3(b,\nu,\eta))\Lambda Q_{\mathcal{P}} 
            - (\nu + \mathcal{P}_4(b,\nu,\eta))\nabla Q_{\mathcal{P}} - i\mathcal{N}_{\mathcal{P}}, \\
             \calN_{\calP} &\coloneqq|Q_{\mathcal{P}}|^2Q_{\mathcal{P}}-Q^3-Q^2P-2\Re(Q^2P).
        \end{aligned}
    \end{equation}
We will iteratively determine polynomials $\{\mathcal{P}_j\}_{j=1,2,3,4}$ and the corrector profiles $R_{p,q,r}$ to eliminate all terms in $\Psi_{\mathcal{P}}$ of total parameter order $k\le 5$. For instance, to eliminate terms of the order $(ib)^1(i\nu)^0\eta^0$ in  $\Psi_{\mathcal{P}}$, we choose $R_{1,0,0}$ by $$R_{1,0,0} = i^{-1}(L_Q)^{-1}[i\Lambda Q],$$
so that the $b$-coefficient terms in $\mathcal{P}_j$ vanish, and no order $k=1$ term appears in $\Psi_{\mathcal{P}}$. To carry out this procedure for each order $k$, we need to investigate the solvability of the equation $L_Q[f] = g$. A necessary and sufficient condition for the existence of a solution $f\in L^2(\mathbb{R})$ is that $g\in L^2(\mathbb{R})$ satisfies the orthogonality conditions $(g,\nabla Q)_r = (g,iQ)_r = 0$. 
We begin by investigating the decay of $R_{p,q,r}$. Due to the structure of the Hilbert transform and the solvability conditions by choosing the coefficients of $\mathcal{P}_{j}$ to satisfy the required orthogonality conditions, the corrector terms $R_{p,q,r}(y)$ exhibit decay of order $\langle y \rangle^{-2}$. This decay follows inductively from Lemma~\ref{lem:L inverse bdd} inductively. We first construct profiles $R_{p,q,r}(y)$ and coefficients of $\mathcal{P}_j$ and then prove the decay estimate \eqref{decay of R(pqr)}.
    Before we proceed to the proof, we recall the commutator formula.
    \begin{equation}\label{comm 1}
        [\Lambda , \nabla] = -\nabla,
    \end{equation}
    \begin{equation}\label{comm 2}
        [i^{-1}L_Q,\nabla]f = -i\nabla(Q^2)f - 2i\Re\{\nabla(Q^2)f\},
    \end{equation}
    \begin{equation}\label{comm 3}
        [i^{-1}L_Q,\Lambda]f = -i(Df + y\nabla(Q^2)f + 2\Re(y\nabla(Q^2)f)),
    \end{equation}
    for some function $f:\mathbb{R} \rightarrow \mathbb{C}$.

    \vspace{10bp}
    \textbf{Construction of $R_{p,q,r}(y)$.}   We construct the profiles $R_{p,q,r}$ in increasing order of the total parameter $k=p + 2q + 2r$.

    \vspace{5bp}
    
    \noindent $\boldsymbol{k=0}$ \textbf{:}
    Recalling the kernel relation of $L_Q$:
    \begin{equation*}
        i^{-1}L_{Q}[iQ] = i^{-1}L_{Q}[\nabla Q]  =0, \qquad i^{-1}L_Q[\Lambda Q]= iQ,
    \end{equation*}
    we observe that $k=0$-terms in $\Psi_{\mathcal{P}}$.
    
    \vspace{5bp}
    \noindent $\boldsymbol{k=1}$ \textbf{:} There is only one case: $(p,q,r) = (1,0,0)$.

    \vspace{5bp}
    \noindent\textbullet\  Order $\mathcal{O}(b)$: Collecting $\mathcal{O}(b)$-terms in \eqref{modified profile eror} we solve  
    \begin{equation}\label{order b}
        i^{-1}L_{Q} [iR_{1,0,0}] = \Lambda Q.
    \end{equation}
    We have $(i\Lambda Q, iQ)_r = (\Lambda Q,Q)=0$, since $\Lambda $ is antisymmetric. In addition, from real-imaginary consideration we have $(i\Lambda Q, \nabla Q)_r=0$. Thus, the solvability condition of the linearized operator $L_Q$ is satisfied and there is a unique solution $R_{1,0,0}$ to \eqref{order b}. Note that $R_{1,0,0}$ is even and $(iR_{1,0,0},iQ)_r = 0$.

    \vspace{5bp}
    \noindent $\boldsymbol{k=2}$ \textbf{:} We consider the case $(p,q,r) = (2,0,0),\ (0,1,0),\ (0,0,1)$.

    \vspace{5bp}
    \noindent\textbullet \ Order $\mathcal{O}(b^2)$: Collecting $\mathcal{O}(b^2)$-terms in \eqref{modified profile eror} we solve   
    \begin{equation}\label{order b^2}
        \begin{aligned}
            L_Q [R_{2,0,0}] &= -\frac{1}{2}R_{1,0,0} + \Lambda R_{1,0,0} - R_{1,0,0}^2Q \eqqcolon \text{(RHS)}_{b^2},
        \end{aligned}
    \end{equation}
    Note that $\Lambda R_{1,0,0} $ and $ R^2_{1,0,0} $ are from $b\Lambda Q_{\mathcal{P}}$ and $- i\mathcal{N}_{\mathcal{P}}$ in \eqref{modified profile eror} respectively. The first term $-\frac{1}{2}R_{1,0,0}$ arises from the $b^2$-term in the coefficient of $\partial_bQ_{\mathcal{P}}$ in \eqref{modified profile eror}. The coefficient $-\frac{1}{2}$ is determined to satisfy the solvability condition at order $O(b^3)$. We will see later. The solvability condition for \eqref{order b^2} is obvious, from even-odd or real-imaginary products. In addition, the unique solution $R_{2,0,0}$ is even and satisfies $(R_{2,0,0},\nabla Q)_r=0$.
    
    \vspace{5bp}
    \noindent \textbullet \ Order $\mathcal{O}(\nu)$: Collecting $\mathcal{O}(\nu)$-term in \eqref{modified profile eror} we solve
    \begin{equation}\label{order v}
        i^{-1}L_{Q} [iR_{0,1,0}] = - \nabla Q.
    \end{equation}
    Note that $-\nabla Q$ is from $-\nu\nabla Q_{\mathcal{P}}$. The solvability condition at order for \eqref{order v} is obvious, from even-odd or real-imaginary products. Furthermore, the unique solution $R_{0,1,0}$ is odd. In particular, $(R_{0,1,0},Q)_r=0$ by parity.

    \vspace{5bp}
    \noindent \textbullet \ Order $\mathcal{O}(\eta)$: Collecting $\mathcal{O}(\eta)$-term in \eqref{modified profile eror} we solve
    \begin{equation}\label{order eta}
        L_Q[R_{0,0,1}] = R_{1,0,0}.
    \end{equation}
    Note that $R_{1,0,0}$ comes from the multiplication with the $\eta$-term in the coefficient of $\partial_b Q_{\mathcal{P}}$, as well as from the $R_{1,0,0}$-term in $\partial_b Q_{\mathcal{P}}$ in \eqref{modified profile eror}. From \eqref{order b}, the solvability condition is satisfied due to the real-imaginary product, and $R_{1,0,0}$ is an even function.

    \vspace{5bp}
    \noindent $\boldsymbol{k=3}$ \textbf{:} There are three cases: $(p,q,r) = (3,0,0),\ (1,1,0), \ (1,0,1)$.

    \vspace{5bp}
    \noindent \textbullet \ Order $\mathcal{O}(b^3)$: Collecting $\mathcal{O}(b^3)$-term in \eqref{modified profile eror} we solve
    \begin{equation}\label{order b^3}
        \begin{aligned}
            i^{-1}L_Q[iR_{3,0,0}] &= -R_{2,0,0} + \Lambda R_{2,0,0} + \text{NL}_{b^3} \eqqcolon \text{(RHS)}_{b^3},\\
            \text{NL}_{b^3} &= 2R_{2,0,0}R_{1,0,0}Q - R_{1,0,0}^3 . 
        \end{aligned}
    \end{equation}
    Note that $\Lambda R_{2,0,0}$ and $\text{NL}_{b^3}$ originate from the terms $b\Lambda Q_{\mathcal{P}}$ and $-i\mathcal{N}_{\mathcal{P}}$ in \eqref{modified profile eror}, respectively. The first term $-R_{2,0,0}$ arises from the $b^2$-term in the coefficient of $\partial_b Q_{\mathcal{P}}$. The coefficient $-\frac{1}{2}$ in front of $b^2$ is chosen to satisfy the solvability condition, and this choice was first introduced in \cite[Proposition 4.1]{KLR2013ARMAhalfwave}. For the reader's convenience, we recall that this choice of $-\frac{1}{2}$ ensures that the solvability condition 
    $$(i(\text{RHS})_{b^3}, iQ)_r = 0$$
    is satisfied.
    Indeed, from \eqref{order b}, we have
    \begin{align*}
        (i(\text{RHS})_{b^3},iQ)_r 
        &= -(R_{2,0,0},Q)_r + \left(R_{2,0,0}, -i^{-1}L_Q[iR_{1,0,0}] + 2Q^2 R_{1,0,0}\right)_r - (R_{1,0,0}^3,Q)_r \\
        &= (R_{2,0,0}, i^2 Q)_r + \frac{1}{2} \|R_{1,0,0}\|_{L^2}^2.
    \end{align*}
    We now claim that
    \begin{equation}\label{relation of R(200) and R(100)}
        (R_{2,0,0}, i^2 Q)_r + \frac{1}{2} \|R_{1,0,0}\|_{L^2}^2 = 0.
    \end{equation}
    In fact, using the identity $i^{-1}L_Q[\Lambda Q] = iQ$ together with \eqref{order b^2}, we compute
    \begin{align*}
        \text{(LHS of \eqref{relation of R(200) and R(100)})} 
        &= (L_Q R_{2,0,0}, \Lambda Q)_r + \frac{1}{2} (R_{1,0,0}, R_{1,0,0})_r \\
        &= -\frac{1}{2} (iR_{1,0,0}, L_Q[iR_{1,0,0}] - iR_{1,0,0})_r \\
        &\quad + (\Lambda(iR_{1,0,0}), L_Q[iR_{1,0,0}])_r - (R_{1,0,0}^2, Q \Lambda Q)_r.
    \end{align*}
    Moreover, by the commutator identity \eqref{comm 3}, we have
    \begin{equation*}
        (\Lambda(iR_{1,0,0}), L_Q[iR_{1,0,0}])_r = \frac{1}{2} \left( \{ D + y \nabla(Q^2) \} [iR_{1,0,0}], iR_{1,0,0} \right)_r.
    \end{equation*}
    Combining the above identity with the pointwise identity $-\frac{y}{2} \nabla(Q^2) + Q \Lambda Q = \frac{1}{2} Q^2$, we obtain \eqref{relation of R(200) and R(100)}.

    On the other hand, the real-imaginary orthogonality condition yields $(i(\text{RHS})_{b^3}, \nabla Q)_r = 0$. Hence, there exists a unique solution $R_{3,0,0}$ to \eqref{order b^3}. Moreover, $R_{3,0,0}$ is an even function satisfying the orthogonality condition $(iR_{3,0,0}, iQ)_r = 0$.

    \vspace{5bp}
    \noindent \textbullet \  Order $\mathcal{O}(b\nu)$: Collecting $\mathcal{O}(b\nu)$-term in \eqref{modified profile eror} we solve
    \begin{equation}\label{order bv}
        \begin{aligned}
            L_Q[R_{1,1,0}] &= -R_{0,1,0} + \Lambda R_{0,1,0} - \nabla R_{1,0,0} + \text{NL}_{b\nu} \eqqcolon \text{(RHS)}_{b\nu},\\
            \text{NL}_{b\nu} &= - 2R_{1,0,0}R_{0,1,0}Q.
        \end{aligned}
    \end{equation}
    Note that the terms $\Lambda R_{0,1,0}$, $-\nabla R_{1,0,0}$, and $\mathrm{NL}_{b\nu}$ originate from the expansion of $b\Lambda Q_{\mathcal{P}}$, $-\nu\nabla Q_{\mathcal{P}}$, and $-i\mathcal{N}_{\mathcal{P}}$, respectively. The first term $-R_{0,1,0}$ arises from $-b\nu$-term in the coefficient of $\partial_\nu Q_{\mathcal{P}}$. The choice of coefficient $-1$ for the $b\nu$-term, which was found in \cite[Proposition 4.1]{KLR2013ARMAhalfwave}, is designed to satisfy the solvability condition. For the reader's convenience, we recall that this choice ensures that $((\mathrm{RHS})_{b\nu}, \nabla Q)_r=0$. 

    Indeed, applying the commutator formula \eqref{comm 1}, we compute
    \begin{align*}
        (\Lambda R_{0,1,0}, \nabla Q)_r 
        &= \big([\Lambda, \nabla] R_{0,1,0}, Q\big)_r + \big(\nabla R_{0,1,0}, \Lambda Q\big)_r \\
        &= \big(R_{0,1,0}, \nabla Q\big)_r + \big(\nabla R_{0,1,0}, i^{-1} L_Q [i R_{1,0,0}]\big)_r,
    \end{align*}
    where in the last equality we have used \eqref{order b}. 

    On the other hand, by applying \eqref{comm 2} and \eqref{order v}, we obtain
    \begin{align*}
        -\big(\nabla R_{1,0,0}, \nabla Q\big)_r 
        &= \big(\nabla R_{1,0,0}, i^{-1} L_Q [i R_{0,1,0}]\big)_r \\
        &= \big([i^{-1} L_Q, \nabla] i R_{1,0,0}, R_{0,1,0}\big)_r + \big(\nabla (i^{-1} L_Q [i R_{1,0,0}]), R_{0,1,0}\big)_r \\
        &= \big(\nabla (Q^2) R_{1,0,0}, R_{0,1,0}\big)_r - \big(\nabla R_{0,1,0}, i^{-1} L_Q [i R_{1,0,0}]\big)_r.
    \end{align*}
    Combining the above identities, we verify that $((\mathrm{RHS})_{b\nu}, \nabla Q)_r=0$. Moreover, a direct computation using the real-imaginary pairing yields $((\mathrm{RHS})_{b\nu}, iQ)_r=0$. Therefore, one obtains the existence and uniqueness of an odd solution $R_{1,1,0}$ to \eqref{order bv} satisfying the orthogonality condition $(R_{1,1,0},\nabla Q)_r=0$.

    \vspace{5bp}
    \noindent \textbullet \ Order $\mathcal{O}(b\eta)$: Collecting $\mathcal{O}(b\eta)$-term in \eqref{modified profile eror} we solve
    \begin{equation}\label{order b eta}
        \begin{aligned}
            i^{-1}L_Q[iR_{1,0,1}] &= 2R_{2,0,0} + \Lambda R_{0,0,1} + \text{NL}_{b\eta} \eqqcolon (\text{RHS})_{b\eta},\\
            \text{NL}_{b\eta} &= 2R_{0,0,1}R_{1,0,0}Q.
        \end{aligned}
    \end{equation}
    Note that $\Lambda R_{0,0,1}$ and $\mathrm{NL}_{b\eta}$ originate from the expansion of $b\Lambda Q_{\mathcal{P}}$ and $-i\mathcal{N}_{\mathcal{P}}$, respectively. In addition, the first term $2R_{2,0,0}$ arises from the product of the $\eta$-term in the coefficient of $\partial_b Q_{\mathcal{P}}$ and the term $2i^2bR_{2,0,0}$ appearing in $\partial_b Q_{\mathcal{P}}$ in \eqref{modified profile eror}. We now verify the solvability condition.

    By using \eqref{order b}, \eqref{order eta}, and the relation between $R_{2,0,0}$ and $R_{1,0,0}$ given in \eqref{relation of R(200) and R(100)}, we compute
    \begin{align*}
        (i(\mathrm{RHS})_{b\eta}, iQ)_r 
        &= \|R_{1,0,0}\|_{L^2}^2 + \big(R_{0,0,1}, 2Q^2 R_{1,0,0} - i^{-1} L_Q[iR_{1,0,0}] \big)_r \\
        &= \|R_{1,0,0}\|_{L^2}^2 - \big(L_Q[R_{0,0,1}], R_{1,0,0}\big)_r \\
        &= 0.
    \end{align*}
    On the other hand, a direct computation using the real-imaginary pairing yields $(i(\mathrm{RHS})_{b\eta}, \nabla Q)_r = 0$. Therefore, one obtains the existence and uniqueness of an even solution to \eqref{order b eta} satisfying the orthogonality condition $(R_{1,0,1}, Q)_r=0$.
    \vspace{5bp}

    \noindent $\boldsymbol{k=4}$ \textbf{:} There are six cases: $(p,q,r) = (4,0,0),\ (2,1,0),\ (2,0,1),\ (0,1,1)$, $\ (0,2,0),\ (0,0,2)$.
    
    Since $k$ is even, the index $p$ is also even. Thus, we distinguish between the cases where $q$ is even or odd. A straightforward calculation yields
    \begin{equation}\label{dichoto for q}
        \begin{aligned}
            i^{-1}L_Q[i^{p+q}R_{p,q,r}] &= (\text{odd function}), \quad \text{if } q \equiv 1 \pmod{2}, \\
            L_Q[i^{p+q}R_{p,q,r}] &= (\text{even function}), \quad \text{if } q \equiv 0 \pmod{2},
        \end{aligned}
    \end{equation}
    for $k = p + 2q + 2r = 4$. 

    As a consequence, from even-oddness and real-imaginary product yield the solvability conditions of \eqref{dichoto for q} for arbitrary $c_1, c_2, c_3, c_4 \in \mathbb{R}$ appearing in \eqref{modified profile eror}. Each constant $c_j$ that arises in the equation for $R_{p,q,r}$ will be determined at the next order corresponding to the index $(p+1,q,r)$. For later use at order $k=5$, we record here the equation for $R_{p,q,r}$ with $p + 2q + 2r = 4$.

    \vspace{5bp}
    \noindent \textbullet \ Order $\mathcal{O}(b^4)$: Collecting $\mathcal{O}(b^4)$-term in \eqref{modified profile eror} we solve
    \begin{equation}\label{order b^4}
        \begin{aligned}
            L_Q[R_{4,0,0}] &= c_1R_{1,0,0} -\frac{3}{2}R_{3,0,0} + \Lambda R_{3,0,0} + \text{NL}_{b^4} \eqqcolon (\text{RHS})_{b^4},\\
            \text{NL}_{b^4} &= 3R_{2,0,0}^2Q - 2R_{3,0,0}R_{1,0,0}Q - R_{2,0,0}R_{1,0,0}^2.
        \end{aligned}
    \end{equation}
    We note that $c_1 R_{1,0,0}$ arises from the multiplication of $-c_1b^4$-term in the coefficient of $\partial_b Q_{\mathcal{P}}$ and $iR_{1,0,0}$ appearing in $\partial_b Q_{\mathcal{P}}$. Also, $ - \frac{3}{2} R_{3,0,0}$ arises from the multiplication of $-\frac{1}{2}b^2$-term in the coefficient of $\partial_b Q_{\mathcal{P}}$ and $3i^3b^2R_{3,0,0}$-term appearing in $\partial_b Q_{\mathcal{P}}$. Moreover, the terms $\Lambda R_{3,0,0}$ and $\mathrm{NL}_{b^4}$ originate from the expansion of $b\Lambda Q_{\mathcal{P}}$ and $-i \mathcal{N}_{\mathcal{P}}$, respectively. Thus, applying the even–odd and real–imaginary products yields the solvability conditions for \eqref{order b^4}, and hence for each $c_1 \in \bbR$, there exists a unique even solution $R_{4,0,0}$ to \eqref{order b^4}, depending on $c_1$, satisfying $(R_{4,0,0}, \nabla Q)_r = 0$.

    \vspace{5bp}
    \noindent \textbullet \ Order $\mathcal{O}(b^2\nu)$: Collecting $\mathcal{O}(b^2\nu)$-term in \eqref{modified profile eror} we solve
    \begin{equation}\label{order b^2v}
        \begin{aligned}
            i^{-1}L_Q[iR_{2,1,0}] &= c_2\nabla Q -\frac{3}{2}R_{1,1,0} + \Lambda R_{1,1,0} -\nabla R_{2,0,0} + \text{NL}_{b^2\nu} \\
            &\eqqcolon (\text{RHS})_{b^2\nu},\\
            \text{NL}_{b^2\nu} &= 2R_{1,1,0}R_{1,0,0}Q + 2R_{2,0,0}R_{0,1,0}Q - 3R_{0,1,0}R_{1,0,0}^2.
        \end{aligned}
    \end{equation}
    Note that $c_2 \nabla Q$ comes from the multiplication of $c_2 b^2 \nu$-term in the coefficient of $\nabla Q_{\mathcal{P}}$ and $\nabla Q$, appearing in $\nabla Q_{\mathcal{P}}$. Also, $-\frac{3}{2} R_{1,1,0}$ is the sum of two contributions: the product of the $-\frac{1}{2}b^2$-term in the coefficient in $\partial_b Q_{\mathcal{P}}$ and $i^2\nu R_{1,1,0}$, and the product of the $-b\nu$-term in the coefficient of $\partial_\nu Q_{\mathcal{P}}$ with $i^2bR_{1,1,0}$.
    Finally, $\Lambda R_{1,1,0}$ arises from the expansion of $b \Lambda Q_{\mathcal{P}}$, and $-\nabla R_{2,0,0}$ arises from the expansion of $-\nu \nabla Q_{\mathcal{P}}$. Finally, $\mathrm{NL}_{b^2 \nu}$ arises from the expansion of $-i \mathcal{N}_{\mathcal{P}}$. From even-odd or real-imaginary product yield that solvability conditions of \eqref{order b^2v} are obviously satisfied. Therefore, for each $c_2 \in \mathbb{R}$, there exists a unique odd solution $R_{2,1,0}$ to \eqref{order b^2v}, depending on $c_2$, and orthogonal to $Q$.  

    \vspace{5bp}
    \noindent \textbullet \ Order $\mathcal{O}(b^2\eta)$: Collecting $\mathcal{O}(b^2\eta)$-term in \eqref{modified profile eror} we solve
    \begin{equation}\label{order b^2 eta}
        \begin{aligned}
            L_Q[R_{2,0,1}] &= -c_3R_{1,0,0} + 3R_{3,0,0} - \frac{1}{2}R_{1,0,1} + \Lambda R_{1,0,1} + \text{NL}_{b^2\eta} \\
            &\eqqcolon (\text{RHS})_{b^2\eta},\\
            \text{NL}_{b^2\eta} &= 6R_{2,0,0}R_{0,0,1}Q -2R_{1,0,1}R_{1,0,0}Q - R_{0,0,1}R_{1,0,0}^2.
        \end{aligned}
    \end{equation}
    Note that $-c_3R_{1,0,0}$ arises from the multiplication of $-c_3b^2\eta$-term in the coefficient of $\partial_b Q_{\mathcal{P}}$ and $iR_{1,0,0}$ appearing in the expansion of $\partial_b Q_{\mathcal{P}}$. Also,   $3R_{3,0,0}-\frac{1}{2}R_{1,0,1}$ arises from the $\eta$-term and $-\frac{1}{2}b^2$-term in the coefficient of $\partial_b Q_{\mathcal{P}}$. In addition, $\Lambda R_{1,0,1}$ and $\text{NL}_{b^2\eta}$ come from the expansion of $b\Lambda Q_{\mathcal{P}}$ and $-i\mathcal{N}_{\mathcal{P}}$, respectively. Thus, applying the even–odd and real–imaginary products yield the solvability conditions for \eqref{order b^2 eta}. Hence, for each $c_3 \in \bbR$, there is a unique even solution $R_{2,0,1}$ to \eqref{order b^2 eta} depending on $c_3$ so that orthogonal to $\nabla Q$.

    \vspace{5bp}
    \noindent \textbullet \ Order $\mathcal{O}(\nu^2)$: Collecting $\mathcal{O}(\nu^2)$-term in \eqref{modified profile eror} we solve
    \begin{equation}\label{order v^2}
        \begin{aligned}
            L_Q[R_{0,2,0}] &= -c_4R_{1,0,0} - \nabla R_{0,1,0} + \text{NL}_{\nu^2} \eqqcolon (\text{RHS})_{\nu^2}\\
            \text{NL}_{\nu^2} &= - R_{0,1,0}^2Q.
        \end{aligned}
    \end{equation}
    We note that $-c_4R_{1,0,0}$ arises from the multiplication of $-c_4\nu^2$ in the coefficient of $\partial_b Q_{\mathcal{P}}$ in \eqref{modified profile eror},
    $\partial_bQ_{\mathcal P}\big|_{\mathcal P=0} =iR_{1,0,0}.$ Also, $-\nabla R_{0,1,0}$ and $\text{NL}_{\nu^2}$ are from the expansion of $-\nu\nabla Q_{\mathcal{P}}$ and $-\mathcal{N}_{\mathcal{P}}$, respectively. Thus, applying the even–odd and real–imaginary products yield the solvability conditions for \eqref{order v^2}. Therefore, for each $c_4 \in \bbR$, there is a unique even solution to \eqref{order v^2} depending on $c_4$, orthogonal to $\nabla Q$.

    \vspace{5bp}
    \noindent \textbullet \ Order $\mathcal{O}(\nu\eta)$: Collecting $\mathcal{O}(\nu\eta)$-term in \eqref{modified profile eror} we solve
    \begin{equation}\label{order v eta}
        \begin{aligned}
            i^{-1}L_Q[iR_{0,1,1}] &= R_{1,1,0} - \nabla R_{0,0,1} + \text{NL}_{\nu\eta} \eqqcolon (\text{RHS})_{\nu\eta} \\
            \text{NL}_{\nu\eta} &= 2R_{0,1,0}R_{0,0,1}Q.
        \end{aligned}
    \end{equation}
    Note that $R_{1,1,0}$ arises from the multiplication of $\eta$-term in the coefficient of $\partial_b Q_{\mathcal{P}}$ and $i^2\nu R_{1,1,0}$ in the expansion of $\partial_bQ_{\mathcal{P}}$ in \eqref{modified profile eror}. In addition, $-\nabla R_{0,0,1}$ and $\text{NL}_{\nu\eta}$ are from the expansion of $\nabla Q_{\mathcal{P}}$ and $-\mathcal{N}_{\mathcal{P}}$, respectively. Thus, applying the even–odd and real–imaginary products yield the solvability conditions for \eqref{order v eta}. Hence, there is a unique odd solution $R_{0,1,1}$ to \eqref{order v eta}, orthogonal to $Q$.
    
    \vspace{5bp}
    \noindent \textbullet \ Order $\mathcal{O}(\eta^2)$: Collecting $\mathcal{O}(\eta^2)$-term in \eqref{modified profile eror} we solve
    \begin{equation}\label{order eta^2}
        \begin{aligned}
            L_Q[R_{0,0,2}] &= R_{1,0,1} + \text{NL}_{\eta^2} \eqqcolon (\text{RHS})_{\eta^2}\\
            \text{NL}_{\eta^2} &= 3R_{0,0,1}^2Q.
        \end{aligned}
    \end{equation}
    We note that $R_{1,0,1}$ arises from the $\eta$-term in the coefficient of $\partial_b Q_{\mathcal{P}}$ and $i\eta R_{1,0,1}$ in the expansion of $\partial_b Q_{\mathcal{P}}$ in \eqref{modified profile eror}. In addition, $\text{NL}_{\eta^2}$ arises from the expansion of $-i\mathcal{N}_{\mathcal{P}}$. Thus, applying the even–odd and real–imaginary products yield the solvability conditions for \eqref{order eta^2}. Therefore, we conclude that there exists a unique even solution to \eqref{order eta^2}, orthogonal to $\nabla Q$.

    \vspace{5bp}
    \noindent $\boldsymbol{k=5}$ \textbf{:} There are six cases: $(p,q,r) = (5,0,0),\ (3,1,0),\ (3,0,1),\ (1,2,0),\ (1,1,1)$, and $(1,0,2)$. In this order, we determine the coefficients $c_j$ to satisfy the solvability conditions for $L_Q$ at the order $\mathcal{O}(b^p\nu^q\eta^r)$, where $k = p+2q+2r= 5$. More precisely, we first fix $c_1$ by satisfying the solvability conditions for $L_Q$ at $(p,q,r)=(5,0,0)$, and determine $c_2$ from $(3,1,0)$, $c_3$ from $(3,0,1)$, and $c_4$ from $(1,2,0)$. Once these coefficients are fixed, we verify that the corresponding solvability conditions at all remaining order $(p,q,r)$ with $p + 2q + 2r = 5$ are satisfied for the fixed $c_j$.

    We note that at the total order $k = 4$, the solvability conditions for constructing the profiles $R_{4,0,0}$, $R_{2,1,0}$, $R_{2,0,1}$, and $R_{0,2,0}$ are automatically satisfied for each $c_j \in \mathbb{R}$. This is because the right-hand sides of the corresponding equations (see \eqref{order b^4}, \eqref{order b^2v}, \eqref{order b^2 eta}, \eqref{order v^2}) are orthogonal to $\ker L_Q$, due to structural properties such as even–odd parity or real–imaginary cancellation. As a result, each profile can be constructed with an affine dependence on the corresponding parameter $c_j$.

   At order $k = 5$, however, such structural cancellations no longer hold, and the solvability conditions must be taken into account explicitly. To ensure compatibility with these conditions and reduce the modified profile error, we introduced the parameters $\{c_j\}$. This differs from the construction in \cite[Proposition 4.1]{KLR2013ARMAhalfwave}, where no such correction was needed since the profile error was not required to be as small. 
   
   While the profiles $R_{p,q,r}$ at order $k = 4$ are not explicitly linear in $c_j$, the right-hand sides of their governing equations \eqref{order b^4}, \eqref{order b^2v}, \eqref{order b^2 eta}, and \eqref{order v^2} admit an affine structure with respect to $c_j$, and this structure propagates into the solvability conditions of $L_Q$ at order $k=5$. As a result, we are able to derive scalar equations for each $c_j$ from the solvability conditions of $L_Q$ at each order $k = 5$.
    
    We use the positivity of the linearized operator $L_Q$ in \eqref{positivity of L_Q 0} to ensure the uniqueness of the choice of parameter $c_j$: 
    \begin{equation}\label{positivity of L_Q}
        (i^{-1}L_Q[iR_{1,0,0}],R_{1,0,0})_r >0,\quad (i^{-1}L_Q[iR_{0,1,0}],R_{0,1,0})_r >0.
    \end{equation}

    \vspace{5bp}
    \noindent \textbullet \ Order $\mathcal{O}(b^5)$ : Collecting $\mathcal{O}(b^5)$-term in \eqref{modified profile eror} we solve
    \begin{equation}\label{order b^5}
        \begin{aligned}
            &i^{-1}L_Q[iR_{5,0,0}] = 2c_1 R_{2,0,0} -2R_{4,0,0} + \Lambda R_{4,0,0} +\text{NL}_{b^5} \eqqcolon (\text{RHS})_{b^5}\\
            \text{NL}_{b^5} & = 2R_{4,0,0}R_{1,0,0}Q + 2R_{3,0,0}R_{2,0,0}Q - 3R_{3,0,0}R_{1,0,0}^2 + R_{2,0,0}^2R_{1,0,0}. 
        \end{aligned}
    \end{equation}
    Note that $2c_1R_{2,0,0}$ arises from the multiplication of $-c_1b^4$-term of the coefficient of $\partial_b Q_{\mathcal{P}}$ and $2i^2bR_{2,0,0}$ in the expansion of $\partial_b Q_{\mathcal{P}}$, and $-2R_{4,0,0}$ is from $-\frac{1}{2}b^2$-term of the coefficient of $\partial_b Q_{\mathcal{P}}$ and $4i^4b^3 R_{4,0,0}$ of the expansion of $\partial_b Q_{\mathcal{P}}$. Also, $\Lambda Q_{\mathcal{P}}$ and $\text{NL}_{b^5}$ are from the expansion of $b\Lambda Q_{\mathcal{P}}$ and $-i\mathcal{N}_{\mathcal{P}}$, respectively. 
    To obtain the existence of $R_{5,0,0}$, we must choose $c_1$ so that the solvability condition $(i(\text{RHS}_{b^5}), iQ)_r = 0$ is satisfied. From \eqref{order b^4}, $R_{4,0,0}$ depends on $c_1$ through the relation
    \[
    L_Q[R_{4,0,0}] = c_1 R_{1,0,0} + \text{(other terms)}.
    \]
    To isolate the $c_1$-dependence, we define the inhomogeneous component of the right-hand side of \eqref{order b^4} as
    \begin{equation}
    A_{b^4} \coloneqq (\text{RHS})_{b^4} - c_1 R_{1,0,0},
    \end{equation}
    so that \eqref{order b^4} becomes $L_Q[R_{4,0,0}] = c_1 R_{1,0,0} + A_{b^4}$. This decomposition allows us to track the $c_1$-dependence of $R_{4,0,0}$ explicitly when evaluating the solvability condition in \eqref{order b^5}.
    Also, to isolate the contribution of $c_1$, we define 
    \[
        \tilde{\text{NL}}_{b^5} \coloneqq \text{NL}_{b^5} - 2R_{4,0,0}R_{1,0,0}Q,
    \]
    so that the solvability condition becomes
    \[
        (2c_1 R_{2,0,0} - 2R_{4,0,0} + \Lambda R_{4,0,0} + 2R_{4,0,0}R_{1,0,0}Q, Q)_r = -(i \tilde{\text{NL}}_{b^5}, iQ)_r.
    \]
    We now compute each term on the left-hand side in terms of $c_1$. Using \eqref{order b^4}, \eqref{order b}, and the identity $i^{-1}L_Q[\Lambda Q] = iQ$, we obtain
    \[
        -2(R_{4,0,0}, Q)_r = 2(L_Q[R_{4,0,0}], \Lambda Q)_r = 2c_1 (i^{-1}L_Q[iR_{1,0,0}], R_{1,0,0})_r + 2(A_{b^4}, \Lambda Q)_r,
    \]
    and similarly,
    \begin{align*}
        (\Lambda R_{4,0,0}, Q)_r + (2R_{4,0,0}R_{1,0,0}Q, Q)_r 
        &= (R_{4,0,0}, -i^{-1}L_Q[iR_{1,0,0}] + 2Q^2 R_{1,0,0})_r \\
        &= -c_1 \|R_{1,0,0}\|_{L^2}^2 - (A_{b^4}, R_{1,0,0})_r.
    \end{align*}
    Combining all contributions and using \eqref{relation of R(200) and R(100)}, we obtain a scalar equation of the form $Ac_1 + B = 0$, which yields the explicit formula
    \begin{equation}\label{value of c_1}
        c_1 = \frac{(A_{b^4}, R_{1,0,0} - 2\Lambda Q)_r - (i \tilde{\text{NL}}_{b^5}, iQ)_r}{2(i^{-1}L_Q[iR_{1,0,0}], R_{1,0,0})_r},
    \end{equation}
    from \eqref{positivity of L_Q}. Also, from the real-imaginary product, we conclude that $(iR_{5,0,0},\nabla Q)_r = 0$. Hence, there exists a unique even solution $R_{5,0,0}$ to \eqref{order b^5}, orthogonal to $Q$.

    \noindent \textbullet\ Order $\mathcal{O}(b^3\nu)$ : Collecting $\mathcal{O}(b^3\nu)$-term in \eqref{modified profile eror} we solve
    \begin{equation}\label{order b^3v}
        \begin{aligned}
            L_Q[R_{3,1,0}]
            &= c_2\nabla R_{1,0,0} - 2R_{2,1,0} + \Lambda R_{2,1,0} -\nabla R_{3,0,0} + \text{NL}_{b^3\nu}\\
            &\eqqcolon (\text{RHS})_{b^3\nu},\\
             \text{NL}_{b^3\nu} &=  -2R_{2,1,0}R_{1,0,0}Q -2R_{3,0,0}R_{0,1,0}Q - 2R_{2,0,0}R_{0,1,0}R_{1,0,0} \\
             & \quad + 6R_{1,1,0}R_{2,0,0}Q - R_{1,1,0}R_{1,0,0}^2.
        \end{aligned}
    \end{equation}
    Note that $c_2\nabla R_{1,0,0}$ arises from the multiplication of $-c_2b^2\nu$-term in the coefficient of $\nabla Q_{\mathcal{P}}$ and $ib\nabla R_{1,0,0}$ in the expansion of $\nabla Q_{\mathcal{P}}$, and $-\nabla (ib)^3R_{3,0,0}$ is from $\nu \nabla Q_{\mathcal{P}}$. Also, $-2R_{2,1,0}$ arises from combining two contributions: multiplication of $-\frac{1}{2}b^2$-term in the coefficient of $\partial_b Q_{\mathcal{P}}$ and $2i^3b\nu R_{2,1,0}$ in the expansion of $\partial_b Q_{\mathcal{P}}$ and multiplication of $-b\nu$-term of the coefficient of $\partial_\nu Q_{\mathcal{P}}$ and $i^3b^2R_{2,1,0}$ in the expansion of $\partial_\nu Q_{\mathcal{P}}$. Finally, $\text{NL}_{b^3\nu}$ is from $-i\mathcal{N}_{\mathcal{P}}$ in \eqref{modified profile eror}. 
    
    To obtain the existence of $R_{3,1,0}$, we must choose $c_2 \in \mathbb{R}$ so that the solvability condition associated with \eqref{order b^3v}, namely
    \[
        ((\text{RHS})_{b^3\nu}, \nabla Q)_r = 0,
    \]
    is satisfied. As in the case of $c_1$, we separate the $c_2$-dependence from the right-hand side by introducing
    \[
        A_{b^2\nu} \coloneqq (\text{RHS})_{b^2\nu} - c_2 \nabla Q,
    \]
    so that the contribution of $c_2$ can be treated explicitly. We also define 
    \[
        \tilde{\text{NL}}_{b^3\nu} \coloneqq \text{NL}_{b^3\nu} + 2R_{2,1,0}R_{1,0,0}Q,
    \]
    in order to separate $c_2$-dependence from \eqref{order b^3v}. With these definitions, the solvability condition becomes
    \begin{equation}\label{equation for c_2}
    \begin{aligned}
        (c_2\nabla R_{1,0,0} - 2R_{2,1,0} + \Lambda R_{2,1,0} - 2R_{2,1,0}R_{1,0,0}Q, \nabla Q)_r = (\nabla R_{3,0,0} - \tilde{\text{NL}}_{b^3\nu}, \nabla Q)_r.
    \end{aligned}
    \end{equation}
    To compute the left-hand side of \eqref{equation for c_2} in terms of $c_2$, we apply the commutator identity \eqref{comm 1}, \eqref{comm 2} along with the identity $i^{-1}L_Q[iR_{0,1,0}] = -\nabla Q$ from \eqref{order v} and $i^{-1}L_Q[iR_{2,1,0}] = c_2\nabla Q+A_{b^2\nu}$ from \eqref{order b^2v}. Using these, we express each term on the left-hand side of \eqref{equation for c_2} as follows:
    \begin{align*}
        (\Lambda R_{2,1,0}, \nabla Q)_r 
        &= - (R_{2,1,0}, \nabla(i^{-1}L_Q[iR_{1,0,0}]))_r + (R_{2,1,0}, \nabla Q)_r, \\
        (-2R_{2,1,0}R_{1,0,0}Q, \nabla Q)_r 
        &= - (R_{2,1,0}, \nabla(Q^2) R_{1,0,0})_r = - (R_{2,1,0}, [i^{-1}L_Q, \nabla]iR_{1,0,0})_r.
    \end{align*}
    Combining all contributions yields
    \begin{align*}
        (\text{LHS of } \eqref{equation for c_2}) 
        &= c_2(\nabla R_{1,0,0}, \nabla Q)_r - (R_{2,1,0}, \nabla Q)_r \\
        &\quad - (R_{2,1,0}, \nabla(i^{-1}L_Q[iR_{1,0,0}]))_r - (R_{2,1,0}, [i^{-1}L_Q, \nabla]iR_{1,0,0})_r \\
        &= c_2(\nabla Q, R_{0,1,0})_r + (A_{b^2\nu}, R_{0,1,0})_r - (A_{b^2\nu}, \nabla R_{1,0,0})_r.
    \end{align*}
    Therefore, from \eqref{positivity of L_Q} we obtain the explicit formula
    \begin{equation}\label{value of c_2}
        c_2 = - \frac{(A_{b^2\nu}, \nabla R_{1,0,0} - R_{0,1,0})_r + (\nabla R_{3,0,0} - \tilde{\text{NL}}_{b^3\nu}, \nabla Q)_r}{(i^{-1}L_Q[iR_{0,1,0}], R_{0,1,0})_r}.
    \end{equation}
    Also, from the real-imaginary product, we conclude that $(R_{3,1,0}, iQ)_r = 0$. Hence, there exists a unique odd solution $R_{3,1,0}$ to \eqref{order b^3v}, orthogonal to $\nabla Q$.

    \noindent \textbullet \ Order $\mathcal{O}(b^3\eta)$ : Collecting $\mathcal{O}(b^3\eta)$-term in \eqref{modified profile eror} we solve
    \begin{equation}\label{order b^3 eta}
        \begin{aligned}
            i^{-1}L_Q[iR_{3,0,1}] &= -2c_3R_{2,0,0} + 4R_{4,0,0} - R_{2,0,1} + \Lambda R_{2,0,1} + \text{NL}_{b^3\eta} \\
            &\eqqcolon (\text{RHS})_{b^3\eta},\\
            \text{NL}_{b^3\eta} &= 2R_{2,0,1}R_{1,0,0}Q + 2QR_{1,0,1}R_{2,0,0} - 3R_{1,0,1}R_{1,0,0}^2 \\
            & \quad + 2R_{3,0,0}R_{0,0,1}Q + 2R_{2,0,0}R_{0,0,1}R_{1,0,0}.
        \end{aligned}
    \end{equation}
    Note that $-2c_3R_{2,0,0}$ arises from the multiplication of $-c_3b^2\eta$-term in the coefficient of $\partial_b Q_{\mathcal{P}}$ and $2i^2bR_{2,0,0}$ in the expansion of $\partial_b Q_{\mathcal{P}}$, and $4R_{4,0,0}$ and $-R_{2,0,1}$ are from the multiplication of $\eta$-term, $-\frac{1}{2}b^2$-term in the coefficient of $\partial_b Q_{\mathcal{P}}$ and $4i^4b^3R_{4,0,0}$, $2i^2b\eta R_{2,0,1}$ in the expansion of $Q_{\mathcal{P}}$, respectively. Also, $\Lambda R_{2,0,1}$ and $\text{NL}_{b^3\eta}$ are from the expansion of $b\Lambda Q_{\mathcal{P}}$, $-i\mathcal{N}_{\mathcal{P}}$, respectively. 
    From real-imaginary product, we have $((\text{RHS})_{b^3\eta},\nabla Q)_r = 0$. Then, the solvability of \eqref{order b^3 eta} is equivalent to
    \begin{equation}\label{equation for c_3}
    \begin{aligned}
        (-2c_3R_{2,0,0} - R_{2,0,1} + \Lambda &R_{2,0,1} + 2R_{2,0,1}R_{1,0,0}Q,Q)_r \\ &= (i[-4R_{4,0,0} - \tilde{\text{NL}}_{b^3\eta}],iQ)_r,
    \end{aligned}
    \end{equation}
    where 
    \begin{equation*}
        \tilde{\text{NL}}_{b^3\eta} \coloneqq  \text{NL}_{b^3\eta} - 2R_{2,0,1}R_{1,0,0}Q.
    \end{equation*}
    From $i^{-1}L_Q[\Lambda Q]=iQ$, \eqref{order b^2 eta}, and $L_Q[R_{1,0,0}]=\Lambda Q-2Q^2R_{1,0,0},$ we have
    \begin{align*}
        \quad \quad  &(-R_{2,0,1}+\Lambda R_{2,0,1} + 2R_{2,0,1}R_{1,0,0}Q,Q)_r \\
        & = -(iR_{2,0,1},i^{-1}L_Q[\Lambda Q])_r + (R_{2,0,1},-\Lambda Q+ 2R_{1,0,0}Q^2)_r \\
        & = (i^{-1}L_Q[R_{2,0,1}],i^{-1}\Lambda Q)_r + (R_{2,0,1}, 2R_{1,0,0}Q^2 - i^{-1}L_Q[iR_{1,0,0}])_r\\
        & = -c_3(i^{-1}L_Q[iR_{1,0,0}], R_{1,0,0})_r + c_3\norm{R_{1,0,0}}_{L^2}^2  -(A_{b^2\eta},R_{1,0,0}-\Lambda Q)_r,
    \end{align*}
    where we denote $A_{b^2\eta}$ by
    \begin{equation*}
        A_{b^2\eta} \coloneqq (\text{RHS})_{b^2\eta} + c_3 R_{1,0,0},
    \end{equation*}
    to separate the $c_3$-dependence from the right-hand side of \eqref{order b^2 eta}. Hence, from \eqref{positivity of L_Q} we set $c_3$ as
    \begin{equation}\label{value of c_3}
        c_3 = \frac{(A_{b^2\eta},\Lambda Q-R_{1,0,0})_r+(i[4R_{4,0,0} + \tilde{\text{NL}}_{b^3\eta}],iQ)_r}{(i^{-1}L_Q[iR_{1,0,0}], R_{1,0,0})_r}.
    \end{equation}
    Therefore, there is a unique even solution $R_{3,0,1}$ to \eqref{order b^3 eta}, orthogonal to $Q$.

    \noindent \textbullet\ Order $\mathcal{O}(b\nu^2)$ : Collecting $\mathcal{O}(b\nu^2)$-term in \eqref{modified profile eror} we solve
    \begin{equation}\label{order bv^2}
        \begin{aligned}
            i^{-1}L_Q[iR_{1,2,0}] &= -2c_4R_{2,0,0} - \nabla R_{1,1,0} -2R_{0,2,0} + \Lambda R_{0,2,0} + \text{NL}_{b\nu^2} \\
            &\eqqcolon (\text{RHS})_{b\nu^2},\\
            \text{NL}_{b\nu^2} &= 2R_{0,2,0}R_{1,0,0}Q +2R_{1,1,0}R_{0,1,0}Q - 3R_{0,1,0}^2R_{1,0,0}.
        \end{aligned}
    \end{equation}
    We note that $-2c_4R_{2,0,0}$ arises from multiplication of $-c_4\nu^2$-term in the coefficient of $\partial_b Q_{\mathcal{P}}$ and $2i^2bR_{2,0,0}$ in the expansion of $\partial_bQ_{\mathcal{P}}$, and $-2R_{0,2,0}$ is from multiplication of $-b\nu$-term in the coefficient of $\partial_\nu Q_{\mathcal{P}}$ and $2i^2\nu R_{0,2,0}$ in the expansion of $\partial_\nu Q_{\mathcal{P}}$. Also, $-\nabla R_{1,1,0}$, $\Lambda R_{0,2,0}$, and $\mathrm{NL}_{b\nu^2}$ arise from the expansion of $-\nu\nabla Q_{\mathcal{P}}$, $b\Lambda Q_{\mathcal{P}}$, and $-i\mathcal{N}_{\mathcal{P}}$, respectively.

    From real-imaginary product, we obtain $(iR_{1,2,0}, \nabla Q)_r = 0$.
    To this end, we choose $c_4\in\mathbb R$ so that the solvability condition
    $((\mathrm{RHS})_{b\nu^2},Q)_r=0$ is satisfied. Once this condition is satisfied, we fix the solution $R_{1,2,0}$ by imposing $(R_{1,2,0},Q)_r=0$. To separate $c_4$-dependence from right-hand side of \eqref{order v^2} and \eqref{order bv^2}, we define $\tilde{\text{NL}}_{b\nu^2} \coloneqq \text{NL}_{b\nu^2} - 2R_{0,2,0}R_{1,0,0}Q$  and $A_{\nu^2} \coloneqq (\text{RHS})_{\nu^2} + c_4R_{1,0,0}$, Then, the solvability condition of \eqref{order bv^2} is equivalent to 
    \begin{equation}\label{equation for c_4}
        \begin{aligned}
            (-2c_4R_{2,0,0} - 2R_{0,2,0} + \Lambda R_{0,2,0} &+ 2R_{0,2,0}R_{1,0,0}Q,Q)_r \\
            &= (i[\nabla R_{1,1,0} - \tilde{\text{NL}}_{b\nu^2}],iQ)_r.
        \end{aligned}
    \end{equation}
    We express the left-hand side of \eqref{equation for c_4} in terms of $c_4$: From the relation $i^{-1}L_Q[\Lambda Q] = iQ$ and \eqref{order v^2}, we obtain
    \begin{align*}
        (-2R_{0,2,0},Q)_r &= 2(R_{0,2,0},L_Q[\Lambda Q])_r \\
        &= -2c_4(i^{-1}L_Q[iR_{1,0,0}],R_{1,0,0})_r + 2(\Lambda Q,A_{\nu^2})_r,
    \end{align*}
    and 
    \begin{align*}
        \quad \quad \ \ (\Lambda R_{0,2,0}+2R_{0,2,0}R_{1,0,0}Q,Q)_r 
        &= (R_{0,2,0},-i^{-1}L_Q[iR_{1,0,0}]+2Q^2R_{1,0,0})_r \\
        &= c_4 \norm{R_{1,0,0}}_{L^2}^2 - (A_{\nu^2},R_{1,0,0})_r.
    \end{align*}
    Therefore, from \eqref{relation of R(200) and R(100)}, (LHS) of \eqref{equation for c_4} becomes
    \begin{align*}
        (&\text{(LHS) of \eqref{equation for c_4}}) \\
        &= -2c_4(i^{-1}L_Q[iR_{1,0,0}],R_{1,0,0})_r + (A_{\nu^2},2\Lambda Q - R_{1,0,0})_r.
    \end{align*}
    Hence, we set $c_4$ as
    \begin{equation}\label{value of c_4}
        c_4 = \frac{(A_{\nu^2},2\Lambda Q-R_{1,0,0})_r - (i[\nabla R_{1,1,0} - \tilde{\text{NL}}_{b\nu^2}],iQ)_r}{2(i^{-1}L_Q[iR_{1,0,0}],R_{1,0,0})_r}.
    \end{equation}
    In fact, this yields $c_4 < 0$. For this, see Appendix~\ref{Proof of c_4<0}. 
    Also, real-imaginary product gives $(iR_{1,2,0},\nabla Q)_r=0$. Therefore, there is a unique even solution $R_{1,2,0}$ to \eqref{order bv^2}, orthogonal to $Q$.
    
    So far, we have fixed $c_1$ in \eqref{value of c_1}, $c_2$ in \eqref{value of c_2}, $c_3$ in \eqref{value of c_3}, and $c_4$ in \eqref{value of c_4}, so as to satisfy the solvability conditions at the multi-indices $(p,q,r)$ and $(p+1,q,r)$ simultaneously, for $(p,q,r) = (4,0,0)$, $(2,1,0)$, $(2,0,1)$, and $(0,2,0)$, respectively. We now claim that, for these choices of $\{c_j\}$ and the corresponding constructed profiles $R_{p,q,r}$, the solvability conditions at the remaining multi-indices $(1,0,2)$ and $(1,1,1)$ are also satisfied, without the need to introduce any further correction to the polynomial $\mathcal{P}_j(b,\nu,\eta)$. This ensures that the profiles $R_{1,0,2}$ and $R_{1,1,1}$ can be constructed accordingly, and that the remainder $\Psi_{\mathcal{P}}$ is of order at least $k = 6$ with respect to the parameters $b$, $\nu$, and $\eta$. To verify this, we check the solvability conditions at the orders $\mathcal{O}(b\eta^2)$ and $\mathcal{O}(b\nu\eta)$ explicitly.

    \vspace{5bp}
    \noindent \textbullet\ Order $\mathcal{O}(b\eta^2)$ : Collecting $\mathcal{O}(b\eta^2)$-term in \eqref{modified profile eror} we solve
    \begin{equation}\label{order b eta^2}
        \begin{aligned}
            i^{-1}L_Q[iR_{1,0,2}] &= 2R_{2,0,1} + \Lambda R_{0,0,2} + \text{NL}_{b\eta^2} \eqqcolon (\text{RHS})_{b\eta^2},\\
            \text{NL}_{b\eta^2} &= 2R_{0,0,2}R_{1,0,0}Q + 2R_{1,0,1}R_{0,0,1}Q + R_{0,0,1}^2R_{1,0,0}.
        \end{aligned}
    \end{equation}
    We note that $2R_{2,0,1}$ arises from the multiplication of $\eta$-term in the coefficient of $\partial_b Q_{\mathcal{P}}$ and $2i^2bR_{2,0,1}$ in the expansion of $\partial_b Q_{\mathcal{P}}$. Also, $\Lambda R_{0,0,2}$ and $\text{NL}_{b\eta^2}$ are from the expansion of $b\Lambda Q_{\mathcal{P}}$ and $-i\mathcal{N}_{\mathcal{P}}$, respectively.
    From the real-imaginary product, we deduce that  
    \[
        (iR_{1,0,2}, \nabla Q)_r = 0.
    \]  
    Hence, verifying the solvability condition of \eqref{order b eta^2} is equivalent to checking  
    \begin{equation}\label{solva of b eta^2}
        2(iR_{2,0,1}, iQ)_r = (i[-\Lambda R_{0,0,2} - \text{NL}_{b\eta^2}], iQ)_r.
    \end{equation}
    In order to verify the solvability condition, we explicitly compute both the left-hand side and the right-hand side of \eqref{solva of b eta^2}, and confirm that they coincide.

    We begin by analyzing the left-hand side, namely the inner product $(R_{2,0,1}, Q)_r$. Recall that once $c_3$ is fixed as in \eqref{value of c_3}, the function $R_{2,0,1}$ is determined accordingly. In particular, from \eqref{equation for c_3}, we obtain
    \begin{equation}\label{b eta^2 : eq1}
        \begin{aligned}
            (R_{2,0,1}, Q)_r &= (-2c_3 R_{2,0,0} + \Lambda R_{2,0,1} + 2R_{2,0,1} R_{1,0,0} Q, Q)_r \\
            &\quad + (i[4R_{4,0,0} + \tilde{\text{NL}}_{b^3\eta}], iQ)_r \\
            &= -(A_{b^2\eta}, R_{1,0,0})_r + (\tilde{\text{NL}}_{b^3\eta}, Q)_r + 4(R_{4,0,0}, Q)_r.
        \end{aligned}
    \end{equation}
    Thus, it remains to compute the inner products $(A_{b^2\eta}, R_{1,0,0})_r$ and $(R_{4,0,0}, Q)_r$.  
    From the equation \eqref{order b^2 eta}, we obtain
    \[
    \begin{aligned}
        -(A_{b^2\eta}, R_{1,0,0})_r 
        &= -3(R_{3,0,0}, R_{1,0,0})_r + (R_{1,0,1}, R_{1,0,0} + L_Q[R_{2,0,0}])_r + R_1,
    \end{aligned}
    \]
    where $R_1$ collects the remainder terms:
    \begin{equation*}
        R_1 \coloneqq -(\text{NL}_{b^2\eta}, R_{1,0,0})_r + (R_{1,0,1} R_{1,0,0}^2, Q)_r.
    \end{equation*}
    On the other hand, since $R_{4,0,0}$ is determined once the coefficient $c_1$ is fixed in \eqref{value of c_1}, the solvability identity at order $b^5$, together with \eqref{value of c_1}, implies
    \begin{equation}\label{b eta^2 : eq2}
        \begin{aligned}
            4(R_{4,0,0}, Q)_r 
            &= 2(2c_1 R_{2,0,0} + \Lambda R_{4,0,0} + 2R_{4,0,0} R_{1,0,0} Q, Q)_r + 2(i\tilde{\text{NL}}_{b^5}, iQ)_r \\
            &= -2(A_{b^4}, R_{1,0,0})_r + 2(\tilde{\text{NL}}_{b^5}, Q)_r \\
            &= 3(R_{3,0,0}, R_{1,0,0})_r - 2(\Lambda R_{3,0,0}, R_{1,0,0})_r + R_2,
        \end{aligned}
    \end{equation}
    where $R_2$ consists of the nonlinear remainder terms:
    \begin{equation*}
        R_2 \coloneqq -2(\text{NL}_{b^4}, R_{1,0,0})_r + 2(\tilde{\text{NL}}_{b^5}, Q)_r.
    \end{equation*}
    Therefore, combining \eqref{b eta^2 : eq1} and \eqref{b eta^2 : eq2}, we group the terms in $(R_{2,0,1}, Q)_r$ into three parts for the convenience of computation.
    
    We first define $I_1$ to be the collection of inner products involving only $R_{p,0,0}$ for $p \leq 3$:
    \begin{equation}\label{I_1}
        \begin{aligned}
            I_1 &\coloneqq 4(R_{4,0,0}, Q)_r - 3(R_{3,0,0}, R_{1,0,0})_r \\
            &= 2(R_{3,0,0}, \Lambda R_{1,0,0} - Q R_{1,0,0}^2 + 2R_{2,0,0} Q^2)_r 
            - 4(R_{2,0,0}^2 Q, R_{1,0,0})_r \\
            &\quad + 2(R_{2,0,0}, R_{1,0,0}^3)_r \\
            &= -2\|R_{2,0,0}\|_{L^2}^2 + (R_{3,0,0}, R_{1,0,0})_r.
        \end{aligned}
    \end{equation}
    The last equality follows from \eqref{order b^2} and \eqref{order b^3} by a straightforward computation.

    Next, we define $I_2$ to be the contribution from the terms involving $R_{1,0,1}$. Using \eqref{order b eta}, we compute
    \begin{equation}\label{I_2}
        \begin{aligned}
            I_2 &\coloneqq (R_{1,0,1}, R_{1,0,0} + L_Q[R_{2,0,0}] + 2R_{2,0,0} Q^2)_r \\
            &= (R_{1,0,1}, R_{1,0,0})_r + (i^{-1} L_Q[iR_{1,0,1}], R_{2,0,0})_r \\
            &= (R_{1,0,1}, R_{1,0,0})_r + (R_{2,0,0}, \Lambda R_{0,0,1})_r \\
            &\quad + 2(R_{2,0,0} R_{0,0,1}R_{1,0,0}, Q)_r + 2\|R_{2,0,0}\|_{L^2}^2.
        \end{aligned}
    \end{equation}

    Finally, we define $I_3$ to be the remainder of $(R_{2,0,1}, Q)_r$ after subtracting $I_1$ and $I_2$. From \eqref{order eta} and    \eqref{order b^3}, we have
    \begin{equation}\label{I_3}
        \begin{aligned}
            I_3 &= (R_{3,0,0}, 2Q^2 R_{0,0,1})_r - 4(R_{2,0,0} R_{0,0,1} R_{1,0,0}, Q)_r + (R_{0,0,1}, R_{1,0,0}^3)_r \\
            &= -(R_{3,0,0}, R_{1,0,0})_r + (R_{3,0,0}, i^{-1} L_Q[iR_{0,0,1}])_r \\
            &\quad + (R_{1,0,0}^3, R_{0,0,1})_r - 4(R_{2,0,0} R_{0,0,1}R_{1,0,0}, Q)_r \\
            &= -(R_{3,0,0}, R_{1,0,0})_r - (R_{2,0,0}, R_{0,0,1})_r + (\Lambda R_{2,0,0}, R_{0,0,1})_r \\
            &\quad - 2(R_{2,0,0} R_{0,0,1}R_{1,0,0}, Q)_r.
        \end{aligned}
    \end{equation}
    Hence, combining \eqref{I_1}, \eqref{I_2}, and \eqref{I_3}, and using the identity
    \[
        (R_{2,0,1}, Q)_r = I_1 + I_2 + I_3,
    \]
    we conclude that
    \begin{equation*}
        (R_{2,0,1}, Q)_r = (R_{1,0,1}, R_{1,0,0})_r - (R_{2,0,0}, R_{0,0,1})_r.
    \end{equation*}

    We now turn to the right-hand side of the solvability condition \eqref{solva of b eta^2}.  
    From \eqref{order b eta}, \eqref{order eta}, and \eqref{order b}, a direct computation yields
    \begin{align*}
        \text{(RHS) of } \eqref{solva of b eta^2} 
        &= (R_{0,0,2}, \Lambda Q - 2Q^2 R_{1,0,0})_r 
            - (R_{0,0,1}^2 R_{1,0,0}, Q)_r - (R_{1,0,1}, 2R_{0,0,1} Q^2)_r \\
        &= (L_Q[R_{0,0,2}], R_{1,0,0})_r - (R_{0,0,1}^2 R_{1,0,0}, Q)_r - (R_{1,0,1}, 2R_{0,0,1} Q^2)_r \\
        &= 2(R_{1,0,1}, R_{1,0,0})_r - (i^{-1} L_Q[iR_{1,0,1}], R_{0,0,1})_r + 2(Q R_{0,0,1} R_{1,0,0}, R_{0,0,1})_r \\
        &= 2(R_{1,0,1}, R_{1,0,0})_r - 2(R_{2,0,0}, R_{0,0,1})_r \\
        &= 2(R_{2,0,1}, Q)_r.
    \end{align*}
    Therefore, the left-hand side and right-hand side of \eqref{solva of b eta^2} coincide, and we conclude that there is a unique $R_{1,0,2}$ to \eqref{order b eta^2}, orthogonal to $Q$.
    
    \vspace{5bp}
    \noindent \textbullet\ Order $\mathcal{O}(b\nu\eta)$ : Collecting $\mathcal{O}(b\nu\eta)$-term in \eqref{modified profile eror} we solve
    \begin{equation}\label{order b eta v}
        \begin{aligned}
            L_Q[R_{1,1,1}] &= 2R_{2,1,0} - R_{0,1,1} + \Lambda R_{0,1,1} - \nabla R_{1,0,1} + \text{NL}_{b\nu\eta}\\
            &\eqqcolon (\text{RHS})_{b\nu\eta},\\
            \text{NL}_{b\nu\eta} & = 6R_{1,1,0}R_{0,0,1}Q -2R_{1,0,0}R_{0,1,1}Q
            -2R_{1,0,1}R_{0,1,0}Q \\
            &\quad - 2R_{1,0,0}R_{0,1,0}R_{0,0,1}
        \end{aligned}
    \end{equation}
    Note that $2R_{2,1,0}$ arises from multiplication of $-\eta$-term in the coefficient of $\partial_b Q_{\mathcal{P}}$ and $2i^3b\nu R_{2,1,0}$ in the expansion of $\partial_b Q_{\mathcal{P}}$, and $-R_{0,1,1}$ arises from multiplication of $-b\nu$-term in the coefficient of $\partial_\nu Q_{\mathcal{P}}$ and $i\eta R_{0,1,1}$ in the expansion of $\partial_\nu Q_{\mathcal{P}}$. In addition $\Lambda R_{0,1,1}$, $-\nabla R_{1,0,1}$, and $\text{NL}_{b\nu\eta}$ are from the expansion of $b\Lambda Q_{\mathcal{P}}$, $-\nu \nabla Q_{\mathcal{P}}$, and $-i\mathcal{N}_{\mathcal{P}}$, respectively.
    From real-imaginary product, we have $(R_{1,1,1},iQ)_r = 0$. 
    Hence, it follows that the existence and uniqueness of a solution $R_{1,1,1}$ is equivalent to the condition
    \begin{equation}\label{solvability of bv eta}
        ((\text{RHS})_{b\nu\eta}, \nabla Q)_r = 0,
    \end{equation}
    which ensures the solvability of \eqref{order b eta v}. To compute the left-hand side of \eqref{solvability of bv eta}, we group the relevant terms into three parts, according to their dependence on $R_{0,1,1}$, $R_{1,0,1}$, and $R_{1,1,0}$, respectively.

    We first define $J_1$ to be the collection of terms involving $R_{0,1,1}$:
    \begin{equation}\label{J_1}
        \begin{aligned}
            J_1 &\coloneqq (-R_{0,1,1} + \Lambda R_{0,1,1} - 2R_{0,1,1} R_{1,0,0} Q, \nabla Q)_r \\
            &= -(R_{0,1,1}, \nabla Q)_r + ([\Lambda, \nabla] R_{0,1,1}, Q)_r + (\nabla R_{0,1,1}, \Lambda Q)_r \\
            &\quad - (R_{0,1,1}, R_{1,0,0} \nabla(Q^2))_r \\
            &= (\nabla R_{0,1,1}, i^{-1} L_Q[i R_{1,0,0}])_r - (R_{0,1,1}, [i^{-1} L_Q, \nabla](i R_{1,0,0}))_r \\
            &= -(i^{-1} L_Q[i R_{0,1,1}], \nabla R_{1,0,0})_r.
        \end{aligned}
    \end{equation}

    Next, we define $J_2$ to be the contribution from the terms involving $R_{1,0,1}$:
    \begin{equation}\label{J_2}
        \begin{aligned}
            J_2 &\coloneqq (-\nabla R_{1,0,1} - 2R_{1,0,1} R_{0,1,0} Q, \nabla Q)_r \\
            &= (\nabla R_{1,0,1}, i^{-1} L_Q[i R_{0,1,0}])_r - (R_{1,0,1}, [i^{-1} L_Q, \nabla](i R_{0,1,0}))_r \\
            &= -(i^{-1} L_Q[i R_{1,0,1}], \nabla R_{0,1,0})_r.
        \end{aligned}
    \end{equation}

    Finally, we define $J_3$ as the remaining contribution involving $R_{1,1,0}$:
    \begin{equation}\label{J_3}
        \begin{aligned}
            J_3 &\coloneqq (6R_{1,1,0} R_{0,0,1} Q, \nabla Q)_r = (R_{1,1,0}, i[i^{-1} L_Q, \nabla] R_{0,0,1})_r \\
            &= (L_Q[R_{1,1,0}], \nabla R_{0,0,1})_r + (\nabla R_{1,1,0}, L_Q[R_{0,0,1}])_r \\
            &= (L_Q[R_{1,1,0}], \nabla R_{0,0,1})_r + (\nabla R_{1,1,0}, R_{1,0,0})_r.
        \end{aligned}
    \end{equation}
    In \eqref{J_1}, \eqref{J_2}, and \eqref{J_3}, we have used the commutator identities \eqref{comm 1} and \eqref{comm 2}. By combining these with the structural identities \eqref{order v eta}, \eqref{order b eta}, and \eqref{order bv}, and the commutator formula \eqref{comm 1}, we arrive at
    \begin{align*}
        ((\text{RHS})_{b\nu\eta}, \nabla Q)_r 
        &= J_1 + J_2 + J_3 + 2(R_{2,1,0}, \nabla Q)_r - 2(R_{1,0,0} R_{0,1,0} R_{0,0,1}, \nabla Q)_r \\
        &= -(R_{1,1,0} - \nabla R_{0,0,1} + 2R_{0,1,0}R_{0,0,1}Q, \nabla R_{1,0,0})_r \\
        &\quad - (2R_{2,0,0}+\Lambda R_{0,0,1} + 2R_{0,0,1}R_{1,0,0}Q,\nabla R_{0,1,0})_r \\
        &\quad + (-R_{0,1,0}+\Lambda R_{0,1,0}-\nabla R_{1,0,0}-2R_{1,0,0}R_{0,1,0}Q,\nabla R_{0,0,1})_r\\
        &\quad + (\nabla R_{1,1,0},R_{1,0,0})_r - 2(R_{1,0,0} R_{0,1,0} R_{0,0,1}, \nabla Q)_r\\
        &= 2(\nabla R_{1,1,0}, R_{1,0,0})_r + 2(\nabla R_{2,0,0}, R_{0,1,0})_r + 2(R_{2,1,0}, \nabla Q)_r.
    \end{align*}

    Therefore, we now compute the inner product $(R_{2,1,0}, \nabla Q)_r$. Since the constant $c_2$ is chosen so that the solvability condition \eqref{equation for c_2} is satisfied, we obtain
    \begin{equation}\label{cal inner product of R(210) with Q'}
        (R_{2,1,0}, \nabla Q)_r 
        = -(A_{b^2\nu}, \nabla R_{1,0,0})_r + (-\nabla R_{3,0,0} + \tilde{\text{NL}}_{b^3\nu}, \nabla Q)_r.
    \end{equation}

    We now estimate the right-hand side of \eqref{cal inner product of R(210) with Q'} term by term. We begin with the first term, $-(A_{b^2\nu}, \nabla R_{1,0,0})_r$, by recalling that 
    \begin{align*}
        A_{b^2\nu} &= (\text{RHS})_{b^2\nu} - c_2 \nabla Q\\
            &=  -\frac{3}{2}R_{1,1,0}+\Lambda R_{1,1,0} -\nabla R_{2,0,0} + \text{NL}_{b^2\nu}.
    \end{align*}
    and decompose it into two parts according to the presence of $R_{1,1,0}$.

    We denote by $K_1$ the contribution involving $R_{1,1,0}$. Using \eqref{order b^2} and integration by parts, we obtain
    \begin{equation}\label{K_1}
        \begin{aligned}
            K_1 &\coloneqq \left(\tfrac{3}{2} R_{1,1,0} - \Lambda R_{1,1,0} - 2R_{1,1,0} R_{1,0,0} Q, \nabla R_{1,0,0}\right)_r \\
            &= - (\nabla R_{1,1,0}, R_{1,0,0})_r - (\nabla R_{1,1,0}, L_Q[R_{2,0,0}])_r + (R_{1,1,0} R_{1,0,0}^2, \nabla Q)_r.
        \end{aligned}
    \end{equation}
    The last equality follows from \eqref{order b^2}.
    
    Next, we define $K_2$ to be the remaining contribution from $-(A_{b^2\nu}, \nabla R_{1,0,0})_r$ after subtracting $K_1$. Integration by parts yields
    \begin{equation}\label{K_2}
        \begin{aligned}
            K_2 &\coloneqq \left(\nabla R_{2,0,0} - 2R_{2,0,0} R_{0,1,0} Q + 3 R_{0,1,0} R_{1,0,0}^2, \nabla R_{1,0,0}\right)_r \\
            &= (\nabla R_{2,0,0}, \nabla R_{1,0,0})_r - 2(R_{2,0,0} R_{0,1,0} \nabla R_{1,0,0}, Q)_r - (\nabla R_{0,1,0}, R_{1,0,0}^3)_r.
        \end{aligned}
    \end{equation}

    We now turn to the second term on the right-hand side of \eqref{cal inner product of R(210) with Q'}, namely
    \[
        (-\nabla R_{3,0,0} + \tilde{\text{NL}}_{b^3\nu}, \nabla Q)_r.
    \]
    We divide this into three groups of inner products. We begin with the contribution involving $R_{3,0,0}$, which we denote by $K_3$. Using \eqref{order v}, \eqref{order b^3}, and the commutator identities \eqref{comm 1} and \eqref{comm 2}, we compute
    \begin{equation}\label{K_3}
        \begin{aligned}
            K_3 &\coloneqq (\nabla R_{3,0,0} + 2R_{3,0,0} R_{0,1,0} Q, -\nabla Q)_r \\
            &= (\nabla R_{3,0,0}, i^{-1} L_Q[i R_{0,1,0}])_r - (R_{3,0,0}, [i^{-1} L_Q, \nabla](i R_{0,1,0}))_r \\
            &= -(i^{-1} L_Q[i R_{3,0,0}], \nabla R_{0,1,0})_r \\
            &= (R_{2,0,0} - \Lambda R_{2,0,0} - 2R_{2,0,0} R_{1,0,0} Q + R_{1,0,0}^3, \nabla R_{0,1,0})_r \\
            &= -(\nabla R_{2,0,0}, \Lambda R_{0,1,0})_r - 2(R_{2,0,0} \nabla R_{0,1,0} R_{1,0,0}, Q)_r 
            + (\nabla R_{0,1,0}, R_{1,0,0}^3)_r.
        \end{aligned}
    \end{equation}

    Next, we define $K_4$ to be the group of terms involving $R_{1,1,0}$. Using the commutator formula \eqref{comm 2}
    together with \eqref{order b^2} and \eqref{order bv}, we obtain
    \begin{equation}\label{K_4}
        \begin{aligned}
            K_4 &\coloneqq (6R_{1,1,0} R_{2,0,0} Q - R_{1,1,0} R_{1,0,0}^2, \nabla Q)_r \\
            &= (R_{1,1,0}, i[i^{-1} L_Q, \nabla] R_{2,0,0})_r - (R_{1,1,0} R_{1,0,0}^2, \nabla Q)_r \\
            &= (L_Q[R_{1,1,0}], \nabla R_{2,0,0})_r + (\nabla R_{1,1,0}, L_Q[R_{2,0,0}])_r - (R_{1,1,0} R_{1,0,0}^2, \nabla Q)_r \\
            &= -(R_{0,1,0}, \nabla R_{2,0,0})_r + (\nabla R_{2,0,0}, \Lambda R_{0,1,0})_r 
            - 2(\nabla R_{2,0,0} R_{0,1,0} R_{1,0,0}, Q)_r \\
            &\quad + (\nabla R_{1,1,0}, L_Q[R_{2,0,0}])_r - (R_{1,1,0} R_{1,0,0}^2, \nabla Q)_r 
            - (\nabla R_{2,0,0}, \nabla R_{1,0,0})_r.
        \end{aligned}
    \end{equation}

    Let $K_5$ be the rest of $(-\nabla R_{3,0,0} + \tilde{\text{NL}}_{b^3\nu},\nabla Q)_r$ after subtracting $K_3+K_4$:
    \begin{equation}\label{K_5}
        K_5 \coloneqq -2(R_{2,0,0}R_{0,1,0}R_{1,0,0},\nabla Q)_r.
    \end{equation}
    From \eqref{K_1}, \eqref{K_2}, \eqref{K_3}, \eqref{K_4}, and \eqref{K_5},
    we obtain
    \begin{align*}
        (R_{2,1,0},\nabla Q)_r &= K_1+K_2+K_3+K_4+K_5 \\
        &=-(\nabla R_{2,0,0},R_{0,1,0})_r - (\nabla R_{1,1,0},R_{1,0,0})_r.
    \end{align*}
    Thus, we conclude that there exists a unique solution $R_{1,1,1}$ to
    \eqref{order b eta v}, chosen so that $(R_{1,1,1},\nabla Q)_r=0$.

    \textbf{Step 2}
     We now claim \eqref{decay of R(pqr)}. We recall that the decay property of $Q$: For $k=0,1$, since $L_Q[\Lambda Q] = i^2Q$ and $(Q,\Lambda Q)_r=0$, from \cite[Lemma A.1]{KLR2013ARMAhalfwave} we have
    \begin{equation*}
        \norm{\Lambda^k Q}_{H^m} \lesssim_m 1,\quad |\Lambda^k Q(y)| \lesssim \langle y \rangle^{-2}.
    \end{equation*}
    We remark a commutator formula 
    \begin{align}\label{eq:LQ Lambda commute}
        [i^{-1}L_Q,\Lambda]f = -i(Df + y\nabla(Q^2)f + 2\Re(y\nabla(Q^2)f))
    \end{align}
    From this, we obtain the following identity:
    \begin{equation*}
        i^{-1}L_Q[i\left\{\Lambda^2 Q + \Lambda Q + \frac{\norm{\Lambda Q}_{L^2}^2}{\norm{Q}_{L^2}^2}Q\right\}] = - (y\nabla (Q^2) + Q^2)\Lambda Q. 
    \end{equation*}
    Thus, from \cite[Lemma A.1]{KLR2013ARMAhalfwave} again, we have
    \begin{equation*}
        \norm{\Lambda^2 Q}_{H^m} \lesssim 1,\quad |\Lambda^2 Q(y)| \lesssim \langle y \rangle^{-2}.
    \end{equation*}
    Since $i^{-1}L_Q[iR_{1,0,0}] = \Lambda Q$, we obtain $\Lambda i^{-1}L_Q[iR_{1,0,0}] = \Lambda^2 Q$. Thus, by \eqref{eq:LQ Lambda commute}, we conclude
    \begin{equation*}
        \norm{\Lambda^k R_{1,0,0}}_{H^m} \lesssim_m 1,\quad |\Lambda^k R_{1,0,0}(y)| \lesssim \langle y \rangle^{-2},
    \end{equation*}
    for $k=0,1,2$. For each profile $R_{p,q,r}$, commuting its defining equation
    with $\Lambda$ once and twice expresses the equations for $\Lambda R_{p,q,r}$ and $\Lambda^2R_{p,q,r}$ in terms of previously constructed profiles and the commutators \eqref{comm 1}--\eqref{comm 3}. By induction in the total
    order $p+2q+2r$, all resulting right-hand sides belong to $H^m$ and decay like $\langle y\rangle^{-2}$. Applying Lemma~\ref{lem:L inverse bdd} inductively therefore gives \eqref{decay of R(pqr)}.
\end{proof}
\begin{remark}\label{where and why we choose c_j}
    Here we note where the constants $c_1,c_2,c_3,c_4 \in \mathbb{R}$ comes from.
    We find $c_j$ to satisfy the solvability conditions for each order $k=4, 5$. This is equivalent to solving the systems of nonlocal ordinary differential equations of $\{R_{p,q,r}\}_{k\in\{4,5\}}$ and $c_j$, \eqref{order b^4}-\eqref{order eta^2}, \eqref{order b^5}, \eqref{order b^3v}, \eqref{order b^3 eta}, \eqref{order bv^2}, and \eqref{order b eta^2}. The constant $c_1$ first appears in the equation of $R_{4,0,0}$ in \eqref{order b^4}. Since, for each $c_1$, the solvability condition of equation \eqref{order b^4} is automatically satisfied, we choose $R_{4,0,0}$ depending on $c_1$. Then, we choose $c_1$ as in \eqref{value of c_1} to satisfy the solvability condition for the equation for $R_{5,0,0}$ \eqref{order b^5}. Similarly, we choose $c_2, c_3$ and $c_4$ to satisfy the solvability conditions for the equations at $(p,q,r)$ and $(p+1,q,r)$ simultaneously, at $(p,q,r) = (2,1,0), (2,0,1)$ and $(0,2,0)$, respectively. In other words, we fix $c_2$ in \eqref{value of c_2}, $c_3$ in \eqref{value of c_3}, and $c_4$ in \eqref{value of c_4}. In fact, this choice of $c_j$ with \eqref{order v^2} and \eqref{order v eta} gives the solvability conditions for equations \eqref{order b eta^2} and \eqref{order b eta v}. So, we can find $R_{1,0,2}$ and $R_{1,1,1}$ so that the profile error for $Q_{\mathcal{P}}$, $\Psi_{\mathcal{P}}$, is of the order of more than or equal to $k=6$.
\end{remark}

\subsection{Formal modulation law}
In this section, we provide an asymptotic solution to the formal modulation law \eqref{eq:formal modul}. Specifically, we verify that the leading order behavior of the solution presented in Remark~\ref{rem:formal law} asymptotically satisfies \eqref{eq:formal modul}. Since $\eta_s = 0$ is one of the equations in \eqref{eq:formal modul}, it follows that $\eta$ is constant. We consider an arbitrarily small constant $\eta > 0$, and under this assumption, we rewrite \eqref{eq:formal modul} as the following system of ordinary differential equations for $(\lambda_\eta, \nut{x}_\eta, \gamma_\eta, \nu_\eta, b_\eta)(t)$:
\begin{equation}\label{formal modulation law}
    \begin{cases}
         (b_{\eta})_s + \left(\frac{1}{2}+c_3\eta\right)b_{\eta}^2 + c_1b_{\eta}^4 + c_4\nu_{\eta}^2 = -\eta, &\\
         (\nu_{\eta})_s + b_{\eta}\nu_{\eta}=0, &\\
         (\gamma_{\eta})_s =  1, &\\
         -\dfrac{(\lambda_{\eta})_s}{\lambda_{\eta}} = b_{\eta}, &\\
         \dfrac{(\nut{x}_{\eta})_s}{\lambda_{\eta}} - \nu_{\eta} - c_2b_{\eta}^2\nu_{\eta} = 0, &\\
         \dfrac{ds}{dt} = \dfrac{1}{\lambda_{\eta}},
    \end{cases}
\end{equation}
where the constants $c_1, c_2, c_3 \in \mathbb{R}$ and $c_4 < 0$ are defined in Proposition~\ref{Singular profile}. The initial data $(\lambda_\eta,b_\eta,\nu_\eta,\nut{x}_\eta, \gamma_\eta)(0)$ will be chosen later in Lemma~\ref{approximated dynamics}.

We now define the universal constants $C_0$ and $D_0$ by
\begin{equation}\label{definition of C_0 and D_0}
    C_0 \coloneqq \sqrt{\frac{\frac{1}{2}(i^{-1}L_Q[iR_{1,0,0}], R_{1,0,0})_r}{E_0 - E(z_f^*)}},\quad 
    D_0 \coloneqq \frac{P_0 - P(z_f^*)}{2(i^{-1}L_Q[iR_{0,1,0}], R_{0,1,0})_r}.
\end{equation}
Note that, since we assume $\alpha^* \ll 1$ so that $E(z_f^*)<E_0$, and by the positivity property \eqref{positivity of L_Q}, the constant $C_0$ is well-defined and strictly positive. The denominator in the definition of $D_0$ is positive, so $D_0\in\mathbb{R}$ is well-defined, but $D_0$ need not be positive.

We fix a small constant $0 < \eta^* \ll 1$, depending only on $\{c_j\}_{1 \leq j \leq 4}$ and the constants $C_0$ and $D_0$, and choose $t_0 < 0$, independent of $\eta \in (0, \eta^*)$, so that the solution $(\lambda_\eta, b_\eta, \nu_\eta, \nut{x}_\eta, \gamma_\eta)(t)$ exists and satisfies the asymptotics \eqref{asymtotic for parameters for t} on the interval $[t_0, 0]$ for each $\eta \in (0, \eta^*)$.

In order to solve the full system \eqref{formal modulation law} on $[t_0, 0]$, we first solve the reduced system consisting only of the parameters $(\lambda_\eta, b_\eta, \nu_\eta)(t)$, which is
\begin{equation}\label{model for lambda b v}
    \begin{aligned}
        &({b}_\eta)_s + \left(\frac{1}{2} + c_3\eta\right){b}_{\eta}^2 + \eta + c_4\nu_{\eta}^2 + c_1b_\eta^4 = 0, \\
        &(\nu_{\eta})_s + b_\eta\nu_{\eta} = 0, \quad \frac{(\lambda_{\eta})_s}{\lambda_{\eta}} = -b_{\eta}, \quad \frac{ds}{dt} = \frac{1}{\lambda_{\eta}}.
    \end{aligned}
\end{equation}

We then derive the asymptotics of $(\lambda_\eta, b_\eta, \nu_\eta)(t)$ on the interval $[t_0, 0]$ for each small $\eta > 0$. Once the reduced dynamics have been established, we choose suitable initial data for $\gamma_{\eta}(0)$ and $\nut{x}_{\eta}(0)$, and subsequently derive the asymptotics of $\gamma_{\eta}(t)$ and $\nut{x}_{\eta}(t)$ on the interval $[t_0, 0]$.

\begin{lemma}\label{approximated dynamics}
There exists a small constant $\eta^*>0$, depending only on $C_0$, $D_0$, and $\{c_j\}_{1\leq j\leq4}$, such that the following holds. For each $\eta\in(0,\eta^*)$, consider the initial data
\begin{equation}\label{asymtotic for initial data of parameters}
    \lambda_\eta(0)=2C_0^2\eta,\qquad b_\eta(0)=0,\qquad \nu_\eta(0)=D_0\lambda_\eta(0).
\end{equation}
Then there exists a constant $t_0<0$, independent of $\eta\in(0,\eta^*)$, and a solution $(\lambda_\eta,b_\eta,\nu_\eta)(t)$ to the reduced system \eqref{model for lambda b v} on $[t_0,0]$ such that
\begin{equation}\label{asymtotic for parameters for t}
    \begin{aligned}
        \lambda_{\eta}(t) &= \left(\frac{1}{4C_0^2}+o_{\eta\to0+}(1)\right)t^2+2C_0^2\eta+\mathcal{O}(t^4),\\
        b_{\eta}(t) &= -\left(\frac{1}{2C_0^2}+o_{\eta\to0+}(1)\right)t+\mathcal{O}(t^3),\\
        \nu_{\eta}(t) &= \left(\frac{D_0}{4C_0^2}+o_{\eta\to0+}(1)\right)t^2+2C_0^2D_0\eta+\mathcal{O}(t^4).
    \end{aligned}
\end{equation}
In particular,
\begin{equation}\label{order of b,eta,lambda}
    b_\eta^2(t)+\eta\sim\lambda_\eta(t).
\end{equation}
Moreover, there exist initial data
\begin{equation}\label{initial data of gmm_eta}
    \nut{x}_\eta(0)=0,\qquad \gamma_\eta(0)=\left(\frac{\pi}{\sqrt{2}}+o_{\eta\to0+}(1)\right)\eta^{-1/2},
\end{equation}
such that $(\lambda_\eta,b_\eta,\nu_\eta,\nut{x}_\eta,\gamma_\eta)(t)$ solves the full system \eqref{formal modulation law} on $[t_0,0]$. For this solution, we have
\begin{equation}
    \nut{x}_{\eta}(t)=\left(\frac{D_0}{12C_0^2}+o_{\eta\to0+}(1)\right)t^3+2C_0^2D_0\eta t+\mathcal{O}(t^5),
\end{equation}
and
\begin{equation}\label{asymptotic for gamma eta}
    \gamma_\eta(t)=\frac{\sqrt{2}+o_{\eta\to0+}(1)}{\sqrt{\eta}}\left[\arctan\left(\left(\frac{1}{2\sqrt{2}C_0^2}+o_{\eta\to0+}(1)\right)\frac{t}{\sqrt{\eta}}\right)+\frac{\pi}{2}\right]+\mathcal{O}(|t|+\eta).
\end{equation}
\end{lemma}

\begin{proof}
    We first consider the reduced subsystem of \eqref{formal modulation law} that governs the dynamics of $(\lambda_\eta, b_\eta, \nu_\eta)$, together with the scaling relation $ds/dt = 1/\lambda_\eta$, as given in \eqref{model for lambda b v}. We choose the initial data 
    \[
        (\lambda_{\eta}(0), b_{\eta}(0), \nu_{\eta}(0)) \coloneqq (2C_0^2\eta,\, 0,\, 2C_0^2D_0\eta).
    \]
    For each $\eta > 0$, we define $t_\eta^{\text{exist}} < 0$ to be the maximal backward lifespan of the solution $(\lambda_\eta(t), b_\eta(t), \nu_\eta(t))$ to \eqref{model for lambda b v}, with the above initial data at $t = 0$.

    Let $K > 0$ be a large constant to be chosen later. From now on, we restrict our attention to sufficiently small $\eta > 0$: we define 
    \[
        0 < \eta^* = \eta^*(c_1, c_3, c_4, C_0, D_0, K) \ll 1
    \]
    so that
    \[
        |2c_3\eta^*| \ll 1,\quad |c_1\eta^*| \ll 1,\quad \eta^*K \ll 1,\quad \text{and}\quad \eta^* \ll \frac{1}{|c_4|(D_0^2+1)}.
    \]
    Then, for each $\eta \in (0, \eta^*)$, we can choose a small $t_\eta = t_\eta(c_1, c_3, c_4, C_0, D_0, K) < 0$ such that 
    \[
        t_\eta \in (\max\{t_\eta^{\text{exist}}, -1\}, 0),
    \]
    and the following estimates hold for all $t \in [t_\eta, 0]$:
    \begin{equation}\label{eq:almost conserve aprior}
        \begin{aligned}
            |b_\eta(t)| &\leq 4\sqrt{\eta^*}, \\
            |\lambda_\eta(t)| &\leq \min\left\{ \frac{1}{32|c_4|C_0^2(D_0^2+1)},\, 1 \right\}, \\
            \left|\frac{b_\eta^4(t)}{\lambda_\eta^{2 + 2c_3\eta}(t)}\right| &\leq 2K.
        \end{aligned}
    \end{equation}
    These estimates follow from the smallness of the initial data, namely $\lambda_\eta(0) \sim \eta > 0$ and $b_\eta(0) = 0$.
    
    We claim that the tuple $(\lambda_\eta, b_\eta, \nu_\eta)$ admits the following almost conserved quantity \( I_\eta : (t_\eta^{\mathrm{exist}}, 0] \to \mathbb{R} \), defined by
    \begin{equation}\label{almost conservative quantity}
        I_\eta(t) \coloneqq \frac{b_\eta^2(t) + \dfrac{2}{1 + 2c_3\eta} \eta + \dfrac{2c_4}{-1 + 2c_3\eta} \nu_\eta^2(t)}{\lambda_\eta^{1 + 2c_3\eta}(t)}.
    \end{equation}
    Differentiating \eqref{almost conservative quantity} with respect to the rescaled time variable \( s \), and using \eqref{model for lambda b v}, we compute
    \begin{align*}
        \frac{dI_\eta}{ds} &= \frac{2b_\eta}{\lambda_\eta^{1 + 2c_3\eta}} \left\{ (b_\eta)_s + \left( \frac{1}{2} + c_3\eta \right) b_\eta^2 + \eta + c_4 \nu_\eta^2 \right\} \\
        &\quad + \frac{4c_4 \nu_\eta}{(-1 + 2c_3\eta) \lambda_\eta^{1 + 2c_3\eta}} \left\{ (\nu_\eta)_s + b_\eta \nu_\eta \right\} \\
        &= -\frac{2c_1 b_\eta^5}{\lambda_\eta^{1 + 2c_3\eta}}.
    \end{align*}
    Using the a priori bounds \eqref{eq:almost conserve aprior}, and integrating in \( s \), we translate the result to the original time variable \( t \) via the identity \( \frac{ds}{dt} = \frac{1}{\lambda_\eta} \). This yields the estimate
    \begin{equation}\label{estimate for I_eta and I_eta'}
        I_\eta(t) = I_\eta(0) + f_\eta(t), \quad \text{where} \quad |f_\eta(t)| \leq 16 |c_1 t| K \sqrt{\eta^*}, \quad |f_\eta'(t)| \leq 4K |c_1 b_\eta(t)|.
    \end{equation}
    Finally, from the choice of initial data \( (\lambda_\eta(0), b_\eta(0), \nu_\eta(0)) \), we compute
    \begin{equation}\label{I_eta(0)}
        \begin{aligned}
            I_\eta(0) &= \frac{1}{\lambda_\eta^{1 + 2c_3\eta}(0)} \left( \frac{2\eta}{1 + 2c_3\eta} + \frac{2c_4}{-1 + 2c_3\eta} \nu_\eta^2(0) \right) \\
            &= \frac{1}{2C_0^2 \eta} \cdot \frac{(2C_0^2\eta)^{-2c_3\eta}}{1 + 2c_3\eta} \cdot 2\eta + \frac{1}{(2C_0^2\eta)^{1 + 2c_3\eta}} \cdot \frac{2c_4}{-1 + 2c_3\eta} \cdot D_0^2 (2C_0^2\eta)^2 \\
            &= \frac{1}{C_0^2} + o_{\eta \to 0+}(1),
        \end{aligned}
    \end{equation}
    since $\eta^\eta\rightarrow 1$ as $\eta \rightarrow 0+$. Hence, for all $t \in [t_\eta, 0]$, we have the identity
    \begin{equation}\label{formal modulation law : eq1}
        \lambda_\eta^{1 + 2c_3\eta}(t)\left( I_\eta(0) + f_\eta(t) \right) + \frac{2c_4}{1 - 2c_3\eta} \nu_\eta^2(t) = b_\eta^2(t) + \frac{2\eta}{1 + 2c_3\eta}.
    \end{equation}
    Using \eqref{formal modulation law : eq1}, the approximation \eqref{I_eta(0)}, and the fact that $c_4 < 0$, we obtain a uniform lower bound for $\lambda_\eta(t)$ as follows:
    \begin{equation}\label{lower bound for lambda_eta}
        \lambda_\eta(t) \geq \left\{ \left( \frac{1}{C_0^2} + o_{\eta \to 0+}(1) + 16K |c_1| \sqrt{\eta^*} \right)^{-1} \cdot \frac{2\eta}{1 + 2c_3\eta} \right\}^{\frac{1}{1 + 2c_3\eta}} > 0.
    \end{equation}
    Differentiating \eqref{formal modulation law : eq1} with respect to $t$, and using the exact conservation law
    $$
    \frac{\nu_\eta(t)}{\lambda_\eta(t)} = D_0,
    $$
    we obtain the following identity for $\lambda_\eta'(t)$:
    \begin{equation}\label{lambda_eta'(t)}
        \begin{aligned}
            \lambda_\eta'(t) &= \frac{d}{dt}\left[\left(I_\eta(0)+f_{\eta}(t)+\frac{2c_4D_0^2}{1-2c_3\eta}\lambda_\eta^{1-2c_3\eta}\right)^{-\frac{1}{1+2c_3\eta}}\left(b_\eta^2(t)+\frac{2\eta}{1+2c_3\eta}\right)^{\frac{1}{1+2c_3\eta}}\right]\\
            & = \left\{\left(I_{\eta}(0)+f_{\eta}(t) + \frac{2c_4D_0^2}{1-2c_3\eta}\lambda_\eta^{1-2c_3\eta}(t) \right)^{-\frac{2+2c_3\eta}{1+2c_3\eta}} \times \left(b_\eta^2(t)+\frac{2\eta}{1+2c_3\eta}\right) \right.\\
            & \quad \times \frac{-1}{1+2c_3\eta}\left(f'_{\eta}(t) + 2c_4D_0^2\lambda_\eta^{-2c_3\eta}(t)\lambda_\eta'(t)\right) + \frac{1}{1+2c_3\eta}\left\{2b_\eta(t)b_{\eta}'(t)\right\} \\
            &\quad \left. \times
            \left(I_{\eta}(0)+f_{\eta}(t) + \frac{2c_4D_0^2}{1-2c_3\eta}\lambda_\eta^{1-2c_3\eta}(t) \right)^{-\frac{1}{1+2c_3\eta}} \right\} \times \left(b_\eta^2(t) + \frac{2\eta}{1+2c_3\eta}\right)^{-\frac{2c_3\eta}{1+2c_3\eta}}.
        \end{aligned}
    \end{equation}
    Now, from \eqref{estimate for I_eta and I_eta'}, we may reduce $t_\eta < 0$ (still denoted by the same symbol) so that 
    $$
        4K |c_1 t_\eta| < \frac{1}{8C_0^2},
    $$
    and possibly reduce $\eta^* > 0$ (still denoted by $\eta^*$), so that for every $\eta \in (0, \eta^*)$, we obtain the following lower bound:
    \begin{equation}\label{formal modulation law estimate 1}
        \left| I_\eta(0) + f_\eta(t) + \frac{2c_4 D_0^2}{1 - 2c_3 \eta} \lambda_\eta^{1 - 2c_3 \eta}(t) \right| \geq \frac{1}{8C_0^2}.
    \end{equation}
    Furthermore, from \eqref{estimate for I_eta and I_eta'} and the relation $(\lambda_\eta)_t = -b_\eta$, we have
    \begin{equation}\label{formal modulation law estimate 2}
        \left| f_\eta'(t) + 2c_4 D_0^2 \lambda_\eta^{-2c_3\eta}(t) \lambda_\eta'(t) \right| 
        \leq \left( 4|c_1|K + 4|c_4| D_0^2 \right) |b_\eta(t)|.
    \end{equation}
    Here we also use the uniform lower bound of $\lambda_\eta(t)$ from \eqref{lower bound for lambda_eta}, and the fact that $\eta^{\eta} \to 1$ as $\eta \to 0+$. In addition, from the a priori bound $|b_\eta(t)|^2 \leq 16 \eta^*$ in \eqref{eq:almost conserve aprior}, we estimate
    \begin{equation}\label{formal modulation law estimate 3}
        \left( b_\eta^2(t) + \frac{2\eta}{1 + 2c_3 \eta} \right)^{ -\frac{2c_3 \eta}{1 + 2c_3 \eta} } 
        \leq \max \left\{ (20\eta^*)^{ -\frac{2c_3 \eta}{1 + 2c_3 \eta} },\ \left(\frac{2\eta}{1+2c_3\eta}\right)^{ -\frac{2c_3 \eta}{1 + 2c_3 \eta} } \right\} \leq 4.
    \end{equation}

    Combining \eqref{formal modulation law estimate 1}, \eqref{formal modulation law estimate 2}, \eqref{formal modulation law estimate 3}, and the differential identity \eqref{lambda_eta'(t)}, and possibly further reducing $\eta^*$ (still denoted by $\eta^*$), we obtain a uniform bound that holds for all $\eta \in (0, \eta^*)$ and all $t \in [t_\eta, 0]$.
    Since we have chosen $\eta^* K \ll 1$, there exists a small universal constant $0 < c_0 \ll 1$, independent of $\eta$ and $K$, such that the following estimate holds:
    \begin{equation}\label{uniform estimate of formal modulation law}
        -b_\eta(t) = \mathfrak{c}(t) b_\eta(t) b_\eta'(t)+ \mathcal{O} \left( c_0 b_\eta(t) \right),\qquad 0<\mathfrak{c}_1<\mathfrak{c}(t)<\mathfrak{c}_2,
    \end{equation}
    where $0<\mathfrak{c}_1<\mathfrak{c}_2$ are independent of $\eta$, $K$, and $t$.
    On the other hand, from \eqref{model for lambda b v} and \eqref{lower bound for lambda_eta} we have
    \begin{align*}
        (b_\eta)_t = -\frac{1}{\lambda_\eta}\left\{\left(\frac{1}{2}+c_3\eta + c_1b_\eta^2\right)b_\eta^2 + \eta + c_4\nu_{\eta}^2\right\} < -\frac{1}{\lambda_\eta}\left(\frac{1}{4}b_\eta^2 + \frac{1}{2}\eta\right) < 0.
    \end{align*}
    Hence, $b_\eta(t)$ is positive for all $t\in [t_\eta,0)$, since $b_\eta(0)=0$. Dividing both sides of \eqref{uniform estimate of formal modulation law} by $b_\eta(t)$, we obtain
    \[
        -(b_{\eta})_t \sim 1
    \]
    uniformly in small $\eta>0$ and $K>0$. Therefore, from $(\lambda_{\eta})_t = -b_\eta$ and $\lambda_\eta(0) = 2C_0^2\eta$, we obtain the rough asymptotics
    \begin{equation}\label{rought asymptotic of formal modulation law}
        b_{\eta}(t) \sim -t,\quad \lambda_{\eta}(t) \sim t^2 + \eta \quad \text{for all } t \in [t_\eta,0],
    \end{equation}
    uniformly in $\eta \in (0,\eta^*)$ and $K>0$. As a consequence, there exists a constant $C>0$ independent of $\eta$ and $K$ such that
    \[
        \left|\frac{b_\eta^4(t)}{\lambda_\eta^{2+2c_3\eta}(t)}\right| \leq C\frac{t^4}{(t^2+\eta)^{2+2c_3\eta}} \quad \text{for } t \in [t_\eta,0].
    \]
    Now, observe that for all $(t,\eta) \in (-t^*,0) \times (0,t^*)$ with some sufficiently small $t^*>0$, we have
    \[
        \frac{|t|^{2-\eta}}{t^2+\eta} \leq 1.
    \]
    Hence, since $c_3$ is fixed, by possibly reducing $\eta^*>0$ (but keeping the same notation), there exist constants $K>0$ (sufficiently large) and $t_\eta < t^{**} < 0$, independent of $\eta$, such that
    \[
        \frac{t^4}{(t^2+\eta)^{2+2c_3\eta}} \leq \frac{K}{2C} \quad \text{for all } t \in [t^{**},0],
    \]
    for every $\eta \in (0,\eta^*)$. Furthermore, from \eqref{rought asymptotic of formal modulation law}, there exists a small number $t_0$ with $t^{**} < t_0 < 0$ such that
    \[
        |b_{\eta}(t)| \leq 2\sqrt{\eta^*}, \quad |\lambda_{\eta}(t)| \leq \frac{1}{2} \min\left\{\frac{1}{32|c_4|C_0^2(D_0^2+1)}, 1\right\},
    \]
    for all $t \in [t_0,0]$. This concludes the bootstrap argument and shows that the range of $t$ on which \eqref{eq:almost conserve aprior} holds can be chosen uniformly in $\eta$. In addition, note that from \eqref{rought asymptotic of formal modulation law}, and \eqref{estimate for I_eta and I_eta'} we conclude that 
    \[
        |f_{\eta}(t)| \lesssim t^2,\quad \text{ for } t \in [t_0,0].
    \]
    Now, using \eqref{formal modulation law estimate 1}, \eqref{formal modulation law estimate 2}, \eqref{formal modulation law estimate 3}, and \eqref{rought asymptotic of formal modulation law}, we rewrite \eqref{lambda_eta'(t)} as
    \[
        -b_{\eta}(t) = \mathcal{O}(b_\eta(t)\lambda_\eta(t)) + 2b_\eta(t)b'_\eta(t)(1+o_{\eta\rightarrow 0+}(1))\left(\frac{1}{C_0^2}+o_{\eta\rightarrow 0+}(1) + \mathcal{O}(t^2)\right)^{-1},
    \]
    for $t \in [t_0,0]$. Since $b_{\eta}(t)$ does not vanish, we divide both sides by $b_\eta(t)$ and integrate from $t$ to $0$. Using the fact that $b_\eta(0) = 0$, we obtain
    \[
        b_{\eta}(t) = -\left(\frac{1}{2C_0^2}+o_{\eta\rightarrow 0+}(1)\right)t + \mathcal{O}(t^3). 
    \]
    In addition, since $(\lambda_\eta)_t = -b_\eta$ and $\lambda_\eta(0) = 2C_0^2\eta$, we integrate to get
    \begin{equation}\label{exact law of lambda_eta}
        \lambda_\eta(t) = \left(\frac{1}{4C_0^2} + o_{\eta \rightarrow 0+}(1)\right)t^2 + 2C_0^2\eta + \mathcal{O}(t^4).
    \end{equation}
    From the exact conservation law $(\nu_\eta/\lambda_\eta)(t) = D_0$, we deduce
    \[
        \nu_\eta(t) = \left(\frac{D_0}{4C_0^2} + o_{\eta \rightarrow 0+}(1)\right)t^2 + 2C_0^2D_0\eta + \mathcal{O}(t^4).
    \]
    
    Finally, we obtain $\nut{x}_\eta(t)$ from the relation $(\nut{x}_\eta)_t = \nu_\eta + c_2b_\eta^2 \nu_\eta$, along with the initial condition $\nut{x}_\eta(0) = 0$. We now choose the initial data $\gamma_\eta(0)$ as in \eqref{initial data of gmm_eta}.
    Since $(\gamma_\eta)_t = \frac{1}{\lambda_\eta}$, we obtain $\gmm_\eta(t)$:
    \begin{align*}
        \gmm_\eta(t) &= \int_{-\eta}^t \frac{1}{\lambda_\eta(\tau)}\,d\tau - \frac{4C_0^2}{t_0} - \int_{-\eta}^{t_0} \frac{1}{\left(\frac{1}{4C_0^2} + o_{\eta \to 0+}(1)\right)\tau^2 + 2C_0^2\eta} \,d\tau \\
        &= \int_{t_0}^t \frac{1}{\left(\frac{1}{4C_0^2} + o_{\eta \to 0+}(1)\right)\tau^2 + 2C_0^2\eta} \,d\tau - \frac{4C_0^2}{t_0} \\
        &\quad + \mathcal{O}\left(\int_{-\eta}^t \left|\frac{1}{\lambda_\eta(\tau)}-\frac{1}{\left(\frac{1}{4C_0^2} + o_{\eta \to 0+}(1)\right)\tau^2 + 2C_0^2\eta}\right|\,d\tau\right).
    \end{align*}
    From \eqref{exact law of lambda_eta}, we obtain
    \[
        \left|\frac{1}{\lambda_\eta(\tau)}-\frac{1}{\left(\frac{1}{4C_0^2} + o_{\eta \to 0+}(1)\right)\tau^2 + 2C_0^2\eta}\right| \lesssim \left|\frac{\tau^4}{(\tau^2+\eta)^2}\right| \leq 1.
    \]
    This proves the law of $\gmm_\eta$.
\end{proof}

\begin{remark}[Choice of initial data of formal modulation parameters]
The specific choice of the initial modulation parameters 
\[
    (\lambda_{\eta}, b_{\eta}, \nu_{\eta}, \nut{x}_{\eta}, \gamma_{\eta})(0)
\]
in Lemma~\ref{approximated dynamics} is tailored to construct a Bourgain--Wang type solution $u(t)$ with prescribed conserved quantities
\[
    E(u(t)) = E_0, \qquad P(u(t)) = P_0,
\]
for any given $(E_0,P_0) \in \mathbb{R}_+ \times \mathbb{R}$.  

As explained in Proposition~\ref{energy and momentum expansion for singular profile}, the energy $E(Q_{\mathcal{P}})$ admits an expansion with respect to $(b,\eta)$, which follows from the Pohozaev identity for the ground state $Q$, and the momentum $P(Q_{\mathcal{P}})$ can be expanded with respect to $\nu$. These expansions allow us to adjust $b_\eta(0)$ and $\nu_\eta(0)$ so that the resulting solution reaches the desired conserved quantities.

The constants $C_0$ and $D_0$ are chosen in order to make the radiation part $z$ and the singular profile $Q_{\mathcal{P}}$ as decoupled as possible at the $H^{1/2}$-level. In other words, $z$ is constructed so that its interaction with $Q_{\mathcal{P}}$ is negligible in the leading order, and $C_0$, $D_0$ are fixed accordingly to reflect this choice.

For further details on these expansions, see Proposition~\ref{energy and momentum expansion for singular profile}, and for how this parameter selection is used in the final construction, see the proof of Theorem~\ref{main result 1}.
\end{remark}

\begin{remark}[On the unstable direction]
The formal modulation law \eqref{formal modulation law} suggests 
$\nu_\eta(0)\sim \eta$, in agreement with the momentum relation. 
Note that the translation mode $iR_{0,1,0}$ belongs to the purely imaginary
generalized kernel and therefore does not contribute to the $L^2$-mass 
variation in the unstable direction. 
Hence, the anticipated $\eta$-direction $R_{0,0,1}$ is indeed the appropriate
choice for the unstable mode. 
In fact,
\[
    \|Q_{\mathcal{P}(b_\eta(0),\nu_{\eta}(0),\eta)}\|_{L^2}^2 
    = \|Q\|_{L^2}^2 
      - 2\eta(i^{-1}L_Q[iR_{1,0,0}],R_{1,0,0})_r
      + \mathcal{O}(\eta^2) < \norm{Q}_{L^2}^2,
\]
for small $\eta>0$, which formally confirms that the $\eta$-direction corresponds to the instability mechanism.
\end{remark}

\section{Modulation Analysis}\label{Main thm to main bootstrap}
In the rest of the paper, we construct a one-parameter family of solutions $\{u_{\eta}\}_{\eta>0}$, from which, taking $\eta\to 0$ up to sequence, we construct a blow-up solution with the prescribed asymptotic profile $z_f^*$. The construction is based on the approximate dynamics developed in Section~\ref{sec:approx sol}, where both the flow $z(t)$ of \eqref{half-wave} with initial data $z_f^*$ (Proposition~\ref{construction of asymptotic profile}) and the modified singular profile $Q_{\mathcal{P}}$ (Proposition~\ref{Singular profile}) were introduced. We consider initial data of the form
\begin{equation*}
    u_{\eta}(0,x)
    = \frac{1}{\lambda_{\eta}^{1/2}(0)}\,
      Q_{\mathcal{P}(b_{\eta}(0),\nu_{\eta}(0),\eta)}\!\left(
          \frac{x - \nut{x}_{\eta}(0)}{\lambda_{\eta}(0)}
      \right)
      e^{i\gamma_{\eta}(0)}
      + z_f^*(x),
\end{equation*}
where the parameters $(\lambda_{\eta}(0),\nut{x}_{\eta}(0),\gamma_{\eta}(0),b_{\eta}(0),\nu_{\eta}(0))$ are chosen in Lemma~\ref{approximated dynamics}. 

Starting from these initial data, we solve the half-wave equation backward in time and, for each $\eta>0$, obtain a solution $u_\eta(t)$ that exists uniformly with respect to $\eta$. We represent $u_\eta$ in the form  
\begin{equation}\label{decomposition of u_eta}
    u_{\eta}(t)
    = \bigl[ Q_{\mathcal{P}(b(t),\nu(t),\eta)} + \epsilon \bigr]^{\sharp}
      + z(t),
\end{equation}
where the perturbation $\epsilon$ is controlled so that the modulation parameters 
\[
    (\lambda(t), \bar{x}(t), \gamma(t), b(t), \nu(t))
\]
remain close to the approximate parameters 
\[
    (\lambda_{\eta}(t), \bar{x}_{\eta}(t), \gamma_{\eta}(t), b_{\eta}(t), \nu_{\eta}(t)),
\]
in Lemma~\ref{approximated dynamics}.
Finally, by sending $\eta \to 0^+$, we construct a Bourgain--Wang type blow-up solution to \eqref{half-wave} with the prescribed asymptotic profile and establish its instability from the construction.

\subsection{Set up of the modulation analysis}\label{Set up of the modulation analysis} Based on the construction of the modified profile $Q_{\mathcal{P}(b,\nu,\eta)}$, we incorporate the modulation parameters and decompose the solutions $u_{\eta}$ to \eqref{half-wave} as 
\begin{equation}\label{u_eta form}
    u_{\eta}(t,x) = \frac{1}{\lambda^{\frac{1}{2}}(t)}\left[Q_{\mathcal{P}(b(t),\nu(t),\eta)} + \epsilon \right]\left(t,\frac{x-\nut{x}(t)}{\lambda(t)}\right)e^{i\gamma(t)} + z(t,x).
\end{equation}
Here, $z(t,x)$ is a solution to \eqref{half-wave} with a prescribed value at time $t=0$, given by $z(0,\cdot)=z_f^*$, constructed in Proposition~\ref{construction of asymptotic profile}. The modified profile for the blow-up component, $Q_{\mathcal{P}}$, is derived in Proposition~\ref{Singular profile} and involves a fixed parameter $\eta>0$ and dynamic parameters $(\lambda,\nut{x},\gamma,b,\nu)(t)$. Our aim is to monitor how $\epsilon$ evolves and demonstrate that $\epsilon \to 0$ in a suitable norm. This justifies the formal modulation law \eqref{formal modulation law}. This formal law \eqref{formal modulation law} is derived with the intention of minimizing profile error $\Psi_{\mathcal{P}}$, as discussed in Proposition~\ref{Singular profile}. Therefore, exploring the dynamics of $\epsilon$ is a key step.

Since $u_{\eta}$ is a solution of \eqref{half-wave}, we normalize it by $u^{\flat}_{\eta}$:
\begin{equation}\label{renormalised solution v}
     u^{\flat}_{\eta}(s,y) = \left.\lambda^{\frac{1}{2}}e^{-i\gamma}u_\eta(t,\lambda y+\nut{x})\right|_{t=t(s)},\quad  \frac{ds}{dt} \coloneqq \frac{1}{\lambda},\quad y\coloneqq \frac{x-\nut{x}}{\lambda},
\end{equation}
where $\lambda(t)$ is the $L^2$-critical scaling parameter, $\nut{x}(t)$ is the translation parameter, and $\gamma(t)$ is the phase rotation parameter. Then, $u^{\flat}_{\eta}$ is a solution of 
\begin{equation}\label{renormalised equation for v}
    i\partial_su^{\flat}_{\eta} - Du^{\flat}_{\eta} - u^{\flat}_{\eta} + |u^{\flat}_{\eta}|^2u^{\flat}_{\eta} = i \frac{\lambda_s}{\lambda}\Lambda u^{\flat}_{\eta} + i\frac{\nut{x}_s}{\lambda}\nabla u^{\flat}_{\eta} + \tilde{\gamma}_su^{\flat}_{\eta},
\end{equation}
where $\tilde{\gamma}_s = \gamma_s-1$.
Recalling the definition of the $\sharp$-notation
\begin{equation*}
    f^{\sharp}(t,x) \coloneqq \frac{1}{\lambda^{1/2}}f(s,\frac{x-\nut{x}}{\lambda})e^{i\gamma} \bigg|_{s=s(t)},
\end{equation*}
\eqref{u_eta form} yields
\begin{equation*}
\begin{aligned}
    u_\eta &= Q_{\mathcal{P}}^\sharp + z + \epsilon^{\sharp}, \\
    u^\flat_{\eta} &= Q_{\mathcal{P}} + z^\flat + \epsilon,
\end{aligned}
\end{equation*}
where $Q^{\sharp}_{\mathcal{P}}$ and $z^{\flat}$ are
\begin{equation*}
    \begin{aligned}
        &Q^{\sharp}_{\mathcal{P}}(t,x) \coloneqq \frac{1}{\lambda^{\frac{1}{2}}}Q_{\mathcal{P}}\left(\frac{x-\nut{x}}{\lambda}\right)e^{i\gamma},\quad 
        z^{\flat}(s,y) \coloneqq \lambda^{\frac{1}{2}}z(t,\lambda y + \nut{x})e^{-i\gamma}.
    \end{aligned}
\end{equation*}

We introduce $W$ notation for the summation of a singular profile and an asymptotic profile in the $(t,x)$-variable. Hence, the profile error $\Psi$ in \eqref{Q_P+z profile error} of $W$ contains the interaction of the singular and regular parts caused by the term $|u_{\eta}|^2u_{\eta}$.
\begin{equation}\label{def of W}
    W \coloneqq Q_{\mathcal{P}}^{\sharp} + z,\quad W^{\flat} = Q_{\mathcal{P}}+z^{\flat},
\end{equation}
and
\begin{equation*}
    u_\eta=W+\epsilon^{\sharp},\quad u^\flat_{\eta}=W^{\flat} + \epsilon.
\end{equation*}
From the construction of $Q_{\mathcal{P}}^\sharp$ and $z$, we have trivial bounds from scaling:
\begin{equation}\label{a priori bound for W}
    \norm{W}_{L^2} \lesssim 1,\quad \norm{W}_{\dot{H}^{\frac{1}{2}}}\lesssim \lambda^{-\frac{1}{2}},\quad \norm{W}_{\dot{H}^1} \lesssim \lambda^{-1},
\end{equation}
and, the profile $W$ satisfies 
\begin{equation}\label{equation for W}
    i\partial_tW - DW + |W|^2W = \frac{1}{\lambda^{\frac{3}{2}}}\Psi\left(s,\frac{x-\nut{x}}{\lambda}\right)e^{i\gamma} \eqqcolon \frac{1}{\lambda}\Psi^\sharp,
\end{equation}
where, $\frac{1}{\lambda}\Psi^{\sharp}$ is the profile error of $W$ defined by
\begin{equation}\label{Q_P+z profile error}
    \begin{aligned}
        \Psi &\coloneqq \text{Mod} + i\Psi_{\mathcal{P}} + \text{IN}.
    \end{aligned}
\end{equation}
Here, $\Psi_{\mathcal{P}}$ is the profile error of $Q_{\mathcal{P}}$ given by \eqref{modified profile eror}. Also, we denote Mod by the usual dot product of the modulation vector $\overrightarrow{\text{Mod}}(t)$ and the modulation direction $\overrightarrow{V}$:
\begin{equation}\label{definition of Mod}
    \text{Mod} \coloneqq \overrightarrow{\text{Mod}}(t) \cdot \overrightarrow{V},
\end{equation}
where 
\begin{equation}\label{definition of modulation eq and modulation vector}
    \begin{aligned}
        \overrightarrow{\text{Mod}}(t) &\coloneqq
        \begin{pmatrix}
            b_s+\left(\frac{1}{2}+c_3\eta\right)b^2 + \eta + c_1b^4 + c_4\nu^2 \\
            \nu_s + b\nu \\
            \frac{\lambda_s}{\lambda} + b \\
            \frac{\nut{x}_s}{\lambda} - \nu - c_2b^2\nu\\
            \tilde{\gamma}_s
        \end{pmatrix}, \quad
        \overrightarrow{V} \coloneqq
        \begin{pmatrix}
            i\partial_b Q_{\mathcal{P}} \\
            i\partial_\nu Q_{\mathcal{P}} \\
            - i\Lambda Q_{\mathcal{P}} \\
            - i\nabla Q_{\mathcal{P}} \\
            - Q_{\mathcal{P}}
        \end{pmatrix}.
    \end{aligned}
\end{equation}
Finally, we define the interaction of $Q_{\mathcal{P}}$ and radiation $z^{\flat}$ by IN, which is derived from nonlinearity $|u_\eta|^2u_\eta$:
\begin{equation}\label{interaction definition}                       \text{IN}\coloneqq2|Q_{\mathcal{P}}|^2z^{\flat}+2Q_{\mathcal{P}}|z^{\flat}|^2+Q_{\mathcal{P}}^2\overline{z^{\flat}}+\overline{Q_{\mathcal{P}}}(z^{\flat})^2.
\end{equation}

Now, we set up an equation for the error $\epsilon^{\sharp}$. Since $\epsilon^{\sharp}=u_{\eta}-W$, error term $\epsilon^{\sharp}$ satisfies the equation
\begin{equation}\label{equation for epsilon sharp}
    i\partial_t \epsilon^{\sharp} - D\epsilon^{\sharp} + (|u_{\eta}|^2u_{\eta}-|W|^2W) = -\frac{1}{\lambda}\Psi^{\sharp}.
\end{equation}
A direct calculation gives
\begin{equation*}
    |u_{\eta}|^2u_{\eta}-|W|^2W = \frac{1}{\lambda^{\frac{3}{2}}}\left[2\epsilon|Q_{\mathcal{P}}|^2+Q_{\mathcal{P}}^2\bar{\epsilon}+R(\epsilon)\right]\left(s,\frac{x-\nut{x}}{\lambda}\right)e^{i\gamma},
\end{equation*}
where $R(\epsilon)$ is a remaining higher order of $\epsilon$ with the interaction of radiation and $\epsilon$, that is,
\begin{equation}\label{remainder definition}
    \begin{aligned}
        R(\epsilon) &\coloneqq 2\left(|Q_{\mathcal{P}}+z^{\flat}|^2-|Q_{\mathcal{P}}|^2 \right) \epsilon + \left((Q_{\mathcal{P}}+z^{\flat})^2 - (Q_{\mathcal{P}})^2\right)\bar{\epsilon} \\
        &\quad +\left(2(Q_{\mathcal{P}}+z^{\flat})|\epsilon|^2 + (\overline{Q_{\mathcal{P}}+z^{\flat}})\epsilon^2+\epsilon|\epsilon|^2\right).
    \end{aligned}
\end{equation}
Then, renormalizing to the $(s,y)$ scale, \eqref{equation for epsilon sharp} is equivalent to
\begin{equation}\label{equation for epsilon}
    \begin{aligned}
        \partial_s\epsilon + &iL_{Q_{\mathcal{P}}}[\epsilon] -(\nu + c_2b^2\nu)\nabla \epsilon + b\Lambda \epsilon \\
        & = \overrightarrow{\text{Mod}}(t) \cdot (-\partial_b Q_{\mathcal{P}}, - \partial_\nu Q_{\mathcal{P}}, \Lambda(Q_{\mathcal{P}}+\epsilon), \nabla(Q_{\mathcal{P}}+\epsilon), -i(Q_{\mathcal{P}}+\epsilon))\\
        & \quad - \Psi_{\mathcal{P}} + i(\text{IN} + R(\epsilon)).
    \end{aligned}
\end{equation}
Here, $L_{Q_{\mathcal{P}}}$ is a linearized operator at $v=Q_{\mathcal{P}}$.
We compute the relation for $L_Q$ and $L_{Q_{\mathcal{P}}}$ as follows;
\begin{equation}\label{difference of L_(Q_p) and L_Q}
    \begin{aligned}
        L_{Q_{\mathcal{P}}}[\epsilon] &= L_Q[\epsilon] + 2(Q^2-|Q_{\mathcal{P}}|^2)\epsilon + (Q^2-Q_{\mathcal{P}}^2)\overline{\epsilon}
        \\
        &=D\epsilon+\epsilon-|Q_{\mathcal{P}}|^2\epsilon-2\Re\left\{\ol{Q_{\mathcal{P}}}\epsilon\right\}Q_{\mathcal{P}}.
    \end{aligned}
\end{equation}
Let $\mathcal{E}(s,y)$ be a function to denote the small part in the inhomogeneous part of \eqref{equation for epsilon} defined by
\begin{equation}\label{definition of mathcal{E}}
    \mathcal{E} \coloneqq \text{IN} + R(\epsilon) + i\Psi_{\mathcal{P}}.
\end{equation}
Then, the evolution of $\epsilon$ is 
\begin{equation}\label{approximated evolution of epsilon}
    \partial_s \epsilon + iL_Q[\epsilon] \approx i\overrightarrow{\text{Mod}}(t) \cdot \overrightarrow{V} + i\mathcal{E}.
\end{equation}
As we have five modulation parameters $(\lambda,\nut{x},\gamma,\nu,b)(t)$, we fix these modulation parameters and $\epsilon(s,y)$ by imposing an orthogonality on $\epsilon$,
in order to guarantee the coercivity of the linearized operator $L_Q$ and maintaining smallness of $\epsilon$ backward in time. The natural candidates of the orthogonality  directions are the generalized kernel of the linearized operator $L_Q$, which is coming from the symmetry group of the equation \eqref{half-wave}. But, as in the Lemma~\ref{Coercivity 1} and Lemma~\ref{Coercivity 2}, the existence of a negative eigenvalue and corresponding eigenfunction of $L_Q$, which is not from the symmetry of \eqref{half-wave}, leads our choice of direction as follows.
\begin{equation*}
    (\epsilon,i\Lambda Q_{\mathcal{P}})_r = (\epsilon,i\partial_b Q_{\mathcal{P}})_r=(\epsilon,i\partial_\eta Q_{\mathcal{P}})_r=(\epsilon, i\nabla Q_{\mathcal{P}})_r=(\epsilon,i\partial_\nu Q_{\mathcal{P}})_r = 0.
\end{equation*}
These five orthogonal conditions give modulation estimates for $(\lambda,\nut{x},\gamma,\nu,b)$ and yield degeneracy for $(\epsilon,Q)_r$. 
Indeed, by differentiating five orthogonality identities with respect to $s$-variable and substituting $\partial_s\epsilon$ using \eqref{approximated evolution of epsilon}, we obtain
\begin{equation}\label{formal calculation of modulation estimates}
\begin{aligned}
    0&=\partial_s(\epsilon,i\Lambda Q_{\mathcal{P}})_r \approx
    (\epsilon,Q)_r + \left|b_s + \left(\frac{1}{2}+c_3\eta\right)b^2 + \eta + c_1b^4 + c_4\nu^2 \right| + (\mathcal{E},\Lambda Q)_r,\\
    0&= \partial_s (\epsilon,i\partial_b Q_{\mathcal{P}})_r \approx
    (\epsilon,i\Lambda Q)_r + \left|\frac{\lambda_s}{\lambda}+b\right| + (i\mathcal{E},R_{1,0,0})_r,\\
    0&= \partial_s (\epsilon,i\partial_\eta Q_{\mathcal{P}})_r \approx
    (\epsilon,R_{1,0,0})_r + |\tilde{\gamma}_s| + (\mathcal{E},R_{0,0,1})_r,\\
    0&= \partial_s (\epsilon, i\nabla Q_{\mathcal{P}})_r \approx
    b(\epsilon,i\nabla Q)_r + |\nu_s + b\nu| + (\mathcal{E},\nabla Q)_r,\\
    0&= \partial_s (\epsilon,i\partial_\nu Q_{\mathcal{P}})_r \approx
    (\epsilon, i\nabla Q)_r + \left|\frac{\nut{x}_s}{\lambda} - \nu - c_2b^2\nu\right| + (i\mathcal{E}, R_{0,1,0})_r.
\end{aligned}
\end{equation}
Now, we investigate the derivative of $(\epsilon, Q)_r$ with respect to $s$.
\begin{equation}
    \partial_s (\epsilon, Q_{\mathcal{P}})_r \approx (\epsilon, L_Q[iQ])_r +|(\partial_b Q_{\mathcal{P}}, Q_{\mathcal{P}})_r||\overrightarrow{\text{Mod}}(t)| + (i\mathcal{E},Q)_r. 
\end{equation}
From \eqref{formal calculation of modulation estimates} and orthogonality conditions for $\epsilon$, we have $|\overrightarrow{\text{Mod}}(t)| \lesssim |(\epsilon,Q)_r| + |(i\mathcal{E}, \langle \cdot \rangle^{-2})_r|$, forming the decaying property of the profile $Q_{\mathcal{P}}$ in the Proposition~\ref{Singular profile}. Then, identity $\norm{R_{1,0,0}}_{L^2}^2 = 2(Q,R_{2,0,0})_r$ yields 
\begin{equation*}
    (\partial_b Q_{\mathcal{P}},Q_{\mathcal{P}})_r = \mathcal{O}(b^2 + \nu + \eta).
\end{equation*}
Since we consider the regime $b^2 + \eta \lesssim \lambda$ and $\nu \sim \lambda$ from the formal modulation law \eqref{formal modulation law} and the smallness of profile error $\Psi$ in Lemma~\ref{interaction error order}, we have additional $\lambda^{1/2-}$ gain of smallness of $(\epsilon, Q_{\mathcal{P}})_r$ than the trivial bound $\norm{\epsilon}_{L^2}$ from the Cauchy--Schwarz inequality. Simultaneously, we obtain modulation estimate
\begin{equation*}
    |\overrightarrow{\text{Mod}}(t)| \lesssim \lambda^{1/2-}\norm{\epsilon}_{L^2} .
\end{equation*}
The rigorous proof for modulation estimates and degeneracy of unstable direction is in the Lemma~\ref{M estimate} and Proposition~\ref{Modulation estimates}.

The precise setup is as follows. For each small $\eta>0$, we consider an initial data
\begin{equation*}
    u_{\eta}(0,x) \coloneqq \frac{1}{\lambda_{\eta}^{1/2}(0)}Q_{\mathcal{P}(b_\eta(0),\nu_{\eta}(0),\eta)}\left(\frac{x-\nut{x}_{\eta}(0)}{\lambda_{\eta}(0)}\right)e^{i\gamma_{\eta}(0)} + z_f^*(x),
\end{equation*}
where $(\lambda_\eta,\nut{x}_{\eta},\gamma_\eta,b_\eta,\nu_{\eta})(0)$ is chosen in the Lemma~\ref{approximated dynamics}, and $z_f^*$ is an asymptotic profile constructed in the Proposition~\ref{construction of asymptotic profile}.
We evolve $u_{\eta}$ by \eqref{half-wave} with the above initial data backward in time. As we assume the decomposition 
\begin{equation*}
    u_{\eta}(t,x) = \frac{1}{\lambda^{\frac{1}{2}}(t)}\left[Q_{\mathcal{P}(b(t),\nu(t),\eta)} + \epsilon \right]\left(t,\frac{x-\nut{x}(t)}{\lambda(t)}\right)e^{i\gamma(t)} + z(t,x),
\end{equation*}
where $z$ is a flow of \eqref{half-wave} with initial data $z_f^*$, $\epsilon$ is a function of $u_{\eta}, z$ and modulation parameters $\lambda, \nut{x}, \gamma , b$ and $\nu$.

\begin{remark}
    In view of the definition of $\mathcal{E}$ in \eqref{definition of mathcal{E}} and the evolution equation for $\epsilon^{\sharp}$ in \eqref{equation for epsilon sharp}, the profile error $\Psi_{\mathcal{P}}$ associated with the singular profile $Q_{\mathcal{P}}$ (constructed in Proposition~\ref{Singular profile}) plays a crucial role in the modulation estimates, particularly in controlling the phase parameter $\gamma$. 

    In fact, if $\|\Psi_{\mathcal{P}}\|_{L^2} = \mathcal{O}(\lambda^2)$ as in \cite{KLR2013ARMAhalfwave}, the modulation analysis yields only a relatively rough bound 
    \[
        |\tilde{\gamma}_s| \lesssim \lambda^{\frac{3}{2}+\omega}
        \quad \text{(see Proposition~\ref{Modulation estimates})},
    \]
    which in turn leads to 
    \[
        |\gamma - \gamma_\eta| \lesssim \alpha^* \lambda^\omega
    \]
    after closing the bootstrap argument. However, such a weak control on phase modulation is not sufficient to establish the key energy-level estimate \eqref{estimate for u_c for H^{1/2}} required in the proof of Theorem~\ref{main result 1} when passing to the limit $\eta \to 0$.  

    For this reason, it is essential to improve the precision of the singular profile by pushing the expansion of $\Psi_{\mathcal{P}}$ to a higher order, specifically beyond $k=5$. This higher order tail computation in Proposition~\ref{Singular profile} ensures that the phase modulation remains sharp enough to achieve the desired energy-level control for the limiting Bourgain--Wang solution.
\end{remark}

\subsection{Reduction of Main theorem to main bootstrap lemma}
In this section, we construct a family of solutions $\{u_{\eta}\}_{\eta \in (0,\eta^*]}$ to \eqref{half-wave} such that the limit function, $u_0$ of $u_\eta$ as $\eta \rightarrow 0+$ up to subsequence becomes our desired blow-up solution via a compactness argument.

\begin{proposition}[Construction of one parameter family of solutions]\label{construction of one parameter family of solutions}
    Let $m = 32$, $\delta = \frac{1}{16}$, and $\omega = \frac{1}{8}$. Fix any nonzero function $f \in \mathcal{S}(\mathbb{R})$. Let $\alpha^* \ll 1$ be a sufficiently small constant, possibly smaller than the one chosen in Proposition~\ref{construction of asymptotic profile}. Let $z_f^*$ be the function constructed in Proposition~\ref{construction of asymptotic profile}, and denote by $z(t,x)$ the solution to the half-wave equation with initial data $z(0,x) = z_f^*(x)$. Then, there exists a constant $\eta^* = \eta^*(\alpha^*) > 0$ such that the following holds:
    
    For each $\eta \in (0, \eta^*(\alpha^*))$, consider the initial data at $t = 0$ given by
    \begin{equation}\label{initial data of u_eta}
        u_\eta(0,x) = \frac{1}{\lambda^{\frac{1}{2}}_\eta(0)} Q_{\mathcal{P}(b_\eta(0), \nu_\eta(0), \eta)}\left( \frac{x - \nut{x}_\eta(0)}{\lambda_\eta(0)} \right) e^{i\gamma_\eta(0)} + z_f^*(x),
    \end{equation}
    where the parameters $(\lambda_\eta(0), b_\eta(0), \nu_\eta(0), \nut{x}_\eta(0), \gamma_\eta(0))$ are those chosen in Lemma~\ref{approximated dynamics}, and $Q_{\mathcal P(b_{\eta},\nu_{\eta},\eta)}$ is a modified profile constructed in Proposition~\ref{Singular profile}.

    Then, there exists a constant $t_1 \in [-1, 0)$, independent of $\eta$, satisfying $|t_1| < |t_0|$, where $t_0 < 0$ is the constant from Lemma~\ref{approximated dynamics}, and a solution $u_\eta$ to the half-wave equation~\eqref{half-wave} defined on $[t_1,-t_1]$ such that
    \[
        u_\eta\in\mathcal C([t_1,-t_1];H^{1/2+\delta}(\mathbb R)).
    \]

    In addition, $u_\eta$ admits a geometric decomposition: for each $\eta \in (0, \eta^*(\alpha^*))$ and for all $t \in [t_1, 0)$,
    \begin{equation}\label{decomposition of solution u_eta}
        u_{\eta}(t,x) = \frac{1}{\lambda^{1/2}(t)}\left[Q_{\mathcal{P}(b(t),\nu(t),\eta)} + \epsilon \right]\left(t,\frac{x-\nut{x}(t)}{\lambda(t)}\right)e^{i\gamma(t)} + z(t,x),
    \end{equation}
    with the following uniform control:
    \begin{equation}\label{uniform control for modulation parameters}
        \begin{aligned}
            \frac{|\lambda-\lambda_\eta|}{\lambda^{2+\omega}} + \frac{|b-b_\eta|}{\lambda^{3/2+\omega}} &+ \frac{|\nu-\nu_\eta|}{\lambda^{2+\omega}} + \frac{|\nut{x}-\nut{x}_\eta|}{\lambda^{5/2+\omega}} + \frac{|\gamma - \gamma_\eta|}{\lambda^{1/2+\omega}}\\
            &+\frac{\norm{\epsilon}_{H^{1/2}}^2}{\lambda^{3+4\omega}} + \frac{\norm{\epsilon^{\sharp}}_{H^{1/2+\delta}}^2}{\lambda^{2-4\delta}}
            \leq C\alpha^*.
        \end{aligned}
    \end{equation}
    Here, $(\lambda_\eta, b_\eta, \nu_\eta, \nut{x}_\eta, \gamma_\eta)(t)$ satisfies the formal modulation law~\eqref{formal modulation law} for $t \in [t_1, 0]$, and $C > 0$ is a universal constant.
\end{proposition}

\begin{proof}[Proof of Proposition~\ref{construction of one parameter family of solutions} assuming Lemma~\ref{Main bootstrap}]
    Let $\eta^*>0$ be a constant which is chosen in the Lemma~\ref{approximated dynamics}. For $\eta \in (0,\eta^*)$, let $u_{\eta}: (T^{(\eta)}_{\text{maximal}},0]\times \mathbb{R} \rightarrow \mathbb{C}$ be the maximal lifespan solution of the half-wave equation \eqref{half-wave} with the initial data \eqref{initial data of u_eta},
    \begin{equation*}
        u_{\eta}(0,x) \coloneqq \frac{1}{\lambda_{\eta}^{1/2}(0)}Q_{\mathcal{P}(b_\eta(0),\nu_{\eta}(0),\eta)}\left(\frac{x-\nut{x}_{\eta}(0)}{\lambda_{\eta}(0)}\right)e^{i\gamma_{\eta}(0)} + z_f^*(x),
    \end{equation*}
    where $(\lambda_\eta,\nut{x}_{\eta},\gamma_\eta,b_\eta,\nu_\eta)(0)$ is initial data chosen in the Lemma~\ref{approximated dynamics}. Indeed, since $u_\eta(0) \in H^{1/2+\delta}(\mathbb{R})$, from Theorem~\ref{Cauchy theory}, we have $u_\eta \in \mathcal{C}((T^{(\eta)}_{\text{maximal}},0];H^{1/2+\delta}(\mathbb{R}))$. 

    Possibly reducing $\eta^*>0$ below the one chosen in Lemma~\ref{approximated dynamics}, for each $\eta \in (0,\eta^*)$ we define $\epsilon(t) \in H^{1/2+\delta}$ and modulation parameters $(\lambda,b,\nu,\nut{x},\gamma)(t)$ as
    \begin{equation}\label{definition of epsilon}
        \frac{1}{\lambda^{1/2}(t)}\epsilon\left(t,\frac{x-\nut{x}(t)}{\lambda(t)}\right)e^{i\gamma(t)} \coloneqq 
        u_{\eta}(t,x) - \frac{1}{\lambda^{1/2}(t)}Q_{\mathcal{P}(b(t),\nu(t),\eta)}\left(\frac{x-\nut{x}(t)}{\lambda(t)}\right)e^{i\gamma(t)}-z(t,x),
    \end{equation}
    with initial data $(\lambda,\nut{x},\gamma,b,\nu)(0) = (\lambda_\eta,\nut{x}_{\eta},\gamma_{\eta},b_\eta,\nu_{\eta})(0)$ and orthogonal conditions on $\epsilon$
        \begin{align}
        (\epsilon, i\Lambda Q_\calP)_r &=0, \label{orthogonality for Lambda Q}
        \\
        (\epsilon, i\partial_b Q_\calP)_r&=0, \label{orthogonality for partial_b Q}
        \\
        (\epsilon, i\partial_{\eta} Q_\calP)_r &=0, \label{orthogonality for partial_eta Q}
        \\
        (\epsilon, i\nabla Q_\calP)_r &=0, \label{orthogonality for nabla Q}
        \\
        (\epsilon, i\partial_{\nu} Q_\calP)_r &=0. \label{orthogonality for partial_v Q}
    \end{align}
    For each $\eta \in (0,\eta^*)$, there is a time $T_{\text{dec}}^{(\eta)}<0$ so that the above geometric decomposition \eqref{definition of epsilon} holds and modulation parameters $(\lambda,\nut{x},\gmm,b,\nu)(t)$ are uniquely determined by the orthogonality conditions \eqref{orthogonality for Lambda Q}--\eqref{orthogonality for partial_v Q}, for each $t \in (T_{\text{dec}}^{(\eta)},0]$. Indeed, from Lemma~\ref{decomposition lemma}, we define
    \begin{equation}\label{definition of T_dec}
        T_{\text{dec}}^{(\eta)} \coloneqq \inf\{\tilde{t} <0 \:|\: \tilde{t} \in \mathcal{A}_{\text{dec}}^{(\eta)} \},
    \end{equation}
    where $\mathcal{A}_{\text{dec}}^{(\eta)}$ is defined in \eqref{definition of A_{dec}}. Equivalently, $T_{\text{dec}}^{(\eta)}$ is a maximal backward time of existence of modulation parameters $(\lambda,b,\nu,\nut{x},\gamma)(t)$ of following system:
    \begin{equation}\label{decomposition of solution u_eta-2}
        \begin{aligned}
            \begin{cases}
                u_\eta(t,x) = \frac{1}{\lambda^{1/2}(t)}Q_{\mathcal{P}(b(t),\nu(t),\eta)}\left(\frac{x-\nut{x}(t)}{\lambda(t)}\right)e^{i\gamma(t)} + z(t,x) + \epsilon^{\sharp}(t,x),&\\
                \partial_t(\epsilon,i\Lambda Q_{\mathcal{P}})_r = \partial_t(\epsilon,i\partial_b Q_{\mathcal{P}})_r=\partial_t(\epsilon,i\partial_\eta Q_{\mathcal{P}})_r
                =\partial_t(\epsilon, i\nabla Q_{\mathcal{P}})_r=\partial_t(\epsilon,i\partial_\nu Q_{\mathcal{P}})_r = 0,&\\
                (\lambda,b,\nu,\nut{x},\gmm)(0) = (\lambda_{\eta},b_{\eta},\nu_{\eta},\nut{x}_{\eta},\gmm_{\eta})(0).
            \end{cases}
        \end{aligned}
    \end{equation}
    We show that there is a constant $t_1<0$, independent of $\eta$ so that 
    \eqref{decomposition of solution u_eta} and \eqref{uniform control for modulation parameters} holds for $t \in [t_1,0]$, by introducing the main bootstrap lemma.
    \begin{lemma}\label{Main bootstrap}(Main bootstrap)
        Let $\omega = 1/8$. There exists $t_1<0\ (t_0<t_1)$ such that for all sufficiently small $\eta>0$ and $\alpha^*>0$, we have the following property. Assume that the decomposition
        \eqref{decomposition of solution u_eta-2} of $u_\eta$ satisfies the weak bootstrap assumptions
        \begin{equation}\label{bootstrap bounds}
            \begin{aligned}
                &\sup_{t\in[t_2,0]}\left\{\frac{\norm{\epsilon}^2_{H^{\frac{1}{2}}}}{\lambda^{3+4\omega}}+\frac{\norm{\epsilon^{\sharp}}_{H^{\frac{1}{2}+\delta}}^2}{\lambda^{2-4\delta}}\right\} \leq 1,\\
                &\sup_{t\in [t_2,0]}\left\{\frac{|\lambda_{\eta}-\lambda|}{\lambda^{2+\omega}} + \frac{|b_{\eta}-b|}{\lambda^{3/2+\omega}}+\frac{|\nu_{\eta}-\nu|}{\lambda^{2+\omega}} + \frac{|\nut{x}_{\eta}-\nut{x}|}{\lambda^{5/2+\omega}}\right\} \leq 1,
            \end{aligned}
        \end{equation}
        for some $t_2 \in [t_1,0] \cap (T_{\text{dec}}^{(\eta)},0]$. Then, for all $t \in [t_2,0]$ we have 
        \begin{equation}\label{uniform control of parameters and error}
            \begin{aligned}
                \frac{|\lambda-\lambda_\eta|}{\lambda^{2+\omega}} + \frac{|b-b_\eta|}{\lambda^{3/2+\omega}} &+ \frac{|\nu-\nu_\eta|}{\lambda^{2+\omega}} + \frac{|\nut{x}-\nut{x}_\eta|}{\lambda^{5/2+\omega}} + \frac{|\gamma - \gamma_\eta|}{\lambda^{1/2+\omega}}\\
                &+\frac{\norm{\epsilon}_{H^{1/2}}^2}{\lambda^{3+4\omega}} + \frac{\norm{\epsilon^{\sharp}}_{H^{1/2+\delta}}^2}{\lambda^{2-4\delta}} \leq C\alpha^* < \frac{1}{2},
            \end{aligned}
        \end{equation}
        where $C>0$ is a universal constant. 
    \end{lemma}
    Now, we possibly replace $\eta^*$ by smaller than one chosen in Lemma~\ref{approximated dynamics} so that there is a constant $t_1<0$ in Lemma~\ref{Main bootstrap} independent of $\eta \in (0,\eta^*)$. 
    Then, from the choice of initial data of modulation parameter $(\lambda,\nut{x},\gamma, b,\nu)(0) = (\lambda_{\eta},\nut{x}_{\eta},\gamma_{\eta}, b_{\eta},\nu_{\eta})(0)$ as in Lemma~\ref{approximated dynamics}, modulation parameters and error satisfy \eqref{bootstrap bounds} in a small neighborhood of $0$. From a standard continuity argument, the bootstrap conclusion \eqref{uniform control of parameters and error} is satisfied on the time interval $[t_1,0] \cap (T_{\text{dec}}^{(\eta)},0]$. If $T_{\mathrm{dec}}^{(\eta)}>t_1$, then \eqref{uniform control of parameters and error} and the implicit function theorem allow the decomposition to be extended to $(T_{\mathrm{dec}}^{(\eta)}-\delta,0]$ for some $\delta>0$. This contradicts the definition of $T_{\mathrm{dec}}^{(\eta)}$. Hence, $T_{\text{dec}}^{(\eta)}\leq t_1$. Finally, we extend $u_\eta$ uniformly to positive times. Since $u_\eta(0)$ is not necessarily real-valued, we apply the same backward construction to the conjugate initial datum $\overline{u_\eta(0)}$. Indeed, since all the profiles $R_{p,q,r}$ are real-valued,
    \[
        \overline{Q_{\mathcal P(b,\nu,\eta)}} =Q_{\mathcal P(-b,-\nu,\eta)},
    \]
    and the conjugate radiation satisfies the same regularity and degeneracy estimates. Hence, after possibly reducing $t_1<0$, there exists a solution $v_\eta$ on $[t_1,0]$, uniformly in $\eta$, such that
    \[
        v_\eta(0)=\overline{u_\eta(0)}.
    \]
    By the time-reversal symmetry, $\overline{v_\eta(-t)}$ is a solution on $[0,-t_1]$ with initial datum $u_\eta(0)$. By uniqueness of the Cauchy problem, this is the forward continuation of $u_\eta$. Therefore, $u_\eta$ exists on $[t_1,-t_1]$ and does not blow up at $t=0$.
\end{proof}

\subsection{Proof of Theorem~\ref{main result 1} and Theorem~\ref{main result 2}}
We construct a Bourgain--Wang type solution to \eqref{half-wave}. We use a compactness argument. Let $\{u_\eta\}_{0<\eta<\eta^*}$ be a family of solutions to \eqref{half-wave} which are constructed in Proposition \ref{construction of one parameter family of solutions}. First, note that $u_\eta \in \mathcal{C}([t_1,0];H^{1/2+\delta}(\mathbb{R}))$, where $t_1<0$ is a constant independent of $\eta$.

\begin{proof}[Proof of Theorem~\ref{main result 1}] 

Let $t_1<0$ and $\eta^*>0$ be the constants chosen in Lemma~\ref{Main bootstrap}.  
We claim that the family $\{u_\eta(t_1)\}_{0<\eta<\eta^*}$ is precompact in $H^{1/2}(\mathbb{R})$.  

First, the uniform boundedness of $\|\epsilon^{\sharp}\|_{H^{1/2+\delta}}$ in \eqref{uniform control of parameters and error} implies
\begin{equation}\label{eq:1}
    \|u_{\eta}(t_1)\|_{H^{1/2+\delta}}
    \lesssim \lambda(t_1)^{-\frac{1}{2}-\delta}
    + \lambda(t_1)^{1-2\delta}
    \leq M(t_1),
\end{equation}
for some $M(t_1)>0$, uniformly in $\eta$.  

Next, we show the tightness of $\{u_\eta(t_1)\}_\eta$ in $H^{1/2}$ by localizing the mass.  
Let $\chi:\mathbb{R}\to\mathbb{R}$ be a smooth cutoff such that $\chi(x)=0$ for $|x|\le1$ and $\chi(x)=1$ for $|x|\ge2$, and define $\chi_R(x):=\chi(x/R)$.  
Using the mass conservation law and Lemma~\ref{Lemma:commutator estimate for localizing mass}, we have, uniformly for $\eta\in(0,\eta^*)$,
\begin{equation}\label{eq:2 - tightness}
\begin{aligned}
    \left|\frac{d}{dt}\int_{\mathbb{R}}\chi_R|u_\eta(t)|^2dx\right|
    &\lesssim 
    \left|\int_{\mathbb{R}}u_\eta(t)\,[\chi_R,iD]\,\overline{u_\eta}(t)\,dx\right|
    \\
    &\lesssim
    \|\nabla\chi_R\|_{L^\infty}\|u_\eta(t)\|_{L^2}^2
    \lesssim \frac{1}{R}.
\end{aligned}
\end{equation}
Here, we used
\[
    \|u_\eta(t)\|_{L^2}
    \le \|Q_{\mathcal{P}}\|_{L^2} + \|z_f^*\|_{L^2}
    \le \|Q\|_{L^2} + \alpha^*
    \lesssim 1.
\]
Integrating \eqref{eq:2 - tightness} from $t_1$ to $0$, and using the
uniform tightness of the initial data at $t=0$, we obtain
\[
    \lim_{R\to\infty}\;
    \sup_{0<\eta<\eta^*}
    \int_{|x|>R}|u_\eta(t_1,x)|^2\,dx
    =0.
\]
Thus, $\{u_\eta(t_1)\}_{0<\eta<\eta^*}$ is tight in $L^2(\mathbb R)$.

Therefore, after passing to a subsequence $\eta_n\to0$, there exists
$u_c(t_1)\in H^{1/2+\delta}(\mathbb R)$ such that $u_{\eta_n}(t_1)\rightharpoonup u_c(t_1)
$ in $H^{1/2+\delta}(\mathbb R)$. By local Rellich compactness and the above $L^2$-tightness, the convergence is strong in $L^2(\mathbb R)$. Interpolating with the uniform $H^{1/2+\delta}$ bound, we obtain
\[
    u_{\eta_n}(t_1)\to u_c(t_1) \quad\text{in }H^{1/2}(\mathbb R).
\]

Now, let $u_c \in \mathcal{C}([t_1,T_c);H^{1/2+\delta}(\mathbb{R}))$ be a solution to \eqref{half-wave} with initial data $u_c(t_1)$, where $T_c$ is a maximal forward time of $u_c$. We claim that $u_c$ is a Bourgain--Wang type solution associated with energy and momentum
\begin{equation}\label{given energy and momentum}
    E(u_c(t)) = E_0,\text{ and } P(u_c(t)) = P_0,
\end{equation}
for given $(E_0,P_0) \in \bbR_+ \times \bbR$, which blows up at $T=0$. By the continuous dependence of the flow in $H^{1/2}$, for each $t\in[t_1,T_c)$,
\begin{equation}\label{strong convergence of H^{1/2}}
    u_{\eta_n}(t)\rightarrow u_c(t) \quad \text{in } H^{1/2}, \qquad n\to\infty.
\end{equation}
Let $(\lambda, b, \nu, \nut{x}, \gamma)(t)$ be a geometrical decomposition associated to $u_\eta$. From strong convergence in $H^{1/2}$ \eqref{strong convergence of H^{1/2}} and uniqueness of modulation parameters, $u_c$ admits on $[t_1,T_c)$ a geometrical decomposition of the form
\begin{equation}
    u_c(t,x) = \frac{1}{\lambda_c^{\frac{1}{2}}}[Q_{\mathcal{P}(b_c,\nu_c,0)} + \epsilon_c]\left(t,\frac{x-\nut{x}_c}{\lambda_c}\right)e^{i\gamma_c} + z(t,x),
\end{equation}
and
\begin{equation}\label{decompose of u_c}
    \lambda \rightarrow \lambda_c, \quad b \rightarrow b_c, \quad \nu\rightarrow\nu_c,\quad \nut{x} \rightarrow \nut{x}_c,\quad \gamma \rightarrow\gamma_c, \text{  and  } \epsilon \rightarrow \epsilon_c \text{ in } H^{1/2},
\end{equation}
for each $t \in [t_1,\min\{T_c,0\})$ as $\eta_n \rightarrow 0$. Therefore, from \eqref{uniform control of parameters and error} in Lemma~\ref{Main bootstrap}, we obtain
\begin{equation}\label{difference estimate for modulation parameter}
\begin{aligned}
    &|\lambda_c-\lambda_0|+ |\nu_c - \nu_0| \lesssim \alpha^*|t|^{4+2\omega},\quad |b_c - b_0| \lesssim \alpha^*|t|^{3+2\omega},\\
    &|\nut{x}_c-\nut{x}_0| \lesssim \alpha^*|t|^{5+2\omega}, \quad |\gamma_c - \gamma_0| \lesssim \alpha^*|t|^{1+2\omega},\\
    & \norm{\epsilon_c}_{H^{1/2}} \lesssim \sqrt{\alpha^*}|t|^{3+4\omega}, \quad \norm{\epsilon_c^{\sharp}}_{H^{1/2+\delta}} \lesssim \sqrt{\alpha^*}|t|^{2-4\delta},
\end{aligned}
\end{equation}
where
\[
    (\lambda_0,\nut{x}_0,\gmm_0,b_0,\nu_0)(t) \coloneqq \lim_{n \to \infty}(\lambda_{\eta_n},\nut{x}_{\eta_n},\gmm_{\eta_n},b_{\eta_n},\nu_{\eta_n})(t).
\]
Hence, letting $\eta_n \to 0^+$ as given by \eqref{asymtotic for parameters for t} in 
Lemma~\ref{approximated dynamics}, we obtain
\begin{equation}\label{asymptotics for modulation parameters}
    \begin{aligned}
        \lambda_c(t) &= \left(\frac{1}{4C_0^2}+ o_{t \to 0-}(1)\right)t^2, \quad 
        b_c(t) = -\left(\frac{1}{2C_0^2}+o_{t \to 0-}(1)\right)t, \\
        \nu_c(t) &= \left(\frac{D_0}{4C_0^2}+o_{t \to 0-}(1)\right)t^2 , \quad
        \nut{x}_c(t) = \left(\frac{D_0}{12C_0^2}+o_{t \to 0-}(1)\right)t^3,\\
        \gamma_c(t) &= -\left(4C_0^2+o_{t \to 0-}(1)\right)\frac{1}{t}.
    \end{aligned}
\end{equation}
Hence, we conclude that $T_c=0$. Moreover, \eqref{decompose of u_c}, \eqref{difference estimate for modulation parameter} and $\lambda_c(t) \sim t^2$ from \eqref{asymptotics for modulation parameters} give
\begin{equation}\label{estiamte for u_c for L^2}
    \begin{aligned}
        &\norm{u_c(t) - \frac{1}{\lambda_0^{1/2}}Q_{\mathcal{P}_0}\left(\frac{\cdot-\nut{x}_0}{\lambda_0}\right)e^{i\gamma_0} -z(t)}_{L^2} \\
        & \lesssim |b_c(t)-b_0(t)|\norm{\partial_b Q_{\mathcal{P}} |_{\mathcal{P}(b_0,\nu_0,0)}}_{L^2} + |\nu_c(t)-\nu_0(t)|\norm{\partial_\nu Q_{\mathcal{P}}|_{\mathcal{P}(b_0,\nu_0,0)}}_{L^2}\\
        & \quad + \left|1-\left(\frac{\lambda_c(t)}{\lambda_0(t)}\right)^{\frac{1}{2}}\right|\norm{\Lambda Q_{\mathcal{P}(b_0,\nu_0,0)}}_{L^2} + \left|\frac{\nut{x}_c(t)-\nut{x}_0(t)}{\lambda_0(t)}\right|\norm{\nabla Q_{\mathcal{P}(b_0,\nu_0,0)}}_{L^2}\\
        & \quad + |\gmm_c(t)-\gmm_0(t)|\norm{iQ_{\mathcal{P}(b_0,\nu_0,0)}}_{L^2} + \norm{\epsilon_c}_{L^2} \lesssim \alpha^*|t|^{1+2\omega}.
    \end{aligned}
\end{equation}
In addition, a similar calculation yields
\begin{equation}\label{estimate for u_c for H^{1/2}}
    \norm{u_c(t) - \frac{1}{\lambda_0^{1/2}}Q_{\mathcal{P}_0}\left(\frac{\cdot-\nut{x}_0}{\lambda_0}\right)e^{i\gamma_0} -z(t)}_{H^{1/2}} \lesssim \alpha^*|t|^{2\omega}.
\end{equation}
Moreover, by Lemma~\ref{interaction error order}, in particular \eqref{interaction degeneracy 2}, and \eqref{decay of R(pqr)}, we have $(Q_{\mathcal P_0},z^\flat)_r=o_{t\to0^-}(1)$. Hence, using \eqref{estimate for u_c for H^{1/2}} and the mass conservation of $z$, we obtain
\[
    \|u_c\|_{L^2}^2=\|Q\|_{L^2}^2+\|z_f^*\|_{L^2}^2.
\]

To this end, we prove that the constructed solution $u_c$ satisfies \eqref{given energy and momentum}. Using Proposition~\ref{energy and momentum expansion for singular profile} with $\eta_n \to 0+$, \eqref{asymptotics for modulation parameters}, and \eqref{strong convergence of H^{1/2}} together with the mass and energy conservation laws, we obtain
\begin{align*}
    E(u_c(t)) = \lim_{t \to 0-}\left(\frac{(i^{-1}L_Q[iR_{1,0,0}],R_{1,0,0})_r}{2}\cdot \frac{b_c^2(t)}{\lambda_c(t)} + E(z_f^*) + \mathcal{O}(|t|^{\frac{1}{2}})\right) = E_0, \\
    P(u_c(t)) = \lim_{t \to 0-}\left(2(i^{-1}L_Q[iR_{0,1,0}],R_{0,1,0})_r\cdot \frac{\nu_c(t)}{\lambda_c(t)} + P(z_f^*) + \mathcal{O}(|t|^{\frac{1}{2}})\right) = P_0,
\end{align*}
from the choice of $C_0$ and $D_0$ in \eqref{definition of C_0 and D_0}.
\end{proof}

\begin{proof}[Proof of  Theorem~\ref{main result 2}]
    The proof is almost immediate by combining Proposition~\ref{construction of one parameter family of solutions} with the proof of Theorem~\ref{main result 1}.  
    For a given $(E_0, P_0) \in \mathbb{R}_+ \times \mathbb{R}$, let  
    \[
        u_c \in \mathcal{C}([t_1,0); H^{1/2+\delta})
    \]  
    be the Bourgain--Wang type solution associated with the prescribed energy and momentum, constructed in the proof of Theorem~\ref{main result 1}.  

    Now, we choose a sequence of functions $\{\tilde{u}_{n}\}_{n} \subset H^{1/2+\delta}$ defined by  
    \[
        \tilde{u}_n(x) \coloneqq 
        \frac{1}{\lambda_{\eta_n}^{1/2}(0)}\,
        Q_{\mathcal{P}(0, \nu_{\eta_n}(0), \eta_n)}\!\left(\frac{x-\nut{x}_{\eta_n}(0)}{\lambda_{\eta_n}(0)}\right)
        e^{i\gmm_{\eta_n}(0)} + z_f^*(x),
    \]  
    where $\eta_n$ is a subsequence chosen in the proof of Theorem~\ref{main result 1}, and  
    $(\lambda_{\eta_n}, \nut{x}_{\eta_n}, \gmm_{\eta_n}, b_{\eta_n}, \nu_{\eta_n})(0)$  
    are the initial data selected in Lemma~\ref{approximated dynamics} for each $n \in \mathbb{N}$.  

    Then, by the Cauchy theory (Theorem~\ref{Cauchy theory}), the solution to \eqref{half-wave} with initial data $\tilde{u}_{n}$ satisfies  
    \[
        u_{n}(t) = u_{\eta_n}(t), \quad \text{for all } t \in [t_1, -t_1].
    \]  
    Hence $u_n \in \mathcal{C}([t_1, -t_1]; H^{1/2+\delta})$ by Proposition~\ref{construction of one parameter family of solutions}.  
    Therefore, from the construction of $u_c$ in the proof of Theorem~\ref{main result 1} we obtain a sequence $\{u_n(t_1)\}$ such that  
    \[
        u_n(t_1) \to u_c(t_1) \quad \text{in } H^{1/2} \text{ as } n \to \infty,
    \]  
    but $u_n$ does not blow up at $T=0$.
\end{proof}

\begin{remark}[Notion of instability]\label{remark:instability}
    Our result of the instability of Bourgain--Wang type solutions is slightly weaker than those that are somewhat less robust compared to the results in \eqref{NLS} \cite{MRS2013AJM} and the Chern-Simons-Schr\"odinger equation \cite{KimKwon2019}. Specifically, we demonstrate the instability via a subsequence of one-parameter family of solutions. This limitation stems from the absence of conditional uniqueness in the constructed Bourgain--Wang type solutions. Instead, we are able to construct a sequence of solutions 
    $\{u_n\}$ such that $u_n(t_1) \to u_c(t_1)$ in $H^{1/2}$ as $n \to \infty$ for some $t_1<0$, 
    but the instability is obtained only along a carefully selected subsequence of modulation 
    parameters $\{\eta_n\}$ satisfying $\eta_n \to 0$. 
   The conditional uniqueness question is more perplexed than initially presumed. It is noteworthy that the uniqueness of the minimal blow-up solution remains open, as noted in \cite{KLR2013ARMAhalfwave}. We believe that these two issues belong to the same category and can be handled simultaneously.  

\end{remark}

\section{Proof of the main bootstrap, Lemma~\ref{Main bootstrap}}\label{proof of main bootstrap lemma}
In this section, we prove the main bootstrap lemma, Lemma~\ref{Main bootstrap}. We now fix small constants $\alpha^*>0$ and $\eta^*>0$ as follows. The constant $\alpha^*$ is selected in accordance with Lemma~\ref{construction of asymptotic profile}, while $\eta^*$ is chosen to be less than the threshold used in the decomposition of the solution $u_\eta$ in \eqref{decomposition of solution u_eta} for $t \in (T_{\text{dec}}^{(\eta)}, 0]$. Throughout the proof, these notations will remain unchanged, and any reduction in their values, if necessary, will be explicitly stated, maintaining the original notation.

In the following, we fix $\eta \in (0, \eta^*)$. For simplicity in notation, we often omit the fixed parameter $\eta$, such as using $u$ to indicate $u_\eta$. By Lemma~\ref{approximated dynamics} and the weak bootstrap assumptions \eqref{bootstrap bounds} in Lemma~\ref{Main bootstrap}, we obtain the following rough bounds:
\begin{equation}\label{orders of modulation parameters}
    |\nu(t)| \lesssim \lambda(t),\quad b^2(t)+\eta \lesssim \lambda(t),\quad |\nut{x}(t)| \lesssim |t| \lambda(t)
\end{equation}
and 
\begin{equation}\label{rough bound for lambda}
    |\lambda(t)| \leq \frac{1}{2}, \quad |t| \lesssim \lambda^{1/2}
\end{equation}
for $t \in (T_{\text{dec}}^{(\eta)}, 0] \cap [t_0,0]$. We first choose $t_1<0$ so that
$t_1 \coloneqq t_0$. By possibly reducing $t_1<0$ but still denoting the same notation, we are going to choose our desired $t_1<0$ in Lemma~\ref{Main bootstrap}.

\subsection{Estimates of the profile error, $\mathcal{E}$}
As mentioned above, we claim the smallness of profile error $\mathcal{E}$, defined in \eqref{definition of mathcal{E}},  which contains the interaction term $\text{IN}$ in \eqref{interaction definition}, the profile error of the singular profile $\Psi_{\mathcal{P}}$ in \eqref{modified profile eror}, and a higher-order term with respect to $\epsilon$ $R(\epsilon)$ in \eqref{remainder definition}.
Because of the slow decay of the singular profile $Q_{\mathcal{P}}(y) \sim \langle y \rangle^{-2}$ as $|y| \rightarrow \infty$ in Proposition~\ref{Singular profile}, the interaction of the asymptotic profile and the singular part does not become arbitrarily small. Hence, we introduce the lemma which quantitatively controls the term containing the profile error $\mathcal{E}$. The bound \eqref{interaction degeneracy}, \eqref{estimate of profile error}, and \eqref{L2 norm of IN and profile error of Q_P}, found in the lemma below, will be applied subsequently in the modulation estimate (Proposition~\ref{Modulation estimates}) and the energy estimate (Lemma~\ref{lemma : differentiation of Energy functional}).

\begin{lemma}\label{interaction error order}
    Let the interaction term $\mathrm{IN}$, the nonlinear remainder $R(\epsilon)$, and the profile error $\Psi_{\mathcal{P}}$ be defined as above. Fix \( n = 8 \). For all \( \beta \in [1,2] \), the interaction between \( z^\flat \) and other terms appearing in \eqref{equation for epsilon} satisfies localized decay estimates:  
\begin{equation}\label{interaction degeneracy}
    \left\| z^{\flat} \langle \cdot \rangle^{-2\beta} \right\|_{L^{\infty}} \lesssim \alpha^* \lambda^{\frac{1}{2} + \left(2 - \frac{2}{n} \right)\beta},
\end{equation}
and
\begin{equation}\label{interaction degeneracy 2}
    \int \langle y \rangle^{-2}|z^\flat|\,dy \lesssim \alpha^* \lambda^{\frac{5}{4}}.
\end{equation}

From \eqref{interaction degeneracy}, \eqref{interaction degeneracy 2}, we have the estimates for the interaction of \( z^{\flat} \) with the profile error and nonlinear terms:
\begin{equation}\label{estimate of profile error}
    \begin{aligned}
        \left| \int \mathrm{IN} \cdot \langle y \rangle^{-2} \, dy \right| &\lesssim \alpha^* \lambda^{\frac{5}{2} + \omega}, \\
        \left| \int R(\epsilon) \cdot \langle y \rangle^{-2} \, dy \right| &\lesssim \left(\alpha^* + \lambda^\omega\right)\lambda^{5/2+\omega}, \\
        \left| \int \Psi_{\mathcal{P}}(y) \cdot \langle y \rangle^{-2} \, dy \right| &\lesssim  \lambda^{3}.
    \end{aligned}
\end{equation}
In particular, we have
\begin{equation}\label{cal(E)-esitmate}
    |(i\mathcal{E}, \langle \cdot \rangle^{-2})_r| \lesssim (\alpha^* + \lambda^\omega)\lambda^{\frac{5}{2}+\omega}.
\end{equation}

Finally, we also have the $L^2$ estimates for the interaction term and the profile error:
\begin{equation}\label{L2 norm of IN and profile error of Q_P}
    \left\| \mathrm{IN} \right\|_{L^2} \lesssim \alpha^* \lambda^{\frac{1}{2} + \left(2 - \frac{2}{n} \right)}, \quad 
    \left\| \Psi_{\mathcal{P}} \right\|_{L^2} \lesssim \lambda^3.
\end{equation}

\end{lemma}

\begin{proof}
From the construction of the asymptotic profile flow $z(t,x)$ in Proposition~\ref{construction of asymptotic profile}, we have the pointwise bound
$$
|z(t,x)| \lesssim \alpha^* (|t| + |x|)^{m+1}.
$$
Renormalizing $z$ to the $(s,y)$-scale, we obtain
\begin{equation}\label{renormalized degeneracy of z}
    |z^{\flat}(s,y)| \lesssim \alpha^* \lambda^{\frac{1}{2}} (|t| + |\lambda y + \nut{x}|)^{m+1}.
\end{equation}

We estimate the $L^{\infty}$-norm of $z^{\flat}(s,y)\langle y \rangle^{-2\beta}$ for $\beta \in [1,2]$. To this end, we divide the domain into two regions: the inner region $|y| \leq \lambda^{-1 + \frac{1}{n}}$, and the outer region $|y| \geq \lambda^{-1 + \frac{1}{n}}$, with $n = 8$.

In the inner region, using \eqref{asymtotic for parameters for t} and \eqref{renormalized degeneracy of z}, we estimate
\begin{equation}\label{inner region}
    |z^{\flat}(s,y)\langle y \rangle^{-2\beta}| \lesssim \alpha^* \lambda^{\frac{1}{2}} \left( |t| + \lambda^{\frac{1}{n}} \right)^{m+1} \leq \alpha^* \lambda^{\frac{1}{2} + \frac{m+1}{n}}.
\end{equation}

In the outer region, from \eqref{control of norm z(t) for geq 1/2} we obtain
\begin{equation}\label{outer region}
    |z^{\flat}(s,y)\langle y \rangle^{-2\beta}| \lesssim \alpha^*\lambda^{\frac{1}{2}} \langle \lambda^{-1 + \frac{1}{n}} \rangle^{-2\beta} \leq \alpha^* \lambda^{\frac{1}{2} + \left(2 - \frac{2}{n} \right)\beta}.
\end{equation}

Since $m+1 = 33 > 32 = 4n$, it follows from \eqref{inner region} and \eqref{outer region} that
\begin{equation}\label{reason for loss of smallness due to 2/n}
    \norm{z^{\flat} \langle \cdot \rangle^{-2\beta}}_{L^{\infty}} \lesssim \alpha^* \max \left\{ \lambda^{\frac{1}{2} + \frac{m+1}{n}}, \lambda^{\frac{1}{2} + \left(2 - \frac{2}{n} \right)\beta} \right\} = \alpha^* \lambda^{\frac{1}{2} + \left(2 - \frac{2}{n} \right)\beta}.
\end{equation}

Moreover, from \eqref{inner region}, we compute
\begin{align*}
    \left| \int \langle y \rangle^{-2} z^{\flat}(s,y) \, dy \right| 
    &\leq \int_{|y| \leq \lambda^{-1 + \frac{1}{n}}} \langle y \rangle^{-2} |z^{\flat}| \, dy + \int_{|y| \geq \lambda^{-1 + \frac{1}{n}}} \langle y \rangle^{-2} |z^{\flat}| \, dy \\
    &\lesssim \alpha^* \lambda^{\frac{1}{2} + \frac{m+1}{n} - 1 - \frac{1}{n}} + \alpha^* \int_{|\tilde{y}| \geq 1} \frac{\lambda^{1/2}}{\lambda^{-1 + \frac{1}{n}}} \langle \tilde{y} \rangle^{-2} d\tilde{y} \\
    &\lesssim \alpha^* \lambda^{1 + \frac{1}{4}}.
\end{align*}

From the schematic forms \eqref{interaction definition} and \eqref{remainder definition}, we have
$$
    \mathrm{IN} \sim Q_{\mathcal{P}}^2 z^{\flat} + Q_{\mathcal{P}} (z^{\flat})^2, \qquad R(\epsilon) \sim Q_{\mathcal{P}}z^{\flat}\epsilon + (z^{\flat})^2 \epsilon + W^{\flat} \epsilon^2 + \epsilon^3.
$$
Applying the Gagliardo–Nirenberg inequality and the bootstrap bounds for $\epsilon$, \eqref{bootstrap bounds}, we estimate:
\begin{align*}
    \left| \int \mathrm{IN} \cdot \langle y \rangle^{-2} \, dy \right| 
    &\lesssim \norm{Q_{\mathcal{P}}}_{L^1} \norm{z^{\flat} Q_{\mathcal{P}} \langle \cdot \rangle^{-2}}_{L^{\infty}} + \norm{Q_{\mathcal{P}}}_{L^1}^{1/2} \norm{z^{\flat}}_{L^2} \norm{\langle \cdot \rangle^{-2} |Q_{\mathcal{P}}|^{1/2} z^{\flat}}_{L^{\infty}} \\
    &\lesssim \alpha^* \lambda^{\frac{5}{2} + \omega}, \\
    \left| \int R(\epsilon) \cdot \langle y \rangle^{-2} \, dy \right|
    &\lesssim \norm{z^{\flat} \langle \cdot \rangle^{-2}}_{L^{\infty}}
        \norm{Q_{\mathcal{P}}}_{L^2}\norm{\epsilon}_{L^2} +  \norm{z^{\flat} \langle \cdot \rangle^{-2}}_{L^{\infty}} \norm{z^{\flat}}_{L^2} \norm{\epsilon}_{L^2} \\
        &\quad + \norm{W^{\flat}\langle \cdot \rangle^{-2}}_{L^{\infty}}  \norm{\epsilon}_{L^2}^2 
        + \norm{\langle \cdot \rangle^{-2}}_{L^2} \norm{\epsilon}_{H^{1/2}}^2 \norm{\epsilon}_{L^2} \\
    &\lesssim \left(\alpha^* + \lambda^\omega\right)\lambda^{5/2+\omega}.
\end{align*}

From \eqref{profile error estimate} in Proposition~\ref{Singular profile}, we have
\[
    \left|
        \int \Psi_{\mathcal P}(y)\langle y\rangle^{-2}\,dy
    \right|
    \lesssim \lambda^3,
    \qquad
    \|\Psi_{\mathcal P}\|_{L^2}
    \lesssim \lambda^3.
\]

Finally, a direct computation with the Cauchy--Schwartz inequality yields
$$
\norm{\mathrm{IN}}_{L^2} \lesssim \norm{Q_{\mathcal{P}} z^{\flat}}_{L^{\infty}} \lesssim \alpha^* \lambda^{\frac{1}{2} + \left(2 - \frac{2}{n} \right)}.
$$

\end{proof}

\begin{remark}
    In the estimate for the interaction term of $f$ and the flow of the asymptotic profile $z^{\flat}$ in \eqref{interaction degeneracy}, there is a loss of size  $\frac{2}{n}$, which cannot be removed. In view of \eqref{reason for loss of smallness due to 2/n}, we will choose large $m$ so that the loss becomes small. But since we will keep $z^*$ to be generic in \textit{finite} co-dimension set, $m$ and $n$ are finite and the loss is indispensable.
\end{remark}

\subsection{Modulation estimates}
In this section, we derive modulation estimates and smallness of unstable direction from the orthogonal relation on $\epsilon$, \eqref{orthogonality for Lambda Q}-\eqref{orthogonality for partial_v Q}. More precisely, we derive modulation estimate by differentiating the orthogonality conditions on $\epsilon$ with respect to rescaled time variable $s$. Thus, we investigate the inner product of $\partial_s \epsilon$ with orthogonal directions $\{i\Lambda Q_{\mathcal{P}}, i\nabla Q_{\mathcal{P}}, i\partial_b Q_{\mathcal{P}}, i\partial_\nu Q_{\mathcal{P}}, i\partial_\eta Q_{\mathcal{P}}\}$. 
Hence, recall the evolution of \eqref{equation for epsilon},
\begin{align*}
    \partial_s\epsilon + &iL_{Q_{\mathcal{P}}}[\epsilon] -(\nu + c_2b^2\nu)\nabla \epsilon + b\Lambda \epsilon \\
    &= \overrightarrow{\text{Mod}}(t) \cdot (-\partial_b Q_{\mathcal{P}}, - \partial_\nu Q_{\mathcal{P}}, \Lambda(Q_{\mathcal{P}}+\epsilon), \nabla(Q_{\mathcal{P}}+\epsilon), -i(Q_{\mathcal{P}}+\epsilon)) + i\mathcal{E}.
\end{align*}
Also, from \eqref{difference of L_(Q_p) and L_Q} and \eqref{modified profile of Q}, we have
\begin{equation}\label{diff for mod eq}
        L_{Q_{\mathcal{P}}}[\epsilon] = L_Q[\epsilon] - 2ibQR_{1,0,0}\overline{\epsilon} + \mathcal{O}(|\lambda|\epsilon).
\end{equation}
Now, for notational convenience, we denote $M(\epsilon)$ and $\overrightarrow{V}(\epsilon)$ by
\begin{align*}
        M(\epsilon) &\coloneqq i^{-1}L_{Q_{\mathcal{P}}}[\epsilon] + (\nu+c_2b^2\nu)\nabla \epsilon - b\Lambda \epsilon,\\
        \overrightarrow{V}(\epsilon) &\coloneqq (-\partial_b Q_{\mathcal{P}}, - \partial_\nu Q_{\mathcal{P}}, \Lambda(Q_{\mathcal{P}}+\epsilon), \nabla(Q_{\mathcal{P}}+\epsilon), -i(Q_{\mathcal{P}}+\epsilon)),
\end{align*}
respectively.
Then, \eqref{equation for epsilon} is written
\begin{equation}
    \partial_s \epsilon = M(\epsilon) + \overrightarrow{\text{Mod}}(t)\cdot \overrightarrow{V}(\epsilon) + i\mathcal{E}.
\end{equation}
We start from the estimate of the inner product of $M(\epsilon)$ orthogonal directions.
\begin{lemma}\label{M estimate}
    Under the weak bootstrap assumption \eqref{bootstrap bounds}, we have
    \begin{align}
            &(M(\epsilon), i\Lambda Q_{\mathcal{P}})_r = (\epsilon,Q_{\mathcal{P}})_r + \mathcal{O}(\lambda \norm{\epsilon}_{L^2}), \label{Lambda Q}\\
            & (M(\epsilon), i\partial_bQ_{\mathcal{P}})_r =  \mathcal{O}(\lambda \norm{\epsilon}_{L^2}), \label{partial_b Q}\\
            & (M(\epsilon), i\partial_\eta Q_{\mathcal{P}})_r = \mathcal{O}(\lambda\norm{\epsilon}_{L^2}), \label{partial_eta Q} \\
            & (M(\epsilon), i\nabla Q_{\mathcal{P}})_r = 
            \mathcal{O}(\lambda \norm{\epsilon}_{L^2}), \label{nabla Q}\\
            & (M(\epsilon), i\partial_\nu Q_{\mathcal{P}})_r = 
            \mathcal{O}(\lambda\norm{\epsilon}_{L^2}). \label{partial_v Q}
    \end{align}
\end{lemma}
\begin{proof}
Note that from \eqref{orders of modulation parameters}, we have $b^2+\eta \lesssim \lambda$ and $|\nu|\lesssim\lambda$. 

\vspace{10pt}
\noindent\textbf{Estimate \eqref{Lambda Q}:} Since $\Lambda Q_{\mathcal{P}} = \Lambda Q + ib\Lambda R_{1,0,0} + \mathcal{O}(\lambda)$, we have
\begin{align}
    (-M(\epsilon), i\Lambda Q_{\mathcal{P}})_r
    =& (L_{Q_{\mathcal{P}}}[\epsilon],\Lambda Q_{\mathcal{P}})_r + b(\Lambda \epsilon, i\Lambda Q)_r + \mathcal{O}(\lambda\norm{\epsilon}_{L^2}) \nonumber\\
    =& (L_Q[\epsilon],\Lambda Q)_r  \label{eq:M esti Lambda 1}
    \\
    &+ b\{(L_Q[\epsilon],i\Lambda R_{1,0,0})_r + (\Lambda \epsilon, i\Lambda Q)_r -2(iQR_{1,0,0},\epsilon\Lambda Q)_r\}  \label{eq:M esti Lambda 2}
    \\
    &+ \mathcal{O}(\lambda\norm{\epsilon}_{L^2}). \nonumber
\end{align}
For \eqref{eq:M esti Lambda 1}, using $L_Q[\Lambda Q]=-Q$ and \eqref{modified profile of Q}, we have
\begin{align*}
    \eqref{eq:M esti Lambda 1}
    =-(\epsilon,Q)
    =-(\epsilon,Q_{\mathcal{P}})_r+ b(\epsilon,iR_{1,0,0})_r + \mathcal{O}(\lambda\norm{\epsilon}_{L^2}).
\end{align*}
    From the commutator formula \eqref{eq:LQ Lambda commute} and $L_Q[iR_{1,0,0}]=i\Lambda Q$, we compute \eqref{eq:M esti Lambda 2} as follows:
\begin{align*}
    &(i^{-1}L_Q[\epsilon],\Lambda R_{1,0,0})_r + (\Lambda \epsilon, i\Lambda Q)_r -2(iQR_{1,0,0},\epsilon\Lambda Q)_r \\
    &= (i^{-1}\epsilon,i^{-1}L_Q[\Lambda iR_{1,0,0}])_r - (i^{-1}\epsilon,\Lambda (i^{-1}L_Q[iR_{1,0,0}]))_r -2(\epsilon,iQ\Lambda QR_{1,0,0})_r \\
    &= (\epsilon, i\{DR_{1,0,0} + y\nabla(Q^2)R_{1,0,0}\})_r - (\epsilon,i\{Q^2R_{1,0,0} + y\nabla(Q^2)R_{1,0,0}\})_r \\
    &= -(\epsilon, iR_{1,0,0})_r + (\epsilon,L_Q[iR_{1,0,0}])_r
\end{align*}
From the orthogonal condition on $\epsilon$ \eqref{orthogonality for Lambda Q}, we have
\begin{equation*}
    0=(\epsilon,i\Lambda Q_{\mathcal{P}})_r = (\epsilon, i\Lambda Q)_r + \mathcal{O}(b\norm{\epsilon}_{L^2}) = (\epsilon, L_Q[iR_{1,0,0}])_r + \mathcal{O}(b\norm{\epsilon}_{L^2}).
\end{equation*}
Thus, we obtain, 
\begin{align*}
    \eqref{eq:M esti Lambda 2}=
    -b(\epsilon, iR_{1,0,0})_r+ \mathcal{O}(\lambda \norm{\epsilon}_{L^2}).
\end{align*}
We deduce
\begin{align*}
    \eqref{eq:M esti Lambda 1}+\eqref{eq:M esti Lambda 2}
    =-(\epsilon,Q_{\mathcal{P}})_r+ \mathcal{O}(\lambda \norm{\epsilon}_{L^2}),
\end{align*}
which implies
\begin{equation*}
    (-i^{-1}L_{Q_{\mathcal{P}}}[\epsilon]  - (\nu+c_2b^2\nu)\nabla \epsilon + b\Lambda \epsilon, i\Lambda Q)_r = -(\epsilon,Q_{\mathcal{P}})_r + \mathcal{O}(\lambda \norm{\epsilon}_{L^2}).
\end{equation*}

\vspace{10pt}
\noindent\textbf{Estimate \eqref{partial_b Q}:}
 Since $\partial_b Q_{\mathcal{P}} = iR_{1,0,0} - 2bR_{2,0,0} + \mathcal{O}(\lambda)$, we have
\begin{align*}
    (-i^{-1}L_{Q_{\mathcal{P}}}[\epsilon]  &- (\nu+c_2b^2\nu)\nabla \epsilon + b\Lambda \epsilon, i\partial_bQ_{\mathcal{P}})_r \\
    &= (L_{Q_{\mathcal{P}}}[\epsilon],iR_{1,0,0}-2bR_{2,0,0})_r - b(\Lambda \epsilon, R_{1,0,0})_r + \mathcal{O}(\lambda \norm{\epsilon}_{L^2})\\
    &= (L_Q[\epsilon],iR_{1,0,0})_r + b\{(\epsilon, \Lambda R_{1,0,0})_r-2(L_Q[\epsilon],R_{2,0,0})_r \\
    & \quad -2(\epsilon,QR_{1,0,0}^2)_r\} + \mathcal{O}(\lambda\norm{\epsilon}_{L^2}).
\end{align*}
From \eqref{order b^2}, we have
\begin{align*}
    &(\epsilon, \Lambda R_{1,0,0})_r-2(L_Q[\epsilon],R_{2,0,0})_r -2(\epsilon,QR_{1,0,0}^2)_r \\
    &= (\epsilon,\Lambda R_{1,0,0})_r -2(\epsilon,-\frac{1}{2}R_{1,0,0} + \Lambda R_{1,0,0} -QR_{1,0,0}^2)_r -2(\epsilon,R_{1,0,0}^2Q)_r\\
    &= (\epsilon,R_{1,0,0})_r - (\epsilon,\Lambda R_{1,0,0})_r.
\end{align*}
From the orthogonal condition \eqref{orthogonality for partial_b Q}, and \eqref{modified profile of Q}, we obtain
\begin{equation*}
    (\epsilon,i\partial_b Q_{\mathcal{P}})_r = -(\epsilon,R_{1,0,0})_r + \mathcal{O}(b \norm{\epsilon}_{L^2}).
\end{equation*}
Also, from $L_Q[iR_{1,0,0}]=i\Lambda Q$, $Q_{\mathcal{P}} = Q+ibR_{1,0,0} + \mathcal{O}(\lambda)$ and the orthogonal condition for $\epsilon$ in \eqref{orthogonality for Lambda Q}, we have
\begin{align*}
    &(-i^{-1}L_{Q_{\mathcal{P}}}[\epsilon]  - (\nu+c_2b^2\nu)\nabla \epsilon + b\Lambda \epsilon, i\partial_bQ_{\mathcal{P}})_r \\
    &= (\epsilon,i\Lambda Q_{\mathcal{P}})_r + \mathcal{O}(\lambda \norm{\epsilon}_{L^2}) = \mathcal{O}(\lambda \norm{\epsilon}_{L^2}).
\end{align*}

\vspace{10pt}
\noindent\textbf{Estimate \eqref{partial_eta Q}:}
 Since $\partial_\eta Q_{\mathcal{P}} = R_{0,0,1} + ibR_{1,0,1} + \mathcal{O}(\lambda)$, we have
\begin{align*}
     (-i^{-1}L_{Q_{\mathcal{P}}}[\epsilon] &- (\nu+c_2b^2\nu)\nabla \epsilon + b\Lambda \epsilon, i\partial_\eta Q_{\mathcal{P}})_r \\
     & = (L_Q[\epsilon],R_{0,0,1})_r + b\{(L_Q[\epsilon],iR_{1,0,1})_r \\
     &\quad - 2(iQR_{1,0,0}\overline{\epsilon},R_{0,0,1})_r + (\Lambda \epsilon,iR_{0,0,1})_r\} + \mathcal{O}(\lambda \norm{\epsilon}_{L^2}). 
\end{align*}
From \eqref{order b eta}, we have
\begin{align*}
    &(L_Q[\epsilon],iR_{1,0,1})_r - 2(iQR_{1,0,0}\overline{\epsilon},R_{0,0,1})_r + (\Lambda \epsilon, iR_{0,0,1})_r \\
    &= (i^{-1}\epsilon,2R_{2,0,0}+\Lambda R_{0,0,1} + 2QR_{0,0,1}R_{1,0,0})_r\\
    & \quad - 2(\epsilon, iQR_{0,0,1}R_{1,0,0})_r - (\epsilon, i\Lambda R_{0,0,1})_r\\
    &= (\epsilon, 2iR_{2,0,0})_r.
\end{align*} 
Thus, from the orthogonality condition on $\epsilon$, \eqref{orthogonality for partial_b Q}, and \eqref{order eta}, we obtain,
\begin{align*}
    (-i^{-1}L_{Q_{\mathcal{P}}}[\epsilon] &- (\nu+c_2b^2\nu)\nabla \epsilon + b\Lambda \epsilon, i\partial_\eta Q_{\mathcal{P}})_r \\
    & = -(\epsilon, i\{iR_{1,0,0}-2bR_{2,0,0}\})_r + \mathcal{O}(\lambda \norm{\epsilon}_{L^2}) \\
    &= -(\epsilon,i\partial_bQ_{\mathcal{P}})_r + \mathcal{O}(\lambda\norm{\epsilon}_{L^2}).
\end{align*}

\vspace{10pt}
\noindent\textbf{Estimate \eqref{nabla Q}:}
 Since $\nabla Q_{\mathcal{P}}=\nabla Q + ib\nabla R_{1,0,0} + \mathcal{O}(\lambda)$, we have
\begin{align*}
     (-i^{-1}L_{Q_{\mathcal{P}}}[\epsilon] & - (\nu+c_2b^2\nu)\nabla \epsilon + b\Lambda \epsilon, i\nabla Q_{\mathcal{P}})_r \\
     &= (L_Q[\epsilon],\nabla Q)_r + b\{(L_Q[\epsilon],i\nabla R_{1,0,0})_r \\
     & \quad -2(iR_{1,0,0}Q\overline{\epsilon},\nabla Q)_r + (\Lambda \epsilon, i\nabla Q)_r\} + \mathcal{O}(\lambda \norm{\epsilon}_{L^2}).
\end{align*}
From the commutator formula $[i^{-1}L_Q, \nabla]f = -i\nabla(Q^2)f - 2i\Re\{\nabla(Q^2)f\}$, $[\Lambda ,\nabla] = -\nabla$ and \eqref{order b} which is $L_Q[iR_{1,0,0}] = i\Lambda Q$, we obtain
\begin{align*}
    &(L_Q[\epsilon],i\nabla R_{1,0,0})_r -2(iR_{1,0,0}Q\overline{\epsilon},\nabla Q)_r + (\Lambda \epsilon, i\nabla Q)_r \\
    &= (i^{-1}\epsilon,i^{-1}L_Q[\nabla iR_{1,0,0}])_r - (\epsilon,iR_{1,0,0}\nabla (Q^2))_r +(\Lambda \epsilon, i\nabla Q)_r \\
    &= (i^{-1}\epsilon, [i^{-1}L_Q,\nabla](iR_{1,0,0}))_r + (i^{-1}\epsilon, \nabla (\Lambda Q))_r\\
    & \quad - (\epsilon,i\Lambda\nabla Q)_r - (\epsilon,i\nabla(Q^2)R_{1,0,0})_r\\
    &= (\epsilon, i\nabla Q)_r
\end{align*}
Now, by the orthogonal condition, we obtain $(\epsilon, i\nabla Q)_r = \mathcal{O}(b\norm{\epsilon}_{L^2})$. Since $\nabla Q$ is a kernel of $L_Q$, we conclude
\begin{equation*}
    (-i^{-1}L_{Q_{\mathcal{P}}}[\epsilon] - (\nu+c_2b^2\nu)\nabla \epsilon + b\Lambda \epsilon, i\nabla Q_{\mathcal{P}})_r =\mathcal{O}(\lambda \norm{\epsilon}_{L^2}).
\end{equation*}

\vspace{10pt}
\noindent \textbf{Estimate \eqref{partial_v Q}:}
 Since $\partial_\nu Q_{\mathcal{P}} = iR_{0,1,0} - bR_{1,1,0} + \mathcal{O}(\lambda)$, we have
\begin{align*}
    (-i^{-1}L_{Q_{\mathcal{P}}}[\epsilon] - &(\nu+c_2b^2\nu)\nabla \epsilon + b\Lambda \epsilon, i\partial_\nu Q_{\mathcal{P}})_r \\
    &= (L_Q[\epsilon], iR_{0,1,0})_r -b\{(L_Q[\epsilon],R_{1,1,0})_r \\
    &\quad + 2(\epsilon,QR_{1,0,0}R_{0,1,0})_r + (\Lambda \epsilon,R_{0,1,0})_r\} + \mathcal{O}(\lambda\norm{\epsilon}_{L^2}).
\end{align*}
From, \eqref{order bv} we have
\begin{align*}
    &(L_Q[\epsilon],R_{1,1,0})_r + 2(\epsilon,QR_{1,0,0}R_{0,1,0})_r + (\Lambda \epsilon,R_{0,1,0})_r \\
    &= (\epsilon, L_Q[R_{1,1,0}])_r +(\epsilon,2QR_{1,0,0}R_{0,1,0})_r - (\epsilon,\Lambda R_{0,1,0})_r \\
    &= -(\epsilon, R_{0,1,0})_r - (\epsilon, \nabla R_{1,0,0})_r.
\end{align*}
Thus, from the orthogonality condition for $\epsilon$, \eqref{orthogonality for partial_v Q} and \eqref{orthogonality for nabla Q}, we have
\begin{align*}
    &(-i^{-1}L_{Q_{\mathcal{P}}}[\epsilon] - (\nu+c_2b^2\nu)\nabla \epsilon + b\Lambda \epsilon, i\partial_\nu Q_{\mathcal{P}})_r \\
    &= -(\epsilon, i\{\nabla Q + ib\nabla R_{1,0,0}\})_r - b(\epsilon, i^2R_{0,1,0})_r + \mathcal{O}(\lambda \norm{\epsilon}_{L^2}) \\
    &= -(\epsilon, i\nabla Q_{\mathcal{P}})_r - b(\epsilon,i\partial_\nu Q_{\mathcal{P}})_r + \mathcal{O}(\lambda\norm{\epsilon}_{L^2}).
\end{align*}
\end{proof}

Now, we are ready to derive modulation estimates and degeneracy of unstable direction. By using the Lemma~\ref{interaction error order} and Lemma~\ref{M estimate}, we obtain the estimates below.

\begin{proposition}[Modulation estimates]\label{Modulation estimates} The following holds true: 
    \begin{enumerate}
        \item (Degeneracy of the unstable direction) There holds
        \begin{equation}\label{unstable direction degeneracy}
        |(\epsilon,Q_{\mathcal{P}})_r| \lesssim (\lambda^{\omega}+\alpha^*)\lambda^{2+\omega}.
        \end{equation}

        \item (Modulation equation) We have the modulation estimates
        \begin{equation}\label{Modulation equation}
        |\overrightarrow{\text{Mod}}(t)| \lesssim (\lambda^{\omega}+\alpha^*)\lambda^{2+\omega},
        \end{equation}
        and improved bounds
        \begin{equation}\label{improved bounds}
            \left|\frac{\lambda_s}{\lambda}+b\right| + |\nu_s+b\nu|+\left|\frac{\nut{x}_s}{\lambda}-\nu-c_2b^2\nu\right|\lesssim (\lambda^{\omega}+\alpha^*)\lambda^{\frac{5}{2}+\omega}.
        \end{equation}
    \end{enumerate}
\end{proposition}

\begin{proof}
Typically, modulation laws are derived from differentiating orthogonality conditions.

   \noindent \textbf{Step 1. } Law for $b$: We differentiate the orthogonal condition \eqref{orthogonality for Lambda Q} on $\epsilon$ with respect to $s$. Then, we have
    \begin{align*}
        \frac{d}{ds}(\epsilon,i\Lambda Q_{\mathcal{P}})_r = (\partial_s \epsilon,i\Lambda Q_{\mathcal{P}})_r + (\epsilon,\partial_si\Lambda Q_{\mathcal{P}})_r = 0.
    \end{align*}
    From the construction of $Q_{\mathcal{P}}$ in \eqref{modified profile of Q}, we have
    \begin{equation*}
        |(\epsilon,\partial_s i\Lambda Q_{\mathcal{P}})_r| \lesssim \norm{\epsilon}_{L^2}(|\overrightarrow{\text{Mod}}(t)| + \lambda).
    \end{equation*}
    On the other hand, from \eqref{equation for epsilon}, we have
    \begin{align*}
        (\partial_s \epsilon, i\Lambda Q_{\mathcal{P}})_r &= (M(\epsilon),i\Lambda Q_{\mathcal{P}})_r + (\overrightarrow{\text{Mod}}\cdot \overrightarrow{V}(\epsilon),i\Lambda Q_{\mathcal{P}})_r + (i\mathcal E, i\Lambda Q_{\mathcal{P}})_r. 
    \end{align*}
    From \eqref{Lambda Q}, we have
    \begin{equation*}
        (M(\epsilon),i\Lambda Q_{\mathcal{P}})_r = (\epsilon,Q_{\mathcal{P}})_r + \mathcal{O}(\lambda \norm{\epsilon}_{L^2}).
    \end{equation*}
    Since, from \eqref{decay of R(pqr)}, we have $\Lambda Q_{\mathcal{P}}(y) \sim \langle y \rangle^{-2}$ as $|y| \rightarrow \infty$. Thus, Lemma \ref{interaction error order} implies
    \begin{equation*}
        |(i\mathcal{E}, i\Lambda Q_{\mathcal{P}})_r| \lesssim (\lambda^\omega+\alpha^*)\lambda^{\frac{5}{2}+\omega}.
    \end{equation*}
    Finally we estimate $(\overrightarrow{\text{Mod}}\cdot \overrightarrow{V}(\epsilon),\Lambda Q_{\mathcal{P}})_r$. From \eqref{orders of modulation parameters} and the evenness and oddness property of $R_{p,q,r}$ in Proposition \ref{Singular profile}, we have
    \begin{equation}\label{innerproduct with Lambda Q}
        \begin{aligned}
            (-\partial_b Q_{\mathcal{P}},i\Lambda Q_{\mathcal{P}})_r &= -(iR_{1,0,0} - 2bR_{2,0,0},i\{\Lambda Q + ib\Lambda R_{1,0,0}\})_r + \mathcal{O}(\lambda)\\
            &= -(R_{1,0,0},i^{-1}L_Q[iR_{1,0,0}])_r + \mathcal{O}(\lambda), \\
            (-\partial_\nu Q_{\mathcal{P}}, i\Lambda Q_{\mathcal{P}})_r &= -(iR_{0,1,0}-bR_{1,1,0},i\{\Lambda Q + ib\Lambda R_{1,0,0}\})_r + \mathcal{O}(\lambda)\\
            &= -(iR_{0,1,0},i\Lambda Q)_r + \mathcal{O}(\lambda) = \mathcal{O}(\lambda),\\
            (\Lambda Q_{\mathcal{P}} + \Lambda \epsilon,i\Lambda Q_{\mathcal{P}})_r &= \mathcal{O}(\norm{\epsilon}_{L^2}),\\
            (\nabla Q_{\mathcal{P}}+\nabla \epsilon, i\Lambda Q_{\mathcal{P}})_r &= (\nabla Q + ib\nabla R_{1,0,0},i\{\Lambda Q + ib\Lambda R_{1,0,0}\})_r + \mathcal{O}(\lambda + \norm{\epsilon}_{L^2})\\
            &= \mathcal{O}(\lambda),\\
            (-i(Q_{\mathcal{P}}+\epsilon),i\Lambda Q_{\mathcal{P}})_r &= \mathcal{O}(\norm{\epsilon}_{L^2}).
        \end{aligned}
    \end{equation}
    Therefore, we have
    \begin{equation}\label{modulation equation for b}
        \left|b_s + \left(\frac{1}{2}+ c_3\eta\right)b^2 + \eta + c_1b^4 + c_4\nu^2\right| \lesssim |(\epsilon,Q_{\mathcal{P}})_r| + \lambda|\overrightarrow{\text{Mod}}(t)| + \lambda\norm{\epsilon}_{L^2} + (\lambda^{\omega}+\alpha^*)\lambda^{\frac{5}{2}+\omega}.
    \end{equation}
    \noindent \textbf{Step 2. } Law for $\lambda$: Differentiating the orthogonal condition \eqref{orthogonality for partial_b Q} on $\epsilon$ with respect to $s$, we obtain
    \begin{align*}
        \frac{d}{ds}(\epsilon,i\partial_b Q_{\mathcal{P}})_r = (\partial_s \epsilon,i\partial_b Q_{\mathcal{P}})_r + (\epsilon,\partial_si\partial_b Q_{\mathcal{P}})_r = 0.
    \end{align*}
    From the construction of $Q_{\mathcal{P}}$ in \eqref{modified profile of Q}, we have
    \begin{equation*}
        |(\epsilon,\partial_s i\partial_b Q_{\mathcal{P}})_r| \lesssim \norm{\epsilon}_{L^2}(|\overrightarrow{\text{Mod}}(t)| + \lambda).
    \end{equation*}
    On the other hand, from \eqref{equation for epsilon}, we have
    \begin{align*}
        (\partial_s \epsilon, i\partial_b Q_{\mathcal{P}})_r &= (M(\epsilon),i\partial_b Q_{\mathcal{P}})_r + (\overrightarrow{\text{Mod}}\cdot \overrightarrow{V}(\epsilon),i\partial_b Q_{\mathcal{P}})_r + (i\mathcal E, i\partial_b Q_{\mathcal{P}})_r. 
    \end{align*}
    From \eqref{partial_b Q}, we have
    \begin{equation*}
        (M(\epsilon),i\partial_b Q_{\mathcal{P}})_r =  \mathcal{O}(\lambda \norm{\epsilon}_{L^2}).
    \end{equation*}
    Thanks to \eqref{decay of R(pqr)}, we have $\partial_b Q_{\mathcal{P}}(y) \sim \langle y \rangle^{-2}$ as $|y| \rightarrow \infty$. Hence, Lemma \ref{interaction error order} yields
    \begin{equation*}
        |(i\mathcal{E}, i\partial_b Q_{\mathcal{P}})_r| \lesssim (\lambda^{\omega}+\alpha^*)\lambda^{\frac{5}{2}+\omega}.
    \end{equation*}
    Finally, we estimate $(\overrightarrow{\text{Mod}}\cdot \overrightarrow{V}(\epsilon),i\partial_b Q_{\mathcal{P}})_r$. From \eqref{orders of modulation parameters} and evenness and oddness property of $R_{p,q,r}$ in Proposition \ref{Singular profile}, we have
    \begin{align*}
        (-\partial_b Q_{\mathcal{P}},i\partial_b Q_{\mathcal{P}})_r &=0,\\
        (-\partial_\nu Q_{\mathcal{P}}, i\partial_b Q_{\mathcal{P}})_r &= -(iR_{0,1,0}-bR_{1,1,0},i\{iR_{1,0,0}-2bR_{2,0,0}\})_r + \mathcal{O}(\lambda)\\
        &= -b(R_{1,1,0},R_{1,0,0}) +\mathcal{O}(\lambda) = \mathcal{O}(b),\\
        (\Lambda Q_{\mathcal{P}} + \Lambda \epsilon,i\partial_b Q_{\mathcal{P}})_r &= (\Lambda Q + ib\Lambda R_{1,0,0},i\{iR_{1,0,0}-2bR_{2,0,0}\})_r + \mathcal{O}(\lambda + \norm{\epsilon}_{L^2})\\
        &= -(R_{1,0,0}, i^{-1}L_Q[iR_{1,0,0}])_r + \mathcal{O}(\lambda),\\
        (\nabla Q_{\mathcal{P}}+\nabla \epsilon, i\partial_b Q_{\mathcal{P}})_r &= (\nabla Q + ib\nabla R_{1,0,0},i\{iR_{1,0,0}-2bR_{2,0,0}\})_r + \mathcal{O}(\lambda + \norm{\epsilon}_{L^2})\\
        &= \mathcal{O}(\lambda),\\
        (-i(Q_{\mathcal{P}}+\epsilon),i\partial_b Q_{\mathcal{P}})_r &= 
        -(Q+ibR_{1,0,0},iR_{1,0,0}-2bR_{2,0,0})_r + \mathcal{O}(\lambda + \norm{\epsilon}_{L^2}) \\
        &= \mathcal{O}(\lambda).
    \end{align*}
    The last equality follows from \eqref{relation of R(200) and R(100)}.
    Therefore, we have
    \begin{equation}\label{modulation equation for lambda}
        \left|\frac{\lambda_s}{\lambda}+b\right| \lesssim  \lambda^{\frac{1}{2}}|\overrightarrow{\text{Mod}}(t)| + \lambda\norm{\epsilon}_{L^2} + (\lambda^{\omega}+\alpha^*)\lambda^{\frac{5}{2}+\omega}.
    \end{equation}
    
    \noindent \textbf{Step 3. } Law for $\tilde{\gamma}$: Similarly, from the orthogonal condition \eqref{orthogonality for partial_eta Q}, we derive
    \begin{align*}
        \frac{d}{ds}(\epsilon,i\partial_\eta Q_{\mathcal{P}})_r = (\partial_s \epsilon,i\partial_\eta Q_{\mathcal{P}})_r + (\epsilon,\partial_si\partial_\eta Q_{\mathcal{P}})_r = 0.
    \end{align*}
    From the construction of $Q_{\mathcal{P}}$ in \eqref{modified profile of Q}, we have
    \begin{equation*}
        |(\epsilon,\partial_s i\partial_\eta Q_{\mathcal{P}})_r| \lesssim \norm{\epsilon}_{L^2}(|\overrightarrow{\text{Mod}}(t)| + \lambda).
    \end{equation*}
    On the other hand, from \eqref{equation for epsilon}, we have
    \begin{align*}
        (\partial_s \epsilon, i\partial_\eta Q_{\mathcal{P}})_r &= (M(\epsilon),i\partial_\eta Q_{\mathcal{P}})_r + (\overrightarrow{\text{Mod}}\cdot \overrightarrow{V}(\epsilon),i\partial_\eta Q_{\mathcal{P}})_r + (i\mathcal E, i\partial_\eta Q_{\mathcal{P}})_r. 
    \end{align*}
    From \eqref{partial_eta Q}, we have
    \begin{equation*}
        (M(\epsilon),i\partial_\eta Q_{\mathcal{P}})_r = \mathcal{O}(\lambda \norm{\epsilon}_{L^2}).
    \end{equation*}
    Since, from \eqref{decay of R(pqr)}, we have $\partial_\eta Q_{\mathcal{P}}(y) \sim \langle y \rangle^{-2}$ as $|y| \rightarrow \infty$. Hence, Lemma \eqref{interaction error order} yields
    \begin{equation*}
        |(i\mathcal{E}, i\partial_\eta Q_{\mathcal{P}})_r| \lesssim (\lambda^{\omega}+\alpha^*)\lambda^{\frac{5}{2}+\omega}.
    \end{equation*}
    Finally we estimate $(\overrightarrow{\text{Mod}}\cdot \overrightarrow{V}(\epsilon),i\partial_\eta Q_{\mathcal{P}})_r$. From \eqref{orders of modulation parameters} and the evenness and oddness property of $R_{p,q,r}$ in Proposition \ref{Singular profile}, we have
    \begin{align*}
        (-\partial_b Q_{\mathcal{P}},i\partial_\eta Q_{\mathcal{P}})_r &= -(iR_{1,0,0} - 2bR_{2,0,0},i\{R_{0,0,1}+ibR_{1,0,1}\})_r + \mathcal{O}(\lambda)\\
        &= - (R_{1,0,0},R_{0,0,1})_r + \mathcal{O}(\lambda), \\
        (-\partial_\nu Q_{\mathcal{P}}, i\partial_\eta Q_{\mathcal{P}})_r &= -(iR_{0,1,0}-bR_{1,1,0},i\{R_{0,0,1} + ibR_{1,0,1}\})_r + \mathcal{O}(\lambda)\\
        &= \mathcal{O}(\lambda),\\
        (\Lambda Q_{\mathcal{P}} + \Lambda \epsilon,i\partial_\eta Q_{\mathcal{P}})_r &= (\Lambda Q + ib\Lambda R_{1,0,0},i\{R_{0,0,1}+ibR_{1,0,1}\})_r + \mathcal{O}(\lambda + \norm{\epsilon}_{L^2})\\
        &= \mathcal{O}(b),\\
        (\nabla Q_{\mathcal{P}}+\nabla \epsilon, i\partial_\eta Q_{\mathcal{P}})_r &= (\nabla Q + ib\nabla R_{1,0,0},i\{R_{0,0,1}+ibR_{1,0,1}\})_r + \mathcal{O}(\lambda + \norm{\epsilon}_{L^2})\\
        &= \mathcal{O}(\lambda),\\
        (-i(Q_{\mathcal{P}}+\epsilon),i\partial_\eta Q_{\mathcal{P}})_r &= -(Q+ibR_{1,0,0},R_{0,0,1}+ibR_{1,0,1})_r + \mathcal{O}(\lambda + \norm{\epsilon}_{L^2}) \\
        &= -(Q,R_{0,0,1})_r + \mathcal{O}(\lambda) = (R_{1,0,0},i^{-1}L_Q[iR_{1,0,0}])_r + \mathcal{O}(\lambda).
    \end{align*}
    Therefore, from \eqref{modulation equation for b} and \eqref{modulation equation for lambda}, we obtain
    \begin{equation}\label{modulation equation for tilde gamma}
        \begin{aligned}
            \left|\tilde{\gamma}_s\right| &\lesssim (\text{(LHS) of } \eqref{modulation equation for b}) + \lambda^{1/2}\left|\frac{\lambda_s}{\lambda} + b\right| + \lambda\norm{\epsilon}_{L^2}+\lambda|\overrightarrow{\text{Mod}}(t)| +(\lambda^{\omega}+\alpha^*)\lambda^{\frac{5}{2}+\omega}\\
            & \lesssim |(\epsilon,Q_{\mathcal{P}})_r| + \lambda|\overrightarrow{\text{Mod}}(t)| +  (\lambda^{\omega}+\alpha^*)\lambda^{\frac{5}{2}+\omega}.
        \end{aligned}
    \end{equation}

 \noindent \textbf{Step 4. } Law for $\nu$: From the orthogonal condition \eqref{orthogonality for nabla Q}, we have
    \begin{align*}
        \frac{d}{ds}(\epsilon,i\nabla Q_{\mathcal{P}})_r = (\partial_s \epsilon,i\nabla Q_{\mathcal{P}})_r + (\epsilon,\partial_si\nabla Q_{\mathcal{P}})_r = 0.
    \end{align*}
    From the construction of $Q_{\mathcal{P}}$ in \eqref{modified profile of Q}, we have
    \begin{equation*}
        |(\epsilon,\partial_s i\nabla Q_{\mathcal{P}})_r| \lesssim \norm{\epsilon}_{L^2}(|\overrightarrow{\text{Mod}}(t)| + \lambda).
    \end{equation*}
    On the other hand, from \eqref{equation for epsilon}, we have
    \begin{align*}
        (\partial_s \epsilon, i\nabla Q_{\mathcal{P}})_r &= (M(\epsilon),i\nabla Q_{\mathcal{P}})_r + (\overrightarrow{\text{Mod}}\cdot \overrightarrow{V}(\epsilon),i\nabla Q_{\mathcal{P}})_r + (i\mathcal E, i\nabla Q_{\mathcal{P}})_r. 
    \end{align*}
    From \eqref{nabla Q}, we have
    \begin{equation*}
        (M(\epsilon),i\nabla Q_{\mathcal{P}})_r = \mathcal{O}(\lambda \norm{\epsilon}_{L^2}).
    \end{equation*}
    Since, from \eqref{decay of R(pqr)}, we have $\nabla Q_{\mathcal{P}}(y) \sim \langle y \rangle^{-2}$ as $|y| \rightarrow \infty$. Hence, Lemma \eqref{interaction error order} yields
    \begin{equation*}
        |(i\mathcal{E}, i\nabla Q_{\mathcal{P}})_r| \lesssim (\lambda^{\omega}+\alpha^*)\lambda^{\frac{5}{2}+\omega}.
    \end{equation*}
    Finally we estimate $(\overrightarrow{\text{Mod}}\cdot \overrightarrow{V}(\epsilon),i\nabla Q_{\mathcal{P}})_r$. From \eqref{orders of modulation parameters} and evenness and oddness property of $R_{p,q,r}$ in Proposition \ref{Singular profile} and previous computation, we have
    \begin{align*}
        (-\partial_b Q_{\mathcal{P}},i\nabla Q_{\mathcal{P}})_r 
        &= \mathcal{O}(\lambda), \\
        (-\partial_\nu Q_{\mathcal{P}}, i\nabla Q_{\mathcal{P}})_r &= -(iR_{0,1,0}-bR_{1,1,0},i\{\nabla Q + ib\nabla R_{1,0,0}\})_r + \mathcal{O}(\lambda)\\
        &= (R_{0,1,0},i^{-1}L_Q[iR_{0,1,0}])_r + \mathcal{O}(\lambda),\\
        (\Lambda Q_{\mathcal{P}} + \Lambda \epsilon,i\nabla Q_{\mathcal{P}})_r &= \mathcal{O}(\lambda + \norm{\epsilon}_{L^2}),\\
        (\nabla Q_{\mathcal{P}}+\nabla \epsilon, i\nabla Q_{\mathcal{P}})_r &= \mathcal{O}(\lambda + \norm{\epsilon}_{L^2}),\\
        (-i(Q_{\mathcal{P}}+\epsilon),i\nabla Q_{\mathcal{P}})_r &= \mathcal{O}(\lambda + \norm{\epsilon}_{L^2}).
    \end{align*}
    Therefore, we obtain
    \begin{equation}\label{modulation equation for v}
        \begin{aligned}
            |\nu_s + b\nu| \lesssim \lambda|\overrightarrow{\text{Mod}}(t)| + \lambda\norm{\epsilon}_{L^2} + (\lambda^{\omega}+\alpha^*)\lambda^{\frac{5}{2}+\omega}.
        \end{aligned}
    \end{equation}

   \noindent \textbf{Step 5. } Law for $\nut{x}$: From the orthogonal condition \eqref{orthogonality for partial_v Q}, we have
    \begin{align*}
        \frac{d}{ds}(\epsilon,i\partial_\nu Q_{\mathcal{P}})_r = (\partial_s \epsilon,i\partial_\nu Q_{\mathcal{P}})_r + (\epsilon,\partial_si\partial_\nu Q_{\mathcal{P}})_r = 0.
    \end{align*}
    From the construction of $Q_{\mathcal{P}}$ in \eqref{modified profile of Q}, we have
    \begin{equation*}
        |(\epsilon,\partial_s i\partial_\nu Q_{\mathcal{P}})_r| \lesssim \norm{\epsilon}_{L^2}(|\overrightarrow{\text{Mod}}(t)| + \lambda).
    \end{equation*}
    On the other hand, from \eqref{equation for epsilon}, we have
    \begin{align*}
        (\partial_s \epsilon, i\partial_\nu Q_{\mathcal{P}})_r &= (M(\epsilon),i\partial_\nu Q_{\mathcal{P}})_r + (\overrightarrow{\text{Mod}}\cdot \overrightarrow{V}(\epsilon),i\partial_\nu Q_{\mathcal{P}})_r + (i\mathcal E, i\partial_\nu Q_{\mathcal{P}})_r. 
    \end{align*}
    From \eqref{partial_v Q}, we have
    \begin{equation*}
        (M(\epsilon),i\partial_\nu Q_{\mathcal{P}})_r = \mathcal{O}(\lambda \norm{\epsilon}_{L^2}).
    \end{equation*}
    Since, from \eqref{decay of R(pqr)}, we have $\partial_\nu Q_{\mathcal{P}}(y) \sim \langle y \rangle^{-2}$ as $|y| \rightarrow \infty$. Hence, Lemma \eqref{interaction error order} yields
    \begin{equation*}
        |(i\mathcal{E}, i\partial_\nu Q_{\mathcal{P}})_r| \lesssim (\lambda^{\omega}+\alpha^*)\lambda^{\frac{5}{2}+\omega}.
    \end{equation*}
    Finally we estimate $(\overrightarrow{\text{Mod}}\cdot \overrightarrow{V}(\epsilon),i\partial_\nu Q_{\mathcal{P}})_r$. From \eqref{orders of modulation parameters} and the evenness and oddness property of $R_{p,q,r}$ in Proposition \ref{Singular profile} and the previous computation, we have
    \begin{align*}
        (-\partial_b Q_{\mathcal{P}},i\partial_\nu Q_{\mathcal{P}})_r 
        &= \mathcal{O}(b), \\
        (-\partial_\nu Q_{\mathcal{P}}, i\partial_\nu Q_{\mathcal{P}})_r &= 0,\\
        (\Lambda Q_{\mathcal{P}} + \Lambda \epsilon,i\partial_\nu Q_{\mathcal{P}})_r &= \mathcal{O}(\lambda + \norm{\epsilon}_{L^2}),\\
        (\nabla Q_{\mathcal{P}}+\nabla \epsilon, i\partial_\nu Q_{\mathcal{P}})_r &= (R_{0,1,0},i^{-1}L_Q[iR_{0,1,0}])_r + \mathcal{O}(\lambda + \norm{\epsilon}_{L^2}),\\
        (-i(Q_{\mathcal{P}}+\epsilon),i\partial_\nu Q_{\mathcal{P}})_r &= \mathcal{O}(\lambda + \norm{\epsilon}_{L^2}).
    \end{align*}
    Therefore, from \eqref{modulation equation for b}, we obtain
    \begin{equation}\label{modulation equation for x}
        \begin{aligned}
            \left|\frac{\nut{x}_s}{\lambda} -\nu -c_2b^2\nu\right| &\lesssim \lambda^{\frac{1}{2}}(\text{(LHS) of } \eqref{modulation equation for b}) + \lambda|\overrightarrow{\text{Mod}}(t)| + \lambda \norm{\epsilon}_{L^2} + (\lambda^{\omega}+\alpha^*)\lambda^{\frac{5}{2}+\omega}\\
            & \lesssim \lambda^{\frac{1}{2}}|(\epsilon,Q_{\mathcal{P}})_r| + \lambda|\overrightarrow{\text{Mod}}(t)| +(\lambda^{\omega}+\alpha^*)\lambda^{\frac{5}{2}+\omega}.
        \end{aligned}
    \end{equation}
    
    From \eqref{modulation equation for b}, \eqref{modulation equation for lambda}, \eqref{modulation equation for tilde gamma}, \eqref{modulation equation for v}, and \eqref{modulation equation for x}, we have
    \begin{align}\label{Mod priori estimate}
        |\overrightarrow{\text{Mod}}(t)| \lesssim |(\epsilon, Q_{\mathcal{P}})_r| + (\lambda^{\omega}+\alpha^*)\lambda^{\frac{5}{2}+\omega}.
    \end{align}

    Hence, we investigate the smallness of $(\epsilon, Q_{\mathcal{P}})_r$. We differentiate this inner product with respect to $s$:
    \begin{equation*}
        \frac{d}{ds}(\epsilon, Q_{\mathcal{P}})_r = (\partial_s \epsilon, Q_{\mathcal{P}})_r + (\epsilon, \partial_s Q_{\mathcal{P}})_r.
    \end{equation*}
    Since, from \eqref{modified profile of Q}, we have
    \begin{equation*}
        |(\epsilon, \partial_s Q_{\mathcal{P}})_r| \lesssim \norm{\epsilon}_{L^2}(\lambda + |\overrightarrow{\text{Mod}}(t)|).
    \end{equation*}
    From \eqref{equation for epsilon}, we obtain
    \begin{equation*}
        (\partial_s\epsilon, Q_{\mathcal{P}})_r = (M(\epsilon),Q_{\mathcal{P}})_r + (\overrightarrow{\text{Mod}}(t) \cdot \overrightarrow{V}(\epsilon),Q_{\mathcal{P}})_r + (i\mathcal{E}, Q_{\mathcal{P}})_r.
    \end{equation*}
    We investigate $(M(\epsilon),Q_{\mathcal{P}})_r$. From \eqref{diff for mod eq}, we have
    \begin{align*}
        &(i^{-1}L_{Q_{\mathcal{P}}}[\epsilon] + (\nu+c_2b^2\nu)\nabla \epsilon - b\Lambda \epsilon, Q_{\mathcal{P}})_r\\
        &= (L_Q[\epsilon] -2ibQR_{1,0,0}\overline{\epsilon},i\{Q+ibR_{1,0,0}\})_r - b(\Lambda \epsilon, Q_{\mathcal{P}})_r + \mathcal{O}(\lambda \norm{\epsilon}_{L^2})\\
        &= (L_Q[\epsilon],iQ)_r - b\{(L_Q[\epsilon],R_{1,0,0})_r +2(iQR_{1,0,0}\overline{\epsilon},iQ)_r +(\Lambda \epsilon, Q_{\mathcal{P}})_r\} + \mathcal{O}(\lambda \norm{\epsilon}_{L^2}).
    \end{align*}
    Since $iQ$ is a kernel of the linearized operator $L_Q$ and since the identity $i^{-1}L_Q[iR_{1,0,0}] = \Lambda Q$ yields
    \begin{align*}
        &(L_Q[\epsilon],R_{1,0,0})_r +2(iQR_{1,0,0}\overline{\epsilon},iQ)_r - (\epsilon,\Lambda Q_{\mathcal{P}})_r \\
        &= (\epsilon, i^{-1}L_Q[iR_{1,0,0}])_r - (\epsilon, \Lambda Q_{\mathcal{P}})_r \\
        &= \mathcal{O}(\lambda \norm{\epsilon}_{L^2}).
    \end{align*}
    Thus, we conclude
    \begin{equation*}
        (M(\epsilon),Q_{\mathcal{P}})_r = \mathcal{O}(\lambda \norm{\epsilon}_{L^2}).
    \end{equation*}
    We now estimate $(\overrightarrow{\text{Mod}}(t) \cdot \overrightarrow{V}(\epsilon),Q_{\mathcal{P}})_r$. From \eqref{orders of modulation parameters} and the evenness and oddness property of $R_{p,q,r}$ in Proposition \ref{Singular profile}, we have
    \begin{align*}
        (\partial_b Q_{\mathcal{P}}, Q_{\mathcal{P}})_r &= (iR_{1,0,0}-2bR_{2,0,0},Q+ibR_{1,0,0})_r + \mathcal{O}(\lambda)\\
        &= b(\norm{R_{1,0,0}}_{L^2}^2 - 2(Q,R_{2,0,0})_r) + \mathcal{O}(\lambda) =\mathcal{O}(\lambda), \\
        (\partial_\nu Q_{\mathcal{P}},Q_{\mathcal{P}})_r &= (iR_{0,1,0}-bR_{1,1,0},Q+ibR_{1,0,0})_r + \mathcal{O}(\lambda) = \mathcal{O}(\lambda),\\
        (\Lambda Q_{\mathcal{P}} + \Lambda \epsilon, Q_{\mathcal{P}})_r &= \mathcal{O}(\norm{\epsilon}_{L^2}),\\
        (\nabla Q_{\mathcal{P}} + \nabla \epsilon, Q_{\mathcal{P}})_r &= \mathcal{O}(\norm{\epsilon}_{L^2}),\\
        (-i(Q_{\mathcal{P}}+\epsilon),Q_{\mathcal{P}})_r &=\mathcal{O}(\norm{\epsilon}_{L^2}).
    \end{align*}
    Since, $Q_{\mathcal{P}}(y) \sim \langle y \rangle^{-2}$ as $|y| \rightarrow \infty$, Lemma \ref{interaction error order} yields
    \begin{equation*}
        |(i\mathcal{E},Q_{\mathcal{P}})_r| \lesssim (\lambda^{\omega}+\alpha^*)\lambda^{\frac{5}{2}+\omega}.
    \end{equation*}
    Therefore, from \eqref{Mod priori estimate} we obtain 
    \begin{equation}
        \begin{aligned}
            \left|\frac{d}{ds}(\epsilon, Q_{\mathcal{P}})_r\right| &\lesssim \lambda|\overrightarrow{\text{Mod}}(t)| + \lambda\norm{\epsilon}_{L^2} + (\lambda^{\omega}+\alpha^*)\lambda^{\frac{5}{2}+\omega} \\
            & \lesssim \lambda|(\epsilon, Q_{\mathcal{P}})_r| +(\lambda^{\omega}+\alpha^*)\lambda^{\frac{5}{2}+\omega}.
        \end{aligned}
    \end{equation}
    By integrating from $s(t)$ to $s(0)$ with respect to $s$, we obtain
    \begin{equation}\label{inner product of ep and Q}
        |(\epsilon,Q_{\mathcal{P}})_r| \lesssim (\lambda^{\omega} + \alpha^*)\lambda^{2+\omega}.
    \end{equation}
    Thus, \eqref{Mod priori estimate} yields \eqref{Modulation equation}. 
    Furthermore, from \eqref{modulation equation for lambda}, \eqref{modulation equation for v}, and \eqref{modulation equation for x} we obtain improved bounds
    \eqref{improved bounds}.
\end{proof}

\begin{remark}
In the construction of minimal mass blow-up solutions to \eqref{half-wave} in \cite{KLR2013ARMAhalfwave}, the smallness condition arises from the threshold condition. More precisely, $|(\epsilon, Q_{\mathcal{P}})_r|=\mathcal{O}(\lambda^2)$. However, a trivial bound without such a threshold condition $|(\epsilon, Q_{\mathcal{P}})_r| \lesssim \|\epsilon\|_{L^2} = \mathcal{O}(\lambda^{\frac{3}{2}+2\omega})=\mathcal{O}(\lambda^{\frac{7}{4}}) $ is not sufficient to close the bootstrapping. Here, we exploit the structural identity
\[
    \|R_{1,0,0}\|_{L^2}^2 = 2(Q, R_{2,0,0})_r
\]
from \eqref{relation of R(200) and R(100)}. The profiles $R_{1,0,0}$ and $R_{2,0,0}$ constructed in Proposition~\ref{Singular profile} correspond to the second and third terms in the expansion of $e^{i b \frac{|y|^2}{4}} Q_{\text{NLS}}$ from \eqref{NLS}, namely $i b \frac{|y|^2}{4} Q_{\text{NLS}}$ and $(i b)^2 \frac{|y|^4}{16} Q_{\text{NLS}}$, which arise from the Taylor expansion of the pseudo-conformal phase. This identity leads to a cancellation in the computation of $(\partial_b Q_{\mathcal{P}}, Q_{\mathcal{P}})_r$, resulting in the improved bound
\[
    (\partial_b Q_{\mathcal{P}}, Q_{\mathcal{P}})_r = \mathcal{O}(\lambda).
\]
As a consequence, we obtain an additional smallness in $(\epsilon, Q_{\mathcal{P}})_r = \mathcal{O}(\lambda^{2+\omega})$ as in \eqref{inner product of ep and Q}.
\end{remark}

As a consequence of the Proposition~\ref{Modulation estimates}, we show that $b(t)$ is nonnegative.

\begin{lemma}
    Under the bootstrap assumption \eqref{bootstrap bounds}, possibly after reducing $t_1$, we have, for any $\eta \in (0,\eta^*)$,
    \begin{equation}\label{sign of b-parameter}
        b(t) \ge 0, \quad \text{for all } t \in \bigl(T_{\mathrm{dec}}^{(\eta)},0\bigr] 
        \cap [t_1,0].
    \end{equation}
\end{lemma}

\begin{proof}
    From modulation estimate \eqref{Modulation equation} and \eqref{rough bound for lambda} we have
    \begin{equation*}
        \left|b_s + \frac{1}{2}b^2+\eta + c_1b^4 + c_3b^2\eta + c_4\nu^2 \right| \lesssim \lambda^{2+\omega}.
    \end{equation*}
    In addition, from \eqref{orders of modulation parameters}, \eqref{bootstrap bounds}, asymptotics for \eqref{asymtotic for parameters for t} and relation $\frac{ds}{dt} =\frac{1}{\lambda}$ we obtain
    \begin{align*}
        b_t = -\frac{b^2+2\eta}{2\lambda} + \mathcal{O}(\lambda) &= -\frac{b^2_{\eta}+2\eta}{2\lambda_{\eta}} + \mathcal{O}\left(\lambda + \left|\frac{b^2+2\eta}{2\lambda} - \frac{b^2_{\eta}+2\eta}{2\lambda_{\eta}}\right|\right)\\
        & = -\left(\frac{1}{2C_0^2} + o_{\eta \to 0+}(1)\right) + \mathcal{O}(\lambda).
    \end{align*}
    Hence, we conclude that $b(t) \geq 0$, for $t \in \bigl(T_{\mathrm{dec}}^{(\eta)},0\bigr] \cap [t_1,0]$. In particular, $b(t) \sim -t$. 
\end{proof}

We now turn to the expansion of conserved quantities in terms of the modulation parameters.

\begin{proposition}\label{energy and momentum expansion for singular profile}
    We have the following expansion of Energy and Momentum:
    \begin{equation}\label{energy expansion}
        \begin{aligned}
            E(Q_{\mathcal{P}}+z^{\flat} + \epsilon) &= E(Q_{\mathcal{P}})+E(z^{\flat})+\mathcal{O}(\lambda^{\frac{5}{4}})\\
            &= \frac{1}{2}(i^{-1}L_Q[iR_{1,0,0}],R_{1,0,0})_r\cdot \left(b^2+2\eta\right) + \lambda E(z_f^*) + \mathcal{O}(\lambda^{\frac{5}{4}}),
        \end{aligned}
    \end{equation}
    and
    \begin{equation}\label{momentum expansion}
        \begin{aligned}
        P(Q_{\mathcal{P}} + z^{\flat}+\epsilon) &= P(Q_{\mathcal P})+P(z^{\flat})+\mathcal{O}(\lambda^{\frac{5}{4}})\\
        &= 2(i^{-1}L_Q[iR_{0,1,0}],R_{0,1,0})_r\cdot\nu + \lambda P(z_f^*) + \mathcal{O}(\lambda^{\frac{5}{4}}).
        \end{aligned}
    \end{equation}
\end{proposition}
\begin{proof}
    \textbf{Step 1.} Energy computation: we compute
    \begin{align*}
        E(Q_{\mathcal{P}}+z^{\flat} + \epsilon) &=
        \frac{1}{2}\int D(Q_{\mathcal{P}} + z^{\flat} + \epsilon)\cdot (\overline{Q_{\mathcal{P}} + z^{\flat} + \epsilon}) dy - \frac{1}{4}\int |Q_{\mathcal{P}} + z^{\flat} + \epsilon|^4 dy \\
        &= E(Q_\mathcal{P}) + E(z^\flat) + E(\epsilon) + \mathcal{R}_{E},
    \end{align*}
    where $\mathcal{R}_{E}$ is an remainder term containing interaction term of $Q_{\mathcal{P}},\ z^{\flat},\text{ and } \epsilon$ which is defined by
    \begin{align*}
        \mathcal{R}_E &\coloneqq \frac{1}{2}\int D(Q_{\mathcal{P}} + z^{\flat})\cdot \overline{\epsilon} + D\epsilon \cdot (\overline{Q_{\mathcal{P}} + z^\flat}) dy \\
        & \quad + \frac{1}{2}\int DQ_{\mathcal{P}}\cdot \overline{z^{\flat}} + Dz^{\flat}\cdot \overline{Q_{\mathcal{P}}} dy\\
        & \quad - \frac{1}{4}\int \{|Q_{\mathcal{P}} + z^{\flat} + \epsilon|^4 - |Q_{\mathcal{P}}|^4 - |z^{\flat}|^4 - |\epsilon|^4 \} dy.
    \end{align*}
    From bootstrap assumption on $\epsilon$, \eqref{bootstrap bounds}, we have
    \begin{equation}\label{Energy computation 1}
        \left|\int D(Q_{\mathcal{P}} + z^{\flat})\cdot \overline{\epsilon} + D\epsilon \cdot (\overline{Q_{\mathcal{P}} + z^\flat}) dy \right| \lesssim 
        \norm{\epsilon}_{H^{1/2}}\norm{Q_{\mathcal{P}}+z^{\flat}}_{H^{1/2}} \lesssim \lambda^{3/2+2\omega}.
    \end{equation}
    Also, from \eqref{interaction degeneracy 2}, and $Q_{\mathcal{P}} \sim \langle y \rangle^{-2}$ from \eqref{decay of R(pqr)}, we have
    \begin{equation}\label{Energy computation 2}
        \left|\int DQ_{\mathcal{P}}\cdot \overline{z^{\flat}} + Dz^{\flat}\cdot \overline{Q_{\mathcal{P}}} dy\right| \lesssim \int \langle y \rangle^{-2}|z^{\flat}| \lesssim \alpha^*\lambda^{\frac{5}{4}}.
    \end{equation}
    Finally, the remainder term $\tilde{\mathcal{R}}_{E}$ after subtracting \eqref{Energy computation 1} and \eqref{Energy computation 2} from $\mathcal{R}_E$ is schematically of the form 
    \begin{align*}
        \tilde{\mathcal{R}}_E \sim &\int Q_{\mathcal{P}}^3\cdot (\epsilon+z^\flat) + Q_{\mathcal{P}}^2 \cdot (\epsilon^2 + \epsilon \cdot z^{\flat} + (z^{\flat})^2) \\ 
        & \quad + Q_{\mathcal{P}}\cdot(\epsilon^3 + \epsilon^2\cdot z^\flat + \epsilon\cdot (z^\flat)^2 + (z^{\flat})^3) \\
        & \quad + (z^\flat)^3\cdot \epsilon + (z^\flat)^2\cdot \epsilon^2 + z^\flat\cdot \epsilon^3 dy.
    \end{align*}
    Using H\"older inequality and Sobolev embedding, we have
    \begin{equation}
        |\tilde{\mathcal{R}}_E| \lesssim \norm{\epsilon}_{H^{1/2}} + \norm{z^{\flat}Q_{\mathcal{P}}}_{L^{\infty}} \lesssim \lambda^{\frac{5}{4}}.
    \end{equation}
    On the other hand, since $|E(\epsilon)| \lesssim \norm{\epsilon}_{H^{1/2}}^2 \lesssim \lambda^{3+4\omega}$ and from the energy conservation of the flow $z(t)$ constructed in Proposition~\ref{construction of asymptotic profile} and scaling property, we have $E(z^\flat) = \lambda E(z^*_f)$. Hence, it remains to compute $E(Q_{\mathcal{P}})$. From Pohozaev identity and \eqref{modified profile eror}, we have
    \begin{align*}
        E(Q_{\mathcal{P}}) &= (DQ_{\mathcal{P}}+Q_{\mathcal{P}}-|Q_{\mathcal{P}}|^2Q_{\mathcal{P}},\Lambda Q_{\mathcal{P}})_r \\
        &= b(i\Lambda Q_{\mathcal{P}},\Lambda Q_{\mathcal{P}})_r - (\nu + c_2\nu b^2)\cdot (i\nabla Q_{\mathcal{P}},\Lambda Q_{\mathcal{P}})_r \\
        & \quad -\left(\frac{1}{2}b^2 + \eta + c_1b^4 + c_3b^2\eta + c_4\nu^2\right)\cdot(i\partial_b Q_{\mathcal{P}},\Lambda Q_{\mathcal{P}})_r \\
        & \quad - b\nu(i\partial_\nu Q_{\mathcal{P}},\Lambda Q_{\mathcal{P}})_r - (i\Psi_\mathcal P, \Lambda Q_\mathcal P)_r\\
        &= \frac{1}{2}(i^{-1}L_Q[iR_{1,0,0}],R_{1,0,0})_r \cdot (b^2+2\eta) + \mathcal{O}(\lambda^2).
    \end{align*}
    In the last equality, we use the estimate \eqref{innerproduct with Lambda Q}. This proves the energy expansion \eqref{energy expansion}.\\

    \textbf{Step 2.} Momentum computation: We have
    \begin{align*}
        P(Q_{\mathcal{P}}+z^\flat + \epsilon) &= \int -i\nabla(Q_{\mathcal{P}}+z^\flat + \epsilon)\cdot(\overline{Q_{\mathcal{P}}+z^\flat + \epsilon}) \\
        &= P(Q_{\mathcal{P}}) + P(z^\flat) + P(\epsilon) + \mathcal{O}\left(\int |\nabla Q_{\mathcal{P}}\cdot z^\flat| + |\nabla \{Q_{\mathcal{P}} + z^\flat\} \cdot \epsilon|\right).
    \end{align*}
    From the momentum conservation of the flow $z(t)$ and scaling property, we have $P(z^\flat) = \lambda P(z_f^*)$ and from \eqref{bootstrap bounds}, we have $|P(\epsilon)| \leq \norm{\epsilon}_{H^{1/2}}^2 \leq K\lambda^{3+4\omega}$.
    In addition, since $|\nabla Q_{\mathcal{P}}(y)| \lesssim \langle y \rangle^{-2}$ for all $y \in \bbR$ from \eqref{decay of R(pqr)}, \eqref{interaction degeneracy 2}, and bootstrap bound for $\epsilon$ in \eqref{bootstrap bounds}, we have
    \begin{align*}
        \int |\nabla Q_{\mathcal{P}}\cdot z^\flat| + |\nabla \{Q_{\mathcal{P}} + z^\flat\} \cdot \epsilon| &\lesssim \int \langle y \rangle^{-2} |z^\flat| dy + \norm{\epsilon}_{H^{1/2}}\\
        &\lesssim \alpha^*\lambda^{\frac{5}{4}} + \lambda^{3/2+2\omega}.
    \end{align*}
    Hence, it remains to calculate $P(Q_{\mathcal{P}})$. Indeed, from the expansion of $Q_{\mathcal{P}}$ in Proposition~\ref{Singular profile}, even-odd or real-imaginary product, we have 
    \begin{align*}
        P(Q_{\mathcal{P}}) &= \int -i\nabla(Q+ibR_{1,0,0}+i\nu R_{0,1,0} + \eta R_{0,0,1})\\
        & \quad \times (\overline{Q+ibR_{1,0,0}+i\nu R_{0,1,0} + \eta R_{0,0,1}}) + \mathcal{O}(\lambda^2)\\
        & = 2(i^{-1}L_Q[iR_{0,1,0}],R_{0,1,0})_r \cdot \nu + \mathcal{O}(\lambda^2).
    \end{align*}
    This proves the expansion \eqref{momentum expansion}.
\end{proof}

\subsection{Energy bounds}\label{closing the bootstrap arguments}
In this subsection, we prove the $H^{1/2}$ estimate. The localized Lyapunov functional $\mathcal J_A$ defined in \eqref{energy functional J_A} satisfies, after fixing $A\gg1$ and shrinking the common time interval,
\begin{equation}\label{summary coercivity and monotonicity of J_A}
\begin{aligned}
    \mathcal J_A(t)
    &\ge
    \frac{\kappa_1}{4\lambda(t)}
    \|\epsilon(t)\|_{H^{1/2}}^2
    -C(\alpha^*+\lambda^\omega(t))
    \lambda^{3+2\omega}(t),\\
    \frac{d}{dt}\mathcal J_A(t)
    &\ge
    -C\left[
        \alpha^*\lambda^2(t)
        +\lambda^{2+4\omega}(t)
        \log^{1/2}\!\left(2+\lambda^{-1-2\omega}(t)\right)
    \right].
\end{aligned}
\end{equation}
Since $\epsilon(0)=0$, we have $\mathcal J_A(0)=0$. Integrating the second estimate backward from $0$ and using the first one yields a strict improvement of the weak $H^{1/2}$ bootstrap bound.

Considering the coercivity of the bilinear form $(L_Q[\epsilon],\epsilon)_r$ in Lemma~\ref{Coercivity 2}, the natural candidate for $J$ is a sum of quadratic terms of linearized energy on $W^{\flat}=Q_\mathcal{P} +z^{\flat}$ with respect to $\epsilon$ and $\norm{\epsilon}_{L^2}^2$:
\begin{align*}
    J(u^{\flat})(s) &= E(W^{\flat} + \epsilon) - E(W^{\flat}) - (\nabla E)(W^{\flat})\cdot \epsilon + \frac{1}{2}\norm{\epsilon}_{L^2}^2,\\
    &= \frac{1}{2}(L_{W^{\flat}}[\epsilon],\epsilon)_r + h.o.t.
\end{align*}
Using $ds/dt = 1/\lambda$, we have
$J(u)(t) = \lambda^{-1}\cdot J(u^{\flat})(s)$. Also, from the evolution of $\epsilon^{\sharp}$ in \eqref{equation for epsilon sharp}, we have
\begin{equation}\label{energy bound: flow of epsilon}
    i\partial_s \epsilon - L_{W^{\flat}}[\epsilon] + ib\Lambda \epsilon \approx 0.
\end{equation}
Note that $|\nu| \lesssim \lambda$ and $|b| \lesssim \lambda^{1/2}$. We correct $J$ by adding a localized virial correction to control the non-perturbative term $ib\Lambda\epsilon$ in the evolution of $\epsilon$. To introduce the localized virial term, we define an even cutoff function $\phi:\mathbb{R} \rightarrow \mathbb{R}$ so that $\phi''(y)\geq 0$ and 
\begin{equation}
        \phi '(y) \coloneqq \begin{cases}
        y \quad \text{for} \quad 0 \leq y\leq 1,&\\
        3-e^{-y} \quad \text{for} \quad y\geq 2.
    \end{cases}
\end{equation}
Also, we denote
\begin{equation*}
    F(u)= \frac{1}{4}|u|^4, \quad f(u) = u|u|^2, \quad \text{and} \quad F'(u)\cdot h = \text{Re}(f(u)\cdot \overline{h}).
\end{equation*}
Note that $F'(u)\cdot h$ is a linear part of $F(u+h)-F(u)$ with respect to $h$ and $F'=f$. 
Finally, we define the energy functional $\mathcal{J}_{A}(u)$ with respect to $t$-variable:
\begin{equation}\label{energy functional J_A}
    \begin{aligned}
        \mathcal{J}_A(u) \coloneqq& \frac{1}{2}\int |D^{\frac{1}{2}}\epsilon^{\sharp}|^2 + \frac{1}{2\lambda}\int |\epsilon^{\sharp}|^2 - \int \left[F(W+\epsilon^{\sharp})-F(W)-F'(W)\cdot \epsilon^{\sharp}\right]\\
        & +\frac{b}{2}\text{Im}\left(\int A\phi'\left(\frac{x-\nut{x}}{A\lambda}\right)\nabla \epsilon^{\sharp}\overline{\epsilon^{\sharp}}\right),
    \end{aligned}
\end{equation}
for positive constant $A>0$, which is chosen later (see the proof of Lemma~\ref{coercivity of C(u) and P(u)}).
We now recall the \textit{a priori} bound and equation for $W$ and $\epsilon$:
From \eqref{equation for W}, we have
\begin{equation}\label{energy estimate : epsilon flow}
    i\partial_tW - DW + |W|^2W = \frac{1}{\lambda}\Psi^{\sharp},
\end{equation}
and from \eqref{a priori bound for W}, we have
\begin{equation*}
    \norm{W}_{L^2} \lesssim 1,\quad \norm{W}_{\dot{H}^{\frac{1}{2}}}\lesssim \lambda^{-\frac{1}{2}},\quad \norm{W}_{\dot{H}^1} \lesssim \lambda^{-1}.
\end{equation*}
for $W = Q_{\mathcal{P}}^{\sharp} + z$. Also, from the bootstrap assumption on $\epsilon$ \eqref{bootstrap bounds}, we have
\begin{equation}
    \frac{1}{\lambda}\norm{\epsilon}_{H^{1/2}}^2 \sim \norm{\epsilon^{\sharp}}_{\dot{H}^{1/2}}^2 +\frac{1}{\lambda}\norm{\epsilon^{\sharp}}_{L^2}^2 \leq \lambda^{2+4\omega},\quad  \norm{\epsilon^{\sharp}}_{\dot{H}^{1/2+\delta}} \leq \lambda^{1-2\delta}.
\end{equation}

Due to the nonlocality of $D$, the control of the localized virial term is not trivial. We follow the techniques used in \cite[Section 6]{KLR2013ARMAhalfwave}. For the readers' convenience, we write some preliminary facts and lemmas used in \cite[Section 6]{KLR2013ARMAhalfwave}, for closing the bootstrap argument. From the fractional calculus for $x>0$, $0<\beta<1$, we have
\begin{equation*}
    x^\beta = \frac{\sin(\pi\beta)}{\pi} \int_0^{\infty} \sigma^{\beta-1} \frac{x}{x+\sigma} \, d\sigma.
\end{equation*}
Using this formula and the spectral theorem applied to the self-adjoint operator $p^2$, we obtain the commutator formula,
\begin{equation*}
    [|p|^\alpha, B] = \frac{\sin(\pi\alpha/2)}{\pi} \int_0^{+\infty} \sigma^{\alpha/2} \frac{1}{p^2 + \sigma} [p^2, B] \frac{1}{p^2 + \sigma} \, d\sigma, \quad \text{for } 0<\alpha<2
\end{equation*}
for $0<\alpha<2$, and self-adjoint operator $B$ whose domain is $\mathcal{S}(\mathbb{R})$. In particular, for any smooth function $\tilde{\phi}$, we have
\begin{equation}\label{commutator D-1}
    \left[ |p|, \nabla \widetilde{\phi} \cdot p + p \cdot \nabla \widetilde{\phi} \right] = \frac{1}{\pi} \int_0^{\infty} \sqrt{\sigma} \frac{1}{p^2 + \sigma} [p^2, \nabla \widetilde{\phi} \cdot p + p \cdot \nabla \widetilde{\phi}] \frac{1}{p^2 + \sigma} \, d\sigma.
\end{equation}
Moreover, we have
\begin{equation}\label{commutator D-2}
    [p^2, \nabla \widetilde{\phi} \cdot p + p \cdot \nabla \widetilde{\phi}] = -4i\, p \cdot (\Delta \widetilde{\phi}) p + i\, \Delta^2 \widetilde{\phi}.
\end{equation}
Now, we define an auxiliary function $g_{\epsilon^{\sharp},\sigma}$ so that 
\begin{equation*}
    g_{\epsilon^{\sharp},\sigma}(t,x) \coloneqq \sqrt{\frac{2}{\pi}} \frac{1}{-\Delta + \sigma} \epsilon^{\sharp}(t,x),
\end{equation*}
for $\sigma>0$. In other words, $g_{\epsilon^{\sharp},\sigma}$ is a solution to the elliptic equation 
\begin{equation}\label{commutator D-3}
    -\Delta g_{\epsilon^{\sharp},\sigma} + \sigma g_{\epsilon^{\sharp},\sigma} = \sqrt{\frac{2}{\pi}} \epsilon^{\sharp}.
\end{equation}
Note that the integral kernel for the resolvent $(-\Delta +\sigma)^{-1}$ in one dimension is explicitly given by $\frac{1}{2\sqrt{\sigma}}e^{-\sqrt{\sigma}|x-\tilde{x}|}$. Thus, we have a convolution formula
\begin{equation}\label{commutator D_4}
    g_{\epsilon^{\sharp},\sigma}(t,x) = \frac{1}{\sqrt{2\pi \sigma}} \int_{\mathbb{R}} e^{-\sqrt{\sigma} |x-\tilde{x}|} \epsilon^{\sharp}(t,\tilde{x}) \, d\tilde{x}.
\end{equation}
Now, we define $\nabla \tilde{\phi}(t,x) = bA\phi'(\frac{x-\nut{x}}{A\lambda})$. Then, since $(-\Delta + \sigma)^{-1}$ is self-adjoint, from Fubini's theorem, we obtain
\begin{equation}\label{fubini theroem}
\begin{aligned}
    -&\frac{1}{4} \Re\left( \int \overline{\epsilon^{\sharp}} \left[ -i|p|, \nabla \widetilde{\phi} \cdot p + p \cdot \nabla \widetilde{\phi} \right] \epsilon^{\sharp} \right)
    \\
    =& \frac{b}{2\lambda} \int_0^{+\infty} \sqrt{\sigma} \int_{\mathbb{R}} \Delta\phi\left( \frac{x-\nut{x}}{A\lambda} \right) |\nabla g_{\epsilon^{\sharp},\sigma}|^2 \, dx\, d\sigma \\
    &- \frac{1}{8} \frac{b}{A^2\lambda^3} \int_0^{+\infty} \sqrt{\sigma} \int_{\mathbb{R}} \Delta^2\phi\left( \frac{x-\nut{x}}{A\lambda} \right) |g_{\epsilon^{\sharp},\sigma}|^2 \, dx\, d\sigma,
\end{aligned}
\end{equation}
for $p = -i\nabla$. Now, we are going to introduce a lemma which is for the estimate of the localized virial correction, which we shall use later.
\begin{lemma}(\cite[Appendix F]{KLR2013ARMAhalfwave}) \label{estimate for localized virial correction}
    Let $\phi : \mathbb{R} \rightarrow \mathbb{R}$ be a smooth function such that
    $\Delta \phi$ and $\nabla \phi$ are in $L^{\infty}$. Also, let $v \in H^{1/2}$ Then, we have
    \begin{equation*}
        \left|\int_{\mathbb{R}} \overline{v}(x) \nabla \phi(x) \cdot \nabla v(x) \ dx\right| \lesssim \norm{\nabla \phi}_{L^{\infty}}\norm{v}_{\dot{H}^{1/2}}^2 + \norm{\Delta \phi}_{L^{\infty}}\norm{v}_{L^2}^2.
    \end{equation*}
\end{lemma}
The proof of the Lemma~\ref{estimate for localized virial correction} is using the Fourier transform and a standard interpolation theorem, and can be found in \cite[Appendix F]{KLR2013ARMAhalfwave}.

Finally, we note that the estimate for the nonlinearity from $|u|^2u$, we need to control $L^{\infty}$-norm. We handle this quantity by higher order norm bootstrap, from Lemma~\ref{log loss of the L-infty estiamte}.

We start by differentiating the functional $\mathcal{J}_A$ with respect to $t$.

\begin{lemma}(Differentiation of energy functional $\mathcal{J}_A$)\label{lemma : differentiation of Energy functional}
    Let $\mathcal{J}_A$ be a function defined in \eqref{energy functional J_A}, and $u = W + \epsilon^{\sharp}$ be a decomposition of the solution to \eqref{half-wave} under the weak bootstrap bounds \eqref{bootstrap bounds}. Then, for $t \in (T_{\text{dec}}^{(\eta)},0]\cap[t_1,0]$, we have
     \begin{equation*}
        \frac{d\mathcal{J}_A}{dt} = \mathcal{C}(u) + \mathcal{P}(u) + \mathcal{O}\left(\frac{1}{\lambda}\norm{\Psi^{\sharp}}_{L^2}^2 + \frac{1}{\lambda}\norm{\epsilon}_{H^{1/2}}^2 + \log^{\frac{1}{2}}(2+\norm{\epsilon^{\sharp}}_{H^{1/2}}^{-1})\norm{\epsilon^{\sharp}}_{H^{1/2}}^2 \right),
    \end{equation*}
    where $\mathcal{C}(u)$ and $\mathcal{P}(u)$ are
    \begin{equation}\label{definition of C(u)}
        \begin{aligned}
            \mathcal{C}(u) \coloneqq &-\frac{1}{\lambda}\textnormal{Im}\left\{\int W^2\overline{\epsilon^{\sharp}}^2\right\} - \textnormal{Re}\left\{\int \partial_tW(\overline{2|\epsilon^{\sharp}|^2W+(\epsilon^{\sharp})^2\overline{W}})\right\}\\
            &+\frac{b}{2\lambda}\int \frac{|\epsilon^{\sharp}|^2}{\lambda} + \frac{b}{2\lambda}\int_0^{\infty} \sqrt{\sigma}\int_{\mathbb{R}} \phi''\left(\frac{x-\nut{x}}{A\lambda}\right)|\nabla g_{\epsilon^{\sharp},\sigma}|^2 \,dxd\sigma\\
            & -\frac{b}{8A^2\lambda^3}\int_0^{\infty}\sqrt{\sigma}\int_{\mathbb{R}}\phi''''\left(\frac{x-\nut{x}}{A\lambda}\right)|g_{\epsilon^{\sharp},\sigma}|^2 \,dxd\sigma \\
            & + b\textnormal{Re}\left\{\int A\phi'\left(\frac{x-\nut{x}}{A\lambda}\right)(2|\epsilon^{\sharp}|^2W+(\epsilon^{\sharp})^2\overline{W})\cdot \overline{\nabla W}\right\},
        \end{aligned}
    \end{equation}
    and
    \begin{equation}\label{definition of P(u)}
        \begin{aligned}
            \mathcal{P}(u) \coloneqq \frac{1}{\lambda}\textnormal{Im}&\left\{\int \left[ -D\Psi^{\sharp} -\frac{\Psi^{\sharp}}{\lambda} +2|W|^2\Psi^{\sharp}-W^2\overline{\Psi^{\sharp}}  \right.\right.\\
            & \left.\left. +ibA\phi'\left(\frac{x-\nut{x}}{A\lambda}\right)\cdot \nabla\Psi^{\sharp} + i\frac{b}{2\lambda}\phi''\left(\frac{x-\nut{x}}{A\lambda}\right)\Psi^{\sharp} \right]\overline{\epsilon^{\sharp}}\right\},
        \end{aligned}
    \end{equation}
    where $g_{\epsilon^{\sharp},\sigma}$ is an auxiliary function which depends on $\epsilon^{\sharp}$ and $\sigma>0$ so that
    \begin{equation*}
        g_{\epsilon^{\sharp},\sigma} = \sqrt{\frac{2}{\pi}}\frac{1}{-\Delta+\sigma}\epsilon^{\sharp},
    \end{equation*}
\end{lemma}

\begin{proof}
     We calculate the differentiation of functional $\mathcal{J}_A$ by dividing the energy part and the localized virial part.
     
    \textbf{Step 1.} We differentiate the energy part of the functional $\mathcal{J}_A$.
    \begin{align*}
        \frac{d}{dt}&\left\{\frac{1}{2}\int |D^{\frac{1}{2}}\epsilon^{\sharp}|^2 + \frac{1}{2\lambda}\int |\epsilon^{\sharp}|^2 - \int \left[F(W+\epsilon^{\sharp})-F(W)-F'(W)\cdot \epsilon^{\sharp}\right]\right\}\\
        =& \text{Re}\left\{\int \partial_t \epsilon^{\sharp}\left(\overline{D\epsilon^{\sharp}+\frac{\epsilon^{\sharp}}{\lambda}-(f(W+\epsilon^{\sharp})-f(W))}\right)\right\} - \frac{\lambda_t}{2\lambda^2}\int |\epsilon^{\sharp}|^2\\
        & -\text{Re}\left\{\int \partial_tW\left(\overline{f(u)-f(W)-f'(W)\cdot \epsilon^{\sharp}}\right)\right\}\\
        =& - \frac{1}{\lambda}\textnormal{Im}\left\{\int W^2\overline{\epsilon^{\sharp}}^2\right\} -\text{Re}\left\{\int \partial_tW\left(\overline{\overline{W}(\epsilon^{\sharp})^2+2W|\epsilon^{\sharp}|^2}\right)\right\} \\
        &-\text{Im}\left\{\int \frac{\Psi^{\sharp}}{\lambda}\left(\overline{D\epsilon^{\sharp}+\frac{\epsilon^{\sharp}}{\lambda}-(2|W|^2\epsilon^{\sharp}+W^2\overline{\epsilon^{\sharp}})}\right)\right\}  \\
        & -\frac{\lambda_t}{2\lambda^2}\int |\epsilon^{\sharp}|^2 + \text{Im}\left\{\int \left(\frac{\Psi^{\sharp}}{\lambda}+\frac{\epsilon^{\sharp}}{\lambda}\right)\left(\overline{\overline{W}(\epsilon^{\sharp})^2 + 2W|\epsilon^{\sharp}|^2+ \epsilon^{\sharp}|\epsilon^{\sharp}|^2}\right)\right\}\\
        & - \text{Re}\left\{\int \partial_tW(\overline{\epsilon^{\sharp}|\epsilon^{\sharp}|^2})\right\}
    \end{align*}
    In the last equality, we use the equation for $\epsilon^{\sharp}$ \eqref{equation for epsilon sharp}. Also, we denote $f'(W)\cdot \epsilon^{\sharp}$ by the linear term of $f(W+\epsilon^{\sharp})-f(W)$ with respect to the $\epsilon^{\sharp}$ which is written by
    \begin{equation}\label{linear term of f(u)-f(W) with respect to epsilon}
        f'(W)\cdot\epsilon^{\sharp} \coloneqq 2|W|^2\epsilon^{\sharp} + W^2\overline{\epsilon^{\sharp}}.
    \end{equation}
    We now estimate 
    From the modulation equation \eqref{Modulation equation}, we have
    \begin{equation}\label{estimate for lambda_t term}
        -\frac{\lambda_t}{2\lambda^2}\int |\epsilon^{\sharp}|^2 = 
        \frac{b}{2\lambda^2} \int |\epsilon^{\sharp}|^2
        -\frac{1}{2\lambda^2}(\lambda_t + b)\norm{\epsilon^{\sharp}}_{L^2}^2 = \frac{b}{2\lambda^2} \int |\epsilon^{\sharp}|^2 + \mathcal{O}(\frac{\norm{\epsilon}_{H^{1/2}}^2}{\lambda}).
    \end{equation}
    Next, from \eqref{a priori bound for W}, \eqref{bootstrap bounds}, and the Gagliardo-Nirenberg inequality, we estimate
    \begin{equation}\label{estimates for epsilon^2 term}
        \begin{aligned}
            &\left|\text{Im}\left\{\int \left(\frac{\Psi^{\sharp}}{\lambda}+\frac{\epsilon^{\sharp}}{\lambda}\right)\left(\overline{\overline{W}(\epsilon^{\sharp})^2 + 2W|\epsilon^{\sharp}|^2+ \epsilon^{\sharp}|\epsilon^{\sharp}|^2}\right)\right\}\right|\\
            & \lesssim \frac{1}{\lambda} \left(\norm{\Psi^{\sharp}}_{L^2}+\norm{\epsilon^{\sharp}}_{L^2}\right)\norm{\epsilon^{\sharp}}_{L^6}^2(\norm{W}_{L^6}+\norm{\epsilon^{\sharp}}_{L^6})\\
            & \lesssim
            \frac{1}{\lambda}\left(\norm{\Psi^{\sharp}}_{L^2}+\norm{\epsilon^{\sharp}}_{L^2}\right) \norm{\epsilon^{\sharp}}_{\dot{H}^{1/2}}^{4/3}\norm{\epsilon^{\sharp}}_{L^2}^{2/3}\norm{W}_{\dot{H}^{1/2}}^{2/3}\norm{W}_{L^2}^{1/3}\\
            & \lesssim
            \frac{\norm{\Psi^{\sharp}}_{L^2}^2}{\lambda} + \frac{\norm{\epsilon}_{H^{1/2}}^2}{\lambda}.
        \end{aligned}
    \end{equation}
    We now estimate the cubic term of $\epsilon^{\sharp}$ hitting the 
    $\partial_tW$. Since $W$ satisfies \eqref{equation for W}, and from \eqref{a priori bound for W}, we have
    \begin{equation}\label{estiamte for cubic term of epsilon}
        \begin{aligned}
            \left|\int \partial_tW(\overline{\epsilon^{\sharp}|\epsilon^{\sharp}|^2})\right| & \lesssim \norm{W}_{\dot{H}^{3/4}}\norm{|\epsilon^{\sharp}|^2\epsilon^{\sharp}}_{\dot{H}^{1/4}} + \norm{W}_{L^6}^3\norm{\epsilon^{\sharp}}_{L^6}^3 + \frac{1}{\lambda}\norm{\Psi^{\sharp}}_{L^2}\norm{\epsilon^{\sharp}}_{L^6}^3\\
            &\lesssim \frac{1}{\lambda^{3/4}}\norm{\epsilon^{\sharp}}_{L^2}^{\frac{1}{2}}\norm{\epsilon^{\sharp}}_{\dot{H}^{\frac{1}{2}}}^{\frac{5}{2}}+\frac{1}{\lambda}\norm{\epsilon^{\sharp}}_{L^2}\norm{\epsilon^{\sharp}}_{\dot{H}^{1/2}}^2 + \frac{1}{\lambda}\norm{\Psi^{\sharp}}_{L^2}\norm{\epsilon^{\sharp}}_{\dot{H}^{1/2}}^2\norm{\epsilon^{\sharp}}_{L^2}\\
            &\lesssim \frac{\norm{\Psi^{\sharp}}_{L^2}^2}{\lambda} + \frac{\norm{\epsilon}_{H^{1/2}}^2}{\lambda}.
        \end{aligned}
    \end{equation}
    Thus, from \eqref{estimate for lambda_t term}, \eqref{estimates for epsilon^2 term}, \eqref{estiamte for cubic term of epsilon} and integration by parts, we yield
    \begin{equation*}
        \begin{aligned}
            \frac{d}{dt}&\left\{\frac{1}{2}\int |D^{\frac{1}{2}}\epsilon^{\sharp}|^2 + \frac{1}{2\lambda}\int |\epsilon^{\sharp}|^2 - \int \left[F(W+\epsilon^{\sharp})-F(W)-F'(W)\cdot \epsilon^{\sharp}\right]\right\}\\
            =& -\frac{1}{\lambda}\text{Im}\left\{\int W^2\overline{\epsilon^{\sharp}}^2\right\} - \text{Re}\left\{\int \partial_tW\left(\overline{2|\epsilon^{\sharp}|^2W+(\epsilon^{\sharp})^2\overline{W}}\right)\right\} + \frac{b}{2\lambda}\int \frac{|\epsilon^{\sharp}|^2}{\lambda} \\
            &+ \frac{1}{\lambda}\text{Im}\left\{\int \left[-D\Psi^{\sharp}-\frac{\Psi^{\sharp}}{\lambda}+2|W|^2\Psi^{\sharp}-W^2\overline{\Psi^{\sharp}}\right]\overline{\epsilon^{\sharp }}\right\} \\
            &+ \mathcal{O}\left(\frac{\norm{\Psi^{\sharp}}_{L^2}^2}{\lambda} + \frac{\norm{\epsilon}_{H^{1/2}}^2}{\lambda}\right).
        \end{aligned}
    \end{equation*}
    \textbf{Step 2.} We now calculate the differentiation of the localized virial part.
    
    For convenience of the notation, we let $\tilde{\phi}(t,x)$ be a function so that 
    \begin{equation*}
        \nabla \tilde{\phi}(t,x)\coloneqq bA \phi'\left(\frac{x-\nut{x}}{A\lambda}\right).
    \end{equation*}
    Then, we obtain
    \begin{equation}\label{Virial differentiation}
        \begin{aligned}
            \frac{1}{2}\frac{d}{dt}&\left(\text{Im}\left\{\int \nabla \tilde{\phi}\cdot \nabla \epsilon^{\sharp}\overline{\epsilon^{\sharp}}\right\}\right)\\
            &=\frac{1}{2}\text{Im}\left\{\int(\partial_t\nabla\tilde{\phi}) \cdot \nabla \epsilon^{\sharp}\overline{\epsilon^{\sharp}}\right\} + \frac{1}{2}\text{Im}\left\{\int \nabla \tilde{\phi}\cdot\left((\nabla \partial_t \epsilon^{\sharp})\overline{\epsilon^{\sharp}}+\nabla\epsilon^{\sharp}\overline{\partial_t\epsilon^{\sharp}}\right)\right\}.
        \end{aligned}
    \end{equation}
    Since we have
    \begin{equation}\label{estimate of time derivative of localized virial}
        \begin{aligned}
            &|\partial_t\nabla \tilde{\phi}| \lesssim |b_t|+\left|\frac{b\nut{x}_t}{\lambda}\right|+\left|b\frac{\lambda_t}{\lambda}\right|\lesssim 1,\\
            & |\nabla (\partial_t \nabla \tilde{\phi})| \lesssim \lambda^{-1}.
        \end{aligned}
    \end{equation}
    Thus, from \eqref{estimate of time derivative of localized virial} and using Lemma~\ref{estimate for localized virial correction} we have
    \begin{equation}\label{estimate of first term virial}
        \left|\int \partial_t(\nabla \tilde{\phi})\cdot \nabla \epsilon^{\sharp}\overline{\epsilon^{\sharp}}\right| \lesssim \norm{\epsilon^{\sharp}}_{\dot{H}^{1/2}}^2+\frac{1}{\lambda}\norm{\epsilon^{\sharp}}_{L^2}^2.
    \end{equation}
    We now estimate the second term of \eqref{Virial differentiation}.
    Integration by parts and \eqref{equation for epsilon sharp} yield
    \begin{align}
        \frac{1}{2}\text{Im}&\left\{\int \nabla \tilde{\phi}\cdot\left((\nabla \partial_t \epsilon^{\sharp})\overline{\epsilon^{\sharp}}+\nabla\epsilon^{\sharp}\overline{\partial_t\epsilon^{\sharp}}\right)\right\} \nonumber \\
        =& -\frac{1}{2}\text{Re}\left\{\int \nabla \tilde{\phi} \cdot (\nabla(D\epsilon^{\sharp})\overline{\epsilon^{\sharp}}+\epsilon^{\sharp}\nabla\overline{D\epsilon^{\sharp}})+\nabla(\nabla \tilde{\phi})\cdot \epsilon^{\sharp}\overline{D\epsilon^{\sharp}}\right\} \label{first term of differentiation virial part - (2)}\\
        & -\text{Re}\left\{\int \nabla\tilde{\phi}\cdot \left(f(W+\epsilon^{\sharp})-f(W)+\frac{\Psi^{\sharp}}{\lambda}\right) \overline{\nabla\epsilon^{\sharp}}\right\}
        \label{second term of differentiation virial part - (2)}\\
        & -\frac{1}{2}\text{Re}\left\{\int \Delta \tilde{\phi} \cdot \left(f(W+\epsilon^{\sharp})-f(W)+\frac{\Psi^{\sharp}}{\lambda}\right)\overline{\epsilon^{\sharp}}\right\} \label{third term of differentiation virial part - (2)}.
    \end{align}
    We divide the term $\eqref{second term of differentiation virial part - (2)} + \eqref{third term of differentiation virial part - (2)}$ by the integration of the quadratic term and the cubic term with respect to $\epsilon^{\sharp}$ and the term containing  $\Psi^{\sharp}$. We express $f(W+\epsilon^{\sharp})-f(W)$ by
    \begin{equation}
        \begin{aligned}
            f(W+\epsilon^{\sharp})-f(W) = f'(W)\cdot \epsilon^{\sharp} + \mathcal{N}(\epsilon^{\sharp}),
        \end{aligned}
    \end{equation}
    where $f'(W)\cdot \epsilon^{\sharp}$ defined in \eqref{linear term of f(u)-f(W) with respect to epsilon} and by definition $\mathcal{N}(\epsilon^{\sharp})$ is computed by
    \begin{align*}
        \mathcal{N}(\epsilon^{\sharp}) = 2|\epsilon^{\sharp}|^2W + (\epsilon^{\sharp})^2\overline{W} + |\epsilon^{\sharp}|^2\epsilon^{\sharp}. 
    \end{align*}
    Hence, we obtain
    \begin{align}
        &\eqref{second term of differentiation virial part - (2)} + \eqref{third term of differentiation virial part - (2)} \nonumber \\
        &=-b\text{Re}\left\{\int A \phi'\left(\frac{x-\nut{x}}{A\lambda}\right)(f'(W)\cdot\epsilon^{\sharp})\overline{\nabla\epsilon^{\sharp}}\right\}
        -\frac{b}{2\lambda}\text{Re}\left\{\int \phi '' \left(\frac{x-\nut{x}}{A\lambda}\right)(f'(W)\cdot\epsilon^{\sharp})\overline{\epsilon^{\sharp}} \right\} \label{linear part of the last two term of viral diff}\\
        &\quad - b \Re\left\{ \int A \phi'\left(\frac{x-\nut{x}}{A \lambda}\right)\mathcal{N}(\epsilon^{\sharp}) \cdot \overline{\nabla \epsilon^{\sharp}}\right\} - \frac{1}{2} \frac{b}{\lambda} \Re\left\{\int \phi''\left(\frac{x-\nut{x}}{A \lambda}\right)\mathcal{N}(\epsilon^{\sharp})\overline{\epsilon^{\sharp}}\right\} \label{nonlinear part of the last two term of virial diff}\\
        &\quad -b \Re\left\{\int \frac{\Psi^{\sharp}}{\lambda} A \phi'\left(\frac{x-\nut{x}}{A \lambda}\right) \cdot \overline{\nabla \epsilon^{\sharp}}\right\}-\frac{1}{2} \frac{b}{\lambda} \Re\left\{\int \frac{\Psi^{\sharp}}{\lambda} \phi''\left(\frac{x-\nut{x}}{A \lambda}\right) \overline{\epsilon^{\sharp}}\right\} \label{psi part of the last two term of virial diff}
    \end{align}
    
    \textbf{1.} We first estimate \eqref{first term of differentiation virial part - (2)}: We use the commutator formula. From direct calculation, we obtain
    \begin{equation*}
        \begin{aligned}
            \eqref{first term of differentiation virial part - (2)} = - \frac{1}{4}\text{Re}\left\{\int \overline{\epsilon^{\sharp}}\left[-iD,\nabla \tilde{\phi}\cdot (-i\nabla) - i\nabla\cdot(\nabla\tilde{\phi})\right]\epsilon^{\sharp}\right\}.
        \end{aligned}
    \end{equation*}
    Then, from \eqref{commutator D-1}--\eqref{fubini theroem} we obtain
    \begin{equation*}
        \begin{aligned}
            \eqref{first term of differentiation virial part - (2)} &= \frac{b}{2\lambda} \int_0^{+\infty} \sqrt{\sigma} \int_{\mathbb{R}} \Delta\phi\left( \frac{x-\nut{x}}{A\lambda} \right) |\nabla g_{\epsilon^{\sharp},\sigma}|^2 \, dx\, d\sigma \\
            &- \frac{1}{8} \frac{b}{A^2\lambda^3} \int_0^{+\infty} \sqrt{\sigma} \int_{\mathbb{R}} \Delta^2\phi\left( \frac{x-\nut{x}}{A\lambda} \right) |g_{\epsilon^{\sharp},\sigma}|^2 \, dx\, d\sigma
        \end{aligned}
    \end{equation*}
    
    \textbf{2.}
    We now estimate \eqref{linear part of the last two term of viral diff}--\eqref{psi part of the last two term of virial diff}: 
    Integration by parts yields 
    \begin{equation*}
        \eqref{linear part of the last two term of viral diff} = b \Re\left\{\int A  \phi '\left(\frac{x-\nut{x}}{A \lambda}\right)\left(2|\epsilon^{\sharp}|^2 W+(\epsilon^{\sharp})^2 \bar{W}\right) \cdot \overline{\nabla W}\right\}.
    \end{equation*}
    In addition, using the bounded-Lipschitz property of the localized virial weight, we have
    \[
        \|\nabla\tilde{\phi}\cdot\mathcal{N}(\epsilon^{\sharp})\|_{\dot H^{1/2}}
        \lesssim_A |b|\|\mathcal{N}(\epsilon^{\sharp})\|_{\dot H^{1/2}}
        +\frac{|b|}{\lambda^{1/2}}\|\mathcal{N}(\epsilon^{\sharp})\|_{L^2}.
    \]
    Hence, by the fractional Leibniz rule and the bootstrap assumption \eqref{bootstrap bounds},
    \begin{align*}
        |\eqref{nonlinear part of the last two term of virial diff}|
        &\lesssim_A \|\epsilon^{\sharp}\|_{\dot H^{1/2}}
        \left(|b|\|\mathcal{N}(\epsilon^{\sharp})\|_{\dot H^{1/2}}
        +\frac{|b|}{\lambda^{1/2}}\|\mathcal{N}(\epsilon^{\sharp})\|_{L^2}\right)
        +\frac{|b|}{\lambda}\|\mathcal{N}(\epsilon^{\sharp})\|_{L^{4/3}}\|\epsilon^{\sharp}\|_{L^4}\\
        &\lesssim \log^{1/2}\left(2+\|\epsilon^{\sharp}\|_{H^{1/2}}^{-1}\right)\|\epsilon^{\sharp}\|_{H^{1/2}}^2
        +\frac{1}{\lambda}\|\epsilon\|_{H^{1/2}}^2.
    \end{align*}
    Here, we use Lemma~\ref{log loss of the L-infty estiamte} and the weak $H^{1/2+\delta}$ bootstrap bound.

    Finally, a direct computation employing integration by parts gives
    \begin{equation*}
        \begin{aligned}
            \eqref{psi part of the last two term of virial diff} =\frac{1}{\lambda}\Im\left\{\int\left[i b A  \phi'\left(\frac{x-\nut{x}}{A \lambda}\right) \cdot \nabla \Psi^{\sharp}+i \frac{b}{2 \lambda} \phi '' \left(\frac{x-\nut{x}}{A \lambda}\right) \Psi^{\sharp}\right] \overline{\epsilon^{\sharp}}\right\}.
        \end{aligned}
    \end{equation*}
    This proves our lemma.
\end{proof}
We now prove the coercivity of $\frac{d\mathcal{J}_A}{dt}$.
To get the almost positivity of $\mathcal{C}(u)$, we are going to introduce two lemmas. One is for the coercivity of $L_A(\epsilon)$ which appears after renormalizing \eqref{eq for C(u) to get coercivity}. The other is for the estimate of the term containing $\phi''''$.
\begin{lemma}[\cite{KLR2013ARMAhalfwave}, Appendix B]\label{coercivity for C(u)}
    Let $A>0$, and $L_A[\epsilon]$ be a functional of $\epsilon \in H^{1/2}$ so that
    \begin{equation*}
        L_A[\epsilon] = \int_{\sigma=0}^{\infty} \sqrt{\sigma}\int_{\mathbb{R}} \phi''\left(\frac{y}{A}\right)|\nabla g_{\epsilon,\sigma}|^2 dyd\sigma + \int |\epsilon|^2 dy - \Re\left\{\int (2|\epsilon|^2Q+\epsilon^2Q)Q dy\right\},
    \end{equation*}
    where $g_{\epsilon,\sigma} = \sqrt{\frac{2}{\pi}}\frac{1}{-\Delta + \sigma}\epsilon$ for $\sigma>0$. Then there is a universal constant $A_0>0$ and $\kappa_2>0$ so that for $A\geq A_0$ and $\epsilon \in H^{1/2}$, we have
    \begin{equation*}
         L_A[\epsilon] \geq \kappa_2 \int_{\mathbb{R}}|\epsilon|^2 \, dy - \frac{1}{\kappa_2}\left\{ (\epsilon, Q)_r^2 + (\epsilon, R_{1,0,0})_r^2+(\epsilon, R_{0,1,0})_r^2+(\epsilon, iR_{0,0,1})_r^2\right\},
    \end{equation*}
    where $R_{1,0,0}$, $R_{0,1,0}$, and $R_{0,0,1}$ are defined in Proposition~\ref{Singular profile}.
\end{lemma}

\begin{lemma}[\cite{KLR2013ARMAhalfwave},Appendix B]\label{smallness of phi''''}
    For $A>0$ and $\epsilon \in H^{1/2}$ we have
    \begin{equation}\label{estimate 1}
        \left|\frac{1}{A^2}\int_0^{\infty}\sqrt{\sigma}\int_{\mathbb{R}}\Delta^2 \phi\left(\frac{y}{A}\right)|g_{\epsilon,\sigma}|^2\,dyd\sigma\right| \lesssim \frac{1}{A}\norm{\epsilon}^2_{L^2}.
    \end{equation}
\end{lemma}

\begin{lemma}[Almost positivity of $\mathcal{C}(u)$ and smallness of $\mathcal{P}(u)$] \label{coercivity of C(u) and P(u)}
    Let $\mathcal{C}(u)$ and $\mathcal{P}(u)$ be functionals defined in Lemma \ref{lemma : differentiation of Energy functional}. Also, let $A>0$ be a sufficiently large constant so that $A \gtrsim \max\{A_0,\frac{4}{\kappa_2}\}$. From the weak bootstrap assumption \eqref{bootstrap bounds}, we have
    \begin{equation}
        \mathcal{C}(u) \geq \frac{\kappa_2b}{4\lambda^2}\norm{\epsilon}_{L^2}^2 + \mathcal{O}\left(\lambda^{5/2+2\omega} + \frac{1}{\lambda}\norm{\epsilon}_{H^{1/2}}^2\right),
    \end{equation}
    and
    \begin{equation}
        |\mathcal{P}(u)| \lesssim \alpha^*\lambda^{1/2-2/n}\norm{\epsilon}_{L^2} +\lambda^{2+4\omega},
    \end{equation}
    where small constant $\kappa_2>0$ is independent of $u_\eta$ and $\eta$.
    Also, $A_0$ is a constant which is introduced in the previous Lemma~\ref{coercivity for C(u)}.
\end{lemma}

\begin{proof}
    \textbf{Step 1.} Estimate for $\mathcal{C}(u)$ \\
    We proceed with a proof by decomposing $W$ into $Q_{\mathcal{P}}^{\sharp}+z$, and we use the coercivity property for the perturbed linearized bilinear operator $L_A[\epsilon]$, Lemma~\ref{coercivity for C(u)}. 
    First, we estimate the first line in the right hand side of \eqref{definition of C(u)}:
    \begin{equation}\label{first line of C(u)}
        -\frac{1}{\lambda}\text{Im}\left\{\int W^2\overline{\epsilon^{\sharp}}^2\right\} - \text{Re}\left\{\int \partial_tW(\overline{2|\epsilon^{\sharp}|^2W+(\epsilon^{\sharp})^2\overline{W}})\right\}
    \end{equation}
    We consider the first term of \eqref{first line of C(u)}. From Lemma \ref{interaction error order}, we have $\norm{Q_{\mathcal{P}}z^{\flat}}_{L^{\infty}} \leq \alpha^*\lambda^{2+\frac{1}{2}-\frac{2}{n}}$. Since we choose $n=8$, we obtain $\norm{Q_{\mathcal{P}}^{\sharp}z}_{L^{\infty}}\lesssim \lambda$. Hence, we have
    \begin{equation*}
        \begin{aligned}
            \frac{1}{\lambda}\Im\left\{\int W^2\overline{(\epsilon^{\sharp})^2}\right\} &=\frac{1}{\lambda}\Im\left\{\int  (Q_{\mathcal{P}}^{\sharp})^2\overline{(\epsilon^{\sharp})^2}\right\} + \mathcal{O}\left(\frac{1}{\lambda}\{\norm{Q_{\mathcal{P}}^{\sharp}z}_{L^{\infty}}+\norm{z}_{L^{\infty}}^2\}\norm{\epsilon^{\sharp}}_{L^2}^2\right) \\
            &=\frac{1}{\lambda}\Im\left\{\int  (Q_{\mathcal{P}}^{\sharp})^2\overline{(\epsilon^{\sharp})^2}\right\} + \mathcal{O}\left(\frac{1}{\lambda}\norm{\epsilon}_{H^{1/2}}^2\right).
        \end{aligned}
    \end{equation*}
    We now consider the second term of \eqref{first line of C(u)}. Note that 
    \begin{equation}\label{Eq : 2}
        \begin{aligned}
            &\Re\left\{\int \partial_tW(\overline{2|\epsilon^{\sharp}|^2W+(\epsilon^{\sharp})^2\overline{W}})\right\} 
            = \Re\left\{\int \partial_tQ_{\mathcal{P}}^{\sharp}(\overline{2|\epsilon^{\sharp}|^2Q_{\mathcal{P}}^{\sharp}+(\epsilon^{\sharp})^2\overline{Q_{\mathcal{P}}^{\sharp}}})\right\} \\
            & + \Re\left\{\int \partial_tQ_{\mathcal{P}}^{\sharp}(\overline{2|\epsilon^{\sharp}|^2z+(\epsilon^{\sharp})^2\overline{z}}) + \int \partial_tz (\overline{2|\epsilon^{\sharp}|^2W+(\epsilon^{\sharp})^2\overline{W}})\right\}.
        \end{aligned}
    \end{equation}
    We first consider the interaction term of $Q_{\mathcal{P}}^{\sharp}$ and $z$, which is the second line in (RHS) of \eqref{Eq : 2}. Differentiation of $Q_{\mathcal{P}}^{\sharp}$ with respect to $t$ is
    \begin{equation}\label{diff of Q_P of t}
        \begin{aligned}
            \partial_tQ_{\mathcal{P}}^{\sharp} &= e^{i\gamma}\frac{1}{\lambda^{1/2}}\left[-\frac{\lambda_t}{\lambda}\Lambda Q_{\mathcal{P}} + i\gamma_tQ_{\mathcal{P}}+b_t\frac{\partial Q_{\mathcal{P}}}{\partial b}+\nu_t\frac{\partial Q_{\mathcal{P}}}{\partial\nu}-\frac{\nut{x}_t}{\lambda}\nabla Q_{\mathcal{P}}\right]\left(\frac{x-\nut{x}}{\lambda}\right) \\
            & = \left(\frac{i}{\lambda}+\frac{b}{2\lambda}\right)Q_{\mathcal{P}}^{\sharp} +\frac{b}{\lambda}\left[y\nabla Q_{\mathcal{P}}\right]^{\sharp} + \mathcal{O}\left(\left[\frac{|\text{Mod}(t)|+\lambda}{\lambda^{3/2}}\right]\langle y \rangle^{-2}\right),
        \end{aligned}
    \end{equation}
    where we denote $y = \frac{x-\nut{x}(t)}{\lambda(t)}$. 
    Here, we use $|\nu| \lesssim \lambda$ and $b^2+\eta \lesssim \lambda$ in \eqref{orders of modulation parameters}.  Using the modulation estimates \eqref{Modulation equation}, we obtain
    \begin{equation}
        \begin{aligned}
            &\left|\int \partial_tQ_{\mathcal{P}}^{\sharp}(\overline{2|\epsilon^{\sharp}|^2z+(\epsilon^{\sharp})^2\overline{z}})\right| + \left|\int \partial_tz (\overline{2|\epsilon^{\sharp}|^2W+(\epsilon^{\sharp})^2\overline{W}})\right| \\  
            &\lesssim
            \frac{1}{\lambda}\left(\norm{Q_{\mathcal{P}}^{\sharp}z}_{L^{\infty}}+\norm{[y\nabla Q_{\mathcal{P}}]^{\sharp}z}_{L^{\infty}}\right)\norm{\epsilon^{\sharp}}_{L^2}^2 + \norm{\partial_tz}_{L^4}\norm{W}_{L^4}\norm{\epsilon^{\sharp}}_{L^4}^2\\
            &\lesssim \frac{1}{\lambda}\norm{\epsilon^{\sharp}}_{L^2}^2 + \frac{1}{\lambda^{1/4}}\norm{\epsilon^{\sharp}}_{L^2}\norm{\epsilon^{\sharp}}_{H^{1/2}}\\
            &\lesssim \frac{1}{\lambda}\norm{\epsilon}_{H^{1/2}}^2.
        \end{aligned}
    \end{equation}
    The last inequality follows from the fact that $\norm{W}_{L^4}\lesssim \norm{W}_{H^{1/2}}^{1/2}\norm{W}_{L^2}^{1/2}\lesssim \lambda^{-1/4}$. We now investigate the first line of (RHS) of \eqref{Eq : 2}. From \eqref{diff of Q_P of t}, we obtain
    \begin{equation}\label{eq:4}
        \begin{aligned}
            -&\Re\left\{\int \partial_tQ_{\mathcal{P}}^{\sharp}(\overline{2|\epsilon^{\sharp}|^2Q_{\mathcal{P}}^{\sharp}+(\epsilon^{\sharp})^2\overline{Q_{\mathcal{P}}^{\sharp}}})\right\} \\
            =& \frac{1}{\lambda}\Im\left\{\int Q_{\mathcal{P}}^{\sharp}(\overline{2|\epsilon^{\sharp}|^2Q_{\mathcal{P}}^{\sharp}+(\epsilon^{\sharp})^2\overline{Q_{\mathcal{P}}^{\sharp}}})\right\}-\frac{b}{2\lambda}\Re\left\{\int (2|\epsilon^{\sharp}|^2Q_{\mathcal{P}}^{\sharp}+(\epsilon^{\sharp})^2\overline{Q_{\mathcal{P}}^{\sharp}})\overline{Q_{\mathcal{P}}^{\sharp}}\right\}\\
            &-b\Re\left\{\int \left(\frac{x-\nut{x}}{\lambda}\right)(2|\epsilon^{\sharp}|^2Q_{\mathcal{P}}^{\sharp}+(\epsilon^{\sharp})^2\overline{Q_{\mathcal{P}}^{\sharp}})\cdot \nabla Q_{\mathcal{P}}^{\sharp}\right\} + \mathcal{O}\left(\frac{1}{\lambda}\norm{\epsilon}_{H^{1/2}}^2\right).
        \end{aligned}
    \end{equation}
    Second, we estimate the last line of (RHS) in \eqref{definition of C(u)}:
    \begin{equation}\label{last line of C(u)}
        b\text{Re}\left\{\int A\phi'\left(\frac{x-\nut{x}}{A\lambda}\right)(2|\epsilon^{\sharp}|^2W+(\epsilon^{\sharp})^2\overline{W})\cdot \overline{\nabla W}\right\}
    \end{equation}
    Since $W = Q_{\mathcal{P}}^{\sharp}+z$, we have
    \begin{align*}
        \eqref{last line of C(u)} &= b\Re\left\{\int A\phi'\left(\frac{x-\nut{x}}{A\lambda}\right)(2|\epsilon^{\sharp}|^2Q_{\mathcal{P}}^{\sharp}+(\epsilon^{\sharp})^2\overline{Q_{\mathcal{P}}^{\sharp}})\cdot\overline{\nabla Q_{\mathcal{P}}^{\sharp}}\right\} \\
        & \quad + \calR,
    \end{align*}
    where $\mathcal{R}$ denotes the part consisting of interaction terms 
    between $z$ and $Q_{\mathcal{P}}$, as well as terms involving only $z$.
    \begin{align*}
         \calR= b\Re\left\{\int A\phi'\left(\frac{x-\nut{x}}{A\lambda}\right)\left\{(2|\epsilon^{\sharp}|^2W+(\epsilon^{\sharp})^2\overline{W})\cdot\overline{\nabla z}+(2|\epsilon^{\sharp}|^2z+(\epsilon^{\sharp})^2\overline{z})\cdot\overline{\nabla Q_{\mathcal{P}}^{\sharp}}\right\}\right\}
    \end{align*}
    Since the function $\phi'$ is a bounded function on $\mathbb{R}$, we obtain the $\calR$-estimate:
    \begin{equation}\label{eq:2}
        \begin{aligned}
            |\calR| \lesssim bA\norm{W}_{L^{\infty}}\norm{\nabla z}_{L^{\infty}}\norm{\epsilon^{\sharp}}_{L^2}^2 + bA \norm{z\nabla Q_{\mathcal{P}}^{\sharp}}_{L^{\infty}}\norm{\epsilon^{\sharp}}_{L^2}^2 \lesssim \frac{1}{\lambda}\norm{\epsilon}_{H^{1/2}}^2.
        \end{aligned}
    \end{equation}
    Here, we use $b^2+\eta \lesssim \lambda$ from \eqref{orders of modulation parameters} 
    and the construction of the asymptotic profile $z$ in Proposition~\ref{construction of asymptotic profile}. 
    In particular, we rely on the 
    uniform-in-time bound established there (see \eqref{control of norm z(t) for geq 1/2}):
    \[
        \|\nabla z(t)\|_{L^{\infty}} \lesssim \alpha^*, \quad 
        t \in (T_{\mathrm{dec}}^{(\eta)},0] \cap [t_1,0].
    \]
    Thus, gathering the estimate of \eqref{first line of C(u)} and \eqref{last line of C(u)}, we obtain
    \begin{equation}\label{eq for C(u) to get coercivity}
        \begin{aligned}
            \mathcal{C}(u) =& -\frac{b}{2\lambda}\Re\left\{\int (2|\epsilon^{\sharp}|^2Q_{\mathcal{P}}^{\sharp}+(\epsilon^{\sharp})^2\overline{Q_{\mathcal{P}}^{\sharp}})\overline{Q_{\mathcal{P}}^{\sharp}}\right\} + \frac{b}{2\lambda^2}\int |\epsilon^{\sharp}|^2\\
            &+ b\Re\left\{\int \left(A \phi'\left(\frac{x-\nut{x}}{A\lambda}\right)-\left(\frac{x-\nut{x}}{\lambda}\right)\right)(2|\epsilon^{\sharp}|^2Q_{\mathcal{P}}^{\sharp}+(\epsilon^{\sharp})^2\overline{Q_{\mathcal{P}}^{\sharp}})\cdot \overline{\nabla Q_{\mathcal{P}}^{\sharp}}\right\}\\
            & + \frac{b}{2\lambda} \int_0^{\infty} \sqrt{\sigma} \int_{\mathbb{R}} \phi''\left(\frac{x-\nut{x}}{A\lambda}\right) |\nabla g_{\epsilon^{\sharp},\sigma}|^2 \,dx d\sigma 
            \\
            &-\frac{b}{8A^2\lambda^3}\int_0^{\infty} \int_{\mathbb{R}}\sqrt{\sigma} \phi''''\left(\frac{x-\nut{x}}{A\lambda}\right)|g_{\epsilon^{\sharp},\sigma}|^2 \,dxd\sigma\\
            &+ \mathcal{O}\left(\frac{1}{\lambda}\norm{\epsilon}_{H^{1/2}}^2\right).
        \end{aligned}
    \end{equation}
    Now by changing of variable $y=\frac{x-\nut{x}}{\lambda}$ and applying the identity $|Q_{\mathcal{P}}(y) - Q(y)| \lesssim \lambda^{1/2}\langle y \rangle^{-2}$, which is from Proposition~\ref{Singular profile} and $b^2+\eta + |\nu| \lesssim \lambda^{1/2}$ from \eqref{orders of modulation parameters}, we obtain
    \begin{align*}
        \mathcal{C}(u) &= \frac{b}{2\lambda^2}\left\{ L_A[\epsilon] + \mathcal{O}(\lambda^{1/2}\norm{\epsilon}_{L^2}^2)
        -\frac{1}{4A^2}\int_0^{\infty} \sqrt{\sigma}\int_{\mathbb{R}} \phi''''\left(\frac{y}{A}\right)|g_{\epsilon,\sigma}|^2 dyd\sigma \right.\\
        &+\left.2\Re\left\{\int\left(A\phi'\left(\frac{y}{A}\right)-y\right)(2|\epsilon|^2Q_{\mathcal{P}}+\epsilon^2\overline{Q_{\mathcal{P}}})\cdot \overline{\nabla Q_{\mathcal{P}}}\right\}\right\}+ \mathcal{O}\left(\frac{1}{\lambda}\norm{\epsilon}_{H^{1/2}}^2\right).
    \end{align*}
    Note that $A\nabla\phi(x/A)-x =0 $ for $|x|\leq A$, $\norm{\nabla Q_{\mathcal{P}}}_{L^{\infty}} \lesssim 1$, and $Q_{\mathcal{P}}(y) \lesssim \langle y \rangle^{-2}$. Therefore, we have
    \begin{equation}\label{estimate 2}
        \begin{aligned}
            &\left|\int\left(A\nabla\phi\left(\frac{y}{A}\right)-y\right)(2|\epsilon|^2Q_{\mathcal{P}}+\epsilon^2\overline{Q_{\mathcal{P}}})\cdot \overline{\nabla Q_{\mathcal{P}}}\right|\\
            &\lesssim \norm{(A+|y|)Q_{\mathcal{P}}}_{L^{\infty}(x\geq A)}\norm{\nabla Q_{\mathcal{P}}}_{L^{\infty}}\norm{\epsilon}^2_{L^2} \lesssim \frac{1}{A}\norm{\epsilon}_{L^2}^2.
        \end{aligned}
    \end{equation}
    Combining the estimate \eqref{estimate 1} in Lemma~\ref{smallness of phi''''} 
    with \eqref{estimate 2}, and applying Lemma \ref{coercivity for C(u)} for 
    sufficiently large $A>0$ such that 
    $A\gtrsim \max\{A_0,\frac{4}{\kappa_2}\}$, together with the bound 
    $b \lesssim \lambda^{1/2}$, we obtain
    \begin{align*}
        \mathcal{C}(u) =&  \frac{b}{2\lambda^2}\left\{ L_A[\epsilon] + \mathcal{O}\left(\frac{1}{A}\norm{\epsilon}_{L^2}^2\right)\right\}
        + \mathcal{O}\left(\frac{1}{\lambda}\norm{\epsilon}_{H^{1/2}}^2\right)\\
        \geq& \frac{b}{2\lambda^2} \left[\frac{\kappa_2}{2}\int |\epsilon|^2 dy - \frac{1}{\kappa_2} \left\{ (\epsilon, Q)_r^2 + (\epsilon, R_{1,0,0})_r^2+(\epsilon, R_{0,1,0})_r^2+(\epsilon, iR_{0,0,1})_r^2\right\}\right] 
        \\
        &+ \mathcal{O}\left(\frac{1}{\lambda}\norm{\epsilon}_{H^{1/2}}^2\right)\\
        \geq& \frac{\kappa_2b}{4\lambda^2}\norm{\epsilon}_{L^2}^2 + \mathcal{O}\left(\lambda^{5/2+2\omega} + \frac{1}{\lambda}\norm{\epsilon}_{H^{1/2}}^2\right).
    \end{align*}
    The last inequality follows from the degeneracy of unstable direction $(\epsilon,Q_{\mathcal{P}})_r = \mathcal{O}(\lambda^{2+\omega})$ in \eqref{unstable direction degeneracy} and orthogonality conditions on $\epsilon$, \eqref{orthogonality for partial_b Q}, \eqref{orthogonality for partial_eta Q}, and \eqref{orthogonality for partial_v Q}.

    \textbf{Step 2.} We now consider the $\mathcal{P}(u)$. Recall that $\mathcal{P}(u)$ is defined by
    \begin{align*}
        \mathcal{P}(u) = \frac{1}{\lambda}\Im&\left\{\int \left[-D\Psi^{\sharp}-\frac{1}{\lambda}\Psi^{\sharp} +2|W|^2\Psi^{\sharp}-W^2\overline{\Psi^{\sharp}} \right.\right.\\
        &+ \left. \left. ibA \phi '\left(\frac{x-\nut{x}}{A\lambda}\right)\cdot \nabla \Psi^{\sharp} + i\frac{b}{2\lambda}\phi''\left(\frac{x-\nut{x}}{A\lambda}\right)\Psi^{\sharp}\right]\overline{\epsilon^{\sharp}} \right\}.
    \end{align*}
    First, we deal with the interaction between $Q_{\mathcal{P}}^{\sharp}$ and $z$, 
    as well as the terms involving only $z$, in
    \begin{equation}\label{profile error part involving W}
        \frac{1}{\lambda}\Im\left\{\int(2|W|^2\Psi^{\sharp}-W^2\overline{\Psi^{\sharp}})
        \overline{\epsilon^{\sharp}}\right\}.
    \end{equation}
    Afterwards, we handle the term involving only $Q_{\mathcal{P}}^{\sharp}$.

    From the definition of $\Psi$ in \eqref{Q_P+z profile error} with estimate \eqref{L2 norm of IN and profile error of Q_P} and modulation estimates \eqref{Modulation equation}, we have
    \begin{equation}\label{L2 norm of Psi}
       \norm{\Psi}_{L^2} \lesssim \norm{\text{Mod}}_{L^2} + \norm{\Psi_{\mathcal{P}}}_{L^2} + \norm{\text{IN}}_{L^2} \lesssim (\alpha^*+\lambda^{\omega})\lambda^{2+\omega}.
    \end{equation}
    Hence, we estimate
    \begin{align*}
        &\frac{1}{\lambda}\int (|z|^2+|Q_{\mathcal{P}}^{\sharp}z|)|\Psi^{\sharp}||\epsilon^{\sharp}|
        \lesssim
        \alpha^{*}\frac{1}{\lambda}\norm{\Psi}_{L^2}\norm{\epsilon^{\sharp}}_{L^2} \lesssim (\alpha^*+\lambda^\omega)\lambda^{1+\omega}\norm{\epsilon^{\sharp}}_{L^2}.
    \end{align*}
    Now, we consider the term involving $Q_{\mathcal{P}}$ in \eqref{profile error part involving W} and the rest part of $\mathcal{P}(u)$ after subtracting \eqref{profile error estimate}. To estimate this, we decompose $\Psi^{\sharp}$ by $\Psi_1^{\sharp}$ and $\Psi_2^{\sharp}$ where they are defined by
    \begin{equation}\label{definition of phi1 and phi2}
        \begin{aligned}
            \frac{1}{\lambda}\Psi_1^{\sharp} &\coloneqq \frac{1}{\lambda^{3/2}}\{\text{Mod}(t)\}\left(\frac{x-\nut{x}}{\lambda}\right)e^{i\gamma},\\
            \frac{1}{\lambda}\Psi_2^{\sharp} &\coloneqq \frac{1}{\lambda}\Psi^{\sharp}-\frac{1}{\lambda}\Psi_1^{\sharp} = \frac{1}{\lambda^{3/2}}\{\text{IN} + i \Psi_{\mathcal{P}}\}\left(\frac{x-\nut{x}}{\lambda}\right)e^{i\gamma}.
        \end{aligned}
    \end{equation}
    \textbf{1.} We first estimate the terms involving $\Psi_1^{\sharp}$: Since $b \lesssim \lambda^{1/2}$, we have
        \begin{align*}
         &\left|\frac{1}{\lambda}\Im\left\{\int \left[-D\Psi^{\sharp}_1-\frac{1}{\lambda}\Psi^{\sharp}_1 +2|Q_{\mathcal{P}}^{\sharp}|^2\Psi_1^{\sharp}
        -(Q_{\mathcal{P}}^{\sharp})^2\overline{\Psi_1^{\sharp}}   \right.\right.\right.
        \\
        &\quad + ibA\phi'\left(\frac{x-\nut{x}}{A\lambda}\right)\cdot \nabla \Psi_1^{\sharp} +\left.\left.\left.i\frac{b}{2\lambda}\phi''\left(\frac{x-\nut{x}}{A\lambda}\right)\Psi_1^{\sharp}\right]\overline{\epsilon^{\sharp}} \right\}\right| \\
        \lesssim& \frac{|\overrightarrow{\text{Mod}}(t)|}{\lambda^2}\left[|(i\partial_bQ_{\mathcal{P}},i^{-1}L_Q[\epsilon])_r|+|(i\partial_\nu Q_{\mathcal{P}},i^{-1}L_Q[\epsilon])_r|+|(Q_{\mathcal{P}},i^{-1}L_Q[\epsilon])_r|+\mathcal{O}(\lambda^{1/2}\norm{\epsilon}_{L^2})\right]\\
        & + \frac{1}{\lambda^2}\left|\frac{\lambda_s}{\lambda}+b\right||(i^{-1}L_Q[\epsilon],i\Lambda Q_{\mathcal{P}})_r+\mathcal{O}(\lambda^{1/2}\norm{\epsilon}_{L^2})| \\
        & + \frac{1}{\lambda^2}\left|\frac{\nut{x}_s}{\lambda}-\nu-c_2b^2\nu\right||(i^{-1}L_Q[\epsilon],i\nabla Q_{\mathcal{P}})_r + \mathcal{O}(\lambda^{1/2}\norm{\epsilon}_{L^2})| \\
        & + \lambda^{-1/2}\norm{\nabla \Psi_1^{\sharp}}_{L^2}\norm{\epsilon}_{L^2}+\lambda^{-3/2}\norm{\Psi_1^{\sharp}}_{L^2}\norm{\epsilon}_{L^2}.
    \end{align*}
    Using $i^{-1}L_Q[iR_{1,0,0}]=\Lambda Q$, $i^{-1}L_Q[iR_{0,1,0}]=-\nabla Q$, $i^{-1}L_Q[iQ]=0$, and $|Q_{\mathcal{P}}(y) - Q(y)| \lesssim \lambda^{1/2}\langle y \rangle^{-2}$, from the Proposition \ref{Singular profile} with orthogonality conditions, \eqref{orthogonality for Lambda Q}, \eqref{orthogonality for nabla Q}, and the kernel property of $\text{ker}(L_Q) = \{iQ,\nabla Q\}$, we have 
    \begin{equation*}
        |(i\partial_bQ_{\mathcal{P}},i^{-1}L_Q[\epsilon])_r|+|(i\partial_\nu Q_{\mathcal{P}},i^{-1}L_Q[\epsilon])_r|+|(Q_{\mathcal{P}},i^{-1}L_Q[\epsilon])_r| = \mathcal{O}(\lambda^{1/2}\norm{\epsilon}_{L^2}).
    \end{equation*}
    Also, from improved bounds \eqref{improved bounds} and \eqref{unstable direction degeneracy} we have
    \begin{equation*}
         \frac{1}{\lambda^2}\left|\frac{\lambda_s}{\lambda}+b\right||(i^{-1}L_Q[\epsilon],i\Lambda Q_{\mathcal{P}})_r+\mathcal{O}(\lambda^{1/2}\norm{\epsilon}_{L^2})| \lesssim \lambda^{5/2+2\omega} + \lambda^{1+\omega}\norm{\epsilon}_{L^2}.
    \end{equation*}
    Finally noting that $i^{-1}L_Q[\nabla Q] =0$ and using a change of variable, we estimate
    \begin{align*}
        \frac{1}{\lambda^{1/2}}\norm{\nabla\Psi_1^{\sharp}}_{L^2}\norm{\epsilon}_{L^2} + \frac{1}{\lambda^{3/2}}\norm{\Psi_1^{\sharp}}_{L^2} \norm{\epsilon}_{L^2}  &\lesssim \frac{1}{\lambda^{3/2}}|\overrightarrow{\text{Mod}}(t)|\norm{\epsilon}_{L^2}\\
        &\lesssim (\alpha^*+\lambda^{\omega})\lambda^{1/2+\omega}\norm{\epsilon}_{L^2}.
    \end{align*}
    \textbf{2.} We estimate the terms involving $\Psi_2^\sharp$: From the definition of \eqref{definition of phi1 and phi2}, we have
    \begin{equation}\label{L2 norm of nabla psi}
        \norm{D\Psi_2^{\sharp}}_{L^2} \sim \norm{\nabla \Psi_2^{\sharp}}_{L^2} \lesssim \frac{1}{\lambda}\left(\norm{\nabla \text{IN}}_{L^2} + \norm{\nabla \Psi_{\mathcal{P}}}_{L^2}\right).
    \end{equation}
    Note that schematically we have
    \begin{equation*}
        \nabla \text{IN} \sim \nabla (Q_{\mathcal{P}})^2z^{\flat} + (Q_{\mathcal{P}})^2\nabla z^{\flat} + \nabla Q_{\mathcal{P}}(z^{\flat})^2 + Q_{\mathcal{P}}z^{\flat}\nabla z^{\flat}.
    \end{equation*}
    Using \eqref{interaction definition}, the degeneracy estimate for $\nabla z$ in \eqref{degeneracy for z}, and \eqref{decay of R(pqr)}, together with the same inner--outer decomposition as in the proof of Lemma~\ref{interaction error order}, we have
    \begin{equation}\label{differentiation of interaction estimate}
        \norm{\nabla \text{IN}}_{L^2} \lesssim \alpha^*\lambda^{2+\frac{1}{2}-\frac{2}{n}}, \quad \norm{\nabla \Psi_{\mathcal{P}}}_{L^2} \lesssim \lambda^3
    \end{equation}
    Here, the second estimate follows directly from \eqref{profile error estimate}.
    Therefore, using \eqref{differentiation of interaction estimate} and the fact that 
    $b \lesssim \lambda^{1/2}$, we estimate
    \begin{align*}
        &\left|\frac{1}{\lambda}\Im\left\{\int \left[-D\Psi_2^{\sharp}-\frac{1}{\lambda}\Psi_2^{\sharp} +2|Q_{\mathcal{P}}^{\sharp}|^2\Psi_2^{\sharp}
        -(Q_{\mathcal{P}}^{\sharp})^2\overline{\Psi_2^{\sharp}} + ibA\phi'\left(\frac{x-\nut{x}}{A\lambda}\right)\cdot \nabla \Psi_2^{\sharp} \right.\right.\right. \\
        &\left.\left.\left.+i\frac{b}{2\lambda}\phi''\left(\frac{x-\nut{x}}{A\lambda}\right)\Psi_2^{\sharp}\right]\overline{\epsilon^{\sharp}} \right\}\right|\\
        &\lesssim \frac{1}{\lambda}\norm{\nabla \Psi_2^{\sharp}}_{L^2}\norm{\epsilon}_{L^2} + \frac{1}{\lambda^2}\norm{\Psi_2^{\sharp}}_{L^2}\norm{\epsilon}_{L^2}+\frac{1}{\lambda^{1/2}}\norm{\nabla\Psi_2^{\sharp}}_{L^2}\norm{\epsilon}_{L^2}+\frac{1}{\lambda^{3/2}}\norm{\Psi_2^{\sharp}}_{L^2}\norm{\epsilon}_{L^2}\\
        & \lesssim \alpha^*\lambda^{1/2-2/n}\norm{\epsilon}_{L^2}.
    \end{align*}
    Henceforth, we conclude
    \begin{equation*}
        \mathcal{P}(u) = \mathcal{O}\left(\alpha^*\lambda^{1/2-2/n}\norm{\epsilon}_{L^2}+\lambda^{2+4\omega}\right).
    \end{equation*}
\end{proof}

In the rest of this section, we finish the proof of our main bootstrap lemma.
\begin{proof}[Finish of the proof of Lemma~\ref{Main bootstrap}] $\,$ \\
    \textbf{Step 1.} We first claim the coercivity of the functional $\mathcal{J}_A$:
    \begin{equation}\label{coercivity of J_A}
        \mathcal{J}_A(u) \geq \frac{\kappa_1}{4\lambda}\norm{\epsilon}_{H^{1/2}}^2 + \mathcal{O}\left((\alpha^*+\lambda^{\omega})\lambda^{3+2\omega} + \frac{1}{\lambda^{1/2}}\norm{\epsilon}_{H^{1/2}}^2\right).
    \end{equation}
    From the Lemma~\ref{estimate for localized virial correction}, since $\norm{\nabla\tilde{\phi}}_{L^{\infty}} = \mathcal{O}(1)$ and $\norm{\Delta \tilde{\phi}}_{L^{\infty}} = \mathcal{O}(\lambda^{-1})$, we have
    \begin{equation*}
        \left|b\Im\left\{\int A\phi'\left(\frac{x-\nut{x}}{A\lambda}\right)\cdot\nabla\epsilon^{\sharp}\overline{\epsilon^{\sharp}}\right\}\right| \lesssim b\left(\norm{D^{\frac{1}{2}}\epsilon^{\sharp}}_{L^2}^2+\frac{1}{\lambda}\norm{\epsilon^{\sharp}}_{L^2}^2\right) \lesssim \frac{1}{\lambda^{1/2}}\norm{\epsilon}_{H^{1/2}}^2.
    \end{equation*}
    By noting that $\norm{Q_{\mathcal{P}}z^{\flat}}_{L^{\infty}} = \mathcal{O}(\lambda^{2+\frac{1}{8}})$ from the Lemma~\ref{interaction error order}, we have
    \begin{equation}
        \begin{aligned}
            \mathcal{J}_A(u) &= \frac{1}{2\lambda}(L_{W^{\flat}}[\epsilon],\epsilon)_r + \mathcal{O}\left(\frac{1}{\lambda^{1/2}}\norm{\epsilon}_{H^{1/2}}^2\right)\\
            &= \frac{1}{2\lambda}\left\{(L_Q[\epsilon],\epsilon)_r + \mathcal{O}\left(\left((\alpha^*)^2+\lambda^{1/2}\right)\norm{\epsilon}_{H^{1/2}}^2\right)\right\}.
        \end{aligned}
    \end{equation}
    Now we use the coercivity of the bilinear $(L_Q[\epsilon],\epsilon)_r$ in the Lemma~\ref{Coercivity 2}. Then, from the degeneracy of unstable direction $(\epsilon,Q_{\mathcal{P}})_r = \mathcal{O}(\lambda^{2+\omega})$ in \eqref{unstable direction degeneracy} and orthogonality conditions on $\epsilon$, \eqref{orthogonality for partial_b Q}, \eqref{orthogonality for partial_eta Q}, and \eqref{orthogonality for partial_v Q}, we have
    \begin{equation}\label{smallness of degeneracy direction}
        \begin{aligned}
            (\epsilon, Q)_r^2 + (\epsilon, R_{1,0,0})_r^2&+(\epsilon, R_{0,1,0})_r^2+(\epsilon, iR_{0,0,1})_r^2 \\
            &= \mathcal{O}\left((\alpha^*+\lambda^{\omega})\lambda^{4+2\omega}+\lambda\norm{\epsilon}_{L^2}^2\right),
        \end{aligned}
    \end{equation}
    since we have $\alpha^*<1$ and $|\lambda(t)| \leq 1$. By possibly replacing $\alpha^*>0$ with a smaller one so that it satisfies $\alpha^* \lesssim \frac{\kappa_1}{2}$. Henceforth, we obtain \eqref{coercivity of J_A}.

    We now close the bootstrap for the $\norm{\epsilon}_{H^{1/2}}$.
    From the Lemma \ref{lemma : differentiation of Energy functional} and the Lemma \ref{coercivity of C(u) and P(u)}, we have
    \begin{align*}
        \frac{d}{dt}\mathcal{J}_A(u) \geq \frac{\kappa_2b}{4\lambda^2}\norm{\epsilon}_{L^2}^2 &+ \mathcal{O}\left(\alpha^*\lambda^{1/2-2/n}\norm{\epsilon}_{L^2}+ \lambda^{\frac52+2\omega}  \right.\\
        &\left.+ \frac{1}{\lambda}\norm{\epsilon}_{H^{1/2}}^2+ \frac{1}{\lambda}\norm{\Psi^{\sharp}}_{L^2}^2 + \log^{\frac{1}{2}}(2+\norm{\epsilon^{\sharp}}_{H^{1/2}}^{-1})\norm{\epsilon^{\sharp}}_{H^{1/2}}^2 \right).
    \end{align*}
    Note that $b\geq 0$ from Lemma~\ref{sign of b-parameter}, and 
    $\norm{\Psi^{\sharp}}_{L^2} \lesssim (\alpha^*+\lambda^\omega)\lambda^{2+\omega}$ from \eqref{L2 norm of Psi}. In addition, bootstrap assumption on $\epsilon$ in \eqref{bootstrap bounds} with \eqref{orders of modulation parameters} gives almost positivity of $\frac{d}{dt}\mathcal{J}_A(u)$: There is a universal constant $C_1>0$ so that
    \begin{equation}\label{almost positivity of d/dt J_A}
        \frac{d}{dt}\mathcal{J}_A(u) \geq -C_1\cdot(\alpha^*\lambda^2 + \lambda^{2+4\omega}\cdot \log^{\frac{1}{2}}(2 + \lambda^{-1-2\omega})),
    \end{equation}
    for each $\eta \in (0,\eta^*)$ and $t \in (T_{\text{dec}}^{(\eta)},0] \cap [t_1,0]$. Here we use the fact that $n=8$ and $\omega = 1/8$. By integrating from $t$ to $0$ both sides of \eqref{almost positivity of d/dt J_A} with initial conditions on $\epsilon(0)=0$ in $H^{1/2+\delta}$ and using \eqref{coercivity of J_A}, we obtain
    \begin{equation}
        \frac{\kappa_1}{4\lambda}\norm{\epsilon}_{H^{1/2}}^2 \leq C_1\lambda^{5/2}\cdot\left(\alpha^* + \lambda^{1/2}\cdot (1+\alpha^* + \log^{\frac{1}{2}}(2 + \lambda^{-1-2\omega}))\right).
    \end{equation}
    We reduce $\alpha^*>0$ if necessary so that  
    \begin{equation}\label{smallness of alpha for energy bootstrapping}
        \alpha^* < \frac{\kappa_1}{16C_1}.
    \end{equation}  
    Then we further shrink $\eta^*=\eta^*(\alpha^*)>0$ (depending on $\alpha^*$) and choose $t_1=t_1(\alpha^*)<0$ sufficiently close to $0$ such that, for every  
    \[
        \eta \in (0,\eta^*(\alpha^*)) \quad \text{and} \quad  
        t \in (T_{\mathrm{dec}}^{(\eta)},0]\cap [t_1,0],
    \]  
    we have  
    \begin{equation}\label{choice of t_1(alpha^*) and eta^*(alpha^*) - (1)}
        \lambda^{1/2}(t)\,\Bigl(1+\alpha^* + \log^{1/2}\!\bigl(2 + \lambda(t)^{-1-2\omega}\bigr)\Bigr) \le \alpha^*.
    \end{equation}  
    The existence of such $\eta^*(\alpha^*)$ and $t_1(\alpha^*)$ follows from the modulation estimate  
    \begin{equation*}
        \lambda(t) \lesssim C\,(t^2+\eta),
    \end{equation*}  
    which is a direct consequence of the bootstrap control on the modulation parameter $\lambda$ in \eqref{bootstrap bounds}. Therefore, we conclude
    \begin{equation*}
        \frac{\norm{\epsilon}_{H^{1/2}}^2(t)}{\lambda^{3+4\omega}(t)} \leq \frac{8C_1}{\kappa_1}\alpha^* < \frac{1}{2} \text{ for } t\in [t_2,0],
    \end{equation*}
    where $t_2 \in [t_1,0] \cap (T^{(\eta)}_{\text{dec}},0]$, assuming \eqref{bootstrap bounds}.
    
    \textbf{Step 2.} We close the bootstrap for modulation parameters.
    First note that $b^2+2\eta \lesssim \lambda$ and $b_{\eta}^2+2\eta \lesssim \lambda_{\eta}$ from \eqref{asymtotic for parameters for t} and \eqref{orders of modulation parameters}. Then from modulation equation \eqref{Modulation equation} and improved bounds \eqref{improved bounds}, we have
    \begin{equation}\label{b bootstrap 1}
        \begin{aligned}
            \left\{\frac{\sqrt{b^2+2\eta}}{\lambda^{1/2}}\right\}_t
            &= \frac{1}{\lambda^{3/2}}\frac{b}{\sqrt{b^2+2\eta}}\left\{\lambda b_t+\frac{1}{2}b^2+\eta\right\}
            -\frac{1}{2}\frac{\sqrt{b^2+2\eta}}{\lambda^{3/2}}\left(b+\lambda_t\right) \\
            &= \frac{1}{\lambda^{3/2}}\frac{b}{\sqrt{b^2+2\eta}}\left\{\lambda b_t+\frac{1}{2}b^2+\eta\right\}
            +\mathcal{O}\left((\lambda^\omega+\alpha^*)\lambda^{3/2+\omega}\right).
        \end{aligned}
    \end{equation}
    Also, from the formal modulation law \eqref{formal modulation law}, we have
    \begin{equation}\label{b bootstrap 2}
        \begin{aligned}
            \left\{\frac{\sqrt{b_{\eta}^2+2\eta}}{\lambda_{\eta}^{1/2}}\right\}_t
            &= \frac{1}{\lambda_{\eta}^{3/2}}\frac{b_{\eta}}{\sqrt{b_\eta^2+2\eta}}
            \left\{\lambda_\eta(b_{\eta})_t+\frac{1}{2}b_{\eta}^2+\eta\right\} \\
            &= \frac{1}{\lambda^{3/2}}\frac{b}{\sqrt{b^2+2\eta}}
            \left\{\lambda_\eta(b_{\eta})_t+\frac{1}{2}b_{\eta}^2+\eta\right\}
            +\mathcal{O}\left(\lambda^{3/2+\omega}\right).
        \end{aligned}
    \end{equation}
    Last equality follows from $|b_{\eta}-b|\leq \lambda^{3/2+\omega}$ and $|\lambda-\lambda_\eta|\leq \lambda^{2+\omega}$ in \eqref{bootstrap bounds}. Indeed, using $b_\eta^2+2\eta\sim\lambda_\eta\sim\lambda$ and $b^2+2\eta\sim\lambda$, we have
    \begin{align*}
        \frac{1}{\lambda_{\eta}^{3/2}}\frac{b_{\eta}}{\sqrt{b_\eta^2+2\eta}}
        -\frac{1}{\lambda^{3/2}}\frac{b}{\sqrt{b^2+2\eta}} &=
        \frac{b_{\eta}}{\sqrt{b_\eta^2+2\eta}}
        \left(\frac{1}{\lambda_{\eta}^{3/2}}-\frac{1}{\lambda^{3/2}}\right)
        +\frac{1}{\lambda^{3/2}}
        \left(\frac{b_{\eta}}{\sqrt{b_\eta^2+2\eta}}-\frac{b}{\sqrt{b^2+2\eta}}\right)\\
        &=\mathcal{O}(\lambda^{-1/2+\omega}),
    \end{align*}
    and from \eqref{formal modulation law} and \eqref{asymtotic for parameters for t}, $\lambda_\eta(b_\eta)_t+\frac{1}{2}b_\eta^2+\eta=\mathcal{O}(\lambda_\eta^2)$. Hence, combining \eqref{b bootstrap 1} and \eqref{b bootstrap 2}, we obtain
    \begin{equation}\label{eq:quadratic order t deriv control}
        \begin{aligned}
            \left| \left\{\frac{\sqrt{b^2+2\eta}}{\lambda^{1/2}}\right\}_t -
            \left\{\frac{\sqrt{b_\eta^2+2\eta}}{\lambda_\eta^{1/2}}\right\}_t
            \right| & \lesssim
            \frac{1}{\lambda^{3/2}}\left|\left\{\lambda b_t+\frac{1}{2}b^2+\eta\right\}-
            \left\{\lambda_\eta(b_\eta)_t+\frac{1}{2}b_\eta^2+\eta\right\}
            \right| +\lambda^{3/2+\omega}\\
            &\lesssim(\lambda^\omega+\alpha^*)\lambda^{1/2+\omega}.
        \end{aligned}
    \end{equation}
    Indeed, by taking the difference between the modulation equation \eqref{Modulation equation} and the formal modulation law \eqref{formal modulation law}, we have
    \begin{align*}
        &\left\{\lambda b_t+\frac{1}{2}b^2+\eta\right\}
        -
        \left\{\lambda_\eta(b_\eta)_t+\frac{1}{2}b_\eta^2+\eta\right\}\\
        &=
        c_1(b_\eta^4-b^4)+c_4(\nu_\eta^2-\nu^2)
        +c_3\eta(b_\eta^2-b^2)
        +\mathcal{O}\left((\alpha^*+\lambda^\omega)\lambda^{2+\omega}\right)\\
        &=
        \mathcal{O}\left((\alpha^*+\lambda^\omega)\lambda^{2+\omega}\right).
    \end{align*}
    Using \eqref{eq:quadratic order t deriv control} we obtain
    \begin{equation}\label{eq:quadratic order control}
            \left|\frac{b^2+2\eta}{\lambda}-\frac{b_\eta^2+2\eta}{\lambda_\eta}\right|
            \lesssim (\alpha^*+\lambda^\omega)\lambda^{1+\omega}.
    \end{equation}
    By subtracting $b_\eta$ law in \eqref{formal modulation law} from the modulation law for the parameter $b$ in \eqref{Modulation equation} and using \eqref{eq:quadratic order control} gives
    \begin{align*}
        \lambda(b-b_\eta)_t + &\frac{1}{2}(b^2+2\eta - \frac{\lambda}{\lambda_\eta}(b_\eta^2 + 2\eta)) + c_3 \eta( b^2 -  \frac{\lambda}{\lambda_\eta}b_\eta^2) \\
        &+ c_1(b^4 - \frac{\lambda}{\lambda_\eta} b_{\eta}^4)
        + c_4(\nu^2-\frac{\lambda}{\lambda_\eta}\nu_\eta^2) = \mathcal{O}((\lambda^\omega+\alpha^*)\lambda^{2+\omega})
    \end{align*}
    Now, using $|\frac{\lambda}{\lambda_\eta}-1|\leq \lambda^{1+\omega}$ , $|b-b_{\eta}|\leq \lambda^{3/2+\omega}$ and $|\nu-\nu_{\eta}|\leq \lambda^{2+\omega}$, we have
    \begin{equation}
        |b-b_{\eta}| \lesssim (\alpha^*+\lambda^{\omega}) \lambda^{3/2+\omega}.
    \end{equation}
    Then, by using $|\lambda_t + b| \lesssim (\lambda^\omega+\alpha^*)\lambda^{\frac{5}{2}+\omega}$ and $(\lambda_{\eta})_t = - b_{\eta}$, we obtain
    \begin{equation}\label{eq:lambda difference bootstrap}
        |\lambda- \lambda_\eta| \leq \int_t^0 \left(|b-b_{\eta}|+|\lambda_t+b|\right) dt
        \lesssim (\lambda^{\omega}+\alpha^*)\lambda^{2+\omega}.
    \end{equation}
    Closing the bootstrap for $\nu$ and $\nut{x}$: From \eqref{improved bounds} and the formal modulation law, we have
    \[
        |\lambda\nu_t+b\nu|+|\lambda_t+b|\lesssim(\lambda^\omega+\alpha^*)\lambda^{5/2+\omega},\quad \left(\frac{\nu_\eta}{\lambda_\eta}\right)_t=0.
    \]
    Hence, using $|\nu|\lesssim\lambda$,
    \[
        \left|\left(\frac{\nu}{\lambda}\right)_t-\left(\frac{\nu_\eta}{\lambda_\eta}\right)_t\right| \leq \left|\frac{\lambda \nu_t+b\nu}{\lambda^2}\right|+\left|\frac{\nu}{\lambda^2}(\lambda_t+b)\right|\lesssim(\lambda^\omega+\alpha^*)\lambda^{1/2+\omega}.
    \]
    Integrating in $t$ and using the bound \eqref{eq:lambda difference bootstrap}, we obtain
    \[
        |\nu-\nu_\eta| \leq (\lambda^\omega+\alpha^*)\lambda^{2+\omega} + \left|\frac{\nu_\eta}{\lambda_\eta}(\lambda-\lambda_\eta)\right|\lesssim(\lambda^\omega+\alpha^*)\lambda^{2+\omega}.
    \]

    Similarly, subtracting the $\nut{x}_\eta$-law in \eqref{formal modulation law} from the law of $\nut{x}$ in \eqref{improved bounds}, we have
    \[
        (\nut{x}-\nut{x}_\eta)_t-(\nu-\nu_\eta)-c_2(b^2\nu-b_\eta^2\nu_\eta)
        =\mathcal{O}\left((\lambda^\omega+\alpha^*)\lambda^{5/2+\omega}\right).
    \]
    Using the bounds for $b-b_\eta$ and $\nu-\nu_\eta$ and integrating in $t$, we obtain
    \[
        |\nut{x}-\nut{x}_\eta|\lesssim(\lambda^\omega+\alpha^*)\lambda^{5/2+\omega}.
    \]

    Finally, from \eqref{Modulation equation} and the formal law $(\gamma_\eta)_t=1/\lambda_\eta$, we have
    \[
        |(\gamma-\gamma_\eta)_t|\leq \frac{|\lambda\gamma_t-1|}{\lambda} +\frac{|\lambda-\lambda_\eta|}{\lambda\lambda_\eta} \lesssim(\lambda^\omega+\alpha^*)\lambda^\omega.
    \]
    Since $\gamma(0)=\gamma_\eta(0)$, integration in $t$ yields
    \[
        |\gamma-\gamma_\eta|\lesssim(\lambda^\omega+\alpha^*)\lambda^{1/2+\omega}.
    \]
    Hence, there is a universal constant $C_2>0$ so that for each $t \in [t_2,0]$
    where $t_2 \in (T_{\text{dec}}^{(\eta)},0] \cap [t_1,0]$ we obtain
    \[
        \frac{|\lambda-\lambda_\eta|}{\lambda^{2+\omega}} + \frac{|b-b_\eta|}{\lambda^{3/2+\omega}} + \frac{|\nu-\nu_\eta|}{\lambda^{2+\omega}} + \frac{|\nut{x}-\nut{x}_\eta|}{\lambda^{5/2+\omega}} + \frac{|\gamma - \gamma_\eta|}{\lambda^{1/2+\omega}} \leq C_2(\lambda^\omega + \alpha^*).
    \]
    We set $\alpha^*>0$ smaller than the one chosen in \eqref{smallness of alpha for energy bootstrapping} to satisfy 
    \begin{equation}\label{smallness of alpha for modulation parameter bootstrapping}
        \alpha^* < \frac{1}{4C_2},
    \end{equation}
    and possibly shrinking $t_1(\alpha^*)<0$ and $\eta^*(\alpha^*)>0$ so that
    for $\eta \in (0,\eta^*)$ and $t \in (T_{\text{dec}}^{(\eta)},0] \cap [t_1,0]$, $\lambda(t)$ satisfies \eqref{choice of t_1(alpha^*) and eta^*(alpha^*) - (1)} and 
    \begin{equation}\label{choice of t_1(alpha^*) and eta^*(alpha^*) - (2)}
        \lambda^{\omega}(t) \leq \alpha^*
    \end{equation}
    This closes the bootstrap argument for the modulation parameters $(\lambda,b,\nu,\gamma,\nut{x})(t)$.

    \textbf{Step 3.} 
    Controlling the quantity $\|\epsilon^{\sharp}\|_{H^{1/2+\delta}}$ with $\delta = 1/16$ so that $2\delta < 1/2$ follows the approach introduced in \cite{KLR2013ARMAhalfwave}, which relies on a precise product estimate for fractional differential operators. For the reader’s convenience, we state the following standard fractional Leibniz rule.

    \begin{lemma}[Fractional Leibniz rule]\label{Leibniz and product rule for fractional calculus}
    Let $f,g \in \mathcal{S}(\mathbb{R})$ and let $\theta_1, \theta_2 \geq 0$ satisfying $\theta_1+\theta_2 \leq 1$. Then the following estimates hold:
        \begin{equation}\label{product rule for cal}
            \begin{aligned}
            \| D^{\theta_1}(fg) - (D^{\theta_1} f)g \|_{\dot{H}^{\theta_2}}  
            &\lesssim  
            \min \Big\{ \| D^{\theta_1+\theta_2} g \|_{L^2}\|\widehat{f}\|_{L^1},\;
                \|\widehat{D^{\theta_1+\theta_2} g}\|_{L^1}\| f \|_{L^2} \Big\} \\
            & \quad + \min \Big\{ \| D^{\theta_1} g \|_{L^2}\|\widehat{D^{\theta_2}f}\|_{L^1},\;
                \|\widehat{D^{\theta_1} g}\|_{L^1}\| D^{\theta_2}f \|_{L^2} \Big\},
            \end{aligned}
        \end{equation}
        and for $\theta \in [0,1]$ we have
        \begin{equation}\label{product rule for cal2}
            \| D^\theta(fg) \|_{L^2}  
            \lesssim  
            \| f \|_{L^\infty} \| D^\theta g \|_{L^2}  
            + \| g \|_{L^\infty} \| D^\theta f \|_{L^2}.
        \end{equation}
    \end{lemma}

    For notational convenience, we set
    \[
        \Phi \coloneqq \frac{\Psi^{\sharp}}{\lambda}.
    \]
    Recall that $W$, defined in \eqref{def of W}, satisfies the perturbed half-wave equation
    \[
        i\partial_t W - D W + |W|^2 W = \Phi.
    \]

    Unlike NLS, the half-wave equation does not enjoy small-data scattering, which forces us to work directly in the Fourier side. In particular, the control of high-regularity norms requires delicate harmonic analysis estimates based on Lemma \ref{Leibniz and product rule for fractional calculus}. We now begin the main argument to close the bootstrap for
    \[
     \|\epsilon^{\sharp}\|_{H^{1/2+\delta}}.
    \]

    Also, recall that $\epsilon^{\sharp}$ satisfies  
    \begin{equation}\label{eq for ep sharp}
        i\partial_t \epsilon^{\sharp} = D\epsilon^{\sharp} - |\epsilon^{\sharp}|^2\epsilon^{\sharp}- \mathcal{F} - \Phi, 
    \end{equation}
    where  
    \[
        \mathcal{F} = |W+\epsilon^{\sharp}|^2(W+\epsilon^{\sharp}) - |W|^2W - |\epsilon^{\sharp}|^2\epsilon^{\sharp}.
    \]
    We start by differentiating $\norm{\epsilon^{\sharp}}_{\dot{H}^{1/2+\delta}}^2$ with respect to $t$:  
    \begin{equation}\label{differentiaion of H^(1/2+)}
        \begin{aligned}
            \frac{1}{2}\frac{d}{dt}\norm{\epsilon^{\sharp}}_{\dot{H}^{1/2+\delta}}^2
            = -\Im\big(D^{1/2+\delta}(-D\epsilon^{\sharp}+|\epsilon^{\sharp}|^2\epsilon^{\sharp}+\mathcal{F}+\Phi),\,
            D^{1/2+\delta}\epsilon^{\sharp}\big).
        \end{aligned}
    \end{equation}
    The highest-order term vanishes in the imaginary part of the right-hand side of \eqref{differentiaion of H^(1/2+)}, since the operator $D$ is self-adjoint:  
    \[
        \int D^{1/2+\delta}(D\epsilon^{\sharp})\,\overline{D^{1/2+\delta}\epsilon^{\sharp}}\;dx \in \mathbb{R}.
    \]
    This is precisely why we work with the quantity $\norm{D^{1/2+\delta}\epsilon^{\sharp}}_{L^2}^2$ to close the bootstrap and control the $H^{1/2+}$-norm.

    Next, we decompose the nonlinearity as  
    \[
        \mathcal{F}+|\epsilon^{\sharp}|^2\epsilon^{\sharp} = L(\epsilon^{\sharp}) + N(\epsilon^{\sharp}),
    \]
    where $L(\epsilon^{\sharp})$ denotes the terms linear in $\epsilon^{\sharp}$ and $N(\epsilon^{\sharp})$ collects the higher-order nonlinear contributions. That is 
    \begin{align*}
        L(\epsilon^{\sharp}) &\coloneqq |W|^2\epsilon^{\sharp} + 2\Re(\overline{W}\epsilon^{\sharp})W \\
        N(\epsilon^{\sharp}) &\coloneqq W|\epsilon^\sharp|^2 + 2\Re(\overline{W}\epsilon^{\sharp})\epsilon^{\sharp} + |\epsilon^{\sharp}|^2\epsilon^{\sharp}.
    \end{align*}
    Note also that $\Phi$ does not contain any $\epsilon^{\sharp}$-dependence.  
    We will therefore estimate the contributions of $D^{1/2+\delta}L(\epsilon^{\sharp})$, $D^{1/2+\delta}N(\epsilon^{\sharp})$, and $D^{1/2+\delta}\Phi$ against $D^{1/2+\delta}\epsilon^{\sharp}$ separately.

    We begin with the estimate of $L(\epsilon^{\sharp})$ paired against $D^{1/2+\delta}\epsilon^{\sharp}$:
    \begin{equation}\label{D^(1/2+)L(ep)}
        \begin{aligned}
            \Im\big(D^{1/2+\delta}L(\epsilon^{\sharp}),\,D^{1/2+\delta}\epsilon^{\sharp}\big) 
            &= \Im\big(D^{1/2+\delta}(|W|^2\epsilon^{\sharp}),\,D^{1/2+\delta}\epsilon^{\sharp}\big) \\
            &\quad + \Im\big(D^{1/2+\delta}(2\Re(\overline{W}\epsilon^{\sharp})W),\,D^{1/2+\delta}\epsilon^{\sharp}\big).
        \end{aligned}
    \end{equation}

    We first handle the first term on the right-hand side of \eqref{D^(1/2+)L(ep)}.  
    Define the commutator error term by
    \begin{equation*}
        E_1 \coloneqq 
        \Im\Big( \big(D^{1/2+\delta}(|W|^2\epsilon^{\sharp}) - D^{-\delta}(|W|^2D^{1/2+2\delta}\epsilon^{\sharp})\big),\,D^{1/2+\delta}\epsilon^{\sharp}\Big).
    \end{equation*}
    Then we can rewrite
    \begin{equation*}
        \Im\big(D^{1/2+\delta}(|W|^2\epsilon^{\sharp}),\,D^{1/2+\delta}\epsilon^{\sharp}\big)
        = \Im\big(|W|^2D^{1/2+2\delta}\epsilon^{\sharp},\,D^{1/2}\epsilon^{\sharp}\big) + E_1.
    \end{equation*}

    Using the fractional Leibniz rule \eqref{product rule for cal} we bound $E_1$ as
    \begin{equation*}
        |E_1| 
        \lesssim 
        \norm{\epsilon^{\sharp}}_{L^2}
        \norm{D^{1/2}\epsilon^{\sharp}}_{L^2}
        \norm{\widehat{D^{1/2+2\delta}|W|^2}}_{L^1} 
        \lesssim 
        \lambda^{-2-2\delta}\norm{\epsilon}_{H^{1/2}}^2.
    \end{equation*}

    Next, we estimate the main term 
    \[
        \Im\big(|W|^2D^{1/2+2\delta}\epsilon^{\sharp},\,D^{1/2}\epsilon^{\sharp}\big).
    \]
    For notational convenience, we set 
    \[
        f_1 \coloneqq D^{1/2}\epsilon^{\sharp}.
    \]
    By Plancherel’s theorem and the smallness of $\delta$, we have the identity
    \begin{align*}
        2i\,\widehat{\Im\left(D^{1/2+2\delta}\epsilon^\sharp\, \overline{D^{1/2}\epsilon^\sharp} \right)}(\xi) &=
        \int_{\mathbb R} \bigl(|\xi-\zeta|^{2\delta}-|\zeta|^{2\delta}\bigr) \widehat{f_1}(\xi-\zeta) \overline{\widehat{f_1}(-\zeta)}\,d\zeta\\ 
        &\lesssim_\delta |\xi|^{2\delta}\|f_1\|_{L^2}^2.
    \end{align*}
    Therefore, we obtain the bound
    \begin{align*}
    \big|\Im\big(|W|^2D^{1/2+2\delta}\epsilon^{\sharp},\,D^{1/2}\epsilon^{\sharp}\big)\big|
        \lesssim 
        \norm{|\xi|^{2\delta}\,\widehat{|W|^2}}_{L^1_{\xi}}\,
        \norm{f_1}_{L^2}^2
        \lesssim 
        \lambda^{-2-2\delta}\,\norm{\epsilon}_{H^{1/2}}^2.    
    \end{align*}

\medskip

    We now move on to the second term on the right-hand side of \eqref{D^(1/2+)L(ep)}, namely
    \[
        \Im\big(D^{1/2+\delta}(2\Re(\overline{W}\epsilon^{\sharp})W),\,D^{1/2+\delta}\epsilon^{\sharp}\big).
    \]
    This term is more delicate due to the interaction between $\epsilon^{\sharp}$ and $W$.  
    We shall decompose it into several commutator error terms and a main term, making repeated use of the fractional Leibniz rules \eqref{product rule for cal} and \eqref{product rule for cal2}.  
    For this purpose, we introduce the following quantities:
    \begin{align*}
        &E_2 \coloneqq \Im\big(D^{1/2+\delta}\{2\Re(\epsilon^{\sharp}\overline{W})W\},\,D^{1/2+\delta}\epsilon^{\sharp}\big)
        -\Im\big(D^{1/2+\delta}\{2\Re(\epsilon^{\sharp}\overline{W})\}W,\,D^{1/2+\delta}\epsilon^{\sharp}\big),\\
        &E_3 \coloneqq \Im\big(D^{1/2+\delta}\{2\Re(\epsilon^{\sharp}\overline{W})\}W,\,D^{1/2+\delta}\epsilon^{\sharp}\big)
        -\Im\big(2\Re\big((D^{1/2+\delta}\epsilon^{\sharp})\overline{W}\big)W,\,D^{1/2+\delta}\epsilon^{\sharp}\big),\\
        &E_4 \coloneqq \Im\big(2\Re\big((D^{1/2+\delta}\epsilon^{\sharp})\overline{W}\big)W,\,D^{1/2+\delta}\epsilon^{\sharp}\big)
        -\Im\big(2D^{1/2}\big[\Re\big((D^{\delta}\epsilon^{\sharp})\overline{W}\big)\big]W,\,D^{1/2+\delta}\epsilon^{\sharp}\big),\\
        &E_5 \coloneqq \Im\big(2D^{1/2}\big[\Re\big((D^{\delta}\epsilon^{\sharp})\overline{W}\big)\big]W,\,D^{1/2+\delta}\epsilon^{\sharp}\big)
        -\Im\big(2\Re\big((D^{\delta}\epsilon^{\sharp})\overline{W}\big)W,\,D^{1+\delta}\epsilon^{\sharp}\big).
    \end{align*}
    With this notation, we can express the second term in \eqref{D^(1/2+)L(ep)} as
    \begin{equation*}
        \text{(second term of \eqref{D^(1/2+)L(ep)})} 
        = E_2 + E_3 + E_4 + E_5
        + \Im\big(2\Re((D^{\delta}\epsilon^{\sharp})\overline{W})W,\,D^{1+\delta}\epsilon^{\sharp}\big).
    \end{equation*}
    We shall estimate these terms one by one.

    Estimate for $E_2$:
    For notational convenience, we set
    \[
        f_2 \coloneqq 2\Re\bigl(\epsilon^{\sharp}\overline{W}\bigr), \quad
        g_2 \coloneqq W, \quad
        h_2 \coloneqq D^{1/2}\epsilon^{\sharp}.
    \]
    Applying \eqref{product rule for cal} and \eqref{product rule for cal2}, we obtain
    \begin{align*}
        |E_2|
        &\le \Bigl|\bigl(D^{\delta}\bigl(D^{1/2+\delta}(f_2 g_2) - (D^{1/2+\delta}f_2) g_2\bigr),\, h_2\bigr)\Bigr| \\
        &\lesssim \|h_2\|_{L^2}\Bigl(\|\widehat{D^{1/2+2\delta}g_2}\|_{L^1}\|f_2\|_{L^2}
        + \|D^{\delta}f_2\|_{L^2}\|\widehat{D^{1/2+\delta}g_2}\|_{L^1}\Bigr) \\
        &\lesssim \lambda^{-1}\|\epsilon\|_{H^{1/2}}\Bigl(\lambda^{-1-2\delta}\|\epsilon\|_{H^{1/2}}
        + \lambda^{-1/2-\delta}\bigl(\|D^{\delta}\epsilon^{\sharp}\|_{L^2}\|W\|_{L^{\infty}}
        + \|\widehat{D^{\delta}W}\|_{L^1_{\xi}}\|\epsilon\|_{L^2}\bigr)\Bigr) \\
        &\lesssim \lambda^{-2-2\delta}\|\epsilon\|_{H^{1/2}}^2.
    \end{align*}

    Estimate for $E_3$:
    For notational convenience, we set
    \[
        f_3 \coloneqq \epsilon^{\sharp}, \quad
        g_3 \coloneqq \overline{W}, \quad
        h_3 \coloneqq \overline{W}\,D^{1/2+\delta}\epsilon^{\sharp}.
    \]
    From \eqref{product rule for cal}, we have
    \begin{equation*}
        \|h_3 - D^{\delta}(\overline{W}D^{1/2}\epsilon^{\sharp})\|_{L^2}
        \lesssim \|\widehat{D^{\delta}\overline{W}}\|_{L^1}\|D^{1/2}\epsilon^{\sharp}\|_{L^2}
        \lesssim \lambda^{-1-\delta}\|\epsilon\|_{H^{1/2}}.
    \end{equation*}
    Using the triangle inequality, we deduce
    \begin{equation}\label{estimate for E3}
        \begin{aligned}
            |E_3|
            &\lesssim \Bigl|\bigl(D^{1/2+\delta}(f_3 g_3) - (D^{1/2+\delta}f_3) g_3,\,
            h_3 - D^{\delta}(\overline{W}D^{1/2}\epsilon^{\sharp})\bigr)\Bigr| \\
            &\quad + \Bigl|\bigl(D^{\delta}\bigl[D^{1/2+\delta}(f_3 g_3) - (D^{1/2+\delta}f_3) g_3\bigr],\,
            \overline{W}D^{1/2}\epsilon^{\sharp}\bigr)\Bigr|.
        \end{aligned}
    \end{equation}
    Each term on the right-hand side of \eqref{estimate for E3} is estimated as follows.  
    First, by \eqref{product rule for cal},
    \begin{align*}
        &\Bigl|\bigl(D^{1/2+\delta}(f_3 g_3) - (D^{1/2+\delta}f_3) g_3,\,
        h_3 - D^{\delta}(\overline{W}D^{1/2}\epsilon^{\sharp})\bigr)\Bigr| \\
        &\quad \lesssim \|\widehat{D^{1/2+\delta}g_3}\|_{L^1_\xi}\|f_3\|_{L^2}\,
        \lambda^{-1-\delta}\|\epsilon\|_{H^{1/2}}
        \lesssim \lambda^{-2-2\delta}\|\epsilon\|_{H^{1/2}}^2.
    \end{align*}
    Similarly,
    \begin{align*}
        &\Bigl|\bigl(D^{\delta}\bigl[D^{1/2+\delta}(f_3 g_3) - (D^{1/2+\delta}f_3) g_3\bigr],\,
        \overline{W}D^{1/2}\epsilon^{\sharp}\bigr)\Bigr| \\
        &\quad \lesssim \lambda^{-1/2}\|D^{1/2}\epsilon^{\sharp}\|_{L^2}
        \Bigl(\|\widehat{D^{1/2+\delta}g_3}\|_{L^1}\|D^{\delta}f_3\|_{L^2}
        + \|\widehat{D^{1/2+2\delta}g_3}\|_{L^1}\|f_3\|_{L^2}\Bigr) \\
        &\quad \lesssim \lambda^{-2-2\delta}\|\epsilon\|_{H^{1/2}}^2.
    \end{align*}
    Combining the above, we conclude
    \[
        |E_3| \lesssim \lambda^{-2-2\delta}\|\epsilon\|_{H^{1/2}}^2.
    \]

    Estimate for $E_4$:
    For notational convenience, set
    \[
        f_4 \coloneqq D^{\delta}\epsilon^{\sharp}, \quad
        g_4 \coloneqq \overline{W}, \quad
        h_4 \coloneqq \overline{W}\,D^{1/2}\epsilon^{\sharp}, \quad
        k_4 \coloneqq D^{\delta}h_4 - \overline{W}\,D^{1/2+\delta}\epsilon^{\sharp}.
    \]
    Since
    \[
        |E_4| \lesssim 
        \bigl|\bigl(D^{1/2}(f_4 g_4) - (D^{1/2}f_4) g_4,\,
        D^{\delta}h_4 - k_4\bigr)\bigr|,
    \]
    we estimate each term separately.  
    First,
    \begin{align*}
        \bigl|\bigl(D^{\delta}[D^{1/2}(f_4 g_4) - (D^{1/2}f_4) g_4],\, h_4\bigr)\bigr|
        &\lesssim \|h_4\|_{L^2}\left(\|D^{\delta}f_4\|_{L^2}\|\widehat{D^{1/2}g_4}\|_{L^1}
        + \|f_4\|_{L^2}\|\widehat{D^{1/2+\delta}g_4}\|_{L^1}\right)\\
        &\lesssim \lambda^{-2-2\delta}\|\epsilon\|_{H^{1/2}}^2.
    \end{align*}
    Second, by \eqref{product rule for cal},
    \begin{align*}
        \bigl|\bigl(D^{1/2}(f_4 g_4) - (D^{1/2}f_4) g_4,\, k_4\bigr)\bigr|
        &\lesssim \|f_4\|_{L^2}\|\widehat{D^{1/2}g_4}\|_{L^1}\|k_4\|_{L^2}
        \lesssim \lambda^{-2-2\delta}\|\epsilon\|_{H^{1/2}}^2,
    \end{align*}
    since
    \[
        \|k_4\|_{L^2}\lesssim \|\widehat{D^{\delta}\overline{W}}\|_{L^1}\|D^{1/2}\epsilon^{\sharp}\|_{L^2}
        \lesssim \lambda^{-1-\delta}\|\epsilon\|_{H^{1/2}}.
    \]
    Thus, we obtain
    \[
        |E_4| \lesssim \lambda^{-2-2\delta}\|\epsilon\|_{H^{1/2}}^2.
    \]
    
    Estimate for $E_5$:
    For notational convenience, set
    \[
        f_5 \coloneqq 2\Re\bigl(D^{\delta}\epsilon^{\sharp}\,\overline{W}\bigr), \quad
        g_5 \coloneqq W, \quad
        h_5 \coloneqq D^{1/2}\epsilon^{\sharp}.
    \]
    By \eqref{product rule for cal} and \eqref{product rule for cal2},
    \begin{align*}
        |E_5|
        &\lesssim \bigl|\bigl(D^{1/2}(f_5 g_5) - (D^{1/2}f_5) g_5,\, D^{\delta}h_5\bigr)\bigr|\\
        &\lesssim \left(\|\widehat{D^{1/2+\delta}g_5}\|_{L^1}\|f_5\|_{L^2}
        + \|\widehat{D^{1/2}g_5}\|_{L^1}\|D^{\delta}f_5\|_{L^2}\right)\|h_5\|_{L^2} \\
        &\lesssim \lambda^{-2-2\delta}\|\epsilon\|_{H^{1/2}}^2.
    \end{align*}
    
    Finally, we investigate $\Im([2\Re(D^{\delta}\epsilon^{\sharp}\overline{W})W],D^{1+\delta}\epsilon^{\sharp})$: Since $\epsilon^{\sharp}$ satisfies \eqref{eq for ep sharp}, we have
    \begin{align*}
        &\Im([2\Re(D^{\delta}\epsilon^{\sharp}\overline{W})W] ,D^{1+\delta}\epsilon^{\sharp}) = -\frac{d}{dt}\left(\norm{\Re(D^{\delta}\epsilon^{\sharp}\overline{W})}_{L^2}^2\right)\\&+\Re(2\Re(D^{\delta}\epsilon^{\sharp}\overline{W}),\partial_t\overline{W}D^{\delta}\epsilon^{\sharp})
        +\Im(2\Re(D^{\delta}\epsilon^{\sharp}\overline{W})W,D^{\delta}(|\epsilon^{\sharp}|^2\epsilon^{\sharp}+\mathcal{F}+\Phi)).
    \end{align*}
    From Cauchy--Schwartz inequality and \eqref{diff of Q_P of t} and from the fact that $i\partial_tz = Dz - |z|^2z$ and \eqref{control of norm z(t) for geq 1/2}, we estimate
    \begin{align*}
        |\Re(2\Re(D^{\delta}\epsilon^{\sharp}\overline{W}),\overline{\partial_tW}
        D^{\delta}\epsilon^{\sharp})| \lesssim \norm{D^{\delta}\epsilon^{\sharp}}_{L^2}^2\norm{\overline{W}}_{L^\infty}\norm{\overline{\partial_tW}}_{L^{\infty}}\lesssim \lambda^{-2-2\delta}\norm{\epsilon}_{H^{1/2}}^2.
    \end{align*}
    To control the last term of the above equality, we use Lemma~\ref{log loss of the L-infty estiamte}.
    Since we have schematically 
    \begin{equation*}
        |\epsilon^{\sharp}|^2\epsilon^{\sharp}+\mathcal{F} \sim W^2\epsilon^{\sharp} + W(\epsilon^{\sharp})^2 +|\epsilon^\sharp|^2\epsilon^\sharp,
    \end{equation*}
    combining with \eqref{product rule for cal2}, we have
    \begin{align*}
        &\norm{D^{\delta}(W^2\epsilon^{\sharp})}_{L^2} \lesssim \norm{W}_{L^{\infty}}^2\norm{D^{\delta}\epsilon^{\sharp}}_{L^2} + \norm{D^{\delta}W}_{L^2}\norm{W}_{L^{\infty}}\norm{\epsilon^{\sharp}}_{L^{\infty}}\\
        &\lesssim \lambda^{-1-\delta}\norm{\epsilon}_{H^{1/2}} + \lambda^{-1-\delta}\norm{\epsilon}_{H^{1/2}}\left[\log\left(2+\frac{\norm{\epsilon^{\sharp}}_{H^{1/2+\delta}}}{\norm{\epsilon^{\sharp}}_{H^{1/2}}}\right)\right]^{1/2},\\
        &\norm{D^{\delta}(W(\epsilon^{\sharp})^2)}_{L^2} \lesssim \norm{\epsilon^{\sharp}}_{L^{\infty}}^2\norm{D^{\delta}W}_{L^2} + \norm{D^{\delta}\epsilon^{\sharp}}_{L^2}\norm{\epsilon^{\sharp}}_{L^{\infty}}\norm{W}_{L^{\infty}} \\
        &\lesssim \lambda^{-1-\delta}\norm{\epsilon}_{H^{1/2}}^2\left\{\log\left(2+\frac{\norm{\epsilon^{\sharp}}_{H^{1/2+\delta}}}{\norm{\epsilon^{\sharp}}_{H^{1/2}}}\right)+ \left[\log\left(2+\frac{\norm{\epsilon^{\sharp}}_{H^{1/2+\delta}}}{\norm{\epsilon^{\sharp}}_{H^{1/2}}}\right)\right]^{1/2}\right\}
    \end{align*}
    Hence, noting that $\norm{\epsilon}_{H^{1/2}}^2 < \frac{1}{2}\lambda^{3+4\omega} \ll 1$, we obtain
    \begin{align*}
        \norm{D^{\delta}\left(|\epsilon^{\sharp}|^2\epsilon^{\sharp} + \mathcal{F}\right)}_{L^2} \lesssim \lambda^{-1-\delta}\norm{\epsilon}_{H^{1/2}}\log\left(2+\frac{\norm{\epsilon^{\sharp}}_{H^{1/2+\delta}}}{\norm{\epsilon^{\sharp}}_{H^{1/2}}}\right).
    \end{align*}
    Now, we estimate $\norm{D^{\delta}\Phi}_{L^2}$. We use the Gagliardo--Nirenberg inequality, and the previous estimate for $L^2$ and $H^1$-norm of $\Phi$, we have
    \begin{equation*}
        \norm{D^{\delta}\Phi}_{L^2} \lesssim \norm{\Phi}_{L^2}^{1-\delta}\norm{\Phi}_{H^1}^{\delta} \lesssim \alpha^*\lambda^{1-\delta}.
    \end{equation*}
    This estimate follows from \eqref{L2 norm of Psi} and \eqref{L2 norm of nabla psi} which implies 
    \begin{equation}\label{estimate for Phi}
        \norm{\Phi}_{L^2} \lesssim \alpha^*\lambda^{1+\omega}, \quad \norm{\Phi}_{H^1} \lesssim \alpha^*\lambda^{\omega}
    \end{equation}
    Thus, we conclude that
    \begin{align*}
        &|\Im(2\Re(D^{\delta}\epsilon^{\sharp}\bar{W})W,D^{\delta}(|\epsilon^{\sharp}|^2\epsilon^{\sharp}+\mathcal{F}+\Phi))| \\
        & \lesssim  \norm{2\Re({D^{\delta}\epsilon^{\sharp}\overline{W}})W}_{L^2}\left(\norm{D^{\delta}(|\epsilon^{\sharp}|^2\epsilon^{\sharp}+\mathcal{F})}_{L^2}+\norm{D^{\delta}\Phi}_{L^2}\right)\\
        &\lesssim \lambda^{-2-2\delta}\norm{\epsilon}_{H^{1/2}}^2\log\left(2+\frac{\norm{\epsilon^{\sharp}}_{H^{1/2+\delta}}}{\norm{\epsilon^{\sharp}}_{H^{1/2}}}\right)
        + \alpha^*\lambda^{-2\delta}\norm{\epsilon}_{H^{1/2}}.
    \end{align*}
    Therefore, from the bootstrap assumption of $\norm{\epsilon^{\sharp}}_{\dot{H}^{1/2+\delta}} \leq \lambda^{1-2\delta} \leq 1$ and $\norm{\epsilon}_{H^{1/2}}^2 \lesssim \alpha^*\lambda^{3+4\omega}$ which is from the energy estimate, we conclude that
    \begin{equation}\label{estimate for L(ep^sharp)}
        \begin{aligned}
            \Im(D^{1/2+\delta}L(\epsilon^{\sharp}),D^{1/2+\delta}\epsilon^{\sharp}) &= -\frac{d}{dt}\left(\norm{\Re([D^{\delta}\epsilon^{\sharp}]\overline{W})}_{L^2}^2\right) + \mathcal{O}\left(\alpha^* \lambda^{3/2-2\delta}|\log(\lambda)|\right).
        \end{aligned}
    \end{equation}
    
    We continue to estimate $\Im(D^{1/2+\delta}N(\epsilon^{\sharp}), D^{1/2+\delta}\epsilon^{\sharp})$. The terms in $N(\epsilon^\sharp)$ are schematically of the form $W(\epsilon^\sharp)^2$ and $|\epsilon^\sharp|^2\epsilon^\sharp$.
    Thus, using $\norm{\epsilon^{\sharp}}_{\dot{H}^{1/2+\delta}} \leq \lambda^{1-2\delta}$ and \eqref{product rule for cal}, we obtain
    \begin{equation}\label{estimate for N(ep^sharp)}
        \begin{aligned}
            &|\Im(D^{1/2+\delta}(W(\epsilon^{\sharp})^2),D^{1/2+\delta}\epsilon^{\sharp})| \\
            &\lesssim \norm{D^{1/2+\delta}\epsilon^{\sharp}}_{L^2}\left(\norm{D^{1/2+\delta}W}_{L^2}\norm{\epsilon^{\sharp}}_{L^{\infty}}^2 + \norm{\epsilon^{\sharp}}_{L^{\infty}}\norm{W}_{L^{\infty}}\norm{D^{1/2+\delta}\epsilon^{\sharp}}_{L^2}\right)\\
            & \lesssim \lambda^{-3/2-\delta}\norm{D^{1/2+\delta}\epsilon^{\sharp}}_{L^2}\norm{\epsilon}_{H^{1/2}}^2\log\left(2+\frac{\norm{\epsilon^{\sharp}}_{H^{1/2+\delta}}}{\norm{\epsilon^{\sharp}}_{H^{1/2}}}\right)\\
            &\quad + \lambda^{-1}\norm{D^{1/2+\delta}\epsilon^{\sharp}}_{L^2}^2\norm{\epsilon}_{H^{1/2}}\left[\log\left(2+\frac{\norm{\epsilon^{\sharp}}_{H^{1/2+\delta}}}{\norm{\epsilon^{\sharp}}_{H^{1/2}}}\right)\right]^{1/2}\\
            & \lesssim  \alpha^*\lambda^{3/2-2\delta}|\log(\lambda)|
        \end{aligned}
    \end{equation}
    The remaining pure cubic term satisfies
    \begin{align*}
        &\left|\Im\left( D^{1/2+\delta}(|\epsilon^\sharp|^2\epsilon^\sharp),
        D^{1/2+\delta}\epsilon^\sharp\right)\right|\\
        &\qquad\lesssim \|\epsilon^\sharp\|_{L^\infty}^2 \|D^{1/2+\delta}\epsilon^\sharp\|_{L^2}^2 \lesssim
        \alpha^*\lambda^{3/2-2\delta}|\log\lambda|.
    \end{align*}
    Finally, we estimate the term $\Im(D^{1/2+\delta}\Phi,D^{1/2+\delta}\epsilon^{\sharp})$:
    From the self-adjointness of $D^{\delta}$ and the Cauchy--Schwartz inequality, we have
    \begin{equation}\label{estimate for Phi2}
        \begin{aligned}
            |\Im(D^{1/2+2\delta}\Phi,D^{1/2}\epsilon^{\sharp})| \lesssim
            \lambda^{-1/2}\norm{\epsilon}_{H^{1/2}} \norm{D^{1/2+2\delta}\Phi}_{L^2}
            \lesssim \alpha^*\lambda^{3/2-2\delta}.
        \end{aligned}
    \end{equation}
    The last inequality follows from \eqref{estimate for Phi} and
    \begin{equation*}
        \norm{D^{1/2+2\delta}\Phi}_{L^2} \lesssim \norm{\Phi}_{L^2}^{1/2-2\delta}\norm{\Phi}_{H^1}^{1/2+2\delta},
    \end{equation*}
    Hence, combining estimates \eqref{estimate for L(ep^sharp)}, \eqref{estimate for N(ep^sharp)}, and \eqref{estimate for Phi2}, we conclude
    \begin{equation}
        \frac{d}{dt} \left\{ \frac{1}{2} \norm{\epsilon^{\sharp}(t)}_{\dot{H}^{1/2+\delta}}^2-
        \norm{\Re([D^{\delta}\epsilon^{\sharp}]\overline{W})}_{L^2}^2 \right\}
        = \mathcal{O}\left( \alpha^*\lambda^{3/2-2\delta}|\log(\lambda)| \right).
    \end{equation}
    Moreover, using the $H^{1/2}$ estimate, we have
    \begin{align*}
        \norm{\Re([D^{\delta}\epsilon^{\sharp}]\overline{W})}_{L^2}^2 \lesssim
        \norm{W}_{L^\infty}^2\norm{D^\delta\epsilon^\sharp}_{L^2}^2\lesssim
        \lambda^{-1-2\delta}\norm{\epsilon}_{H^{1/2}}^2\lesssim \alpha^*\lambda^{2+4\omega-2\delta}\lesssim \alpha^*\lambda^{2-4\delta}.
    \end{align*}
    Futhermore, we have 
    \[
        \int_t^0 \lambda^{3/2-2\delta}(\tau)|\log\lambda(\tau)|\,d\tau\lesssim
        \lambda^{2-4\delta}(t).
    \]
    Since $\epsilon^\sharp(0)=0$, we integrate in $t$, together with the preceding two estimates, to obtain a universal constant $C_3>0$ satisfying 
    \begin{equation*}
        \norm{\epsilon^{\sharp}(t)}_{\dot{H}^{1/2+\delta}}^2
        \leq C_3\alpha^*\lambda^{2-4\delta},
    \end{equation*}
    By reducing $\alpha^*>0$ to be smaller, if necessary, so that 
    \[
        \alpha^* < \frac{1}{2C_3},
    \]
    and shrinking $t_1(\alpha^*)<0$ and $\eta^*(\alpha^*)>0$ to satisfy
    \eqref{choice of t_1(alpha^*) and eta^*(alpha^*) - (1)} and \eqref{choice of t_1(alpha^*) and eta^*(alpha^*) - (2)}, we finish the proof of Lemma~\ref{Main bootstrap}.
\end{proof}

\appendix

\section{On the Modulation equations}
\subsection{Uniqueness of Modulation Parameters}\label{Uniqueness of modulation parameters proof}
In this section, we prove uniqueness of modulation parameters $\{\lambda , \nut{x}, \gamma, b, \nu\}$ with orthogonal conditions on $\epsilon$, \eqref{orthogonality for Lambda Q}, \eqref{orthogonality for nabla Q}, \eqref{orthogonality for partial_eta Q}, \eqref{orthogonality for partial_b Q}, and \eqref{orthogonality for partial_v Q}. We prove this by the Implicit Function Theorem.
\begin{lemma}\label{decomposition lemma}
    There is a constant $\eta^*>0$ possibly smaller than the one chosen in Lemma~\ref{approximated dynamics} such that the following property holds:
    For $\eta \in (0,\eta^*)$, let $u_{\eta} \in \mathcal{C}((T^{(\eta)}_{\text{exist}},0]:H^{1/2+\delta}(\bbR))$ be a solution to \eqref{half-wave} with initial data $u_{\eta}(0)$ defined in \eqref{initial data of u_eta}, where $T^{(\eta)}_{\text{exist}}<0$ is denoted by the maximal lifespan of $u_{\eta}(t)$ backward in time.
    Let $\mathcal{A}^{(\eta)}_{\text{dec}}$ be a subset of $( T^{(\eta)}_{\text{exist}},0]$ defined by 
    \begin{equation}\label{definition of A_{dec}}
    \begin{aligned}
        \mathcal{A}^{(\eta)}_{\text{dec}} \coloneqq \{\tilde{t}<0 \:|\:& \text{For each } t \in (\tilde{t},0], \text{there is a unique tuple of parameters} \\
        & \text{$(\lambda(t),b(t),\nu(t),\nut{x}(t),\gmm(t))$ such that the solution $u_{\eta}(t)$ admits} \\ & \text{a geometrical  decomposition as following:\ }\\
        &u_{\eta}(t,x) = \frac{1}{\lambda^{1/2}(t)}[Q_{\mathcal{P}(b(t),\nu(t),\eta)} + \epsilon]\left(\frac{x-\nut{x}(t)}{\lambda(t)}\right)e^{i\gmm(t)} + z(t,x), \\
        &\text{with orthogonal conditions \eqref{orthogonality for Lambda Q}--\eqref{orthogonality for partial_v Q}.}
        \}
    \end{aligned}
    \end{equation}
    Then, $\mathcal{A}_{\text{dec}}^{(\eta)}$ is not empty.
\end{lemma}
\begin{proof}
    Let $\eta^*>0$ be a constant chosen in Lemma~\ref{approximated dynamics}.
    For each $\eta \in (0,\eta^*)$, define $\Sigma^{(\eta)} : (T^{(\eta)}_{\text{exist}},0] \rightarrow H^{1/2+\delta}$ so that
    \begin{equation}
        \Sigma^{(\eta)}(t,y) \coloneqq \lambda_{\eta}^{1/2}(0)\left(u_{\eta}-z\right)(t,\lambda_{\eta}(0)y+\nut{x}_{\eta}(0))e^{-i\gmm_{\eta}(0)},
    \end{equation}
    where $(\lambda_{\eta},b_{\eta},\nu_{\eta},\nut{x}_{\eta},\gmm_{\eta})(0)$ is initial data chosen in Lemma~\ref{approximated dynamics}. Note that from the choice of initial data $u_{\eta}(0)$, we have
    \[
        \Sigma^{(\eta)}(0,y) = Q_{\mathcal{P}(b_{\eta}(0),\nu_{\eta}(0),\eta)}(y)\eqqcolon Q_{\eta}.
    \]
    In addition, the Cauchy theory of \eqref{half-wave} in Lemma~\ref{Cauchy theory} for $u_{\eta}(t)$ and $z(t)$ implies that $t \mapsto \Sigma^{(\eta)}(t)$ is continuous in $H^{1/2+\delta}$, and hence for each $0<R\ll 1$, there is a constant $T^{(\eta)}_R \in (T^{(\eta)}_{\text{exist}},0)$ such that
    \begin{equation}
        \norm{\Sigma^{(\eta)}(t) - Q_{\eta}}_{H^{1/2}} < R, \text{ for all } t \in [T_{R}^{(\eta)},0].
    \end{equation}

    We now claim that, possibly after reducing $\eta^*>0$ from the one chosen in Lemma~\ref{approximated dynamics}, there exists a constant $0<R_*\ll1$, independent of $\eta\in(0,\eta^*)$, such that the following holds:

    For every $\Sigma \in H^{1/2}$ satisfying $\norm{\Sigma - Q_\eta}_{H^{1/2}} < R_*$, there exists a unique set of parameters $(\tilde{\lambda}, \tilde{y}, \tilde{\gamma}, b, \nu) \in \mathbb{R}_+ \times \mathbb{R}^4$ such that
    \[
        |\tilde{\lambda} - 1| + |\tilde{y}| + |\tilde{\gamma}| + |b| + |\nu - \nu_\eta(0)| < R_*,
    \]
    and the associated error term
    \[
        \epsilon^{(\eta)}_{\tilde{\lambda},\tilde{y},\tilde{\gamma},b,\nu}(y) \coloneqq 
        e^{i\tilde{\gamma}} \tilde{\lambda}^{1/2} \Sigma(\tilde{\lambda} y - \tilde{y}) - Q_{\mathcal{P}(b,\nu,\eta)}
    \]
    satisfies the orthogonality conditions~\eqref{orthogonality for Lambda Q}--\eqref{orthogonality for partial_v Q}.

    To prove this claim, we employ the implicit function theorem.
    Let $0<R\ll 1$ be given a small constant, and let $B_{R}^{(\eta)} \coloneqq \{\Sigma \in H^{1/2} \:|\: \norm{\Sigma - Q_\eta}_{H^{1/2}} < R\}$. Now, we define a map $F^{(\eta)} : B_{R}^{(\eta)} \times \{|\tilde{\lambda} - 1| + |\tilde{y}| + |\tilde{\gamma}| + |b| + |\nu - \nu_\eta(0)| < R\} \mapsto \bbR^5$ so that
    \begin{equation*}
        F^{(\eta)}(\Sigma,\tilde{\lambda},\tilde{y},\tilde{\gmm},b,\nu) = (F_1^{(\eta)},F_2^{(\eta)},F_3^{(\eta)},F_4^{(\eta)},F_5^{(\eta)}),
    \end{equation*}
    where $(F_j^{(\eta)})_{1\leq j\leq 5}$ are defined by
    \begin{align*}
        F_1^{(\eta)}(\Sigma,\tilde{\lambda},\tilde{y},\tilde{\gmm},b,\nu) &\coloneqq (\epsilon_{\tilde{\lambda}, \tilde{y}, \tilde{\gmm}, b,\nu}^{(\eta)},i\Lambda Q_{\mathcal{P}})_r,\\
        F_2^{(\eta)}(\Sigma,\tilde{\lambda},\tilde{y},\tilde{\gmm},b,\nu) &\coloneqq (\epsilon_{\tilde{\lambda}, \tilde{y}, \tilde{\gmm}, b,\nu}^{(\eta)}, i\nabla Q_{\mathcal{P}})_r,\\
        F_3^{(\eta)}(\Sigma,\tilde{\lambda},\tilde{y},\tilde{\gmm},b,\nu) &\coloneqq (\epsilon_{\tilde{\lambda}, \tilde{y}, \tilde{\gmm}, b,\nu}^{(\eta)}, i\partial_\eta Q_{\mathcal{P}})_r,\\
        F_4^{(\eta)}(\Sigma,\tilde{\lambda},\tilde{y},\tilde{\gmm},b,\nu) &\coloneqq (\epsilon_{\tilde{\lambda}, \tilde{y}, \tilde{\gmm}, b,\nu}^{(\eta)}, i\partial_b Q_{\mathcal{P}})_r,\\
        F_5^{(\eta)}(\Sigma,\tilde{\lambda},\tilde{y},\tilde{\gmm},b,\nu) &\coloneqq (\epsilon_{\tilde{\lambda}, \tilde{y}, \tilde{\gmm}, b,\nu}^{(\eta)}, i\partial_\nu Q_{\mathcal{P}})_r.
    \end{align*}
    Note that $F^{(\eta)}(Q_{\eta},1,0,0,0,\nu_{\eta}(0))=0$. Hence, to choose desired $R_{\eta}>0$, it's suffice to show 
    \begin{equation}\label{det of F^eta}
        \text{det}\left(\frac{\partial(F_1^{(\eta)},F_2^{(\eta)}
        ,F_3^{(\eta)},F_4^{(\eta)},F_5^{(\eta)})}{\partial(\tilde{\lambda},\tilde{y},\tilde{\gamma},b,\nu)} 
        \right)\bigg|_{(Q_{\eta},1,0,0,0,\nu_{\eta}(0))} \neq 0,
    \end{equation}
    from the implicit function theorem. Note that $\nu_{\eta}(0) \sim \eta$. Then by taking partial derivatives of $\epsilon_{\tilde{\lambda}, \tilde{y}, \tilde{\gmm}, b,\nu}^{(\eta)}$ at $(\Sigma,\tilde{\lambda},\tilde{y},\tilde{\gmm}, b,\nu) = (Q_{\eta},1,0,0,0,\nu_{\eta}(0))$ yields 
    \begin{align*}
        &\frac{\partial}{\partial \tilde{\lambda}} \epsilon_{\tilde{\lambda}, \tilde{y}, \tilde{\gamma}, b,\nu}^{(\eta)} = \Lambda Q_{\eta}, \quad \frac{\partial}{\partial \tilde{y}} \epsilon_{\tilde{\lambda}, \tilde{y}, \tilde{\gamma}, b,\nu}^{(\eta)} = -\nabla Q_{\eta}, \quad 
        \frac{\partial}{\partial \tilde{\gamma}} \epsilon_{\tilde{\lambda},\tilde{y}, \widetilde{\gamma}, b,\nu}^{(\eta)} = iQ_{\eta},\quad \\
        & \frac{\partial}{\partial b} \epsilon_{\tilde{\lambda}, \tilde{y}, \tilde{\gamma}, b,\nu}^{(\eta)} = -iR_{1,0,0}+ \mathcal{O}(\eta)\langle y \rangle^{-2},\quad \frac{\partial}{\partial \nu} \epsilon_{\tilde{\lambda}, \tilde{y}, \tilde{\gamma}, b,\nu}^{(\eta)} = -iR_{0,1,0}+ \mathcal{O}(\eta)\langle y \rangle^{-2}.
    \end{align*}
    In addition, we have
    \begin{align*}
        i\Lambda Q_{\mathcal{P}(0,\nu_\eta(0),\eta)} &= i\Lambda Q + \mathcal{O}(\eta)\langle y \rangle^{-2},\quad
        i\nabla Q_{\mathcal{P}(0,\nu_\eta(0),\eta)} = i\nabla Q + \mathcal{O}(\eta)\langle y \rangle^{-2},\\
        i\partial_\eta Q_{\mathcal{P}}|_{\mathcal{P}(0,\nu_\eta(0),\eta)} &= iR_{0,0,1} + \mathcal{O}(\eta)\langle y \rangle^{-2},\quad
        i\partial_b Q_{\mathcal{P}}|_{\mathcal{P}(0,\nu_\eta(0),\eta)} = -R_{1,0,0} + \mathcal{O}(\eta)\langle y \rangle^{-2},\\
        i\partial_\nu Q_{\mathcal{P}}|_{\mathcal{P}(0,\nu_{\eta}(0),\eta)} &= - R_{0,1,0} + \mathcal{O}(\eta)\langle y \rangle^{-2}.
    \end{align*}
    Therefore, the derivative of $F^{(\eta)}$ at $(Q_{\eta},1,0,0,0,\nu_{\eta}(0))$ is 
    \begin{align*}
        &\frac{\partial(F_1^{(\eta)},F_2^{(\eta)},F_3^{(\eta)},F_4^{(\eta)},F_5^{(\eta)})}{\partial(\tilde{\lambda},\tilde{y},\tilde{\gamma},b,\nu)} \bigg|_{(\tilde{\lambda},\tilde{y},\tilde{\gamma},b,\nu) = (1,0,0,0,\nu_{\eta}(0))} \\
        & = \left[
        \begin{matrix}
            0 & 0 & 0 & A_{1,4} & 0 \\
            0 & 0 & 0 & 0 &  A_{2,5} \\
            0 & 0 & A_{3,3} & A_{3,4} & 0 \\
            A_{4,1} & 0 & 0 & 0 & 0 \\
            0 & A_{5,2} & 0 & 0 & 0 \\
        \end{matrix}\right] + \mathcal{O}(\eta),
    \end{align*}
    where $A_{i,j}$ denote by
    \begin{align*}
        A_{1,4} &= -(i^{-1}L_Q[iR_{1,0,0}],R_{1,0,0})_r,\\
        A_{2,5} &=  -(i^{-1}L_Q[iR_{0,1,0}],R_{0,1,0})_r,\\
        A_{3,3} &= (Q,R_{0,0,1})_r, \quad A_{3,4} =  -(R_{1,0,0},R_{0,0,1})_r, \\
        A_{4,1} &= -(i^{-1}L_Q[iR_{1,0,0}],R_{1,0,0})_r, \\
        A_{5,2} &= -(i^{-1}L_Q[iR_{0,1,0}],R_{0,1,0})_r.
    \end{align*}
    We use $(i^{-1}L_Q[iR_{0,1,0}],R_{0,1,0})_r>0$, $(i^{-1}L_Q[iR_{1,0,0}],R_{1,0,0})_r>0$ and from the identity $L_Q[\Lambda Q] = -Q$, we have 
    \begin{align*}
        (Q,R_{0,0,1})_r = -(\Lambda Q,L_Q[R_{0,0,1}])_r = -(i^{-1}L_Q[iR_{1,0,0}],R_{1,0,0})_r \neq 0.    
    \end{align*}
    Thus, from the continuity of the Jacobian matrix and the profiles with respect to $\eta$, after possibly reducing $\eta^*>0$, the Jacobian remains uniformly invertible for $\eta\in(0,\eta^*)$. Hence, the implicit-function neighborhood can be chosen uniformly in $\eta$, and we denote its radius by $R_*>0$. 

    To this end, for each $\eta \in (0,\eta^*)$ we claim that $T^{(\eta)}_{R_*} \in \mathcal{A}_{\text{dec}}^{(\eta)}$. Indeed, for each $t \in [T^{(\eta)}_{R_*},0]$, define $\epsilon(t,y)$ to be
    \[
        \epsilon(t,y) \coloneqq e^{i\tilde{\gmm}(t)}\tilde{\lambda}^{1/2}(t)\Sigma^{(\eta)}(\tilde{\lambda}(t)y - \tilde{y}(t)) - Q_{\mathcal{P}(b(t),\nu(t),\eta)},
    \]
    where $(\tilde{\lambda}(t),\tilde{y}(t),\tilde{\gmm}(t),b(t),\nu(t))$ is the modulation parameter corresponding to $\Sigma^{(\eta)}(t)$. Setting
    \[
        \lambda(t)=\lambda_\eta(0)\tilde{\lambda}(t),\qquad
        \nut{x}(t)=\nut{x}_\eta(0)-\lambda_\eta(0)\tilde{y}(t),\qquad
        \gamma(t)=\gamma_\eta(0)-\tilde{\gmm}(t),
    \]
    the definition of $\Sigma^{(\eta)}$ yields the required decomposition of $u_\eta$, and $\epsilon$ satisfies the orthogonal conditions \eqref{orthogonality for Lambda Q}--\eqref{orthogonality for partial_v Q}. Therefore, we conclude that $\mathcal{A}_{\text{dec}}^{(\eta)}$ is not empty.
\end{proof}

\subsection{Sign of the constant $c_4$}\label{Proof of c_4<0}
\begin{lemma}
    The sign of $c_4$ defined in \eqref{value of c_4} is negative.
\end{lemma}
\begin{proof}
Recall that from \eqref{value of c_4} we have
\begin{equation*}
    c_4 = \frac{(A_{\nu^2},2\Lambda Q- R_{1,0,0})_r - (i[\nabla R_{1,1,0} - \tilde{\text{NL}}_{b\nu^2}],iQ)_r}{2(i^{-1}L_Q[iR_{1,0,0}],R_{1,0,0})_r},
\end{equation*}
where 
\begin{align*}
    A_{\nu^2} &= -\nabla R_{0,1,0} - R_{0,1,0}^2Q,\\
    \tilde{\text{NL}}_{b\nu^2} &= 2R_{1,1,0}R_{0,1,0}Q - 3R_{0,1,0}^2R_{1,0,0}.
\end{align*}
We first calculate the term $(i[\nabla R_{1,1,0} - \tilde{\text{NL}}_{b\nu^2}],iQ)_r$. From the definition of $R_{1,1,0}$ in \eqref{order bv} and $R_{0,1,0}$ in \eqref{order v}, we have
\begin{equation}\label{Appendix eq:1}
    \begin{aligned}
        -(i[\nabla R_{1,1,0} &- \tilde{\text{NL}}_{b\nu^2}],iQ)_r \\
        &= (-\nabla R_{1,1,0} + 2R_{1,1,0}R_{0,1,0}Q - 3R_{0,1,0}^2R_{1,0,0},Q)_r \\
        &= -(R_{1,1,0},i^{-1}L_Q[iR_{0,1,0}])_r + (R_{0,1,0},2Q^2R_{1,1,0})_r - 3(R_{0,1,0}^2R_{1,0,0},Q)_r\\
        &= -(R_{0,1,0},L_Q[R_{1,1,0}])_r - 3(R_{0,1,0}^2R_{1,0,0},Q)_r\\
        &= \norm{R_{0,1,0}}_{L^2}^2 + (R_{0,1,0},\nabla R_{1,0,0})_r - (R_{0,1,0}^2Q,R_{1,0,0})_r.
    \end{aligned}
\end{equation}
Now we calculate the term $(A_{\nu^2},2\Lambda Q- R_{1,0,0})_r$. 
\begin{equation}\label{Appendix eq:2}
    \begin{aligned}
        (A_{\nu^2},&2\Lambda Q- R_{1,0,0})_r \\
        &= -2(\nabla R_{0,1,0}+R_{0,1,0}^2Q,\Lambda Q)_r + (\nabla R_{0,1,0}+R_{0,1,0}^2Q,R_{1,0,0})_r \\
    \end{aligned}
\end{equation}
Therefore, from \eqref{Appendix eq:1}, \eqref{Appendix eq:2}, and the commutator formula $[\Lambda , \nabla] = -\nabla$, the numerator in the right hand side of \eqref{value of c_4} is 
\begin{equation}
    \begin{aligned}
        (&\text{numerator in (RHS) of \eqref{value of c_4}})\\
        &= -2(\nabla R_{0,1,0},\Lambda Q)_r -2(R_{0,1,0}^2,Q\Lambda Q)_r + \norm{R_{0,1,0}}_{L^2}^2\\
        &= 2(\nabla(\Lambda R_{0,1,0}),Q)_r + 2([\Lambda ,\nabla]R_{0,1,0},Q)_r -2(R_{0,1,0}^2,Q\Lambda Q)_r + \norm{R_{0,1,0}}_{L^2}^2\\
        &= -2(i^{-1}L_Q[iR_{0,1,0}],R_{0,1,0})_r + 2(\Lambda R_{0,1,0},i^{-1}L_Q[iR_{0,1,0}])_r - 2(R_{0,1,0}^2,Q\Lambda Q)_r + \norm{R_{0,1,0}}_{L^2}^2.
    \end{aligned}
\end{equation}
Note from commutator formula $[D,\Lambda]=D$ and point-wise identity
\begin{equation*}
    -(y\cdot \nabla Q) Q + Q\Lambda Q = \frac{1}{2}Q^2,
\end{equation*}
we have
\begin{equation}
    \begin{aligned}
        2(\Lambda R_{0,1,0},&i^{-1}L_Q[iR_{0,1,0}])_r \\
        &= (R_{0,1,0},[i^{-1}L_Qi , \Lambda] R_{0,1,0})_r \\
        &= (R_{0,1,0},[D,\Lambda]R_{0,1,0})_r - (R_{0,1,0},[Q^2,\Lambda]R_{0,1,0})_r\\
        &= (DR_{0,1,0},R_{0,1,0})_r + 2((y\cdot \nabla Q) QR_{0,1,0},R_{0,1,0})_r\\
        &= (i^{-1}L_Q[iR_{0,1,0}],R_{0,1,0})_r + 2(R_{0,1,0}^2,Q\Lambda Q)_r - \norm{R_{0,1,0}}_{L^2}^2.
    \end{aligned}
\end{equation}
Thus, the value of $c_4$ is
\begin{equation*}
    c_4 = -\frac{(i^{-1}L_Q[iR_{0,1,0}],R_{0,1,0})_r}{2(i^{-1}L_Q[iR_{1,0,0}],R_{1,0,0})_r} < 0,
\end{equation*}
by noting that $(i^{-1}L_Q[iR_{0,1,0}],R_{0,1,0})_r>0$ and $(i^{-1}L_Q[iR_{1,0,0}],R_{1,0,0})_r>0$.
\end{proof}

 \bibliographystyle{abbrv}
\bibliography{reference}

\end{document}